\pdfoutput=1

\documentclass[letterpaper,12pt,titlepage,final,twoside]{book}

\usepackage[utf8]{inputenc}
\usepackage[T1]{fontenc}

\usepackage{amsmath,amssymb,amstext} 
\usepackage[pdftex]{graphicx} 

\usepackage{etoolbox}
\newtoggle{thesis}
\togglefalse{thesis}

\usepackage{algorithm}
\usepackage{algpseudocode}

\usepackage[
    backend=biber,
    natbib=true,
    giveninits=true,
    style=alphabetic,
    doi=false,
    isbn=false,
    url=false,
    sorting=ynt,
    sortcites,
    maxnames=1000,
    minnames=1000
]{biblatex}
\bibliography{uw-ethesis}

\usepackage[Rejne]{fncychap}

\usepackage{xcolor}
\definecolor{colA}{HTML}{2277AA}
\definecolor{colB}{HTML}{008844}
\definecolor{colC}{HTML}{AA0000}
\colorlet{linkcolor}{colA}
\colorlet{citecolor}{colB}
\colorlet{urlcolor}{colC}
\colorlet{ctrlcolor}{colC}

\usepackage{verbatim} 

\usepackage{paralist}
\usepackage{enumitem}

\makeatletter
\renewcommand{\fnum@figure}{Fig. \thefigure}
\makeatother
\usepackage{wrapfig}
\usepackage[labelfont=it]{caption}
\usepackage{subfig}

\usepackage{lmodern}
\usepackage{microtype}
\usepackage{mathrsfs}
\usepackage{nicefrac}
\renewcommand*{\geq}{\geqslant}
\renewcommand*{\leq}{\leqslant}

\usepackage[bottom]{footmisc}

\usepackage{mathtools}

\usepackage[md]{titlesec}

\usepackage{etoolbox}
\preto\subequations{\ifhmode\unskip\fi}

\usepackage{booktabs}

\newcommand{\textspace}{\quad}

\usepackage{tikz}
\usetikzlibrary{arrows,automata,calc,intersections,patterns,positioning,shapes,decorations.pathreplacing}
\tikzset{>=latex}
\pgfmathdeclarefunction{invgauss}{2}{%
    \pgfmathparse{sqrt(-2*ln(#1))*cos(deg(2*pi*#2))}%
}
\pgfmathsetmacro\gratio{(1+sqrt(5))/2}
\newcommand{\drawArithmeticBrownianMotionPath}[5]{
    \draw [#1] (0, 0) node (ArithmeticBrownianMotionStart_#2) {}
    \foreach \x in {1, ..., #5}
    { -- ++( 1 / #5, { #3 / #5 + #4 * invgauss(rnd, rnd) * sqrt(1 / #5) } ) }
    node (ArithmeticBrownianMotionEnd_#2) {};
}

\usepackage{amsthm}
\usepackage{aliascnt}

\numberwithin{dummy}{section}
\newcommand{\makeNewTheorem}[4]{
    \newaliascnt{#3}{dummy}
    \newtheorem{#1}[dummy]{#2}
}

\usepackage{todonotes}

\iftoggle{thesis}{%
\usepackage[pdftex,pagebackref=false]{hyperref}
}{
\usepackage{hyperref}
}

\usepackage[xindy,toc,nonumberlist]{glossaries}
\newglossaryentry{wiener}
{
    name={\ensuremath{B}},
    description={Brownian motion},
    sort=B
}

\newglossaryentry{interp}
{
    name={\ensuremath{\operatorname{interp}}},
    description={monotone linear interpolation},
    sort=interp
}

\newglossaryentry{diag}
{
    name={\ensuremath{\operatorname{diag}}},
    description={map from vector to diagonal matrix},
    sort=diag
}

\newglossaryentry{meshing_parameter}
{
    name={\ensuremath{h}},
    description={meshing parameter},
    sort=h,
}

\newglossaryentry{wcdd}
{
    name={w.c.d.d.},
    description={weakly chained diagonally dominant}
}

\newglossaryentry{MDP}
{
    name={MDP},
    description={Markov decision process},
    plural={MDPs}
}

\newglossaryentry{proj_i}
{
    name={\ensuremath{\operatorname{proj}_i}},
    description={$i$-th projection map},
    sort=proj
}

\newglossaryentry{sdd}
{
    name={s.d.d.},
    description={strictly diagonally dominant}
}

\newglossaryentry{wdd}
{
    name={w.d.d.},
    description={weakly diagonally dominant}
}

\newglossaryentry{const}
{
    name={\ensuremath{\operatorname{const.}}},
    description={positive constant that may vary from line to line},
    sort=const.
}

\newglossaryentry{FEX}
{
    name={FEX},
    description={foreign exchange}
}

\newglossaryentry{GMWB}
{
    name={GMWB},
    description={guaranteed minimum withdrawal benefit},
    plural={GMWBs}
}

\newglossaryentry{GLWB}
{
    name={GLWB},
    description={guaranteed lifelong withdrawal benefit},
    plural={GLWBs}
}

\newglossaryentry{bdd}{
    name={\ensuremath{B(\overline{\Omega})}},
    description={set of bounded real-valued maps from $\overline{\Omega}$},
    sort={Bdd}
}

\newglossaryentry{trace}{
    name={\ensuremath{\operatorname{trace}}},
    description={trace of a matrix},
    sort={trace}
}

\newglossaryentry{extreme_points}{
    name={\ensuremath{\operatorname{vert}}},
    description={vertices of a polytope},
    sort={vert}
}

\newglossaryentry{domain}{
    name={\ensuremath{\mathfrak{D}}},
    description={domain of stochastic process},
    sort={D}
}

\newglossaryentry{hausdorff}{
    name={\ensuremath{d_{H}}},
    description={Hausdorff metric},
    sort={dH}
}

\usepackage[capitalise]{cleveref}
\crefformat{enumi}{#2#1#3}
\crefrangeformat{enumi}{#3#1#4\textendash #5#2#6}
\crefmultiformat{enumi}{#2#1#3}{ and~#2#1#3}{, #2#1#3}{, and~#2#1#3}
\crefformat{equation}{(#2#1#3)}
\crefrangeformat{equation}{(#3#1#4) to (#5#2#6)}
\crefmultiformat{equation}{(#2#1#3)}{ and~(#2#1#3)}{, (#2#1#3)}{, and~(#2#1#3)}

\usepackage{ifthen}
\newboolean{PrintVersion}
\setboolean{PrintVersion}{false}

\iftoggle{thesis}{%
\hypersetup{
    plainpages=false,       
    unicode=false,          
    pdftoolbar=true,        
    pdfmenubar=true,        
    pdffitwindow=false,     
    pdfstartview={FitH},    
    pdftitle={Impulse Control in Finance: Numerical Methods and Viscosity Solutions},    
    pdfauthor={Parsiad Azimzadeh},    
    pdfnewwindow=true,      
    colorlinks=true,        
    linkcolor=linkcolor,
    citecolor=citecolor,
    filecolor=filecolor,
    urlcolor=urlcolor
}
\ifthenelse{\boolean{PrintVersion}}{   
\hypersetup{	
    citecolor=black,%
    filecolor=black,%
    linkcolor=black,%
    urlcolor=black}
}{} 
}{
\hypersetup{
    pdftitle={Impulse Control in Finance: Numerical Methods and Viscosity Solutions},    
    pdfauthor={Parsiad Azimzadeh},    
    colorlinks=true,        
    linkcolor=linkcolor,
    citecolor=citecolor,
    filecolor=filecolor,
    urlcolor=urlcolor
}
}

\makeNewTheorem{definition}{Definition}{definition}{colA}
\makeNewTheorem{theorem}{Theorem}{theorem}{colB}
\makeNewTheorem{lemma}{Lemma}{lemma}{colB}
\makeNewTheorem{corollary}{Corollary}{corollary}{colB}
\makeNewTheorem{example}{Example}{example}{colC}
\makeNewTheorem{remark}{Remark}{remark}{black}
\makeNewTheorem{proposition}{Proposition}{proposition}{colC}

\newcommand{\introductionHJBQVIInterior}{
    \min \left\{
        - V_t - \sup_{\textcolor{ctrlcolor}{w} \in W} \left\{
            \frac{1}{2} b(\cdot, \textcolor{ctrlcolor}{w})^2 V_{xx}
            {+} a(\cdot, \textcolor{ctrlcolor}{w}) V_x
            {+} f(\cdot, \textcolor{ctrlcolor}{w})
        \right\},
        V - \mathcal{M}V
    \right\} & = 0 \text{ on } [0,T) \times \gls*{domain}
}
\newcommand{\introductionHJBQVIBoundary}{
    \min \left \{
        V(T, \cdot) - g,
        V(T, \cdot) - \mathcal{M} V(T, \cdot)
    \right \}
    & = 0 \text{ on } \gls*{domain}
}

\newcommand{\introductionIntervention}{
    \mathcal{M}V(t,x) = \sup_{\textcolor{ctrlcolor}{z}\in Z(t,x)} \left\{
        V(
            t,
            \Gamma(t, x, \textcolor{ctrlcolor}{z})
        )
        + K(t, x, \textcolor{ctrlcolor}{z})
    \right\}
}

\newcommand{\resultsFEXIntervention}{
    \mathcal{M} V(t,x) =
    \sup_{ \textcolor{ctrlcolor}{z} \in \mathbb{R} }
    \left \{
        V(t, x + \textcolor{ctrlcolor}{z}) - e^{-\beta t} \left(
            \kappa \left| \textcolor{ctrlcolor}{z} \right| + c
        \right)
    \right \}.
}

\newcommand{\schemesDiscretizedIntervention}{
    (\mathcal{M}_n U)_i =
    \sup_{\textcolor{ctrlcolor}{z_i} \in Z^{\gls*{meshing_parameter}}(\tau^n, x_i)}
    \left\{
        \gls*{interp}(
            U,
            \Gamma(\tau^n, x_i, \textcolor{ctrlcolor}{z_i})
        ) + K(\tau^n, x_i, \textcolor{ctrlcolor}{z_i})
    \right\}
}

\newcommand{\partitionFootnote}{\footnote{A partition of an interval $[a,b]$ is understood to be a set of points $\{ x_0, \ldots, x_M \}$ satisfying $a = x_0 < x_1 < \cdots < x_M = b$.}}

\newcommand{\variableAnnuitiesDescription}{variable annuity is a contract between an individual and an insurance company that provides a guaranteed stream of cash flows.
The contract is bootstrapped by an up-front premium from the individual to the insurance company, which is immediately invested in a diversified portfolio (e.g., an exchange-traded or mutual fund).
While the contract guarantees a minimum stream of cash flows, these flows may exceed the guaranteed amount if the portfolio performs well.}

\newcommand{\artificial}{artificial}
\newcommand{\Artificial}{Artificial}

\newcommand{\wideEll}{{\scalebox{1.2}[1]{$\ell$}}}

\let\cleardoublepage\clearemptydoublepage

\assignrefcontextentries[]{*}
\begin{document}

\pagestyle{empty}
\pagenumbering{roman}

\begin{titlepage}
        \begin{center}
        \vspace*{1.0cm}

        \Huge
        {Impulse Control in Finance: \\ Numerical Methods and Viscosity Solutions}

        \vspace*{1.0cm}

        \normalsize
        by \\

        \vspace*{1.0cm}

        \Large
        Parsiad Azimzadeh \\

        \vspace*{3.0cm}

        \normalsize
        A thesis \\
        presented to the University of Waterloo \\
        in fulfillment of the \\
        thesis requirement for the degree of \\
        Doctor of Philosophy \\
        in \\
        Computer Science \\

        \vspace*{2.0cm}

        Waterloo, Ontario, Canada, 2017 \\

        \vspace*{1.0cm}

        \copyright\ Parsiad Azimzadeh 2017 \\
        \end{center}
\end{titlepage}

\pagestyle{plain}
\setcounter{page}{2}

\clearpage 


\begin{center}\textbf{Examining Committee Membership}\end{center}

  \noindent
The following served on the Examining Committee for this thesis. The decision of the Examining Committee is by majority vote.

  \bigskip

  \noindent External Examiner \hfill
  \begin{tabular}[t]{r@{}}
    Catherine Rainer \\
    Universit\'{e} de Brest
  \end{tabular}

  \bigskip

  \noindent Supervisor \hfill
  \begin{tabular}[t]{r@{}}
    George Labahn \\
    Professor
  \end{tabular}

  \bigskip

  \noindent Internal Member \hfill
  \begin{tabular}[t]{r@{}}
    Yuying Li \\
    Professor
  \end{tabular}

  \bigskip

  \noindent Internal Member \hfill
  \begin{tabular}[t]{r@{}}
    Justin Wan \\
    Associate Professor
  \end{tabular}

  \bigskip

  \noindent Internal-External Member \hfill
  \begin{tabular}[t]{r@{}}
    David Saunders \\
    Associate Professor
  \end{tabular}
\clearpage

  \noindent
I hereby declare that I am the sole author of this thesis. This is a true copy of the thesis, including any required final revisions, as accepted by my examiners.

  \bigskip

  \noindent
I understand that my thesis may be made electronically available to the public.

\clearpage


\begin{center}\textbf{Abstract}\end{center}

The goal of this thesis is to provide efficient and provably convergent numerical methods for solving partial differential equations (PDEs) coming from \emph{impulse control problems} motivated by finance.
Impulses, which are controlled jumps in a stochastic process, are used to model realistic features in financial problems which cannot be captured by ordinary stochastic controls.
\iftoggle{thesis}{%
In this thesis, we consider two distinct cases of impulse control: one in which impulses can occur at any time and one in which they occur only at ``fixed'' (i.e., nonrandom and noncontrollable) times.

The first case is used to model features in finance such as fixed transaction costs, liquidity risk, execution delay, etc.
In this case, the corresponding PDEs are Hamilton-Jacobi-Bellman quasi-variational inequalities (HJBQVIs).
}{%

The dynamic programming equations associated with impulse control problems are Hamilton-Jacobi-Bellman quasi-variational inequalities (HJBQVIs)
}
Other than in certain special cases, the numerical schemes that come from the discretization of HJBQVIs take the form of complicated nonlinear matrix equations also known as \emph{Bellman problems}.
We prove that a policy iteration algorithm can be used to compute their solutions.
In order to do so, we employ the theory of weakly chained diagonally dominant (\gls*{wcdd}) matrices.
As a byproduct of our analysis, we obtain some new results regarding a particular family of Markov decision processes which can be thought of as impulse control problems on a discrete state space and the relationship between \gls*{wcdd} matrices and M-matrices.
Since HJBQVIs are nonlocal PDEs, we are unable to directly use the seminal result of Barles and Souganidis (concerning the convergence of \emph{monotone}, \emph{stable}, and \emph{consistent} numerical schemes to the \emph{viscosity solution}) to prove the convergence of our schemes.
We address this issue by extending the work of Barles and Souganidis to nonlocal PDEs in a manner general enough to apply to HJBQVIs.
We apply our schemes to compute the solutions of various classical problems from finance concerning optimal control of the exchange rate, optimal consumption with fixed and proportional transaction costs, and guaranteed minimum withdrawal benefits in variable annuities.

\iftoggle{thesis}{%
The second case of impulse control, involving impulses occurring at fixed times, is frequently used in pricing and hedging insurance contracts.
In this case, the impulses correspond to regular anniversaries (e.g., monthly, yearly, etc.) at which the holder of the contract can perform certain actions (e.g., lapse the contract).
The corresponding pricing equations are a sequence of linear PDEs coupled by nonlinear constraints corresponding to the impulses.
For these problems, our focus is on speeding up the computation associated with the nonlinear constraints by means of a control reduction.
We apply our results to price guaranteed lifelong withdrawal benefits in variable annuities.
}{}

\clearpage


\begin{center}\textbf{Acknowledgements}\end{center}

I would like to thank:

\begin{itemize}
    \itemsep1.3em
    \item[]
        My supervisor, George Labahn, for his support and belief in me.
    \item[]
        Peter Forsyth for introducing me to computational finance.
    \item[]
        Various faculty members at the University of Waterloo and abroad, with special mention to Edward Vrscay for being a positive influence during my doctorate.
    \item[]
        The University of Waterloo Scientific Computing (a.k.a. ``SciCom'') affiliated faculty and lab members (past and present) for the fond memories.
        Special mention goes to Edward Cheung for keeping me sane during the preparation of this document!
    \item[]
        The internal committee members -- Yuying Li, David Saunders, and Justin Wan -- for their time and role in the defense.
        Special mention goes to the external examiner, Catherine Rainer, for her insightful comments and for her wonderful body of work on stochastic differential games.
    \item[]
        Erhan Bayraktar for providing me with fruitful ideas and exciting future employment at the University of Michigan.
    \item[]
        All of my family and friends for their support, with special mention to my mother and father (for the obvious reasons!)
\end{itemize}
\clearpage


\begin{center}\textbf{Dedication}\end{center}

\begin{center}To my parents.\end{center}
\clearpage

\newpage
\phantomsection
\renewcommand\contentsname{Table of Contents}
\begingroup
\let\cleardoublepage\clearpage
\tableofcontents
\endgroup
\clearpage
\phantomsection

\begingroup
\let\cleardoublepage\clearpage
\listoftables
\endgroup
\clearpage
\phantomsection		

\begingroup
\let\cleardoublepage\clearpage
\listoffigures
\endgroup
\clearpage
\phantomsection		



\pagenumbering{arabic}
\pagestyle{headings}

\setcounter{chapter}{0}
\chapter{Introduction}
\label{chap:introduction}

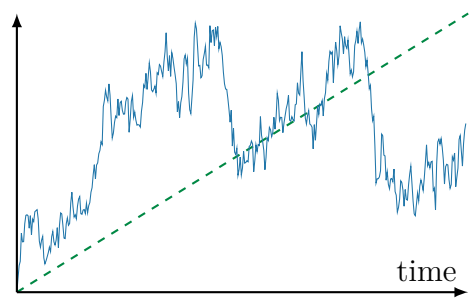
\begin{wrapfigure}{r}{0.4\textwidth}
    \centering
    \begin{tikzpicture}[scale=6]
        \pgfmathsetseed{10}
        \draw [colB,thick,dashed] (0,0) -- (1,1/\gratio);
        \drawArithmeticBrownianMotionPath{colA}{A}{1/\gratio}{1/\gratio}{400}
        \draw [thick, ->] (0,0) -- (1,0) node [above left] {time};
        \draw [thick, ->] (0,0) -- (0,1/\gratio);
    \end{tikzpicture}
    \caption{Sample path of process driven by Brownian motion plotted alongside its mean}
\end{wrapfigure}

Financial problems and stochastic processes have a long, symbiotic history.
The first known use of Brownian motion to evaluate stock options appears in Louis Bachelier's doctoral thesis, ``Th\'{e}orie de la sp\'{e}culation'', defended in 1900 \cite{MR1397712}.
In this document, Bachelier reasoned that increments of stock prices should be independent and normally distributed.
Though a rigorous construction of Brownian motion only appeared in 1923 due to Wiener, Bachelier was able to (at least intuitively) connect the price of an option to a heat equation, which he solved using a Green's function.

In 1931, Kolmogorov proved that a \emph{continuous} random process not depending on its past information (i.e., a Markov process) is wholly characterized by two parameters: the drift, describing the speed of its deterministic evolution, and the diffusion, describing the speed of its random evolution \cite{MR1153022}.
In 1944, It\^{o} introduced a particular class of such processes described by the stochastic differential equation (SDE) \cite{MR0014633}
\begin{equation}
    dX_t = \underbrace{a(X_t)}_{\text{drift}} dt + \underbrace{b(X_t)}_{\text{diffusion}} d\gls*{wiener}_t
\label{eqn:introduction_sde}
\end{equation}
where $\gls*{wiener}$ is a Brownian motion.
Given a twice differentiable function $f$, It\^{o} characterized the rate of change of $f(X_t)$ by means of what is now known as It\^{o}'s formula \cite{MR0044063}:
\[
    df(X_t) = f^\prime(X_t) dX_t + \frac{1}{2} f^{\prime\prime}(X_t) dX_t^2.
\]

In 1965, Samuelson, having rediscovered the work of Bachelier, published two revolutionary works \cite{samuelson1965rational,samuelson1965proof}.
The first gives economic basis for Bachelier's belief that prices fluctuate randomly and suggests that discounted future prices should follow a martingale, laying out the theory of rational option pricing.
The second suggests that a good model for options pricing is geometric Brownian motion (i.e., $a(x) = r x$ and $b(x) = \sigma x$ in \cref{eqn:introduction_sde}).
Therein, he also derived prices for European and American options.
However, these ideas only gained widespread popularity when Black and Scholes gave derivations using It\^{o}'s formula along with replicating portfolio arguments to show that the value of a European option is obtained by solving the partial differential equation (PDE) \cite{MR3363443}
\[
    V_t + \frac{1}{2} \sigma^2 x^2 V_{xx} + r x V_x - r V = 0,
\]
now known as the Black-Scholes PDE. The above is a linear (degenerate) parabolic PDE, and as such, is well-understood.

As the field of finance matured, there came a need to solve problems more complicated than European options pricing. 
As an example, the price of a European option when the volatility is uncertain but is known to lie within the interval $[\sigma_{\min}, \sigma_{\max}]$ is obtained by solving the nonlinear PDE \cite{avellaneda1995pricing,lyons1995uncertain}
\begin{equation}
    V_t + \sup_{\sigma \in [\sigma_{\min}, \sigma_{\max}]} \left\{
        \frac{1}{2} \sigma^2 x^2 V_{xx} + r x V_x - r V
    \right\} = 0.
    \label{eqn:introduction_uncertain}
\end{equation}

Since nonlinear PDEs do not generally admit smooth solutions, a notion of weak solution is required to study them.
The appropriate notion of weak solution in the setting of finance (and more generally, optimal control) is that of a viscosity solution, introduced for first order equations by Crandall and Lions \cite{MR690039} and extended to second order equations by Jensen \cite{MR920674}.
Following these works, Barles and Souganidis gave very general criteria for a numerical method to converge to the viscosity solution of a second order equation \cite{MR1115933}.
This provided a theoretical backing for employing numerical PDE methods to compute solutions of complicated financial problems such as \cref{eqn:introduction_uncertain}.
This brings us to the goal of this thesis: to study numerically one such class of problems, impulse control problems.

\section{Motivation and contributions}

The goal of this thesis is to provide efficient and provably convergent numerical methods for solving PDEs coming from impulse control problems motivated by finance.
Impulse control problems involve a stochastic process $X$ with dynamics 
\begin{subequations}
    \begin{align}
        dX_t & = a(t, X_t, \textcolor{ctrlcolor}{w_t}) dt + b(t, X_t, \textcolor{ctrlcolor}{w_t}) d\gls*{wiener}_t
        & \text{for } t \neq \xi_1, \xi_2, \ldots
        \label{eqn:introduction_controlled_sde_between} \\
        X_{\textcolor{ctrlcolor}{\xi_\ell}} & = 
        \Gamma(\textcolor{ctrlcolor}{\xi_\ell}, X_{\textcolor{ctrlcolor}{\xi_\ell}-}, \textcolor{ctrlcolor}{z_\ell})
        & \text{for } \ell = \phantom{\xi_1} \mathllap{1}, \phantom{\xi_2} \mathllap{2}, \ldots
        \label{eqn:introduction_controlled_sde_across}
    \end{align}
    \label{eqn:introduction_controlled_sde}%
\end{subequations}
where $X_{t-} = \lim_{s\uparrow t} X_s$ is a limit from the left and $\textcolor{ctrlcolor}{\theta} = (\textcolor{ctrlcolor}{w}; \textcolor{ctrlcolor}{\xi_{1}}, \textcolor{ctrlcolor}{z_{1}}; \textcolor{ctrlcolor}{\xi_{2}}, \textcolor{ctrlcolor}{z_{2}}; \ldots)$ is a so-called combined stochastic and impulse control.
Aside from influencing the process $X$ continuously via the stochastic control $\textcolor{ctrlcolor}{w}$, the controller is also able to induce jumps in the process by choosing ``intervention times'' $\textcolor{ctrlcolor}{\xi_{\ell}}$ and ``impulses'' $\textcolor{ctrlcolor}{z_{\ell}}$ (\cref{fig:introduction_controlled_sde}).
The size and direction of these jumps is determined by the function $\Gamma$.

\begin{remark}
    In this thesis, probabilistic ``objects'' such as the process $X$ are used only to motivate, heuristically, several PDE problems (the rigorous link can be established by various arguments; see the discussion in \cref{sec:convergence_viscosity_solution}).
    As such, we ignore issues related to the existence of solutions $X$ to the SDE \cref{eqn:introduction_controlled_sde}. Nevertheless, it is understood that $\gls*{wiener}$ is a Brownian motion generating the (completed and right-continuous) filtration $(\mathcal{F}_t)_{t \geq 0}$, $\textcolor{ctrlcolor}{w}$ is progressively measurable, $(\textcolor{ctrlcolor}{\xi_\ell})_\ell$ is an increasing sequence of stopping times, and each $\textcolor{ctrlcolor}{z_\ell}$ is an $\mathcal{F}_{\textcolor{ctrlcolor}{\xi_\ell}}$ measurable random variable.
\end{remark}

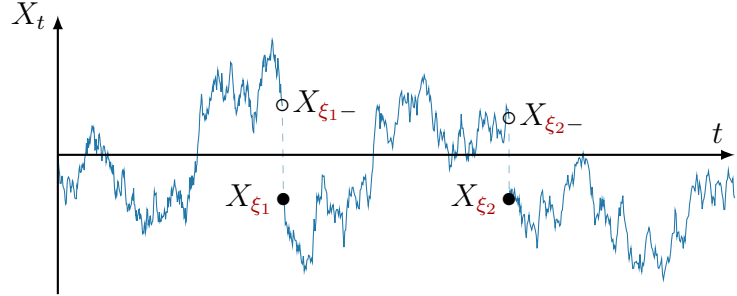
\begin{figure}
    \centering
    \begin{tikzpicture}[scale=3]
        \pgfmathsetseed{32}
        \drawArithmeticBrownianMotionPath{colA}{A}{1/\gratio}{1/\gratio}{400}
        \drawArithmeticBrownianMotionPath{colA,xshift=1cm,yshift=-0.2cm}{B}{1/\gratio}{1/\gratio}{250}
        \drawArithmeticBrownianMotionPath{colA,xshift=2cm,yshift=-0.2cm}{C}{1/\gratio}{1/\gratio}{250}
        \draw [dashed, colA!50] (ArithmeticBrownianMotionStart_B) -- (ArithmeticBrownianMotionEnd_A);
        \draw [dashed, colA!50] (ArithmeticBrownianMotionStart_C) -- (ArithmeticBrownianMotionEnd_B);
        \node at (ArithmeticBrownianMotionEnd_A) {$\circ$};
        \node at (ArithmeticBrownianMotionStart_B) {\textbullet};
        \node at (ArithmeticBrownianMotionEnd_B) {$\circ$};
        \node at (ArithmeticBrownianMotionStart_C) {\textbullet};
        \node [right=of ArithmeticBrownianMotionEnd_A,xshift=-1.15cm] {$X_{\textcolor{ctrlcolor}{\xi_1}-}$};
        \node [left=of ArithmeticBrownianMotionStart_B,xshift=1.15cm] {$X_{\textcolor{ctrlcolor}{\xi_1}}$};
        \node [right=of ArithmeticBrownianMotionEnd_B,xshift=-1.15cm] {$X_{\textcolor{ctrlcolor}{\xi_2}-}$};
        \node [left=of ArithmeticBrownianMotionStart_C,xshift=1.15cm] {$X_{\textcolor{ctrlcolor}{\xi_2}}$};
        \draw [thick, ->] (0,0) -- (3,0) node [above left] {$t$};
        \draw [thick, ->] (0,-1/\gratio) -- (0,1/\gratio) node [left] {$X_t$};
    \end{tikzpicture}
    \caption{Sample path of stochastic process with impulses}
    \label{fig:introduction_controlled_sde}
\end{figure}

Our interest in the SDE \cref{eqn:introduction_controlled_sde} in lieu of the simpler \cref{eqn:introduction_sde} lies in its ability to capture realistic features in financial decision-making including, for example, fixed transaction costs, liquidity risk, or execution delay.
Problems where such features appear include optimal consumption, optimal portfolio selection, hedging and pricing insurance contracts, and optimal liquidation \cite{MR1897195,MR2284012,MR2407322,MR2513141,MR2745264,MR3145856,MR3163980,MR3464413}.
Of course, this list is far from exhaustive.

\subsection{Hamilton-Jacobi-Bellman quasi-variational inequalities}

The general form of an optimal control problem involving impulse control is to maximize, over all controls $\textcolor{ctrlcolor}{\theta}$, the quantity
\begin{equation}
    J(t, x; \textcolor{ctrlcolor}{\theta}) = \mathbb{E} \left[
        \int_t^T f(u, X_u, \textcolor{ctrlcolor}{w_u}) du
        + \sum_{t \leq \textcolor{ctrlcolor}{\xi_\ell} \leq T} K(
            \textcolor{ctrlcolor}{\xi_\ell},
            X_{\textcolor{ctrlcolor}{\xi_\ell}-},
            \textcolor{ctrlcolor}{z_\ell}
        )
        + g(X_T)
        \middle | X_{t-} = x
    \right].
    \label{eqn:introduction_functional}
\end{equation}
Here, the terms $\int f$, $\sum K$, and $g$ represent cash flows obtained continuously, from impulses, and at the terminal time, respectively.
Directly computing $V(t,x) = \sup_{\textcolor{ctrlcolor}{\theta}} J(t,x;\textcolor{ctrlcolor}{\theta})$ is not an easy, nor even feasible, task.
As such, we use an equivalent PDE formulation of the problem which can then be solved via numerical methods.

If we assume that $\textcolor{ctrlcolor}{w_t}$ and $\textcolor{ctrlcolor}{z_{\ell}}$ take values in the control sets $W$ and $Z(\textcolor{ctrlcolor}{\xi_{\ell}},X_{\textcolor{ctrlcolor}{\xi_{\ell}}-})$ respectively, then using dynamic programming arguments \cite{MR2568293}, we obtain the PDE (restricting our attention to the case in which $X$ is a one-dimensional process for simplicity)
\begin{subequations}
    \begin{align}
        \introductionHJBQVIInterior
        \label{eqn:introduction_hjbqvi_interior} \\
        \introductionHJBQVIBoundary
        \label{eqn:introduction_hjbqvi_boundary}
    \end{align}
    \label{eqn:introduction_hjbqvi}%
\end{subequations}
where
\begin{equation}
    \introductionIntervention
    \label{eqn:introduction_intervention}
\end{equation}
is a so-called ``intervention operator'' and $\gls*{domain}$ is a closed and connected subset of $\mathbb{R}$ determined by the support of the process $X$. 
The PDE \cref{eqn:introduction_hjbqvi} is called a Hamilton-Jacobi-Bellman quasi-variational inequality\footnote{While the term ``inequality'' may seem out of place in describing the PDE \cref{eqn:introduction_hjbqvi}, we point out that the equation $\min \{ a, b \} = 0$ is equivalent to $0 \leq a \perp b \geq 0$.} (HJBQVI).

We consider three numerical schemes for tackling the HJBQVI \cref{eqn:introduction_hjbqvi}: the \emph{direct control}, \emph{penalty}, and \emph{explicit-impulse} schemes.
The first two schemes require the solution of complicated nonlinear matrix equations.
We can use policy iteration to solve these equations.
However, in the case of the direct control scheme, the corresponding nonlinear matrix equations involve singular matrices, and as a result, convergence of policy iteration is not immediate.
In our work, we solve the convergence problem by making use of the theory of weakly chained diagonally dominant (\gls*{wcdd}) matrices.
As a byproduct of our analysis, we obtain some new results regarding a particular family of Markov decision processes which can be thought of as impulse control problems on a discrete state space and the relationship between \gls*{wcdd} matrices and M-matrices.
On the other hand, the explicit-impulse scheme involves solving a linear system at each timestep and hence no policy iteration is necessary.
However, it can only be used when the diffusion coefficient $b$ does not depend on the control (i.e., $b(t,x,\textcolor{ctrlcolor}{w})=b(t,x)$) and the time horizon is finite (i.e., $T < \infty$).

Since our PDEs are nonlinear, we cannot expect to obtain unique smooth solutions.
Rather, we seek \emph{viscosity solutions}, the relevant notion of weak solution for optimal control.
The typical method for showing convergence of a numerical scheme to a viscosity solution is via the classical work of Barles and Souganidis which proves the convergence of a monotone, stable, and consistent numerical scheme to the unique viscosity solution of a \emph{local} PDE satisfying a comparison principle \cite{MR1115933}.
However, the intervention operator $\mathcal{M}$ appearing in the HJBQVI is \emph{nonlocal}: the value of $\mathcal{M} V(t, x)$ depends on the value of $V$ at points other than $(t, x)$ (see \cref{eqn:introduction_intervention}).
Therefore, the results of \cite{MR1115933} do not immediately apply.
Instead, we extend the work of Barles and Souganidis to allow for such nonlocal operators, with the main feature being the introduction of the notion of nonlocal consistency.
We are then able to prove convergence of our three schemes to the viscosity solution  by showing that they are all monotone, stable, and nonlocally consistent.
In order to illustrate convergence and compare the schemes, we apply them to obtain numerical solutions of various classical problems from finance.

\subsection{Example: optimal control of the foreign exchange rate}

To make matters concrete, we introduce below an impulse control problem involving a government interested in controlling the foreign exchange (\gls*{FEX}) rate of its currency (i.e., the number of domestic monetary units it takes to buy one foreign monetary unit).
This problem, which will serve as a running example for a majority of the thesis, is perhaps one of the oldest and simplest financial applications of impulse control \cite{MR1614233,MR1705310,MR1802595}.

\begin{example}
    \label{exa:introduction_fex}
    Consider a government with two ways of influencing the \gls*{FEX} rate of its own currency:
    \begin{itemize}
        \item At all times, the government picks the domestic interest rate.
        \item The government picks specific times at which to intervene in the \gls*{FEX} market by buying or selling foreign currency in large quantities.
    \end{itemize}
    
    Let $\textcolor{ctrlcolor}{w_t}$ denote the difference between the domestic and foreign interest rate at time $t$.
    Let $\textcolor{ctrlcolor}{\xi_1} \leq \textcolor{ctrlcolor}{\xi_2} \leq \cdots \leq \infty$ be the times at which the government intervenes in the \gls*{FEX} market, with corresponding amounts $\textcolor{ctrlcolor}{z_1}, \textcolor{ctrlcolor}{z_2}, \ldots$
    A positive value for $\textcolor{ctrlcolor}{z_\ell}$ indicates the purchase of foreign currency at time $\textcolor{ctrlcolor}{\xi_\ell}$, while a negative value indicates the sale of foreign currency.

    Subject to the above, the \gls*{FEX} rate $X$ (in log space) evolves according to
    \begin{align*}
        dX_t & = -\mu \textcolor{ctrlcolor}{w_t} dt + \sigma d\gls*{wiener}_t
        & \text{for } t \neq \xi_1, \xi_2, \ldots
        \\
        X_{\textcolor{ctrlcolor}{\xi_\ell}} & = X_{\textcolor{ctrlcolor}{\xi_\ell}-} + \textcolor{ctrlcolor}{z_\ell}
        & \text{for } \ell = \phantom{\xi_1} \mathllap{1}, \phantom{\xi_2} \mathllap{2}, \ldots
    \end{align*}
    where $\mu$ and $\sigma$ are nonnegative constants representing the drift speed and variability of $X$.
    
    To prevent the domestic government from choosing an arbitrarily large interest rate, we assume that $\textcolor{ctrlcolor}{w_t}$ takes values in the set $W = [-w^{\max}, w^{\max}]$ where $w^{\max} > 0$.
    We assume that $\textcolor{ctrlcolor}{z_\ell}$ takes values in $\mathbb{R}$ 
    (i.e., no restrictions are imposed on the amounts that the domestic government can buy or sell in the FEX market).
    
    The government's objective is to keep $X$ as close as possible to some target rate $m$.
    Letting $\textcolor{ctrlcolor}{\theta} = (\textcolor{ctrlcolor}{w}; \textcolor{ctrlcolor}{\xi_1}, \textcolor{ctrlcolor}{z_1}; \textcolor{ctrlcolor}{\xi_2}, \textcolor{ctrlcolor}{z_2}; \ldots)$ denote a control, the government's costs are captured by the objective function (compare with \cref{eqn:introduction_functional})
    \begin{equation*}
        J(t, x; \textcolor{ctrlcolor}{\theta})
        = \mathbb{E} \left[
            - \int_t^T e^{-\beta u}
                \left(
                    \left( X_u - m \right)^2 + \gamma \textcolor{ctrlcolor}{w_u}^2
                \right)
            du
            - \sum_{t \leq \textcolor{ctrlcolor}{\xi_\ell} \leq T}
                e^{-\beta \textcolor{ctrlcolor}{\xi_\ell}} \left(
                    \kappa \left| \textcolor{ctrlcolor}{z_\ell} \right| + c
                \right)
            \middle | X_{t-} = x
        \right]
    \end{equation*}
    where $\beta$ is a nonnegative discount factor.
    The term $(X_u - m)^2$ penalizes the government for straying from the target rate $m$.
    The term $\gamma \textcolor{ctrlcolor}{w_u}^2$, where $\gamma \geq 0$, penalizes the government for choosing an interest rate that is either too low or too high with respect to the foreign interest rate.
    The term $\kappa \left| \textcolor{ctrlcolor}{z_\ell} \right| + c$, where $\kappa \geq 0$ and $c > 0$, captures the fixed and proportional transaction costs paid by the government for intervening in the \gls*{FEX} market.
    
    Standard dynamic programming arguments (cf. \cite[Example 8.2]{MR2109687}) are used to transform the optimization problem $V(t, x) = \sup_{\textcolor{ctrlcolor}{\theta}} J(t, x; \textcolor{ctrlcolor}{\theta})$ into the equivalent HJBQVI
    \begin{align}
        \min \left \{ 
            -V_t - \sup_{\textcolor{ctrlcolor}{w} \in W} \left \{
                    \frac{\sigma^2}{2} V_{xx}
                    - \mu \textcolor{ctrlcolor}{w} V_x
                    - e^{-\beta t} \left(
                        \left( x - m \right)^2
                        + \gamma \textcolor{ctrlcolor}{w}^2
                    \right )
            \right \},
            V - \mathcal{M} V
        \right \} 
        & = 0 \text{ on } [0,T) \times \mathbb{R}
        \nonumber
        \\
        \min \left\{
            V(T, \cdot), 
            V(T, \cdot) - \mathcal{M} V(T, \cdot)
        \right\}
        & = 0 \text{ on } \mathbb{R}
        \label{eqn:results_fex_hjbqvi}
    \end{align}
    where
    \begin{equation}
        \resultsFEXIntervention
        \label{eqn:results_fex_intervention}
    \end{equation}
    Note that the above is just a special case of the HJBQVI \cref{eqn:introduction_hjbqvi}.
\end{example}

\iftoggle{thesis}{%
\subsection{Fixed intervention times}

We also consider a type of impulse control problem arising from the pricing and hedging of insurance contracts.
In this context, we restrict impulses in \cref{eqn:introduction_functional} to occur at the ``fixed'' (i.e., nonrandom and noncontrollable) times $\xi_\ell$ satisfying
\[
    0 < \xi_1 < \xi_2 < \cdots < \xi_L < T.
\]
These intervention times correspond to regular anniversaries (e.g., monthly, yearly, etc.) on which the holder of the insurance contract can perform certain actions (e.g., lapse the contract).
Away from these times, the holder has no control over the contract.
This corresponds to the functions $a$ and $b$ in \cref{eqn:introduction_controlled_sde} and $f$ in \cref{eqn:introduction_functional} being independent of the control $\textcolor{ctrlcolor}{w}$.
This setting appears in, for example, \cite{MR2655287,MR2407322,chen2008effect,MR2844716,MR3215438,MR3494613,huang2016optimal}. 
In this case, a standard dynamic programming argument gives the pricing problem as finding a function $V(t,x)$ which satisfies (once again restricting our attention to the case in which $X$ is a one-dimensional process for simplicity)
\begin{subequations}
    \begin{align}
        V_t(t, x) & = - \frac{1}{2} b(t, x)^2 V_{xx}(t, x) - a(t, x) V_x(t, x) - f(t, x)
        & \text{for } t \neq \xi_1, \ldots, \xi_L
        \label{eqn:introduction_restricted_pde_interior} \\
        V(t-, x) & = \mathcal{M} V(t, x)
        & \text{for } t = \xi_1, \ldots, \xi_L
        \label{eqn:introduction_restricted_pde_boundary} \\
        V(T, x) & = g(x)
        \label{eqn:introduction_restricted_pde_terminal}
    \end{align}
    \label{eqn:introduction_restricted_pde}%
\end{subequations}
%
where $V(t-, \cdot) = \lim_{s \uparrow t} V(s, \cdot)$ is a limit from the left, $\mathcal{M}$ is the intervention operator \cref{eqn:introduction_intervention}, and it is understood that the equations \cref{eqn:introduction_restricted_pde_interior,eqn:introduction_restricted_pde_boundary,eqn:introduction_restricted_pde_terminal} hold at each $x$ in the domain.
Under reasonable conditions discussed in the sequel, problem \cref{eqn:introduction_restricted_pde} admits a solution $V$ that is smooth everywhere except possibly at $t = \xi_1, \ldots, \xi_L$.
As such, viscosity arguments are not required to ensure convergence of numerical methods for the problem.
In this case, we focus our attention on the numerical discretization of the intervention operator $\mathcal{M}$, as it is the main computational bottleneck.
In particular, since the control set $Z(t,x)$ appearing in \cref{eqn:introduction_intervention} is generally infinite, the numerical method needs to resort to a linear search over a discretization of the control set $Z(t,x)$ to approximate the supremum.
Convergence to a desired tolerance is obtained by refining this discretization.
Since the cost of linear search is proportional to the level of refinement, we identify practical scenarios in which the refinement step can be skipped while retaining a convergent method.
}{}

\clearpage

The main research contributions in this thesis are:
\begin{itemize}
    \item Creating the explicit-impulse scheme and generalizing the direct control and penalty schemes.
    \item Ensuring that a policy iteration algorithm can be applied to solve the direct control scheme by using (and contributing to) the theory of \gls*{wcdd} matrices.
    \item Of the three different schemes -- direct control, penalty, and explicit-impulse -- we show that the second and third are the methods of choice for the HJBQVI.
    \item Developing the concept of nonlocal consistency for handling convergence proofs for nonlocal PDEs and using this to prove the convergence of the three schemes.
    \iftoggle{thesis}{%
        \item Finding sufficient conditions under which impulse control problems with fixed intervention times can be solved without needing to successively refine a discretization of the control set, resulting in a faster numerical method.
    }
\end{itemize}

Our results appear in the following articles (in chronological order):
\begin{enumerate}[label=(\arabic*)]
    \iftoggle{thesis}{%
        \item \fullcite{MR3316194}
    }{}
    \item \fullcite{MR3493959}
    \item \fullcite{azimzadeh2017zero}
    \item \fullcite{azimzadeh2017fast}
    \item \fullcite{azimzadeh2017convergence}
\end{enumerate}

\section{Outline}

This thesis is organized as follows:
\begin{itemize}
    \item
        In \cref{chap:schemes}, we describe three numerical schemes for the HJBQVI: the direct control, penalty, and explicit-impulse schemes.
        These schemes are derived using heuristic arguments, with rigorous proofs of convergence deferred to a later chapter.
    \item
        In \cref{chap:matrix}, we consider solving the nonlinear matrix equations associated with the direct control and penalty schemes by policy iteration.
        In the case of the direct control scheme, the nonlinear matrix equations involve singular matrices, and new techniques involving \gls*{wcdd} matrices are introduced to handle this case.
    \item
        In \cref{chap:convergence}, we introduce the notion of nonlocal consistency in order to extend the results of Barles and Souganidis to viscosity solutions of nonlocal PDEs.
        We apply our findings to prove that the schemes of \cref{chap:schemes} converge to the viscosity solution of the HJBQVI.
    \item
        In \cref{chap:results}, we apply the schemes of \cref{chap:schemes} to compute numerical solutions of three impulse control problems from finance.
        We use our numerical results to compare the relative efficiency of the schemes.
    \iftoggle{thesis}{%
        \item
            In \cref{chap:bang}, we consider the impulse control problem involving fixed (i.e., nonrandom and noncontrollable) intervention times.
            We show that computing solutions of such a problem can often be sped up by skipping the control set refinement step (while retaining provable convergence) and apply our result to obtain a fast numerical pricer for a particular insurance contract.
    }{}
    \item
        In \cref{chap:summary}, we summarize our findings and discuss possible avenues for future research.
\end{itemize}
\setcounter{chapter}{1}
\chapter{Numerical schemes for the HJBQVI}
\label{chap:schemes}

In this chapter, we introduce three ``implicit'' numerical schemes for HJBQVIs: the \emph{direct control}, \emph{penalty}, and \emph{explicit-impulse} schemes.
By implicit, we mean that each scheme requires the solution of a nontrivial (and possibly nonlinear) system of equations at each timestep.
Our motivation for considering such schemes is to avoid the usual timestep restrictions of ``explicit'' schemes.

To simplify presentation, we restrict our attention to the one dimensional HJBQVI \cref{eqn:introduction_hjbqvi}, repeated below for the reader's convenience:
\begin{align}
    \introductionHJBQVIInterior
    \tag{\ref{eqn:introduction_hjbqvi_interior}} \\
    \introductionHJBQVIBoundary
    \tag{\ref{eqn:introduction_hjbqvi_boundary}}
\end{align}
where
\begin{equation}
    \introductionIntervention.
    \tag{\ref{eqn:introduction_intervention}}
\end{equation}
Modifications required to extend the schemes of this chapter to solve HJBQVIs of higher dimension are discussed in \cref{chap:convergence}.
\cref{chap:convergence} also includes the rigorous convergence arguments used to justify the heuristic derivations in this chapter.

%
%

\section{Preliminaries}

\subsection{Terminal condition}
\label{subsec:schemes_terminal_condition}

The terminal condition \cref{eqn:introduction_hjbqvi_boundary} can be written equivalently as
\begin{equation}
    V(T, \cdot)
    = \max \left \{
        \mathcal{M} V(T, \cdot),
        g
    \right \}.
    \label{eqn:schemes_implicit_boundary}
\end{equation}
Unlike most problems in finance, which involve an explicit terminal condition of the form $V(T, \cdot) = g$, \cref{eqn:schemes_implicit_boundary} is an implicit terminal condition: $V(T, \cdot)$ appears on both sides of the equation.
For a function $g$ whose only input is the space variable $x$ (i.e., $g(x)$), define
\[
    \mathcal{M} g(x) = \sup_{\textcolor{ctrlcolor}{z}\in Z(T,x)} \left\{
        g(
            \Gamma(T, x, \textcolor{ctrlcolor}{z})
        )
        + K(T, x, \textcolor{ctrlcolor}{z})
    \right\}
\]
(compare with \cref{eqn:introduction_intervention}).
If we assume the pointwise inequality
\begin{equation}
    \mathcal{M} g \leq g,
    \label{eqn:schemes_implicit_boundary_assumption}
\end{equation}
then $V(T, \cdot) = g$ is a solution of \cref{eqn:schemes_implicit_boundary}, returning us to an explicit terminal condition.
Assumption \cref{eqn:schemes_implicit_boundary_assumption} has the interpretation that performing an impulse at the final time $T$ never increases the value of the payoff \cite{belak2016utility,MR3615458}.
This assumption is satisfied by most reasonable impulse control problems, including those in this thesis.


\subsection{Numerical grid}

Since the terminal condition \cref{eqn:schemes_implicit_boundary} occurs at time $t = T$, the numerical method proceeds ``backwards in time'' to obtain a solution at time $t = 0$.
Therefore, it is natural to refer to the procedure which evolves a numerical solution from time $t = T$ to some earlier time $t = T - \Delta \tau$ as the \emph{first} timestep, that which evolves the numerical solution from time $t = T - \Delta \tau$ to $t = T - 2 \Delta \tau$ as the \emph{second} timestep, etc.

In light of this, we denote by $V_i^n \approx V(\tau^n, x_i)$ the numerical solution at time $\tau^n = T - n \Delta \tau$ and point $x_i$ in space where $0 \leq i \leq M$ and $\Delta \tau = T / N$ (\cref{fig:schemes_grid}).
We denote by $V^n = (V_0^n, \ldots, V_M^n)^\intercal$ the numerical solution vector at the $n$-th timestep.
For brevity, we use the shorthand $a_i^n(\textcolor{ctrlcolor}{w}) = a(\tau^n, x_i, \textcolor{ctrlcolor}{w})$ with $b_i^n$ and $f_i^n$ defined similarly.

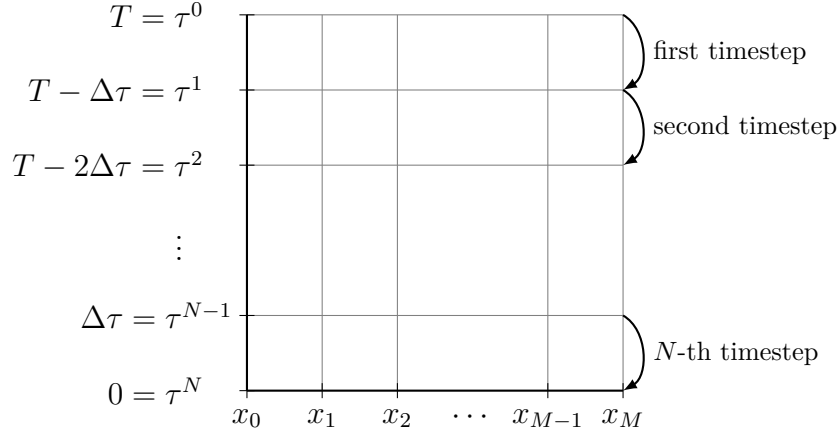
\begin{figure}
    \centering
    \begin{tikzpicture}
        \foreach \n in {1, 3, 4, 5}
            \draw [gray]
                (0, \n) -- (5, \n)
        ;
        \foreach \n in {1, 2, 4, 5}
            \draw [gray]
                (\n, 0) -- (\n, 5)
        ;
        \draw [thick]
            (0, 0) -- (5, 0)
            (0, 0) -- (0, 5)
        ;
        \path [thick, ->] (5,5) edge [bend left=60] node [anchor=west] {\footnotesize first timestep} (5,4);
        \path [thick, ->] (5,4) edge [bend left=60] node [anchor=west] {\footnotesize second timestep} (5,3);
        \path [thick, ->] (5,1) edge [bend left=60] node [anchor=west] {\footnotesize $N$-th timestep} (5,0);
        \foreach \n/\i in {0/0, 1/1, 2/2, 4/M-1, 5/M} {
            \draw (\n, -0.1) -- (\n, 0.1);
            \node [yshift=-0.1cm, below] at (\n, 0) {$x_{\i}$};
        };
        \node [yshift=-0.35cm] at (3,0) {$\cdots$};
        \foreach \n/\i/\extra in {0/N\phantom{-1}/0=, 1/N-1/\Delta\tau=, 3/2\phantom{-1}/T-2\Delta\tau=, 4/1\phantom{-1}/T-\Delta\tau=, 5/0\phantom{-1}/T=} {
            \draw (-0.1, \n) -- (0.1, \n);
            \node [xshift=-0.05cm, left] at (0, \n) {$\extra\tau^{\i}$};
        }
        \node [xshift=-0.9cm] at (0,2) {$\vdots$};
    \end{tikzpicture}
    \caption{Numerical grid}
    \label{fig:schemes_grid}
\end{figure}

\begin{example}
    \label{exa:schemes_grid}
    The HJBQVI \cref{eqn:results_fex_hjbqvi} of \cref{exa:introduction_fex} is posed on the unbounded set $[0,T] \times \mathbb{R}$.
    A numerical grid is obtained by ``truncating'' this set to a bounded region $[0,T] \times [-R, R]$ and taking $\{ x_0, \ldots, x_M \}$ to be a partition\partitionFootnote~of the interval $[-R, R]$.
\end{example}

Following the discussion of the previous section, all schemes in this chapter satisfy
\[
    V_i^0 = g(x_i)
\]
corresponding to the terminal condition $V(T, \cdot) = g$.
Therefore, in describing a scheme, we specify only the equations required to determine $V_i^n$ for $n > 0$.

\subsection{Standard stencils}

Unless otherwise specified, we discretize the time derivative $V_t$ appearing in \cref{eqn:introduction_hjbqvi_interior} by
\[
    V_t(\tau^n, x_i)
    \approx \frac{V(\tau^n + \Delta \tau, x_i) - V(\tau^n, x_i)}{\Delta \tau}
    = \frac{V(\tau^{n-1}, x_i) - V(\tau^n, x_i)}{\Delta \tau}
    \approx \frac{V_i^{n-1} - V_i^n}{\Delta \tau}.
\]

We use $\mathcal{D}_2$ to denote a three point discretization of the second spatial derivative so that $V_{xx}(\tau^n, x_i) \approx (\mathcal{D}_2 V^n)_i$.
Precisely, for each vector $U = (U_0, \ldots, U_M)^\intercal$, we define
\begin{equation}
    (\mathcal{D}_2 U)_i = \begin{cases}
        \displaystyle{
            \frac{U_{i+1} - U_i}{\left(x_{i+1} - x_i\right) \left(x_{i+1} - x_{i-1}\right)}
            - \frac{U_i - U_{i-1}}{\left(x_i - x_{i-1}\right) \left(x_{i+1} - x_{i-1}\right)}
        } & \text{if } 0 < i < M \\
        0 & \text{otherwise}.
    \end{cases}
    \label{eqn:schemes_second_derivative}
\end{equation}
To make matters concrete, we have set $(\mathcal{D}_2 U)_i$ to be zero at the boundary points $x_0$ and $x_M$, corresponding to the artificial Neumann boundary condition $V_{xx}(t, x_0) = V_{xx}(t, x_M) = 0$
(we mention, however, that the techniques in this thesis are not specific to this choice of boundary condition).
Note that when the grid points $\{x_0, \ldots, x_M\}$ are uniformly spaced (i.e., $x_{i+1} - x_i = \Delta x$), $(\mathcal{D}_2 U)_i$ takes the familiar form
\[
    (\mathcal{D}_2 U)_i = 
            \frac{U_{i+1} - 2 U_i + U_{i-1}}{(\Delta x)^2}
        \textspace \text{if } 0 < i < M. \\
\]

Similarly, we use $\mathcal{D}_\pm$ to denote forward ($+$) and backward ($-$) discretizations of the first spatial derivative so that $V_x(\tau^n, x_i) \approx (\mathcal{D}_\pm V^n)_i$.
In particular,
\begin{equation}
    (\mathcal{D}_+ U)_i = \begin{cases}
        \displaystyle{
            \frac{U_{i+1} - U_i}{x_{i+1} - x_i}
        } & \text{if } 0 < i < M \\
        0 & \text{otherwise}
    \end{cases}
    \textspace \text{and} \textspace
    (\mathcal{D}_- U)_i = \begin{cases}
        \displaystyle{
            \frac{U_i - U_{i-1}}{x_i - x_{i-1}}
        } & \text{if } 0 < i < M \\
        0 & \text{otherwise}.
    \end{cases}
    \label{eqn:schemes_first_derivative_pm}
\end{equation}
As usual, to make matters concrete, we have set $(\mathcal{D}_\pm U)_i$ to be zero at the boundary points $x_0$ and $x_M$, corresponding to the artificial Neumann boundary condition $V_x(t, x_0) = V_x(t, x_M) = 0$.
With a slight abuse of notation, we define the ``upwind'' discretization
\begin{equation}
    a_i^n(\textcolor{ctrlcolor}{w})
    (\mathcal{D} U)_i =
    \begin{cases}
        a_i^n(\textcolor{ctrlcolor}{w})
        (\mathcal{D}_+ U)_i
        & \text{if } a_i^n(\textcolor{ctrlcolor}{w}) > 0 \\
        a_i^n(\textcolor{ctrlcolor}{w})
        (\mathcal{D}_- U)_i
        & \text{otherwise}
    \end{cases}
    \label{eqn:schemes_first_derivative}
\end{equation}
which employs a forward difference when the coefficient of the first derivative is positive and a backward difference otherwise.
The choice of upwinding will be useful in the sequel to ensure the convergence of our schemes.

\begin{remark}
    It is often convenient to view $\mathcal{D}_2$ (resp. $\mathcal{D}_\pm$) as a matrix mapping each vector $U$ to the vector $\mathcal{D}_2U$ (resp. $\mathcal{D}_\pm U$) with entries given by \cref{eqn:schemes_second_derivative} (resp. \cref{eqn:schemes_first_derivative_pm}).
\end{remark}

%
%

\subsection{Intervention operator}

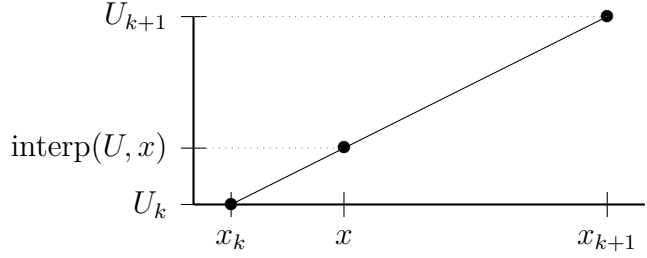
\begin{figure}
    \centering
    \begin{tikzpicture}[scale=5]
        \draw (0, 0) -- (1, 0.5);
        \foreach \x/\y in {0/0, 1/0.5, 0.3/0.15}
            \node at (\x, \y) {\textbullet};
        \draw [thick]
            (-0.1, 0) -- (1+0.1, 0)
            (-0.1, 0) -- (-0.1, 0.5)
        ;
        \foreach \n/\lbl in {0/x_k, 1/x_{k+1}, 0.3/x} {
            \draw (\n, -0.1/3) -- (\n, 0.1/3);
            \node [yshift=-0.2cm, below] at (\n, 0) {$\lbl$};
        };
        \foreach \n/\lbl in {0/U_k, 0.5/U_{k+1}, 0.15/{\gls*{interp}(U,x)}} {
            \draw (-0.1-0.1/3, \n) -- (-0.1+0.1/3, \n);
            \node [xshift=-0.7cm, left] at (0, \n) {$\lbl$};
        }
        \draw [dotted, gray]
            (-0.1, 0.15) -- (0.3, 0.15)
            (-0.1, 0.5) -- (1, 0.5)
        ;
    \end{tikzpicture}
    \caption{Linear interpolation}
    \label{fig:schemes_interpolation}
\end{figure}

The point $
\Gamma(t, x, \textcolor{ctrlcolor}{z})$ appearing in the intervention operator $\mathcal{M}$ defined in \cref{eqn:introduction_intervention} is not necessarily a point on the numerical grid, hence a discretization of $\mathcal{M}$ requires interpolation.
We use $\gls*{interp}(V^n, x)$ to denote the value of the numerical solution at $(\tau^n, x)$ as approximated by a standard monotone linear interpolant.
That is, for each vector $U = (U_0, \ldots, U_M)^\intercal$ and point $x$ contained in the grid (i.e., $x_0 < x < x_M$), we define
\begin{equation}
    \gls*{interp}(U, x) = \alpha U_{k+1} + \left(1 - \alpha\right) U_k
    \label{eqn:schemes_interpolation}
\end{equation}
where $k$ is the unique integer satisfying $x_k \leq x < x_{k+1}$ and $\alpha = (x - x_k) / (x_{k+1} - x_k)$ (\cref{fig:schemes_interpolation}).
Though we have suppressed this in the notation, the quantities $\alpha$ and $k$ depend on $x$.
If $x$ is not contained in the grid (i.e., $x \leq x_0$ or $x \geq x_M$), we define
\[
    \gls*{interp}(U, x) = \begin{cases}
        U_0 & \text{if } x \leq x_0 \\
        U_M & \text{if } x \geq x_M,
    \end{cases}
\]
so that no extrapolation is performed.
We can now discretize the intervention operator according to
\begin{equation}
    \schemesDiscretizedIntervention
    \label{eqn:schemes_discretized_intervention}
\end{equation}
Using the notation above, $\mathcal{M}V(\tau^n, x_i) \approx (\mathcal{M}_n V^n)_i$.

Since the control set $Z(t, x)$ appearing in \cref{eqn:introduction_intervention} is not, in general, a finite set, we have introduced into \cref{eqn:schemes_discretized_intervention} a subset $Z^{\gls*{meshing_parameter}}(\tau^n, x_i)$ of $Z(\tau^n, x_i)$ to serve as an approximation of $Z(\tau^n, x_i)$.
We are purposely vague in our use of the term ``approximation'' (the convergence proofs of \cref{chap:convergence} will require us to attach a rigorous meaning to the term).
For the time being, we assume only that $Z^{\gls*{meshing_parameter}}(t, x)$ is finite and nonempty.
The finitude of $Z^{\gls*{meshing_parameter}}(t, x)$ ensures that the supremum in \cref{eqn:schemes_discretized_intervention} can be computed by performing a linear search over finitely many elements.
To make matters concrete, an example is given below.

\begin{example}
    \label{exa:schemes_intervention}
    Recall the intervention operator \cref{eqn:results_fex_intervention} of \cref{exa:introduction_fex}:
    \begin{equation}
        \resultsFEXIntervention
        \tag{\ref{eqn:results_fex_intervention}}
    \end{equation}
    The control set $Z(t, x) = \mathbb{R}$ is infinite.
    We take $Z^{\gls*{meshing_parameter}}(t, x) = \{ x_0 - x, \ldots, x_M - x \}$ so that
    \begin{align*}
        (\mathcal{M}_n U)_i
        & = \max_{\textcolor{ctrlcolor}{z_i} \in \{x_0 - x_i, \ldots, x_M - x_i\}} \left \{
            \gls*{interp}(U, x_i + \textcolor{ctrlcolor}{z_i})
            - e^{-\beta \tau^n} \left(
                \kappa \left| \textcolor{ctrlcolor}{z_i} \right| + c
            \right)
        \right \}
        \\
        & = \max_{ 0 \leq j \leq M } \left \{
            \gls*{interp}(U, x_j)
            - e^{-\beta \tau^n} \left(
                \kappa \left| x_j - x_i \right| + c
            \right)
        \right \}
        \\
        & = \max_{ 0 \leq j \leq M } \left \{
            U_j
            - e^{-\beta \tau^n} \left(
                \kappa \left| x_j - x_i \right| + c
            \right)
        \right \}.
    \end{align*}
\end{example}


\section{The direct control scheme}

We are now ready to introduce our first scheme, the direct control scheme.
We first introduce an ``auxiliary'' control $\textcolor{ctrlcolor}{d}$ in order to rewrite \cref{eqn:introduction_hjbqvi_interior} without the outer minimum:
\[
    \sup_{\substack{
        \textcolor{ctrlcolor}{d} \in \{0,1\} \\
        \textcolor{ctrlcolor}{w} \in W
    }}
    \left\{
        \left( 1 - \textcolor{ctrlcolor}{d} \right)
        \left(
            V_t
            + \frac{1}{2} b(\cdot, \textcolor{ctrlcolor}{w})^2 V_{xx}
            + a(\cdot, \textcolor{ctrlcolor}{w}) V_x
            + f(\cdot, \textcolor{ctrlcolor}{w})
        \right)
        + \textcolor{ctrlcolor}{d} \left(
            \mathcal{M}V - V
        \right)
    \right\} = 0.
\]
From this point on, we will use the shorthand $\textcolor{ctrlcolor}{\overline{d}} = 1 - \textcolor{ctrlcolor}{d}$ for brevity.
To obtain the direct control scheme, we then replace all operators in the above by their discretized versions (i.e., $\mathcal{D}_2$, $\mathcal{D}$, $\mathcal{M}_n$, etc.):
\begin{multline}
        \sup_{\substack{
            \textcolor{ctrlcolor}{d_i} \in \{0,1\} \\
            \textcolor{ctrlcolor}{w_i} \in W^{\gls*{meshing_parameter}}
        }}
        \biggl \{ 
                \textcolor{ctrlcolor}{\overline{d_i}} \left(
                    \frac{V_i^{n-1} - V_i^n}{\Delta \tau}
                    + \frac{1}{2} b_i^n(\textcolor{ctrlcolor}{w_i})^2 (\mathcal{D}_2 V^n)_i
                    + a_i^n(\textcolor{ctrlcolor}{w_i}) (\mathcal{D} V^n)_i
                    + f_i^n(\textcolor{ctrlcolor}{w_i})
                \right)
                \\
                + \textcolor{ctrlcolor}{d_i}
                \left(
                    (\mathcal{M}_n V^n)_i - V_i^n
                \right)
        \biggr \} 
        = 0.
    \label{eqn:schemes_direct_control}
\end{multline}
It is understood that equation \cref{eqn:schemes_direct_control} holds for all $n = 1, \ldots, N$ and $i = 0, \ldots, M$.

In \cref{eqn:schemes_direct_control}, we have introduced a finite nonempty subset $W^{\gls*{meshing_parameter}}$ of $W$ to serve as an approximation of $W$, analogously to the approximation of $Z(t, x)$ by $Z^{\gls*{meshing_parameter}}(t, x)$.
For example, in the case of the control set $W = [-w^{\max}, w^{\max}]$ of \cref{exa:introduction_fex}, we can take $W^{\gls*{meshing_parameter}}$ to be a partition of $W$.

The direct control scheme was suggested in \cite{MR2302036} to solve a particular problem involving optimal consumption.
In that paper, the authors focused on solving the nonlinear matrix equations associated with the direct control scheme by policy iteration.
No proofs of convergence (in the viscosity sense) or an implementation are given in that paper.

\section{The penalty scheme}

In \cite[Chapter 5, Section 1.6]{bensoussan1984impulse}, the authors study the HJBQVI \cref{eqn:introduction_hjbqvi} in its variational form. In order to show existence of a solution, they introduce what they refer to as the ``penalized'' form of \cref{eqn:introduction_hjbqvi}:
\begin{subequations}
    \begin{align}
        - V_t^\epsilon
        - \sup_{\textcolor{ctrlcolor}{w} \in W} \left\{
            \frac{1}{2} b(\cdot, \textcolor{ctrlcolor}{w})^2 V_{xx}^\epsilon
            {+} a(\cdot, \textcolor{ctrlcolor}{w}) V_x^\epsilon
            {+} f(\cdot, \textcolor{ctrlcolor}{w})
        \right\}
        - \frac{1}{\epsilon} \max \left\{
            \mathcal{M}V^\epsilon - V^\epsilon,
            0
        \right\}
        & = 0 \text{ on } [0,T) \times \gls*{domain}
        \label{eqn:schemes_penalized_interior} \\
        V^\epsilon(T, \cdot) - g
        & = 0 \text{ on } \gls*{domain}.
        \label{eqn:schemes_penalized_boundary}
    \end{align}
    \label{eqn:schemes_penalized}%
\end{subequations}
The authors of \cite{bensoussan1984impulse} then show, subject to some technical conditions, that the solution $V$ of the original HJBQVI \cref{eqn:introduction_hjbqvi} is the pointwise limit of the solution $V^\epsilon$ of \cref{eqn:schemes_penalized} as $\epsilon \downarrow 0$. 

The basic idea behind the penalty scheme is to discretize \cref{eqn:schemes_penalized_interior}.
In particular, we first rewrite \cref{eqn:schemes_penalized_interior} using the auxiliary control $\textcolor{ctrlcolor}{d}$ to get, omitting the superscript $\epsilon$ in $V^\epsilon$,
\[
    \sup_{\substack{
        \textcolor{ctrlcolor}{d} \in \{0,1\} \\
        \textcolor{ctrlcolor}{w} \in W
    }}
    \left\{
        V_t
        + \frac{1}{2} b(\cdot, \textcolor{ctrlcolor}{w})^2 V_{xx}
        + a(\cdot, \textcolor{ctrlcolor}{w}) V_x
        + f(\cdot, \textcolor{ctrlcolor}{w})
        + \frac{1}{\epsilon} \textcolor{ctrlcolor}{d} \left(
            \mathcal{M}V - V
        \right)
    \right\} = 0.
\]
We then replace all operators by their discretized versions (i.e., $\mathcal{D}_2$, $\mathcal{D}$, $\mathcal{M}_n$, etc.) to obtain the penalty scheme:
\begin{multline}
        \sup_{\substack{
            \textcolor{ctrlcolor}{d_i} \in \{0,1\} \\
            \textcolor{ctrlcolor}{w_i} \in W^{\gls*{meshing_parameter}}
        }}
        \biggl \{ 
                \frac{V_i^{n-1} - V_i^n}{\Delta \tau}
                + \frac{1}{2} b_i^n(\textcolor{ctrlcolor}{w_i})^2 (\mathcal{D}_2 V^n)_i
                + a_i^n(\textcolor{ctrlcolor}{w_i}) (\mathcal{D} V^n)_i
                + f_i^n(\textcolor{ctrlcolor}{w_i})
                \\
                + \frac{1}{\epsilon} \textcolor{ctrlcolor}{d_i}
                \left(
                    (\mathcal{M}_n V^n)_i - V_i^n
                \right)
        \biggr \} 
        = 0.
    \label{eqn:schemes_penalty}
\end{multline}
It is understood that equation \cref{eqn:schemes_penalty} holds for all $n = 1, \ldots, N$ and $i = 0, \ldots, M$.

An alternate (and perhaps easier) way to motivate the penalty scheme \cref{eqn:schemes_penalty} is to note that it is an approximation of the direct control scheme \cref{eqn:schemes_direct_control}.
In particular, after some simplification, \cref{eqn:schemes_penalty} is equivalent to
\begin{equation}
    \sup_{\substack{
        \textcolor{ctrlcolor}{d_i} \in \{0,1\} \\
        \textcolor{ctrlcolor}{w_i} \in W^{\gls*{meshing_parameter}}
    }}
    \left\{
        \textcolor{ctrlcolor}{\overline{d_i}} \gamma_i^n(V, \textcolor{ctrlcolor}{w_i})
        + \textcolor{ctrlcolor}{d_i} \left(
            (\mathcal{M}_n V^n)_i - V_i^n + \epsilon \gamma_i^n(V, \textcolor{ctrlcolor}{w_i})
        \right)
    \right\} = 0
    \label{eqn:schemes_penalty_rearranged}
\end{equation}
where
\begin{equation}
    \gamma_i^n(V, \textcolor{ctrlcolor}{w_i}) =
    \frac{V_i^{n-1} - V_i^n}{\Delta \tau}
    + \frac{1}{2} b_i^n(\textcolor{ctrlcolor}{w_i})^2 (\mathcal{D}_2 V^n)_i
    + a_i^n(\textcolor{ctrlcolor}{w_i}) (\mathcal{D} V^n)_i
    + f_i^n(\textcolor{ctrlcolor}{w_i}).
    \label{eqn:schemes_penalty_gamma}
\end{equation}
Aside from the term $\epsilon \gamma_i^n(V, \textcolor{ctrlcolor}{w_i}) = O(\epsilon)$, \cref{eqn:schemes_direct_control} and \cref{eqn:schemes_penalty_rearranged} are identical.
Hence, as the positive parameter $\epsilon$ vanishes, we expect the penalty scheme to behave like the direct control scheme.

A penalty scheme was first applied to an HJBQVI in \cite{MR2597608}.
The authors in \cite{MR2597608} consider a very specific form for the intervention operator $\mathcal{M}$ motivated by the application problem studied therein.
This specific form is exploited in the convergence proofs, which do not generalize.
To the best of our knowledge, the penalty scheme was first considered in full generality (i.e., with a general form for $\mathcal{M}$) in our own work \cite{MR3493959,azimzadeh2017convergence}.

\section{The infinite horizon (steady state) case}
\label{sec:schemes_infinite_horizon}

Assume now that the functions $a$, $b$, $f$, $Z$, $\Gamma$, and $K$ do not depend on time (i.e., $a(t, x, \textcolor{ctrlcolor}{w}) = a(0, x, \textcolor{ctrlcolor}{w})$, $b(t, x, \textcolor{ctrlcolor}{w}) = b(0, x, \textcolor{ctrlcolor}{w})$, etc.).
Then, the infinite horizon analogue of \cref{eqn:introduction_hjbqvi} is given by
\begin{equation}
    \min \left\{
        \beta V - \sup_{\textcolor{ctrlcolor}{w} \in W} \left\{
            \frac{1}{2} b(0, \cdot, \textcolor{ctrlcolor}{w})^2 V_{xx}
            + a(0, \cdot, \textcolor{ctrlcolor}{w}) V_x
            + f(0, \cdot, \textcolor{ctrlcolor}{w})
        \right\},
        V - \mathcal{M} V
    \right\} = 0 \text{ on } \gls*{domain}
    \label{eqn:schemes_infinite_horizon}
\end{equation}
where $\beta > 0$ is a positive discount factor.
Note that in the above, $V$ is no longer a function of time and space but rather a function of space alone (i.e., $V(x)$).

Recall that in the context of impulse control, the finite horizon HJBQVI \cref{eqn:introduction_hjbqvi} was related to picking a control $\textcolor{ctrlcolor}{\theta}$ to maximize the quantity \cref{eqn:introduction_functional}.
Similarly, the infinite horizon HJBQVI \cref{eqn:schemes_infinite_horizon} is related to picking a control $\textcolor{ctrlcolor}{\theta}$ to maximize the quantity \cite{MR2568293}
\begin{equation}
    J_\infty(t, x; \textcolor{ctrlcolor}{\theta}) = \mathbb{E} \left[
        \int_t^\infty e^{-\beta u} f(u, X_u, \textcolor{ctrlcolor}{w_u}) du
        + \sum_{t \leq \textcolor{ctrlcolor}{\xi_\ell}} e^{-\beta \xi_\ell} K(
            \textcolor{ctrlcolor}{\xi_\ell},
            X_{\textcolor{ctrlcolor}{\xi_\ell}-},
            \textcolor{ctrlcolor}{z_\ell}
        )
        \middle | X_{t-} = x
    \right].
    \label{eqn:schemes_infinite_horizon_functional}
\end{equation}
We will see an example of such a problem involving an investor whose goal it is to consume optimally in a market consisting of a risky investment and bank account in \cref{chap:results}.

Dropping the term corresponding to the discretization of the time derivative and introducing a term for the discount factor in \cref{eqn:schemes_direct_control}, we obtain an extension of the direct control scheme for the infinite horizon problem:
\begin{equation}
        \sup_{\substack{
            \textcolor{ctrlcolor}{d_i} \in \{0,1\} \\
            \textcolor{ctrlcolor}{w_i} \in W^{\gls*{meshing_parameter}}
        }} \left\{
                \textcolor{ctrlcolor}{\overline{d_i}} \left(
                    - \beta V_i
                    +\frac{1}{2} b_i^0(\textcolor{ctrlcolor}{w_i})^2 (\mathcal{D}_2 \vec{V})_i
                    + a_i^0(\textcolor{ctrlcolor}{w_i}) (\mathcal{D} \vec{V})_i
                    + f_i^0(\textcolor{ctrlcolor}{w_i})
                \right)
                + \textcolor{ctrlcolor}{d_i}
                \left(
                    (\mathcal{M}_0 \vec{V})_i - V_i
                \right)
        \right\} = 0.
    \label{eqn:schemes_direct_control_infinite_horizon}
\end{equation}
In the above, the numerical solution $\vec{V} = (V_0, \ldots, V_M)^\intercal$ is a vector of $M+1$ points (one per grid point $x_i$) since there is no time dependency.
It is understood that equation \cref{eqn:schemes_direct_control_infinite_horizon} holds for all $n = 1, \ldots, N$ and $i = 0, \ldots, M$.
An identical exercise can be performed to extend the penalty scheme \cref{eqn:schemes_penalty} to the infinite horizon case:
\begin{equation}
        \sup_{\substack{
            \textcolor{ctrlcolor}{d_i} \in \{0,1\} \\
            \textcolor{ctrlcolor}{w_i} \in W^{\gls*{meshing_parameter}}
        }} \left\{
                - \beta V_i
                + \frac{1}{2} b_i^0(\textcolor{ctrlcolor}{w_i})^2
                (\mathcal{D}_2 \vec{V})_i
                + a_i^0(\textcolor{ctrlcolor}{w_i})
                (\mathcal{D} \vec{V})_i
                + f_i^0(\textcolor{ctrlcolor}{w_i})
                + \frac{1}{\epsilon} \textcolor{ctrlcolor}{d_i}
                \left(
                    (\mathcal{M}_0 \vec{V})_i - V_i
                \right)
        \right\} = 0.
    \label{eqn:schemes_penalty_infinite_horizon}
\end{equation}

%

\section{The explicit-impulse scheme}

The direct control and penalty schemes of the previous sections are versatile in the sense that they are able to handle both finite and infinite horizon problems with minimal restrictions on the coefficients of the HJBQVI.
However, as we will see later in \cref{chap:matrix}, the equations associated with these schemes are nonlinear in $V^n$ and as such, require the use of an expensive iterative method at each timestep.

If the horizon is finite (i.e., $T < \infty$) and the second derivative coefficient does not depend on the control $\textcolor{ctrlcolor}{w}$ (i.e., $b(t, x, \textcolor{ctrlcolor}{w}) = b(t, x)$), then we can produce a scheme whose associated equations are linear in $V^n$, thereby requiring only a single linear system be solved at each timestep.
The idea is to rewrite \cref{eqn:introduction_hjbqvi_interior} as
\begin{equation}
    \min \left\{
        - \sup_{\textcolor{ctrlcolor}{w} \in W} \left\{
            \left( V_t + a(\cdot, \textcolor{ctrlcolor}{w}) V_x \right)
            + \frac{1}{2} b(\cdot)^2 V_{xx}
            + f(\cdot, \textcolor{ctrlcolor}{w})
        \right\},
        V - \mathcal{M}V
    \right\} = 0 \text{ on } (0,T] \times \gls*{domain}
    \label{eqn:schemes_hjbqvi_interior_modified}
\end{equation}
and approximate the term $V_t + a(\cdot, \textcolor{ctrlcolor}{w}) V_x$
by
\begin{equation}
    V_t(\tau^n, x_i) + a(\tau^n, x_i, \textcolor{ctrlcolor}{w}) V_x(\tau^n, x_i)
    \approx
    \frac{\gls*{interp}(V^{n-1}, x_i + a_i^n(\textcolor{ctrlcolor}{w}) \Delta \tau) - V_i^n}{\Delta \tau}.
    \label{eqn:schemes_lagrangian_derivative}
\end{equation}
This approximation can be derived in two ways: by a Lagrangian argument involving tracing the path of a particle \cite[Section 2.3.1]{chen2008numerical}, or by a Taylor series (see \cref{sec:schemes_convergence_rates}).
As usual, the second derivative is approximated by $\mathcal{D}_2$ and the intervention operator by $\mathcal{M}_n$.
Substituting these approximations into \cref{eqn:schemes_hjbqvi_interior_modified}:
\begin{multline*}
    \sup_{\substack{
        \textcolor{ctrlcolor}{d_i} \in \{0,1\} \\
        \textcolor{ctrlcolor}{w_i} \in W^{\gls*{meshing_parameter}}
    }}
    \biggl \{ 
            \textcolor{ctrlcolor}{\overline{d_i}}
            \left(
                \displaystyle{
                    \frac{\gls*{interp}(
                        V^{n-1},
                        x_i + a_i^n(\textcolor{ctrlcolor}{w_i}) \Delta \tau
                    ) - V_i^n}{\Delta \tau}
                }
                + \frac{1}{2} (b_i^n)^2
                (\mathcal{D}_2 V^n)_i
                + f_i^n(\textcolor{ctrlcolor}{w_i})
            \right) \\
            + \textcolor{ctrlcolor}{d_i} \left(
                (\mathcal{M}_n V^{n-1})_i
                - V_i^n
            \right)
    \biggr \} 
    = 0
\end{multline*}
where $b_i^n = b(\tau^n, x_i)$.
Note that unlike the previous schemes, we have not used $V^n$ but rather $V^{n-1}$ as an argument to the discretized intervention operator $\mathcal{M}_n$, hence the name ``explicit-impulse''.
Next, an $O(\Delta \tau)$ term is added to the equation:
\begin{multline}
    \sup_{\substack{
        \textcolor{ctrlcolor}{d_i} \in \{0,1\} \\
        \textcolor{ctrlcolor}{w_i} \in W^{\gls*{meshing_parameter}}
    }}
    \biggl \{ 
            \textcolor{ctrlcolor}{\overline{d_i}}
            \left(
                \displaystyle{
                    \frac{\gls*{interp}(
                        V^{n-1},
                        x_i + a_i^n(\textcolor{ctrlcolor}{w_i}) \Delta \tau
                    ) - V_i^n}{\Delta \tau}
                }
                + \frac{1}{2} (b_i^n)^2 (\mathcal{D}_2 V^n)_i
                + f_i^n(\textcolor{ctrlcolor}{w_i})
            \right) \\
            + \textcolor{ctrlcolor}{d_i} \left(
                (\mathcal{M}_n V^{n-1})_i
                - V_i^n
                + \frac{1}{2} (b_i^n)^2 (\mathcal{D}_2 V^n)_i \Delta \tau
            \right)
    \biggr \} 
    = 0.
    \label{eqn:schemes_explicit_impulse_rearranged}
\end{multline}
After some simplification, we can isolate all $V^n$ terms to one side of the equation, arriving at the explicit-impulse scheme:
\begin{multline}
            V_i^n - \frac{1}{2} (b_i^n)^2
            (\mathcal{D}_2 V^n)_i \Delta \tau \\
            = \sup_{\substack{
                \textcolor{ctrlcolor}{d_i} \in \{0,1\} \\
                \textcolor{ctrlcolor}{w_i} \in W^{\gls*{meshing_parameter}}
            }}
            \left\{
                \textcolor{ctrlcolor}{\overline{d_i}}
                \left(
                    \gls*{interp}(
                        V^{n-1},
                        x_i + a_i^n(\textcolor{ctrlcolor}{w_i}) \Delta \tau
                    )
                    + f_i^n(\textcolor{ctrlcolor}{w_i})
                    \Delta \tau
                \right)
                + \textcolor{ctrlcolor}{d_i}
                (\mathcal{M}_n V^{n-1})_i
            \right\}.
    \label{eqn:schemes_explicit_impulse}
\end{multline}
It is understood that equation \cref{eqn:schemes_explicit_impulse} holds for all $n = 1, \ldots, N$ and $i = 0, \ldots, M$.
As we will see below, adding the $O(\Delta \tau)$ term allows us to express the scheme as a linear system of equations at each timestep.
However, adding this term does not come for free: it introduces an additional source of discretization error.

Equation \cref{eqn:schemes_explicit_impulse} is the $i$-th row of the linear system
\begin{equation}
    A V^n = y
    \label{eqn:schemes_linear_system}
\end{equation}
where $y$ is the vector whose $i$-th component is the right hand side of \cref{eqn:schemes_explicit_impulse}: 
\[
    y_i = \sup_{\substack{
        \textcolor{ctrlcolor}{d_i} \in \{0,1\} \\
        \textcolor{ctrlcolor}{w_i} \in W^{\gls*{meshing_parameter}}
    }} \left\{
        \textcolor{ctrlcolor}{\overline{d_i}}
        \left(
            \gls*{interp}(
                V^{n-1},
                x_i + a_i^n(\textcolor{ctrlcolor}{w_i}) \Delta \tau
            )
            + f_i^n(\textcolor{ctrlcolor}{w_i})
            \Delta \tau
        \right)
        + \textcolor{ctrlcolor}{d_i}
        (\mathcal{M}_n V^{n-1})_i
    \right\}
\]
and
\[
    A = I - \frac{\Delta \tau}{2} \gls*{diag}((b_0^n)^2,\ldots, (b_M^n)^2) \mathcal{D}_2.
\]
For example, if the grid points $\{ x_0, \ldots, x_M \}$ are uniformly spaced (i.e., $x_{i+1} - x_i = \Delta x$), the matrix $A$ simplifies to
\begin{equation}
    A = I + \frac{\Delta \tau}{2(\Delta x)^2}
    \begin{pmatrix}
        0 \\
        -(b_1^n)^2 & 2(b_1^n)^2 & -(b_1^n)^2 \\
        & -(b_2^n)^2 & 2(b_2^n)^2 & -(b_2^n)^2 \\
        &  & \ddots & \ddots & \ddots \\
        &  &  & -(b_{M-1}^n)^2 & 2(b_{M-1}^n)^2 & -(b_{M-1}^n)^2 \\
        &  &  &  &  & 0
    \end{pmatrix},
    \label{eqn:schemes_tridiagonal}
\end{equation}
which is strictly diagonally dominant (in fact, $A$ is also strictly diagonally dominant if the grid is not uniform).
Since strict diagonal dominance implies nonsingularity \cite[Theorem 1.21]{MR1753713}, the linear system \cref{eqn:schemes_linear_system} has a unique solution.

\begin{remark}
    If the coefficient of the second spatial derivative does not depend on time (i.e., $b(t, x) = b(0, x)$), then the matrix $A$ does not depend on the timestep $n$.
    In this case, we can speed up computation by factoring (or preconditioning) the matrix $A$ exactly once instead of at each timestep.
    \label{rem:schemes_factorization}
\end{remark}

The explicit-impulse scheme was first introduced in our own work \cite{MR3493959}.

\section{Convergence rates}
\label{sec:schemes_convergence_rates}

We close this chapter with a study of convergence rates.
Doing so will help our understanding of the convergence rates witnessed in experiments appearing in the sequel.

To simplify notation, we assume temporarily that the grid points $\{x_0, \ldots, x_M\}$ are uniformly spaced (i.e., $x_{i+1} - x_i = \Delta x$).
Now, let $\varphi$ be a smooth function (of time and space) and $\varphi^n = (\varphi(\tau^n, x_0), \ldots, \varphi(\tau^n, x_M))^\intercal$ be a vector whose components are obtained by evaluating $\varphi(\tau^n, \cdot)$ on the spatial grid.
Using Taylor series, it follows that
\[
    (\mathcal{D}_2 \varphi^n)_i = \varphi_{xx}(\tau^n, x_i) + O((\Delta x)^2)
    \textspace \text{and} \textspace
    (\mathcal{D}_{\pm} \varphi^n)_i = \varphi_x(\tau^n, x_i) + O(\Delta x)
    \textspace \text{if } 0 < i < M
\]
Similarly,\footnote{We are, for the time being, ignoring ``overstepping'' error that occurs in \cref{eqn:schemes_interpolation_taylor} when the point $x$ does not lie between two grid points. This issue is handled rigorously in the sequel (see, in particular, the text preceding \cref{lem:convergence_explicit_impulse_consistency} of \cref{chap:convergence}).}
\begin{equation}
    \gls*{interp}(\varphi^n, x) = \varphi(\tau^n, x) + O((\Delta x)^2).
    \label{eqn:schemes_interpolation_taylor}
\end{equation}
Assuming $\Gamma$ is a bounded function, it follows from the above that
\[
    (\mathcal{M}_n \varphi^n)_i
    \approx \sup_{\textcolor{ctrlcolor}{z_i} \in Z(\tau^n, x_i)}
    \left\{
        \gls*{interp}(
            \varphi^n,
            \Gamma(\tau^n, x_i, \textcolor{ctrlcolor}{z_i})
        ) + K(\tau^n, x_i, \textcolor{ctrlcolor}{z_i})
    \right\} 
    = \mathcal{M}\varphi(\tau^n, x_i) + O((\Delta x)^2).
\]
Moreover, substituting $\varphi$ into the right hand side of \cref{eqn:schemes_lagrangian_derivative},
\begin{multline*}
    \frac{
        \gls*{interp}(\varphi^{n-1}, x_i + a_i^n(\textcolor{ctrlcolor}{w}) \Delta \tau) - \varphi(\tau^n, x_i)
    }{\Delta \tau} \\
    = \frac{
        \varphi(\tau^{n-1}, x_i + a_i^n(\textcolor{ctrlcolor}{w}) \Delta \tau) - \varphi(\tau^n, x_i)
    }{\Delta \tau}
    + O \left( \frac{(\Delta x)^2}{\Delta \tau} \right) \\
    = \varphi_t(\tau^n, x_i) + a(\tau^n, x_i, \textcolor{ctrlcolor}{w}) \varphi_x(\tau^n, x_i)
    + O \left( \frac{(\Delta x)^2}{\Delta \tau} + \Delta \tau \right).
\end{multline*}

Assuming that the timestep and spatial grid size are of the same order (i.e., $\Delta \tau = \gls*{const} \gls*{meshing_parameter}$ and $\Delta x = \gls*{const} \gls*{meshing_parameter}$), the above approximations suggest that the schemes of this chapter are all first order accurate.
Obtaining higher order accuracy requires a higher order discretization of, for example, the first order spatial derivative.
However, as we will see in \cref{chap:convergence}, the use of higher order approximations violates a key requirement known as monotonicity, which is used to ensure the theoretical convergence of our schemes (nonmonotone schemes may fail to converge in the viscosity sense; see \cite[Section 1.3]{MR2218974} for a simple example).


\setcounter{chapter}{2}
\chapter{Convergence of policy iteration}
\label{chap:matrix}

We saw in \cref{chap:schemes} that a numerical solution of the HJBQVI \cref{eqn:introduction_hjbqvi} can be obtained by the explicit-impulse scheme if the time horizon is finite and the second derivative coefficient is independent of the control.
If either of these assumptions are violated, then we should use the direct control or penalty scheme.
As we will see in this chapter, solutions of these schemes are obtained by solving nonlinear matrix equations of the form
\begin{equation}
    \text{find }
    U \in \mathbb{R}^{M+1}
    \text{ such that }
    \sup_{P\in\mathcal{P}} \left\{
        -A(P) U + y(P)
    \right\} = 0
    \label{eqn:matrix_bellman}
\end{equation}
where $A(P)$ is a real square matrix, $y(P)$ is a real vector, and $\mathcal{P}$ is a set.

Problem \cref{eqn:matrix_bellman} is commonly referred to as a \emph{Bellman problem} and is usually solved by one of two procedures.
The first is value iteration, which -- in the context of numerical schemes for PDEs -- exhibits slow convergence as the grid is refined \cite[Section 6.1]{forsyth2007numerical}.
The second is Howard's \emph{policy iteration}, which is celebrated for its superlinear convergence \cite[Section 3]{MR2551155}.
We consider only the latter.
The policy iteration procedure given in this thesis is our own extension of Howard's original policy iteration that makes no assumptions about the compactness of the set $\mathcal{P}$ or continuity of the functions $A$ and $y$ (compare with, e.g., the standard setting in \cite{MR2551155}).

In the case of the penalty scheme, convergence of policy iteration is a consequence of the strict diagonal dominance of the matrices $A(P)$.
The direct control scheme, however, involves matrices that are singular and as such, more delicate arguments are required to ensure convergence.

The paper \cite{MR2302036} gives sufficient conditions for the convergence of policy iteration as applied to the direct control scheme.
In that paper, the authors consider a more general version of the direct control scheme involving nonlinear operators in lieu of matrices.
As a result of this, the sufficient conditions given therein are too conservative to be applied to the problems appearing in this thesis.
In addition, convergence in \cite{MR2302036} depends on the choice of initial guess. 
Our goal is to give more lenient sufficient conditions for convergence that are also independent of the choice of initial guess.
We also apply our findings to obtain some new results for infinite horizon Markov decision processes (\glspl*{MDP}).

In order to do so, we use \emph{weakly chained diagonally dominant} (\gls*{wcdd}) matrices.
As a byproduct of our analysis, we prove that \gls*{wcdd} matrices give a graph-theoretic characterization of weakly diagonally dominant \emph{M-matrices}.
This generalizes some well-known results regarding M-matrices.
M-matrices are common in the scientific computing community since they arise naturally in discretizations of elliptic operators (e.g., the Laplacian $\Delta$), Markov decision processes, linear complementarity problems, etc.

Our contributions in this chapter are:
\begin{itemize}
    \item
        Extending Howard's original policy iteration to a setting in which no assumptions are made about the compactness of $\mathcal{P}$ or continuity of $A$ and $y$ (\cref{thm:matrix_policy_iteration}).
    \item
        Providing a graph-theoretic characterization of weakly diagonally dominant M-matrices (\cref{thm:matrix_characterization}).
    \item
        Showing that policy iteration applied to the penalty scheme always converges.
    \item
        Showing that a na\"{i}ve application of policy iteration to the direct control scheme can (and often does) fail.
    \item
        Using the theory of \gls*{wcdd} matrices to provide a provably convergent modification of policy iteration for the direct control scheme.
    \item
        Applying our findings to obtain new results for infinite horizon \glspl*{MDP}.
\end{itemize}

The results of this chapter appear in our articles

\fullcite{MR3493959}

\fullcite{azimzadeh2017fast}

\section{Policy iteration}
\label{sec:matrix_policy_iteration}

In this section, we give some results regarding policy iteration.
In order to simplify notation, we work with the general form of the Bellman problem \cref{eqn:matrix_bellman}, for which it is understood that
\begin{enumerate}[label=(\roman*)]
    \item
        $\mathcal{P} = \mathcal{P}_{0} \times \cdots \times \mathcal{P}_{M}$ is a product of nonempty sets.
    \item
        Controls are ``row-decoupled'' (cf. \cite{MR2551155}).
        That is, using the notation $P = (P_0, \ldots, P_M)$ to denote a control in $\mathcal{P}$,
        \[
            [A(P)]_{ij} \text{ and } [y(P)]_i \text{ depend only on } P_i \in \mathcal{P}_i.
        \]
        In other words, the $i$-th row of $A(P)$ and $y(P)$ are determined solely by the $i$-th coordinate of the control $P$.
    \item
        The order on $\mathbb{R}^{M+1}$ (and $\mathbb{R}^{(M+1) \times (M+1)}$) is element-wise:
        \[
            \text{for } x, y \in \mathbb{R}^{M+1} \text{, } x \geq y
            \text{ if and only if } x_i \geq y_i \text{ for all } i.
        \]
    \item
        The supremum in \cref{eqn:matrix_bellman} is with respect to the element-wise order:
        \[
            \text{for }
            \left\{ y(P) \right\}_{P \in \mathcal{P}}
            \subset \mathbb{R}^{M+1} \text{, }
            x = \sup_{P \in \mathcal{P}} y(P)
            \text{ is a vector with components }
            x_i = \sup_{P \in \mathcal{P}}[y(P)]_i.
        \]
\end{enumerate}

We recall Howard's policy iteration procedure \cite[Algorithm Ho-1]{MR2551155} to solve problem \cref{eqn:matrix_bellman} below.

\vspace{0.25em}

\algrenewcommand\algorithmicdo{\textbf{until convergence} (or max iteration count reached)}
\begin{algorithmic}[1]
    \Procedure{Policy-Iteration}{$\mathcal{P}$}
        \State Pick an initial guess $U^0 \in \mathbb{R}^{M + 1}$
        \For{$\ell \gets 1,2,\ldots$}
            \State Pick $P^\ell$ such that $-A(P^\ell) U^{\ell - 1} + y(P^\ell) = \sup_{P \in \mathcal{P}} \{ -A(P) U^{\ell - 1} + y(P) \}$ \label{line:matrix_exact_supremum}
            \State Solve the linear system $A(P^\ell) U^\ell = y(P^\ell)$ to obtain $U^\ell$
        \EndFor
        \State \Return $U^\ell$
    \EndProcedure
\end{algorithmic}

We have purposely not specified a convergence criterion since the results of this chapter characterize convergence of the iterates $U^\ell$ as $\ell \rightarrow \infty$.
In a practical implementation, the iteration should be terminated when the relative error between subsequent iterates is desirably small (see \cref{eqn:results_convergence_criterion} in \cref{chap:results}).

Before continuing, we recall a definition from linear algebra.

\begin{definition}[{\cite[Section II.23]{MR0205126}}]
    A real square matrix $A$ is monotone if 
    it is nonsingular and $A^{-1} \geq 0$ (i.e., the entries of $A^{-1}$ are nonnegative).
\end{definition}

We can now state sufficient conditions for the convergence of {\sc Policy-Iteration}.

\begin{proposition}[\cite{MR2551155} Theorem 2.1]
    Suppose that $\mathcal{P}$ is a compact topological space, $A$ and $y$ are continuous functions, and that $A(P)$ is monotone for each $P \in \mathcal{P}$.
    Then, the sequence $(U^\ell)_{\ell \geq 1}$ defined by \Call{Policy-Iteration}{$\mathcal{P}$} converges from below to the unique solution of \cref{eqn:matrix_bellman}.
    Moreover, if $\mathcal{P}$ is finite, convergence occurs in at most $|\mathcal{P}|$ iterations (i.e., $U^{|\mathcal{P}|} = U^{|\mathcal{P}| + 1} = \cdots$).
    \label{prop:matrix_policy_iteration}
\end{proposition}

The compactness and continuity assumptions in \cref{prop:matrix_policy_iteration} ensure that the supremum on \cref{line:matrix_exact_supremum} of {\sc Policy-Iteration} is attained at a point in $\mathcal{P}$.
To remove these assumptions (specifically in order to strengthen our results in our study of \glspl*{MDP}), we introduce below an extension of Howard's policy iteration.
In the pseudocode description of the procedure, we use $\vec{e} = (1, \ldots, 1)^\intercal$ to denote the vector whose entries are one.

\vspace{0.25em}

\begin{algorithmic}[1]
    \Procedure{$\epsilon$-Policy-Iteration}{$\mathcal{P}$}
        \State Pick an initial guess $U^0 \in \mathbb{R}^{M + 1}$
        \State Pick a sequence $(\epsilon^k)_{k \geq 1}$ of positive numbers satisfying $\sum_k \epsilon^k < \infty$ \Comment{e.g., $\epsilon^k = 1/k^2$}
        \For{$\ell \gets 1,2,\ldots$}
            \State Pick $P^\ell$ such that $-A(P^\ell) U^{\ell - 1} + y(P^\ell) + \epsilon^\ell \vec{e} \geq \sup_{P \in \mathcal{P}} \{ -A(P) U^{\ell - 1} + y(P) \}$ \label{line:matrix_supremum}
            \State Solve the linear system $A(P^\ell) U^\ell = y(P^\ell)$ to obtain $U^\ell$ \label{line:matrix_solve}
        \EndFor
        \State \Return $U^\ell$
    \EndProcedure
\end{algorithmic}
\algrenewcommand\algorithmicdo{\textbf{do}}

We can think of the parameter $\epsilon^\ell$ in the above as the maximum error allowed in approximating the supremum on \cref{line:matrix_supremum} at each step $\ell$.

To establish convergence of {\sc $\epsilon$-Policy-Iteration}, we require some assumptions:

\begin{enumerate}[label=(H\arabic*)]
    \item
        \label{enum:matrix_bounded_inverse}
        The function $P \mapsto A(P)^{-1}$ is bounded on the set $\{ P \in \mathcal{P} \colon A(P) \text{ is nonsingular} \}$.
    \item
        \label{enum:matrix_bounded}
        The functions $A$ and $y$ are bounded.
\end{enumerate}

Note that the boundedness in assumptions \cref{enum:matrix_bounded_inverse,enum:matrix_bounded} is stated without reference to a particular norm (all norms on a finite dimensional vector space are equivalent).
We are now ready to state the convergence result.

\begin{theorem}
    Suppose \cref{enum:matrix_bounded_inverse}, \cref{enum:matrix_bounded}, and that $A(P)$ is monotone for each $P \in \mathcal{P}$.
    Then, the sequence $(U^\ell)_\ell$ defined by \Call{$\epsilon$-Policy-Iteration}{$\mathcal{P}$} converges to the unique solution of \cref{eqn:matrix_bellman}.
    \label{thm:matrix_policy_iteration}
\end{theorem}

To prove the result, we require a few lemmas.
The first is just a general form of the results \cite[Eq. (5.8)]{MR2551155} and \cite[Proposition 3.1]{MR3022201}.

\begin{lemma}
    Let $(x^\ell)_{\ell \geq 0}$ be a sequence of real numbers that is bounded from above.
    Suppose there exists a sequence $(\epsilon^\ell)_{\ell \geq 1}$ of positive numbers such that $\sum_\ell \epsilon^\ell < \infty$ and $x^\ell - x^{\ell - 1} \geq -\epsilon^\ell$ for $\ell \geq 1$.
    Then, the sequence $(x^\ell)_{\ell \geq 0}$ converges to a real number.
    \label{lem:proofs_matrix_policy_iteration_1}
\end{lemma}

\begin{proof}
    Note that
    \begin{equation}
        x^q - x^p
        = \left( x^q - x^{q-1} \right) + \cdots + \left( x^{p+1} - x^p \right)
        \geq - \sum_{\ell = p+1}^q \epsilon^\ell
        > - \sum_{\ell > p} \epsilon^\ell
        \label{eqn:proofs_matrix_policy_iteration_0}
    \end{equation}
    for any integers $q$ and $p$ satisfying $q \geq p \geq 0$.
    Picking $p = 0$ in \cref{eqn:proofs_matrix_policy_iteration_0}, we get $x^q > x^0 - \sum_\ell \epsilon^\ell$, establishing that the sequence $(x^\ell)_\ell$ is also bounded from below.
    Taking limits in \cref{eqn:proofs_matrix_policy_iteration_0}, we get $\liminf_q x^q - \limsup_p x^p \geq 0$.
    Since $\liminf_q x^q - \limsup_p x^p \leq 0$ by definition, this implies that $\liminf_q x^q = \limsup_p x^p$, and hence the sequence $(x^\ell)_\ell$ has a limit.
\end{proof}

\begin{lemma}
    The function $H$ defined by $H(U) = \sup_{P \in \mathcal{P}} \{ -A(P) U + y(P) \}$ is Lipschitz continuous.
    \label{lem:proofs_matrix_policy_iteration_2}
\end{lemma}

\begin{proof}
    The claim follows from the fact that for any two vectors $U$ and $\hat{U}$,
    \begin{multline*}
        \left \Vert
            H(U) - H(\hat{U})
        \right \Vert_\infty
        \leq \sup_{P \in \mathcal{P}} \left \Vert
            - A(P) (U - \hat{U}) + y(P) - y(P)
        \right \Vert_\infty
        \\
        = \sup_{P \in \mathcal{P}} \left \Vert
            A(P) ( U - \hat{U} )
        \right \Vert_\infty
        \leq \left(
            \sup_{P \in \mathcal{P}} \left \Vert
                A(P)
            \right \Vert_\infty
        \right)
        \left \Vert
            U - \hat{U}
        \right \Vert_\infty
        = \gls*{const} \left \Vert
            U - \hat{U}
        \right \Vert_\infty.
    \end{multline*}
    The last equality follows from the boundedness of the function $A$ in \cref{enum:matrix_bounded}.
\end{proof}

\begin{lemma}
    Let $U$ be a solution of \cref{eqn:matrix_bellman} and $\hat{U}$ be a vector.
    Suppose 
    we can find a sequence $(P^\ell)_\ell$ such that $A(P^\ell)$ is monotone for each $\ell$ and
    \begin{equation}
        - A(P^\ell) \hat{U} + y(P^\ell) \rightarrow 0.
        \label{eqn:matrix_one_side}
    \end{equation}
    Then, $U - \hat{U} \geq 0$.
    \label{lem:matrix_one_side}
\end{lemma}

\begin{proof}
    By \cref{eqn:matrix_one_side}, we can find a sequence $(\delta^\ell)_\ell$ of vectors converging to zero such that
    \begin{equation}
        -A(P^\ell) \hat{U} + y(P^\ell) + \delta^\ell = 0.
        \label{eqn:matrix_one_side_1}
    \end{equation}
    Moreover, since $U$ is a solution of \cref{eqn:matrix_bellman},
    \begin{equation}
        0 = \sup_{P \in \mathcal{P}} \left \{
            -A(P) U + y(P)
        \right \}
        \geq -A(P^\ell) U + y(P^\ell).
        \label{eqn:matrix_one_side_2}
    \end{equation}
    Combining \cref{eqn:matrix_one_side_1,eqn:matrix_one_side_2},
    \begin{equation}
        U - \hat{U}
        \geq A(P^\ell)^{-1} (- \delta^\ell )
        \geq - \gls*{const} \delta^\ell.
        \label{eqn:matrix_one_side_3}
    \end{equation}
    In the above, we have used the monotonicity of $A(P^\ell)$ (which implies $A(P^\ell)^{-1} \geq 0$) along with \cref{enum:matrix_bounded_inverse}.
    Taking limits in \cref{eqn:matrix_one_side_3}, we find that $U - \hat{U} \geq 0$.
\end{proof}

We can now prove \cref{thm:matrix_policy_iteration}:

\begin{proof}[Proof of \cref{thm:matrix_policy_iteration}]
    By \cref{line:matrix_solve,line:matrix_supremum} of {\sc $\epsilon$-Policy-Iteration},
    \begin{equation}
        A(P^\ell) \left( U^\ell - U^{\ell - 1} \right)
        = -A(P^\ell) U^{\ell - 1} + y(P^\ell)
        \geq \sup_{P \in \mathcal{P}} \left\{
            -A(P) U^{\ell - 1} + y(P)
        \right\} - \epsilon^{\ell} \vec{e}
        \label{eqn:proofs_matrix_policy_iteration_1}
    \end{equation}
    and, for $\ell > 1$,
    \begin{equation}
        \sup_{P \in \mathcal{P}} \left\{
            -A(P) U^{\ell - 1} + y(P)
        \right\}
        \geq -A(P^{\ell - 1}) U^{\ell - 1} + y(P^{\ell - 1})
        = 0.
        \label{eqn:proofs_matrix_policy_iteration_2}
    \end{equation}
    Combining \cref{eqn:proofs_matrix_policy_iteration_1,eqn:proofs_matrix_policy_iteration_2}, we find that for $\ell > 1$,
    \begin{equation}
        A(P^\ell) \left(
            U^\ell - U^{\ell - 1}
        \right) + \epsilon^\ell \vec{e}
        \geq \sup_{P \in \mathcal{P}} \left\{
            -A(P) U^{\ell - 1} + y(P)
        \right\}
        \geq 0
        \label{eqn:proofs_matrix_policy_iteration_3}
    \end{equation}
    and hence, similarly to the proof of \cref{lem:matrix_one_side},
    \[
        U^\ell - U^{\ell - 1}
        \geq A(P^\ell)^{-1} ( - \epsilon^\ell \vec{e} \, )
        \geq - \gls*{const} \epsilon^\ell \vec{e} \, .
    \]
    Note also that since $U^\ell = A(P^\ell)^{-1} y(P^\ell)$ for $\ell \geq 1$, the sequence $(U^\ell)_\ell$ is bounded by \cref{enum:matrix_bounded_inverse,enum:matrix_bounded}.
    Therefore, we can apply \cref{lem:proofs_matrix_policy_iteration_1} to conclude that the sequence $(U^\ell)_\ell$ converges to some real vector $U$.
    Now, since $U^\ell \rightarrow U$ and $\epsilon^\ell \rightarrow 0$,
    \begin{equation}
        \lim_{\ell \rightarrow \infty} \left\{
            A(P^\ell) \left(
                U^\ell - U^{\ell - 1}
            \right) + \epsilon^\ell \vec{e}
        \right\}
        = 0.
        \label{eqn:proofs_matrix_policy_iteration_4}
    \end{equation}
    Taking limits in \cref{eqn:proofs_matrix_policy_iteration_3} and applying \cref{eqn:proofs_matrix_policy_iteration_4},
    \[
        0
        = \lim_{\ell \rightarrow \infty} \left\{
            A(P^\ell) \left(
                U^\ell - U^{\ell - 1}
            \right) + \epsilon^\ell \vec{e}
        \right\}
        \geq \lim_{\ell \rightarrow \infty} \sup_{P \in \mathcal{P}} \left\{
            -A(P) U^{\ell - 1} + y(P)
        \right\}
        \geq 0
    \]
    and hence
    \[
        \lim_{\ell \rightarrow \infty} \sup_{P \in \mathcal{P}}
        \left\{ -A(P) U^{\ell - 1} + y(P) \right\} = 0.
    \]
    Using the notation of \cref{lem:proofs_matrix_policy_iteration_2}, we can rewrite the above as
    \begin{equation}
        \lim_{\ell \rightarrow \infty} H(U^{\ell-1}) = 0
        \label{eqn:proofs_matrix_policy_iteration_5}
    \end{equation}
    Since limits and continuous functions commute, \cref{lem:proofs_matrix_policy_iteration_2} implies
    \[
        \lim_{\ell \rightarrow \infty} H(U^{\ell-1})
        = H(\lim_{\ell \rightarrow \infty} U^{\ell-1})
        = H(U)
    \]
    and hence $H(U) = 0$ by \cref{eqn:proofs_matrix_policy_iteration_5}.
    Equivalently, substituting in the definition of $H$,
    \[
        \sup_{P \in \mathcal{P}} \left\{ -A(P)U + y(P) \right\} = 0.
    \]
    That is, $U$ is a solution of \cref{eqn:matrix_bellman}.

    It remains only to prove uniqueness.
    Let $U$ and $\hat{U}$ be solutions of \cref{eqn:matrix_bellman}.
    It follows that we can find a sequence $(P^\ell)_\ell$ such that
    \[
        -A(P^\ell) \hat{U} + y(P^\ell) \rightarrow 0.
    \]
    Therefore, by \cref{lem:matrix_one_side}, $U - \hat{U} \geq 0$.
    Switching the roles of $U$ and $\hat{U}$ in the above gives $\hat{U} - U \geq 0$.
    Therefore, $U = \hat{U}$.
\end{proof}


We close this section by pointing out that it is sometimes advantageous to scale the inputs to policy iteration in order to obtain faster convergence.
While this was observed in \cite{MR3102414} for a specific instance of problem \cref{eqn:matrix_bellman}, we give a general result which codifies precisely when this scaling can be performed without changing the set of solutions to \cref{eqn:matrix_bellman}.
This will be used in the sequel 
to speed up computations.

\begin{lemma}
    \label{lem:matrix_scaling}
    Let $s \colon \mathcal{P} \rightarrow \mathbb{R}^{M+1}$ be a vector-valued function such that
    \begin{equation}
        \inf_{P \in \mathcal{P}} \min_i [s(P)]_i > 0
        \textspace \text{and} \textspace
        \sup_{P \in \mathcal{P}} \max_i [s(P)]_i < \infty.
        \label{eqn:matrix_scaling_1}
    \end{equation}
    Then, $U$ is a solution of \cref{eqn:matrix_bellman} if and only if
    \begin{equation}
        \sup_{P \in \mathcal{P}} \left\{
            \gls*{diag}(s(P)) \left( -A(P)U + y(P) \right)
        \right\} = 0.
        \label{eqn:matrix_scaling_2}
    \end{equation}
\end{lemma}

\begin{proof}
    First, we prove that if $U$ is a solution of \cref{eqn:matrix_bellman} then \cref{eqn:matrix_scaling_2} holds.
    Indeed, if $U$ is a solution of \cref{eqn:matrix_bellman}, there exists a sequence $(P^\ell)_\ell$ such that
    \[
        - A(P^\ell) U + y(P^\ell) \uparrow 0
    \]
    where we have used the symbol $\uparrow$ to mean that the limit approaches zero monotonically from below so that $- A(P^\ell) U + y(P^\ell) \leq 0$ for each $\ell$.
    Since $s$ is bounded above by \cref{eqn:matrix_scaling_1},
    \[
        \limsup_{\ell \rightarrow \infty} \left\{
            \gls*{diag}(s(P^\ell))
            \left( -A(P^\ell)U + y(P^\ell) \right)
        \right\}
        \geq
        \gls*{const}
        \lim_{\ell \rightarrow \infty} \left\{
            -A(P^\ell)U + y(P^\ell)
        \right\}
        = 0,
    \]
    and so
    \[
        \sup_{P \in \mathcal{P}} \left\{
            \gls*{diag}(s(P))
            \left( -A(P)U + y(P) \right)
        \right\} \geq 0.
    \]
    Note that if the above holds with equality, we have established \cref{eqn:matrix_scaling_2}, as desired.
    Therefore, in order to arrive at a contradiction, suppose the above holds with strict inequality.
    In this case, we can find some $P_0$ for which
    \[
        \gls*{diag}(s(P_0)) \left(
            -A(P_0)U + y(P_0)
        \right) > 0.
    \]
    Since $s$ is bounded below by \cref{eqn:matrix_scaling_1}, $s(P_0) > 0$ and hence we can multiply both sides of the above inequality by the inverse of $\gls*{diag}(s(P_0))$ to obtain
    \[
        -A(P_0)U + y(P_0) > 0,
    \]
    contradicting the assumption that $U$ is a solution of \cref{eqn:matrix_bellman}.

    The converse is handled by symmetry.
    In particular, suppose that $U$ satisfies \cref{eqn:matrix_scaling_2}.
    Defining $A_0(P) = \gls*{diag}(s(P)) A(P)$ and $y_0(P) = \gls*{diag}(s(P)) y(P)$, \cref{eqn:matrix_scaling_2} is equivalent to
    \[
        \sup_{P \in \mathcal{P}} \left\{
            - A_0(P) U + y_0(P)
        \right\} = 0.
    \]
    The arguments of the previous paragraph imply that
    \[
        \sup_{P \in \mathcal{P}} \left\{
            - A(P) U + y(P)
        \right\}
        = \sup_{P \in \mathcal{P}} \left\{
            \gls*{diag}(s(P))^{-1}
            \left( - A_0(P) U + y_0(P) \right)
        \right\}
        = 0,
    \]
    as desired.
\end{proof}

\section{Weakly chained diagonally dominant matrices}

In this section, we describe weakly chained diagonally dominant (\gls*{wcdd}) matrices \cite{MR0332820} and give some new results relating them to M-matrices.
We will ultimately apply the theory developed in this section to the Bellman problem \cref{eqn:matrix_bellman}.
We first recall a few definitions from linear algebra.

\begin{definition}
    Let $A = (a_{ij})$ be a complex matrix.
    Row $i$ of $A$ is strictly diagonally dominant (\gls*{sdd}) if $|a_{ii}| > \sum_{j \neq i} | a_{ij} |$.
    The matrix $A$ is \gls*{sdd} if all of its rows are \gls*{sdd}
    Similarly, row $i$ of $A$ is weakly diagonally dominant (\gls*{wdd}) if $|a_{ii}| \geq \sum_{j \neq i} | a_{ij} |$.
    The matrix $A$ is \gls*{wdd} if all of its rows are \gls*{wdd}
\end{definition}

\begin{definition}
    Let $A = (a_{ij})_{i,j \in \{0, \ldots, M\}}$ be an $(M+1) \times (M+1)$ complex matrix.
    The (directed) adjacency graph of $A$ is given by the vertices $\{0, \ldots, M\}$ and edges defined as follows: there exists an edge from $i$ to $j$ if and only if $a_{ij} \neq 0$.
\end{definition}

In light of the above, we use the terms ``row'' and ``vertex'' interchangeably.
To simplify notation, we write $i \rightarrow j$ to denote an edge from $i$ to $j$.
Similarly, we write $i_1 \rightarrow \cdots \rightarrow i_k$ to denote a walk from vertex $i_1$ to vertex $i_k$ passing through vertices $i_2$, $i_3$, etc.

We are now ready to define \gls*{wcdd} matrices.

\begin{definition}[\cite{MR0332820}]
    A square complex matrix $A$ is \gls*{wcdd} if it is \gls*{wdd} and for each row $i_1$ that is not \gls*{sdd}, there exists a walk $i_1 \rightarrow \cdots \rightarrow i_k$ in the adjacency graph of $A$ such that $i_k$ is \gls*{sdd} (\cref{fig:matrix_wcdd}).
    \label{def:matrix_wcdd}
\end{definition}

\begin{figure}
    \subfloat[An $(M+1) \times (M+1)$ matrix]{
        $
            A = \begin{pmatrix}
                +1\\
                -1 & +1\\
                & -1 & +1\\
                &  & \ddots & \ddots\\
                &  &  & -1 & +1
            \end{pmatrix}
        $
    }
    \hfill{}
    \subfloat[Directed adjacency graph of $A$ (vertices corresponding to \gls*{sdd} vertices are \colorbox{colA}{\textcolor{white}{highlighted}})]{
        \begin{tikzpicture}[node distance=2cm]
            \node [ellipse, fill=colA, text=white] (0) {0};
            \node [ellipse, fill=gray!20, right of=0] (1) {1};
            \node [right of=1] (dots) {$\cdots$};
            \node [ellipse, fill=gray!20, right of=dots] (end) {$M$};
            \draw [thick, ->, loop above] (0) to (0);
            \draw [thick, ->, loop above] (1) to (1);
            \draw [thick, ->, loop above] (end) to (end);
            \draw [thick, ->] (end) -- (dots);
            \draw [thick, ->] (dots) -- (1);
            \draw [thick, ->] (1) -- (0);
        \end{tikzpicture}
    }
    \caption{An example of a \gls*{wcdd} matrix and its adjacency graph}
    \label{fig:matrix_wcdd}
\end{figure}

Below, we recall the definition of Z-matrices and M-matrices.

\begin{definition}[{\cite{MR0444681}}]
    A Z-matrix is a real matrix with nonpositive off-diagonals.
    A nonsingular M-matrix is a monotone Z-matrix.
\end{definition}

In this thesis, we will often drop the term ``nonsingular'' and simply refer to a nonsingular M-matrix as an M-matrix (there is a notion of singular M-matrix in the literature, though we do not use it).
The main result of this section is given below.

\begin{theorem}
    If $A$ is a square \gls*{wdd} Z-matrix with nonnegative diagonals, then the following are equivalent:
    \begin{enumerate}[label=(\roman*)]
        \item
            \label{enum:matrix_characterization_1}
            $A$ is a (nonsingular) M-matrix.
        \item
            \label{enum:matrix_characterization_2}
            $A$ is a nonsingular matrix.
        \item
            \label{enum:matrix_characterization_3}
            $A$ is a \gls*{wcdd} matrix.
    \end{enumerate}
    \label{thm:matrix_characterization}
\end{theorem}

\begin{proof}
    \cref{enum:matrix_characterization_1} implies \cref{enum:matrix_characterization_2} is true since M-matrices are monotone by definition and monotone matrices are nonsingular by definition.
    \cref{enum:matrix_characterization_3} implies \cref{enum:matrix_characterization_1} is proven in \cite{MR0162367}.\footnote{\cite{MR0162367} refers to \gls*{wcdd} Z-matrices with nonnegative diagonals as \emph{matrices of positive type}.}
    Therefore, it remains to prove that \cref{enum:matrix_characterization_2} implies \cref{enum:matrix_characterization_3}.

    To this end, let $A$ be an $(M+1) \times (M+1)$ \gls*{wdd} Z-matrix with nonnegative diagonals that is not \gls*{wcdd}
    We will prove that $A$ is singular, thereby establishing \cref{enum:matrix_characterization_2} implies \cref{enum:matrix_characterization_3} by contrapositive.
    Let $J$ be the set of all \gls*{sdd} rows (of $A$).
    Let $W$ be the set of all non-\gls*{sdd} rows $i_1$ satisfying \cref{def:matrix_wcdd}.
    Specifically,
    \[
        W = \left \{
            i_1 \notin J \colon
            \text{there exists a walk }
            i_1 \rightarrow \cdots \rightarrow i_k
            \text{ such that }
            i_k \in J
        \right\}
    \]
    Let $R = \{ 0, \ldots, M \} \setminus (J \cup W)$.
    Since $A$ is not \gls*{wcdd}, $R$ is nonempty.
    Without loss of generality, we may assume $R = \{0, \ldots, r\}$ where $0 \leq r \leq M$ since otherwise, we can relabel the rows and columns of $A$ (this corresponds to replacing $A$ by $PAP^\intercal$ where $P$ is an appropriately chosen permutation matrix).

    If $r = M$, it follows that $J$ is empty, and hence every row of $A$ is not \gls*{sdd}
    In this case, since $A$ is a \gls*{wdd} Z-matrix with nonnegative diagonals, we have $a_{ii} = -\sum_{j \neq i} a_{ij}$ for each row $i$.
    In other words, the row sums of $A$ are zero so that $A \vec{e} = 0$ and hence $A$ is singular.

    If $r < M$, the adjacency graph of $A$ has the structure shown in \cref{fig:matrix_proof_adjacency_graph}.
    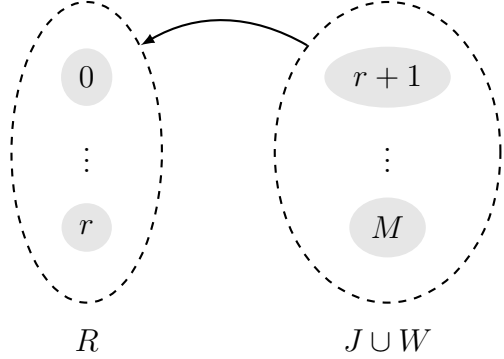
\begin{figure}
        \centering
        \begin{tikzpicture}
            \node [ellipse, fill=gray!20] (0) {$0$};
            \node [below of=0] (Rdots) {$\vdots$};
            \node [ellipse, fill=gray!20, below of=Rdots] (r) {$r$};
            \node [ellipse, fill=gray!20] at (4, 0) (r1) {$r+1$};
            \node [below of=r1] (SWdots) {$\vdots$};
            \node [ellipse, fill=gray!20, below of=SWdots] (M) {$M$};
            \node [ellipse, thick, dashed, draw, minimum width=2cm, minimum height=4cm] at (Rdots) (R) {};
            \node [ellipse, thick, dashed, draw, minimum width=3cm, minimum height=4cm] at (SWdots) (SW) {};
            \path [thick, ->] (SW.north west) edge [bend right] (R.north east);
            \node [below of=R, yshift=-1.5cm] {$R$};
            \node [below of=SW, yshift=-1.5cm] {$J \cup W$};
        \end{tikzpicture}
        \caption[``Block'' adjacency graph of a w.c.d.d. matrix]{``Block'' adjacency graph of the w.c.d.d. matrix in the proof of \cref{thm:matrix_characterization}}
        \label{fig:matrix_proof_adjacency_graph}
    \end{figure}
    In particular, there are no edges from a vertex $i \in R$ to a vertex $j \in J \cup W$ since if there were, $i$ would not be a member of $R$ by definition.
    This implies that $A$ has the block structure
    \[
        A = \left( \begin{array}{c|c}
            B & 0 \\
            \hline
            C & D
        \end{array} \right)
        \text{ where } B \in \mathbb{R}^{(r + 1) \times (r + 1)}.
    \]
    where the zero in the top right quadrant corresponds to the absence of edges from vertices in $R$ to vertices in $J \cup W$.
    Moreover, the partition above guarantees that $D$ is \gls*{wcdd} since $i \notin R$ for all $i > r$.
    Since \gls*{wcdd} matrices are nonsingular \cite{MR0332820}, the linear system $Dx = -C \vec{e}$ has a unique solution $x$.
    Similarly to the case of $r = M$, since every row of $B$ is not \gls*{sdd}, the row sums of $B$ are zero so that $B \vec{e} = 0$.
    It follows that
    \[
        A \begin{pmatrix} \vec{e} \\ x \end{pmatrix}
        = \begin{pmatrix} B \vec{e} \\ C \vec{e} + D x \end{pmatrix}
        = 0,
    \]
    and hence $A$ is singular.
\end{proof}

The result above was first proven in our own work \cite{MR3493959,azimzadeh2017fast} and immediately strengthens some existing results from linear algebra.
For example, due to the above, well-known bounds on the minimum eigenvalue and infinity-norm of the inverse of an arbitrary \gls*{wcdd} M-matrix \cite{MR1384509,MR2350685,MR2433738,MR2535528,MR2577710} apply to \gls*{wdd} M-matrices.

\section{Application to the HJBQVI}
\label{sec:matrix_hjbqvi}

In this section, we revisit the direct control and penalty schemes of \cref{chap:schemes}.
Our goal is to see if given the solution $V^{n-1}$ of one of these schemes at the $(n - 1)$-th timestep, we can use policy iteration to compute the solution $V^n$ at the $n$-th timestep.

Recall the discretized control sets $W^{\gls*{meshing_parameter}}$ and $Z^{\gls*{meshing_parameter}}(t, x)$ introduced in \cref{chap:schemes}.
For the remainder of this section, we let
\[
    \mathcal{P}_i
    = W^{\gls*{meshing_parameter}}
    \times Z^{\gls*{meshing_parameter}}(\tau^n, x_i)
    \times \left \{ 0, 1 \right \}.
\]
Intuitively, each $P_i = (\textcolor{ctrlcolor}{w_i}, \textcolor{ctrlcolor}{z_i}, \textcolor{ctrlcolor}{d_i}) \in \mathcal{P}_i$ is a tuple that contains control information for \emph{some} grid point $x_i$.
Therefore, letting $\mathcal{P} = \mathcal{P}_0 \times \cdots \times \mathcal{P}_M$, each $P = (\textcolor{ctrlcolor}{w_i}, \textcolor{ctrlcolor}{z_i}, \textcolor{ctrlcolor}{d_i})_{i=0}^M \in \mathcal{P}$ is a tuple that contains control information for \emph{all} grid points $x_i$.
Since $W^{\gls*{meshing_parameter}}$ and $Z^{\gls*{meshing_parameter}}(\tau^n, x_i)$ are finite sets by definition, $\mathcal{P}$ is also finite and hence trivially compact.

\subsection{The penalty scheme}
\label{subsec:matrix_penalty}

We begin with the penalty scheme \cref{eqn:schemes_penalty}.
To write the penalty scheme in the form of \cref{eqn:matrix_bellman}, we define the matrix-valued function $A$ and vector-valued function $y$ such that for each vector $U = (U_0, \ldots, U_M)^\intercal$ and integer $i$ satisfying $0 \leq i \leq M$, 
\begin{align}
    [A(P) U]_i
    & = \frac{U_i}{\Delta \tau}
    - \frac{1}{2} b_i^n(\textcolor{ctrlcolor}{w_i})^2 (\mathcal{D}_2 U)_i
    - a_i^n(\textcolor{ctrlcolor}{w_i}) (\mathcal{D} U)_i
    + \frac{1}{\epsilon} \textcolor{ctrlcolor}{d_i} \left(
        U_i
        - \gls*{interp}(
            U,
            \Gamma(\tau^n, x_i, \textcolor{ctrlcolor}{z_i})
        )
    \right)
    \nonumber \\
    [y(P)]_i
    & = \frac{V_i^{n - 1}}{\Delta \tau}
    + f_i^n(\textcolor{ctrlcolor}{w_i})
    + \frac{1}{\epsilon} \textcolor{ctrlcolor}{d_i} K(\tau^n, x_i, \textcolor{ctrlcolor}{z_i}).
    \label{eqn:matrix_penalty_parameters}
\end{align}
With the above choices of $A$ and $y$, $V^n$ is a solution of \cref{eqn:matrix_bellman}.
This is seen immediately by noting that
\begin{multline}
    [ - A(P) U + y(P) ]_i
    = \frac{V_i^{n - 1} - U_i}{\Delta \tau}
    + \frac{1}{2} b_i^n(\textcolor{ctrlcolor}{w_i})^2 (\mathcal{D}_2 U)_i
    + a_i^n(\textcolor{ctrlcolor}{w_i}) (\mathcal{D} U)_i
    + f_i^n(\textcolor{ctrlcolor}{w_i}) \\
    + \frac{1}{\epsilon} \textcolor{ctrlcolor}{d_i} \left( 
        \gls*{interp}(
            U,
            \Gamma(\tau^n, x_i, \textcolor{ctrlcolor}{z_i})
        )
        + K(\tau^n, x_i, \textcolor{ctrlcolor}{z_i})
        - U_i
    \right)
    \label{eqn:matrix_penalty_parameters_combined}
\end{multline}
(compare with \cref{eqn:schemes_penalty}).
The lemma below ensures that in this case, the conditions of \cref{prop:matrix_policy_iteration} are satisfied, so that we may compute $V^n$ by {\sc Policy-Iteration}.

\begin{lemma}
    $A(P)$ given by \cref{eqn:matrix_penalty_parameters} is an M-matrix.
    \label{lem:matrix_penalty_matrix}
\end{lemma}

\begin{proof}
    We will prove that $A(P)$ is an \gls*{sdd} Z-matrix with positive diagonals so that by \cref{thm:matrix_characterization}, it is an M-matrix.

    To simplify notation, we will assume the grid points $\{ x_0, \ldots, x_M \}$ are uniformly spaced (i.e., $x_{i+1} - x_i = \Delta x$) (the same arguments can also be applied in the case of a nonuniform grid).
    Letting $U$ be an arbitrary vector, \cref{eqn:schemes_second_derivative,eqn:schemes_first_derivative,eqn:matrix_penalty_parameters} imply
    \begin{multline}
        [A(P) U]_i
        = \left(
            \frac{1}{\Delta \tau}
            + \frac{b_i^n(\textcolor{ctrlcolor}{w_i})^2}{(\Delta x)^2}
            + \frac{\left| a_i^n(\textcolor{ctrlcolor}{w_i}) \right|}{\Delta x}
            + \frac{1}{\epsilon} \textcolor{ctrlcolor}{d_i}
        \right) U_i
        \\
        - \left(
            \frac{1}{2} \frac{b_i^n(\textcolor{ctrlcolor}{w_i})^2}{(\Delta x)^2}
            + \frac{\left| a_i^n(\textcolor{ctrlcolor}{w_i}) \right|}{\Delta x}
            \boldsymbol{1}_{ \{ a_i^n(\textcolor{ctrlcolor}{w_i}) > 0 \} }
        \right) U_{i+1}
        - \left(
            \frac{1}{2} \frac{b_i^n(\textcolor{ctrlcolor}{w_i})^2}{(\Delta x)^2}
            + \frac{\left| a_i^n(\textcolor{ctrlcolor}{w_i}) \right|}{\Delta x}
            \boldsymbol{1}_{ \{ a_i^n(\textcolor{ctrlcolor}{w_i}) < 0 \} }
        \right) U_{i-1}
        \\
        - \frac{1}{\epsilon} \textcolor{ctrlcolor}{d_i} \alpha U_{k+1}
        - \frac{1}{\epsilon} \textcolor{ctrlcolor}{d_i} \left( 1 - \alpha \right) U_k
        \textspace \text{if } 0 < i < M
        \label{eqn:matrix_penalty_matrix_1}
    \end{multline}
    where $k$ and $\alpha$ depend on $i$ and satisfy $0 \leq k < M$ and $0 \leq \alpha \leq 1$ (recall \cref{eqn:schemes_interpolation}).
    The indicator functions $\boldsymbol{1}_{ \{ a_i^n(\textcolor{ctrlcolor}{w_i}) > 0 \} }$ and $\boldsymbol{1}_{ \{ a_i^n(\textcolor{ctrlcolor}{w_i}) < 0 \} }$ are used to capture the fact that due to the upwind discretization, the choice of stencil depends on the sign of the coefficient of the first derivative.
    Similarly,
    \begin{multline}
        [A(P) U]_0
        = \left(
            \frac{1}{\Delta \tau}
            + \frac{\left| a_0^n(\textcolor{ctrlcolor}{w_0}) \right|}{\Delta x}
            \boldsymbol{1}_{ \{ a_0^n(\textcolor{ctrlcolor}{w_0}) > 0 \} }
            + \frac{1}{\epsilon} \textcolor{ctrlcolor}{d_0}
        \right) U_0
        \\
        - \left(
            \frac{\left| a_0^n(\textcolor{ctrlcolor}{w_0}) \right|}{\Delta x}
            \boldsymbol{1}_{ \{ a_0^n(\textcolor{ctrlcolor}{w_0}) > 0 \} }
        \right) U_{1}
        - \frac{1}{\epsilon} \textcolor{ctrlcolor}{d_0} \alpha U_{k+1}
        - \frac{1}{\epsilon} \textcolor{ctrlcolor}{d_0} \left( 1 - \alpha \right) U_k
        \textspace
        \label{eqn:matrix_penalty_matrix_1_left}
    \end{multline}
    and
    \begin{multline}
        [A(P) U]_M
        = \left(
            \frac{1}{\Delta \tau}
            + \frac{\left| a_M^n(\textcolor{ctrlcolor}{w_M}) \right|}{\Delta x}
            \boldsymbol{1}_{ \{ a_M^n(\textcolor{ctrlcolor}{w_M}) < 0 \} }
            + \frac{1}{\epsilon} \textcolor{ctrlcolor}{d_M}
        \right) U_M
        \\
        - \left(
            \frac{\left| a_M^n(\textcolor{ctrlcolor}{w_M}) \right|}{\Delta x}
            \boldsymbol{1}_{ \{ a_M^n(\textcolor{ctrlcolor}{w_M}) < 0 \} }
        \right) U_{M-1}
        - \frac{1}{\epsilon} \textcolor{ctrlcolor}{d_M} \alpha U_{k+1}
        - \frac{1}{\epsilon} \textcolor{ctrlcolor}{d_M} \left( 1 - \alpha \right) U_k.
        \label{eqn:matrix_penalty_matrix_1_right}
    \end{multline}
    By \cref{eqn:matrix_penalty_matrix_1,eqn:matrix_penalty_matrix_1_left,eqn:matrix_penalty_matrix_1_right}, the matrix $A(P)$ has nonpositive off-diagonals and as such, is a Z-matrix.
    Setting $U = \vec{e}$ in \cref{eqn:matrix_penalty_matrix_1,eqn:matrix_penalty_matrix_1_left,eqn:matrix_penalty_matrix_1_right}, we find that for all $i$,
    \begin{equation}
        \sum_j [A(P)]_{ij}
        = [A(P) \vec{e} \, ]_i
        = \frac{1}{\Delta \tau}
        > 0,
        \label{eqn:matrix_penalty_matrix_2}
    \end{equation}
    from which it follows that $A(P)$ is \gls*{sdd} with positive diagonals since
    \[
        [A(P)]_{ii}
        > -\sum_{j \neq i} [A(P)]_{ij}
        = \sum_{j \neq i} \left| [A(P)]_{ij} \right|
        \geq 0. \qedhere
    \]
\end{proof}

\subsection{The direct control scheme}
\label{subsec:matrix_direct_control}

We now consider the direct control scheme \cref{eqn:schemes_direct_control}.
To write the direct control scheme in the form of \cref{eqn:matrix_bellman}, we define the matrix-valued function $A$ and the vector-valued function $y$ such that for each vector $U = (U_0, \ldots, U_M)^\intercal$ and integer $i$ satisfying $0 \leq i \leq M$, 
\begin{align}
    [A(P) U]_i
    & = \textcolor{ctrlcolor}{\overline{d_i}} \left(
        \frac{U_i}{\Delta \tau}
        - \frac{1}{2} b_i^n(\textcolor{ctrlcolor}{w_i})^2 (\mathcal{D}_2 U)_i
        - a_i^n(\textcolor{ctrlcolor}{w_i}) (\mathcal{D} U)_i
    \right)
    + \textcolor{ctrlcolor}{d_i} \left(
        U_i
        - \gls*{interp}(
            U,
            \Gamma(\tau^n, x_i, \textcolor{ctrlcolor}{z_i})
        )
    \right)
    \nonumber \\
    [y(P)]_i
    & = \textcolor{ctrlcolor}{\overline{d_i}} \left(
        \frac{V_i^{n - 1}}{\Delta \tau}
        + f_i^n(\textcolor{ctrlcolor}{w_i})
    \right)
    + \textcolor{ctrlcolor}{d_i} K(\tau^n, x_i, \textcolor{ctrlcolor}{z_i}).
    \label{eqn:matrix_direct_control_parameters}
\end{align}
With the above choices of $A$ and $y$, $V^n$ is a solution of \cref{eqn:matrix_bellman}.
This is seen immediately by noting that
\begin{multline}
    [ -A(P) U + y(P) ]_i
    = \textcolor{ctrlcolor}{\overline{d_i}} \left(
        \frac{V_i^{n-1} - U_i}{\Delta \tau}
        + \frac{1}{2} b_i^n(\textcolor{ctrlcolor}{w_i})^2 (\mathcal{D}_2 U)_i
        + a_i^n(\textcolor{ctrlcolor}{w_i}) (\mathcal{D} U)_i
        + f_i^n(\textcolor{ctrlcolor}{w_i})
    \right) \\
    + \textcolor{ctrlcolor}{d_i} \left(
        \gls*{interp}(
            U,
            \Gamma(\tau^n, x_i, \textcolor{ctrlcolor}{z_i})
        )
        + K(\tau^n, x_i, \textcolor{ctrlcolor}{z_i})
        - U_i
    \right)
    \label{eqn:matrix_direct_control_parameters_combined}
\end{multline}
(compare with \cref{eqn:schemes_direct_control}).

\begin{remark}
    \label{rem:matrix_scaling}
    Note that the quantities
    \[
        \frac{V_i^{n-1} - U_i}{\Delta \tau}
        + \frac{1}{2} b_i^n(\textcolor{ctrlcolor}{w_i})^2 (\mathcal{D}_2 U)_i
        + a_i^n(\textcolor{ctrlcolor}{w_i}) (\mathcal{D} U)_i
        + f_i^n(\textcolor{ctrlcolor}{w_i})
    \]
    and
    \[
        \gls*{interp}(
            U,
            \Gamma(\tau^n, x_i, \textcolor{ctrlcolor}{z_i})
        )
        + K(\tau^n, x_i, \textcolor{ctrlcolor}{z_i})
        - U_i
    \]
    appearing in \cref{eqn:matrix_direct_control_parameters_combined} have different units: the former is a rate of change while the latter is a unit amount.
    A similar situation is encountered in \cite{MR3102414}, in which the authors 
    find that policy iteration performs better after scaling the problem to normalize the units.
    In the context of \cref{lem:matrix_scaling}, this corresponds to the scaling factor $s(P)$ with entries given by
    \[
        [s(P)]_i = \textcolor{ctrlcolor}{\overline{d_i}} + \textcolor{ctrlcolor}{d_i} \frac{1}{\delta \Delta \tau}
    \]
    where $\delta$ is some positive constant.
    By \cref{lem:matrix_scaling}, the unscaled problem \cref{eqn:matrix_bellman} has the same set of solutions as the scaled problem \cref{eqn:matrix_scaling_2}, for which the units are all rates of change:
    \begin{multline*}
        [ \gls*{diag}(s(P))(-A(P) U + y(P)) ]_i
        \\
        = \textcolor{ctrlcolor}{\overline{d_i}} \left(
            \frac{V_i^{n-1} - U_i}{\Delta \tau}
            + \frac{1}{2} b_i^n(\textcolor{ctrlcolor}{w_i})^2 (\mathcal{D}_2 U)_i
            + a_i^n(\textcolor{ctrlcolor}{w_i}) (\mathcal{D} U)_i
            + f_i^n(\textcolor{ctrlcolor}{w_i})
        \right) \\
        + \frac{1}{\delta} \textcolor{ctrlcolor}{d_i} \left(
            \frac{
                \gls*{interp}(
                    U,
                    \Gamma(\tau^n, x_i, \textcolor{ctrlcolor}{z_i})
                )
                + K(\tau^n, x_i, \textcolor{ctrlcolor}{z_i})
                - U_i
            }{\Delta \tau}
        \right).
    \end{multline*}
    While a similar issue occurs for the penalty scheme (see \cref{eqn:matrix_penalty_parameters_combined}), taking $\epsilon$ to be a multiple of $\Delta \tau$ immediately normalizes the units.
\end{remark}

Next, we give an analogue of \cref{lem:matrix_penalty_matrix} for the direct control setting.

\begin{lemma}
    $A(P)$ given by \cref{eqn:matrix_direct_control_parameters} is a \gls*{wdd} Z-matrix with nonnegative diagonals.
    \label{lem:matrix_direct_control_matrix}
\end{lemma}

\begin{proof}
    As in the proof of \cref{lem:matrix_penalty_matrix}, we have
    \begin{multline}
        [A(P) U]_i
        = \left(
            \textcolor{ctrlcolor}{\overline{d_i}} \left(
                \frac{1}{\Delta \tau}
                + \frac{b_i^n(\textcolor{ctrlcolor}{w_i})^2}{(\Delta x)^2}
                + \frac{\left| a_i^n(\textcolor{ctrlcolor}{w_i}) \right|}{\Delta x}
            \right)
            + \textcolor{ctrlcolor}{d_i}
        \right) U_i
        \\
        - \textcolor{ctrlcolor}{\overline{d_i}} \left(
            \frac{1}{2} \frac{b_i^n(\textcolor{ctrlcolor}{w_i})^2}{(\Delta x)^2}
            + \frac{\left| a_i^n(\textcolor{ctrlcolor}{w_i}) \right|}{\Delta x}
            \boldsymbol{1}_{ \{ a_i^n(\textcolor{ctrlcolor}{w_i}) > 0 \} }
        \right) U_{i+1}
        - \textcolor{ctrlcolor}{\overline{d_i}} \left(
            \frac{1}{2} \frac{b_i^n(\textcolor{ctrlcolor}{w_i})^2}{(\Delta x)^2}
            + \frac{\left| a_i^n(\textcolor{ctrlcolor}{w_i}) \right|}{\Delta x}
            \boldsymbol{1}_{ \{ a_i^n(\textcolor{ctrlcolor}{w_i}) < 0 \} }
        \right) U_{i-1}
        \\
        - \textcolor{ctrlcolor}{d_i} \alpha U_{k+1}
        - \textcolor{ctrlcolor}{d_i} \left( 1 - \alpha \right) U_k
        \textspace \text{if } 0 < i < M
        \label{eqn:matrix_direct_control_matrix_1}
    \end{multline}
    along with
    \begin{multline}
        [A(P) U]_0
        = \left(
            \textcolor{ctrlcolor}{\overline{d_0}} \left(
                \frac{1}{\Delta \tau}
                + \frac{\left| a_0^n(\textcolor{ctrlcolor}{w_0}) \right|}{\Delta x}
                \boldsymbol{1}_{ \{ a_0^n(\textcolor{ctrlcolor}{w_0}) > 0 \} }
            \right)
            + \textcolor{ctrlcolor}{d_0}
        \right) U_0
        \\
        - \textcolor{ctrlcolor}{\overline{d_0}} \left(
            \frac{\left| a_0^n(\textcolor{ctrlcolor}{w_0}) \right|}{\Delta x}
            \boldsymbol{1}_{ \{ a_0^n(\textcolor{ctrlcolor}{w_0}) > 0 \} }
        \right) U_1
        - \textcolor{ctrlcolor}{d_0} \alpha U_{k+1}
        - \textcolor{ctrlcolor}{d_0} \left( 1 - \alpha \right) U_k
        \label{eqn:matrix_direct_control_matrix_1_left}
    \end{multline}
    and
    \begin{multline}
        [A(P) U]_M
        = \left(
            \textcolor{ctrlcolor}{\overline{d_M}} \left(
                \frac{1}{\Delta \tau}
                + \frac{\left| a_M^n(\textcolor{ctrlcolor}{w_M}) \right|}{\Delta x}
                \boldsymbol{1}_{ \{ a_M^n(\textcolor{ctrlcolor}{w_M}) < 0 \} }
            \right)
            + \textcolor{ctrlcolor}{d_M}
        \right) U_M
        \\
        - \textcolor{ctrlcolor}{\overline{d_M}} \left(
            \frac{\left| a_M^n(\textcolor{ctrlcolor}{w_M}) \right|}{\Delta x}
            \boldsymbol{1}_{ \{ a_M^n(\textcolor{ctrlcolor}{w_M}) < 0 \} }
        \right) U_{M-1}
        - \textcolor{ctrlcolor}{d_M} \alpha U_{k+1}
        - \textcolor{ctrlcolor}{d_M} \left( 1 - \alpha \right) U_k
        \label{eqn:matrix_direct_control_matrix_1_right}
    \end{multline}
    from which one sees that $A(P)$ has nonpositive off-diagonals and as such, is a Z-matrix.
    Setting $U = \vec{e}$ in \cref{eqn:matrix_direct_control_matrix_1,eqn:matrix_direct_control_matrix_1_left,eqn:matrix_direct_control_matrix_1_right}, we find that for all $i$,
    \begin{equation}
        \sum_j [A(P)]_{ij} = [A(P) \vec{e} \, ]_i
        = \textcolor{ctrlcolor}{\overline{d_i}} \frac{1}{\Delta \tau}
        = \left( 1 - \textcolor{ctrlcolor}{d_i} \right) \frac{1}{\Delta \tau}
        \geq 0,
        \label{eqn:matrix_direct_control_matrix_2}
    \end{equation}
    from which it follows that $A(P)$ is \gls*{wdd} with nonnegative diagonals since
    \[
        [A(P)]_{ii}
        \geq -\sum_{j \neq i} [A(P)]_{ij}
        = \sum_{j \neq i} \left| [A(P)]_{ij} \right|
        \geq 0. \qedhere
    \]
\end{proof}

Unfortunately, \cref{lem:matrix_direct_control_matrix} does not guarantee the nonsingularity of $A(P)$.
For example, if $\textcolor{ctrlcolor}{d_i} = 1$ for all $i$, \cref{eqn:matrix_direct_control_matrix_2} implies $A(P) \vec{e} = 0$ and hence $A(P)$ is singular.
This is a serious issue for {\sc Policy-Iteration} since if any of the matrices in the set $\{ A(P) \}_{P \in \mathcal{P}}$ are singular, the iterates $U^\ell$ may not even be well-defined.
To see why, consider the linear system on \cref{line:matrix_solve} of the pseudocode.
If $A(P^\ell)$ is singular, the vector $y(P^\ell)$ is not guaranteed to lie in the range of $A(P^\ell)$ and hence the linear system may not have any solutions.
We will resolve this issue in the next section.

We close this subsection by stating a result that we will use in the sequel.


\begin{lemma}
    \label{lem:matrix_sdd_iff}
    Let $A(P)$ be given by \cref{eqn:matrix_direct_control_parameters}.
    Then, row $i$ of $A(P)$ is not \gls*{sdd} if and only if $\textcolor{ctrlcolor}{d_i} = 1$.
    Moreover, $[A(P)]_{ii} \leq 1$ whenever $\textcolor{ctrlcolor}{d_i} = 1$.
\end{lemma}
\begin{proof}
    The first claim is made immediately obvious by rewriting \cref{eqn:matrix_direct_control_matrix_2} as
    \[
        \left| [A(P)]_{ii} \right|
        = [A(P)]_{ii}
        = - \sum_{j \neq i} [A(P)]_{ij}
        + \left( 1 - \textcolor{ctrlcolor}{d_i} \right) \frac{1}{\Delta \tau}
        = \sum_{j \neq i} \left| [A(P)]_{ij} \right|
        + \left( 1 - \textcolor{ctrlcolor}{d_i} \right) \frac{1}{\Delta \tau}.
    \]

    Now, suppose that $\textcolor{ctrlcolor}{d_i} = 1$ for some fixed row $i$.
    Then by \cref{eqn:matrix_direct_control_parameters}, we have
    \[
        [A(P) U]_i
        = U_i - \gls*{interp}(
            U,
            \Gamma(\tau^n, x_i, \textcolor{ctrlcolor}{z_i})
        )
        = U_i - \alpha U_{k+1} - \left( 1 - \alpha \right) U_k
    \]
    where $0 \leq \alpha \leq 1$ and $0 \leq k < M$.
    Letting $\vec{x}$ be the $i$-th standard basis vector (i.e., the vector with entries $\vec{x}_j = \delta_{ij}$ where $\delta_{ij}$ is the Kronecker delta), we have
    \[
        [A(P)]_{ii}
        = [A(P) \vec{x} \,]_i
        = \delta_{ii} - \alpha \delta_{i,k+1} - \left( 1 - \alpha \right) \delta_{ik}
        = 1 - \alpha \delta_{i,k+1} - \left( 1 - \alpha \right) \delta_{ik}
        \leq 1,
    \]
    establishing the second claim in the lemma statement.
\end{proof}

\section{Singularity in the Bellman problem}
\label{sec:matrix_singularity}

In the previous section, we saw that the direct control scheme may involve singular (and hence nonmonotone) matrices $A(P)$.
Motivated by this issue, we now revisit the Bellman problem \cref{eqn:matrix_bellman}, relaxing the assumption that the matrices $A(P)$ are monotone.
Throughout this section, we call upon the assumption below.

\begin{enumerate}[label=(H\arabic*),start=3]
    \item
        \label{enum:matrix_standard_assumption}
        For each $P \in \mathcal{P}$, $A(P)$ is a \gls*{wdd} Z-matrix with nonnegative diagonals satisfying $[A(P)]_{ii} \leq 1$ whenever its $i$-th row 
        is not \gls*{sdd}
\end{enumerate}

We mention that \cref{enum:matrix_standard_assumption} is satisfied by all problems of interest in this chapter (see, e.g., \cref{lem:matrix_direct_control_matrix,lem:matrix_sdd_iff}), including the \glspl*{MDP} studied in the sequel.

\subsection{Uniqueness}

We now establish that solutions of \cref{eqn:matrix_bellman} are unique independent of the nonsingularity of $A(P)$.
The idea behind our technique is to restrict attention to rows of $A(P)$ that are not \gls*{sdd}, as these are the ``problem'' rows that can cause $A(P)$ to be singular.

In light of this, it is convenient to define
\[
    [\hat{y}(P)]_i = \begin{cases}
        [y(P)]_i & \text{if row } i \text{ of } A(P) \text{ is not \gls*{sdd}} \\
        -\infty & \text{otherwise}
    \end{cases}
\]
and the (nonlinear) operator $\mathbb{A}$ by
\[
    \mathbb{A}U = \sup_{P \in \mathcal{P}} \left\{
        - A(P) U + \hat{y}(P)
    \right\}
\]
(compare with \cref{eqn:matrix_bellman}).
We also define $\mathbb{M} = I + \mathbb{A}$.
It is understood that if $[\mathbb{A} U]_i = - \infty$, then $[\mathbb{M} U]_i = [(I + \mathbb{A}) U]_i = U_i + [\mathbb{A} U]_i = -\infty$.
We use superscripts to indicate composition.
For example, $\mathbb{M}^0 = I$ and $\mathbb{M}^k = (\mathbb{M}^{k - 1}) \circ \mathbb{M}$ for any positive integer $k$.

\begin{example}
    Suppose $A$ and $y$ are given by \cref{eqn:matrix_direct_control_parameters}, corresponding to the direct control scheme.
    By \cref{lem:matrix_sdd_iff}, row $i$ of $A(P)$ is not \gls*{sdd} if and only if $\textcolor{ctrlcolor}{d_i} = 1$.
    Therefore,
    \[
        [\hat{y}(P)]_i = \begin{cases}
            K(\tau^n, x_i, \textcolor{ctrlcolor}{z_i})
            & \text{if } \textcolor{ctrlcolor}{d_i} = 1 \\
            -\infty
            & \text{otherwise}
        \end{cases}
    \]
    and hence
    \begin{align*}
        [\mathbb{M} U]_i
        & = U_i + [\mathbb{A} U]_i \\
        & = U_i + \sup_{
            \textcolor{ctrlcolor}{z_i} \in Z^{\gls*{meshing_parameter}}(\tau^n, x_i)
        } \left \{
            \gls*{interp}(
                U,
                \Gamma(\tau^n, x_i, \textcolor{ctrlcolor}{z_i})
            )
            + K(\tau^n, x_i, \textcolor{ctrlcolor}{z_i})
            - U_i
        \right \} \\
        & = (\mathcal{M}_n U)_i
    \end{align*}
    where $\mathcal{M}_n$ is defined in \cref{eqn:schemes_discretized_intervention}.
    More succinctly, we write $\mathbb{M} = \mathcal{M}_n$.
    \label{exa:matrix_intervention}
\end{example}

The following assumption, justified momentarily, is used in the proof of uniqueness:

\begin{enumerate}[label=(H\arabic*),start=4]
    \item
        \label{enum:matrix_eventually_suboptimal}
        For each 
        $U \in \mathbb{R}^{M+1}$ and integer $i$ satisfying $0 \leq i \leq M$, there exist integers $m_1$ and $m_2$ satisfying $0 \leq m_1 < m_2$ and $[\mathbb{M}^{m_1} U]_i > [\mathbb{M}^{m_2} U]_i$.
\end{enumerate}

\begin{example}
    %
    The discretized intervention operator $\mathcal{M}_n$ of \cref{exa:schemes_intervention} satisfies
    \begin{multline*}
        (\mathcal{M}_n^2 U)_i
        = (\mathcal{M}_n (\mathcal{M}_n U))_i
        = \max_{0 \leq j \leq M} \left\{
            (\mathcal{M}_n U)_j
            - e^{-\beta \tau^n} \left(
                \kappa \left| x_j - x_i \right|
                + c
            \right)
        \right\} \\
        = \max_{0 \leq j, k \leq M} \left\{
            U_k
            - e^{-\beta \tau^n} \left(
                \kappa \left| x_j - x_i \right|
                + \kappa \left| x_k - x_j \right|
                + 2c
            \right)
        \right\}.
    \end{multline*}
    Therefore, an application of the triangle inequality yields
    \[
        (\mathcal{M}_n^2 U)_i
        \leq \max_{0 \leq j \leq M} \left\{
            U_j
            - e^{-\beta \tau^n} \left(
                \kappa \left| x_j - x_i \right|
                + 2c
            \right)
        \right\}
        = (\mathcal{M}_n U)_i - e^{-\beta \tau^n} c
        < (\mathcal{M}_n U)_i.
    \]
    Since $\mathbb{M} = \mathcal{M}_n$ by \cref{exa:matrix_intervention}, \cref{enum:matrix_eventually_suboptimal} is satisfied (with $m_1 = 1$ and $m_2 = 2$).
    \label{exa:matrix_transaction_costs}
\end{example}

The above example highlights the intuition behind \cref{enum:matrix_eventually_suboptimal}, which attempts to capture the suboptimality of consecutive applications of the operator $\mathbb{M}$.
In the example, $\mathcal{M}_n^2$ represents the controller performing two consecutive impulses: the first changes the state from $x_i$ to $x_j$ while the second changes it from $x_j$ to $x_k$.
This is suboptimal, since the controller can directly change the state from $x_i$ to $x_k$, paying only the fixed cost of $c$ instead of $2c$.

We can now state our uniqueness result, which is independent of any assumptions regarding nonsingularity.

\begin{theorem}
    \label{thm:matrix_unique}
    Suppose \cref{enum:matrix_bounded_inverse}, \cref{enum:matrix_standard_assumption}, and \cref{enum:matrix_eventually_suboptimal}.
    Then, solutions of \cref{eqn:matrix_bellman} are unique.
\end{theorem}

To prove the result, we require a lemma.

\begin{lemma}
    Let $U$ be a solution of \cref{eqn:matrix_bellman}.
    Then, we can find a sequence $(P^\ell)_\ell$ such that
    \begin{equation}
        -A(P^\ell) U + y(P^\ell) \rightarrow 0
        \label{eqn:matrix_eventually_wcdd_1}
    \end{equation}
    and $A(P^\ell)$ is an M-matrix for each $\ell$.
    \label{lem:matrix_eventually_wcdd}
\end{lemma}

\begin{proof}
    Throughout the proof, we call upon the following facts:
    \begin{enumerate}[label=(\roman*)]
        \item
            By \cref{enum:matrix_standard_assumption,thm:matrix_characterization}, $A(P)$ being an M-matrix is equivalent to it being \gls*{wcdd}
        \item
            Since $U$ is a solution of \cref{eqn:matrix_bellman},
            \begin{equation}
                \mathbb{A} U
                = \sup_{P \in \mathcal{P}} \left \{
                    - A(P) U + \hat{y}(P)
                \right \}
                \leq \sup_{P \in \mathcal{P}} \left \{
                    - A(P) U + y(P)
                \right \}
                = 0
                \label{eqn:matrix_eventually_wcdd_global_1}
            \end{equation}
            and hence $\mathbb{M} U = (I + \mathbb{A}) U = U + \mathbb{A} U \leq U$.
        \item
            $\mathbb{M}$ is a monotone operator (i.e., if $X \leq Y$, then $\mathbb{M} X \leq \mathbb{M} Y$).
            To see this, note that
            \begin{align*}
                [\mathbb{M} X]_i
                & = [(I + \mathbb{A}) X]_i \\
                & = X_i + \textstyle{ \sup_{P \in \mathcal{P}} } [
                    - A(P) X + \hat{y}(P)
                ]_i \\
                & = \textstyle{ \sup_{P \in \mathcal{P}} } \{
                    \left( 1 - [A(P)]_{ii} \right) X_i
                    - \sum_{j \neq i} [A(P)]_{ij} X_j
                    + [\hat{y}(P)]_i
                \} \\
                & \leq \textstyle{ \sup_{P \in \mathcal{P}} } \{
                    \left( 1 - [A(P)]_{ii} \right) Y_i
                    - \sum_{j \neq i} [A(P)]_{ij} Y_j
                    + [\hat{y}(P)]_i
                \} \\
                & = [\mathbb{M} Y]_i
            \end{align*}
            where the inequality above follows from the fact that if $[\hat{y}(P)]_i \neq -\infty$, then $1 - [A(P)]_{ii} \geq 0$ and $[A(P)]_{ij} \leq 0$ for $j \neq i$ by \cref{enum:matrix_standard_assumption}.
    \end{enumerate}

    First, we consider the case where $\mathcal{P}$ is a finite set.
    In this case, the supremum in \cref{eqn:matrix_bellman} is attained at a point $P$:
    \begin{equation}
        - A(P) U + y(P) = 0.
        \label{eqn:matrix_eventually_wcdd_finite_1}
    \end{equation}
    We seek to show that $A(P)$ is \gls*{wcdd} since in this case, the constant sequence $(P^\ell)_\ell$ defined by $P^\ell = P$ satisfies \cref{eqn:matrix_eventually_wcdd_1}.
    Let $G$ be the (directed) adjacency graph of $A(P)$, $J$ be the set of all \gls*{sdd} rows of $A(P)$, and $R = \{ 0, \ldots, M \} \setminus ( J \cup W )$ where
    \[
        W = \left \{
            i_1 \notin J \colon
            \text{there exists a walk }
            i_1 \rightarrow \cdots \rightarrow i_k
            \text{ in }
            G
            \text{ such that }
            i_k \in J
        \right\}
    \]
    (compare with $J$, $R$, and $W$ defined in the proof of \cref{thm:matrix_characterization}).
    Suppose $A(P)$ is not \gls*{wcdd}
    By \cref{def:matrix_wcdd}, it follows that $R$ is nonempty.
    Since $R \cap J = \emptyset$ and $[\hat{y}(P)]_j = [y(P)]_j$ for any $j \notin J$, \cref{eqn:matrix_eventually_wcdd_finite_1} implies
    \begin{equation}
        [\mathbb{A} U]_r
        \geq [- A(P) U + \hat{y}(P)]_r
        = [- A(P) U + y(P)]_r
        = 0
        \textspace \text{for } r \in R.
        \label{eqn:matrix_eventually_wcdd_finite_2}
    \end{equation}
    By \cref{eqn:matrix_eventually_wcdd_global_1}, $\mathbb{A} U \leq 0$, and hence the inequality in \cref{eqn:matrix_eventually_wcdd_finite_2} must hold with equality.
    In particular, we can conclude
    \[
        [\mathbb{A} U]_r = 0
        \textspace \text{for } r \in R.
    \]
    Equivalently, in terms of the operator $\mathbb{M} = I + \mathbb{A}$,
    \begin{equation}
        [\mathbb{M} U]_r = U_r
        \textspace \text{for } r \in R.
        \label{eqn:matrix_eventually_wcdd_finite_3}
    \end{equation}
    Now, note that
    \begin{align}
        [\mathbb{M}^2 U]_r
        & = [\mathbb{M} (\mathbb{M} U)]_r
        \nonumber \\
        & = [(I + \mathbb{A}) (\mathbb{M} U)]_r
        \nonumber \\
        & = [\mathbb{M} U]_r + [\mathbb{A} (\mathbb{M} U)]_r
        \nonumber \\
        & \geq [\mathbb{M} U]_r + [-A(P) (\mathbb{M} U) + y(P)]_r
        \nonumber \\
        & = [\mathbb{M} U]_r
        - { \sum_{j \in R} [A(P)]_{rj} } [\mathbb{M} U]_j
        + [y(P)]_r
        & \text{for } r \in R
        \label{eqn:matrix_eventually_wcdd_finite_4}
    \end{align}
    where the last equality follows since $[A(P)]_{rj} = 0$ if $j \notin R$ (recall, as in \cref{fig:matrix_proof_adjacency_graph}, that there are no edges from vertices in $R$ to vertices in $J \cup W$).
    Applying \cref{eqn:matrix_eventually_wcdd_finite_3} to \cref{eqn:matrix_eventually_wcdd_finite_4} and using the fact that \cref{eqn:matrix_eventually_wcdd_finite_2} holds with equality yields
    \begin{align*}
        [\mathbb{M}^2 U]_r
        & \geq U_r - { \sum_{j \in R} [A(P)]_{rj} } U_j + [y(P)]_r \\
        & = U_r + [-A(P) U + y(P)]_r \\
        & = U_r + [\mathbb{A} U]_r \\
        & = [(I + \mathbb{A}) U]_r \\
        & = [\mathbb{M} U]_r
        & \text{for } r \in R.
    \end{align*}
    By the monotonicity of $\mathbb{M}$ and the inequality $\mathbb{M} U \leq U$ (recall \cref{eqn:matrix_eventually_wcdd_global_1}), $\mathbb{M}^2 U \leq \mathbb{M} U$, and hence $[\mathbb{M}^2 U]_r = [\mathbb{M} U]_r$ for $r \in R$.
    Continuing this procedure, we establish
    \[
        U_r = [\mathbb{M} U]_r = [\mathbb{M}^2 U]_r = [\mathbb{M}^3 U]_r = \cdots
        \textspace \text{for } r \in R,
    \]
    contradicting \cref{enum:matrix_eventually_suboptimal}.

    If $\mathcal{P}$ is infinite, the supremum in \cref{eqn:matrix_bellman}, while not necessarily attained at a point, can be approximated by some sequence $(P^\ell)_\ell$:
    \begin{equation}
        \lim_{\ell \rightarrow \infty} \left \{
            - A(P^\ell) U + y(P^\ell)
        \right \} = 0.
        \label{eqn:matrix_eventually_wcdd_infinite_1}
    \end{equation}
    It is sufficient to show that for $\ell$ sufficiently large, $A(P^\ell)$ is \gls*{wcdd} since we can always drop the first finitely many terms from a sequence without changing its limit.
    In order to arrive at a contradiction, suppose that we can find a subsequence of $(P^\ell)_\ell$ indexed by $(\ell_s)_s$ such that for each $s$, $A(P^{\ell_s})$ is not \gls*{wcdd}
    Let $G^s$ be the directed adjacency graph of $A(P^{\ell_s})$, $J^s$ be the set of \gls*{sdd} rows of $A(P^{\ell_s})$, and $R^s = \{0, \ldots, M\} \setminus (J^s \cup W^s)$ where
    \[
        W^s = \left \{
            i_1 \notin J^s \colon
            \text{there exists a walk }
            i_1 \rightarrow \cdots \rightarrow i_k
            \text{ in }
            G^s
            \text{ such that }
            i_k \in J^s
        \right\}.
    \]
    Since the sequence of triplets $( (G^s, J^s, W^s) )_s$ takes on at most finitely many values, we can use the pigeonhole principle to find a constant subsequence of $( (G^s, J^s, W^s) )_s$ indexed by $(s_t)_t$.
    By constant, we mean that
    \[
        G^{s_1} = G^{s_2} = \cdots
        , \textspace
        J^{s_1} = J^{s_2} = \cdots
        , \textspace \text{and} \textspace
        W^{s_1} = W^{s_2} = \cdots
    \]
    In light of this, it is convenient to define $G = G^{s_1}$, $J = J^{s_1}$, $W = W^{s_1}$, and $R = R^{s_1}$.
    For brevity, we also define $A^{(t)} = A(P^{\ell_{s_t}})$, $y^{(t)} = y(P^{\ell_{s_t}})$, and $\hat{y}^{(t)} = \hat{y}(P^{\ell_{s_t}})$.
    Since a subsequence of a convergent sequence has the same limit, \cref{eqn:matrix_eventually_wcdd_infinite_1} implies
    \begin{equation}
        \lim_{t \rightarrow \infty} \left \{
            - A^{(t)} U + y^{(t)}
        \right \} = 0.
        \label{eqn:matrix_eventually_wcdd_infinite_2}
    \end{equation}
    Since $R \cap J = \emptyset$ and $[\hat{y}(P)]_j = [y(P)]_j$ for any $j \notin J$, \cref{eqn:matrix_eventually_wcdd_infinite_2} implies
    \[
        [\mathbb{A} U]_r
        \geq \lim_{t \rightarrow \infty} [- A^{(t)} U + \hat{y}^{(t)}]_r
        = \lim_{t \rightarrow \infty} [- A^{(t)} U + y^{(t)}]_r
        = 0
        \textspace \text{for } r \in R.
    \]
    Again, by \cref{eqn:matrix_eventually_wcdd_global_1}, $\mathbb{A} U \leq 0$ and hence it must be the case that the above holds with equality.
    As in the previous paragraph, this implies \cref{eqn:matrix_eventually_wcdd_finite_3}.
    Similarly to \cref{eqn:matrix_eventually_wcdd_finite_4},
    \begin{align*}
        [\mathbb{M}^2 U]_r
        & \geq [\mathbb{M} U]_r
        + \lim_{t \rightarrow \infty} [
            -A^{(t)} (\mathbb{M} U) + y^{(t)}
        ]_r \\
        & \geq [\mathbb{M} U]_r
        + \lim_{t \rightarrow \infty}
            - \sum_{j \in R} [-A^{(t)}]_{rj} [\mathbb{M} U]_j
            + [y^{(t)}]_r
        & \text{for } r \in R.
    \end{align*}
    Applying \cref{eqn:matrix_eventually_wcdd_finite_3} to the above
    \begin{align*}
        [\mathbb{M}^2 U]_r
        & \geq U_r + \lim_{t \rightarrow \infty}
            - { \sum_{j \in R} [A^{(t)}]_{rj} } U_j
            + [y^{(t)}]_r \\
        & = U_r + \lim_{t \rightarrow \infty} [-A^{(t)} U + y^{(t)}]_r \\
        & = U_r + [\mathbb{A} U]_r \\
        & = [\mathbb{M} U]_r
        & \text{for } r \in R.
    \end{align*}
    Again, by the monotonicity of $\mathbb{M}$ and the inequality $\mathbb{M} U \leq U$ (recall \cref{eqn:matrix_eventually_wcdd_global_1}), $\mathbb{M}^2 U \leq \mathbb{M} U$, and hence $[\mathbb{M}^2 U]_r = [\mathbb{M} U]_r$ for $r \in R$.
    Continuing this procedure, we establish
    \[
        U_r = [\mathbb{M} U]_r = [\mathbb{M}^2 U]_r = [\mathbb{M}^3 U]_r = \cdots
        \textspace \text{for } r \in R,
    \]
    contradicting \cref{enum:matrix_eventually_suboptimal}.
\end{proof}

We can now prove \cref{thm:matrix_unique}:

\begin{proof}[Proof of \cref{thm:matrix_unique}]
    Let $U$ and $\hat{U}$ be solutions of \cref{eqn:matrix_bellman}.
    By \cref{lem:matrix_eventually_wcdd}, it follows that we can find a sequence $(P^\ell)_\ell$ such that $A(P^\ell)$ is an M-matrix for each $\ell$ and
    \[
        -A(P^\ell) \hat{U} + y(P^\ell) \rightarrow 0.
    \]
    Therefore, by \cref{lem:matrix_one_side}, $U - \hat{U} \geq 0$.
    Switching the roles of $U$ and $\hat{U}$ in the above gives $\hat{U} - U \geq 0$.
    Hence $U = \hat{U}$.
\end{proof}

\subsection{Policy iteration on a subset of controls}
\label{subsec:matrix_restriction}

Recall that policy iteration may fail upon encountering a singular matrix iterate $A(P^\ell)$ (see the discussion at the end of \cref{subsec:matrix_direct_control}).
To avoid singular matrix iterates, we perform policy iteration on a subset $\mathcal{P}^\prime$ of $\mathcal{P}$ obtained by removing certain controls $P$ for which $A(P)$ is singular (\cref{fig:matrix_restriction}), employing \cref{thm:matrix_unique} to ensure uniqueness of the solution.
Below, we state this idea precisely.

\begin{figure}
    \centering
    \begin{tikzpicture}
        \draw [thick] ellipse (2 and 2/\gratio);
        \filldraw [thick, fill=black!20] ellipse (1 and 1/\gratio);
        \draw [thick, xshift=8cm] ellipse (3 and 3/\gratio);
        \draw [thick, xshift=8cm] ellipse (2 and 2/\gratio);
        \filldraw [thick, xshift=8cm, fill=black!20] ellipse (1 and 1/\gratio);
        \node at (0, 0) (subset) {$\mathcal{P}^\prime$};
        \node at (0, -1.5/\gratio) {$\mathcal{P}$};
        \node at (8, 0) (image) {$A(\mathcal{P}^\prime)$};
        \node at (8, -1.4/\gratio) {\small nonsingular};
        \node at (8, -2.4/\gratio) {\small $\mathbb{R}^{(M+1) \times (M+1)}$};
        \path [thick, ->, bend left] (subset) edge node [anchor=center, above] {$A$} (8, 0.25);
    \end{tikzpicture}
    \caption{Restriction to a subset of controls that preclude singularity}
    \label{fig:matrix_restriction}
\end{figure}
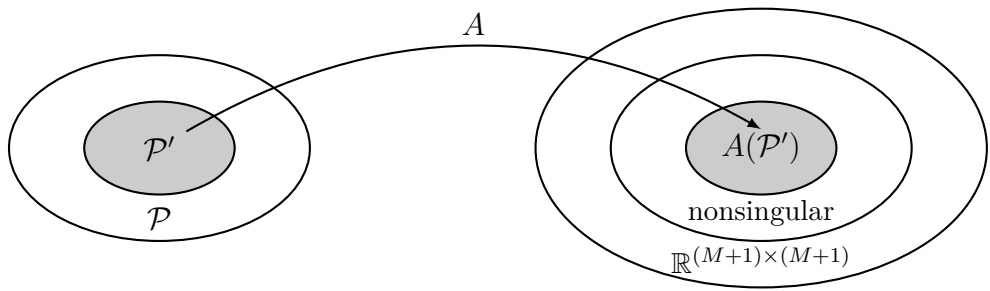

\begin{theorem}
    For each $i$, let $\mathcal{P}_i^\prime$ be a nonempty subset of $\mathcal{P}_i$ and define $\mathcal{P}^\prime = \mathcal{P}_0^\prime \times \cdots \times \mathcal{P}_M^\prime$.
    Suppose \cref{enum:matrix_bounded_inverse,enum:matrix_bounded,enum:matrix_standard_assumption,enum:matrix_eventually_suboptimal}, $A(P)$ is nonsingular for each $P \in \mathcal{P}^\prime$, and for all $U \in \mathbb{R}^{M+1}$,
    \begin{equation}
        H(U; \mathcal{P}) \leq 0
        \textspace \text{whenever} \textspace
        H(U; \mathcal{P}^\prime) = 0
        \label{eqn:matrix_invariance}
    \end{equation}
    where $H(U; \mathcal{P}) = \sup_{P \in \mathcal{P}} \{ -A(P) U + y(P) \}$.
    Then, the sequence $(U^\ell)_\ell$ defined by \Call{$\epsilon$-Policy-Iteration}{$\mathcal{P}^\prime$} converges to the unique solution of \cref{eqn:matrix_bellman}.
    \label{thm:matrix_restriction}
\end{theorem}

\begin{proof}
    By \cref{enum:matrix_standard_assumption,thm:matrix_characterization}, $A(P)$ is an M-matrix for each $P \in \mathcal{P}^\prime$.
    Therefore, by \cref{thm:matrix_policy_iteration}, $(U^\ell)_\ell$ converges to a vector $U$ satisfying $H(U; \mathcal{P}^\prime) = 0$.
    Since $\mathcal{P}^\prime \subset \mathcal{P}$, it follows that $H(U; \mathcal{P}) \geq H(U; \mathcal{P}^\prime) = 0$.
    Therefore, by \cref{eqn:matrix_invariance}, $H(U; \mathcal{P}) = 0$
    and hence $U$ is a solution of \cref{eqn:matrix_bellman}.
    Moreover, by \cref{thm:matrix_unique}, it is the unique solution.
\end{proof}

\begin{remark}
    If, in addition to the requirements of \cref{thm:matrix_restriction}, $\mathcal{P}^\prime$ is a compact topological space and the restrictions of the functions $A$ and $y$ to the set $\mathcal{P}^\prime$ are continuous, we can use \Call{Policy-Iteration}{$\mathcal{P}^\prime$} (instead of \Call{$\epsilon$-Policy-Iteration}{$\mathcal{P}^\prime$}) to compute the unique solution of \cref{eqn:matrix_bellman}.
\end{remark}

The drawback of the approach suggested by \cref{thm:matrix_restriction} is that it is not clear how to select the subset $\mathcal{P}^\prime$ in general.
This set must be chosen on a problem-by-problem basis (an example is given below).
In this sense, the penalty scheme, which requires no additional effort to ensure convergence, is more robust than the direct control scheme.
This begs the following natural question: if both schemes converge to the same solution, is there an advantage to using the direct control scheme? (e.g., could it, at least empirically, exhibit faster convergence?)
This question is answered in \cref{chap:results}.

\begin{example}
    \label{exa:matrix_direct_control}
    Consider the HJBQVI \cref{eqn:results_fex_hjbqvi} of \cref{exa:introduction_fex}.
    To simplify notation, assume the target \gls*{FEX} rate $m$ is zero.

    Let $\{ x_0, \ldots, x_M \}$ be a uniform partition of the interval $[-R, R]$ (i.e., $x_i = (i - M/2) \Delta x$ where $\Delta x = 2R / M$) and $Z^{\gls*{meshing_parameter}}(t, x) = \{ x_0 - x, \ldots, x_M - x \} $ be chosen as per \cref{exa:schemes_intervention}.
    To simplify notation, assume $M$ is even so that $x_{M/2} = m = 0$.

    The direct control scheme applied to the HJBQVI \cref{eqn:results_fex_hjbqvi} is given by the matrix-valued function $A$ and vector-valued function $y$ specified by (compare with \cref{eqn:matrix_direct_control_parameters})
    \begin{align*}
        [A(P) U]_i
        & = \textcolor{ctrlcolor}{\overline{d_i}} \left(
            \frac{U_i}{\Delta \tau}
            - \frac{1}{2} \sigma^2 (\mathcal{D}_2 U)_i
            + \mu \textcolor{ctrlcolor}{w_i} (\mathcal{D}_- U)_i
        \right)
        + \textcolor{ctrlcolor}{d_i} \left(
            U_i
            - \gls*{interp}(
                U,
                x_i + \textcolor{ctrlcolor}{z_i}
            )
        \right)
        \\
        [y(P)]_i
        & = \textcolor{ctrlcolor}{\overline{d_i}} \left(
            \frac{V_i^{n - 1}}{\Delta \tau}
            - e^{-\beta \tau^n} \left( \left(x_i - m\right)^2 + \gamma \textcolor{ctrlcolor}{w_i}^2 \right)
        \right)
        - \textcolor{ctrlcolor}{d_i} e^{-\beta \tau^n} \left( \kappa \left| \textcolor{ctrlcolor}{z_i} \right| + c \right).
    \end{align*}

    As discussed in \cref{subsec:matrix_direct_control}, we cannot apply policy iteration to the above parameterization of the Bellman problem since we have no guarantees on the nonsingularity of the matrices $A(P)$.
    In order to surmount this issue, we seek to construct a subset $\mathcal{P}^\prime$ on which to apply policy iteration.
    In particular, we choose $\mathcal{P}^\prime$ to be the subset of controls $P = (\textcolor{ctrlcolor}{w_i}, \textcolor{ctrlcolor}{z_i}, \textcolor{ctrlcolor}{d_i})_{i=0}^M \in \mathcal{P}$ satisfying
    \begin{equation}
        \textcolor{ctrlcolor}{d_{M/2}} = 0
        \text{,} \textspace
        \textcolor{ctrlcolor}{z_{i}} > 0 \text{ if } x_i < m
        \text{,} \textspace \text{and} \textspace
        \textcolor{ctrlcolor}{z_{i}} < 0 \text{ if } x_i > m.
        \label{eqn:matrix_direct_control_filter}
    \end{equation}
    In \cref{app:matrix_direct_control}, we verify that $\mathcal{P^\prime}$ satisfies the conditions of \cref{thm:matrix_restriction}, so that we may apply \Call{Policy-Iteration}{$\mathcal{P}^\prime$} to compute the solution of the direct control scheme at the $n$-th timestep.
\end{example}

\section{Application to Markov decision processes}

\emph{Undiscounted} infinite horizon Markov decision processes (\glspl*{MDP}) were first considered in \cite{MR0236983,MR0173536,MR0194243} and more recently, for example, in \cite{MR3420429,MR3285895}.
Unlike their \emph{discounted} counterparts, computing solutions of these \glspl*{MDP} is known to be a difficult task (see \cref{rem:matrix_undiscounted_difficult}), often amounting to solving a nonconvex optimization problem (with no known efficient general purpose solver).
In this section, we use our results regarding singularity in the Bellman problem to show that we can often apply policy iteration to solve these \glspl*{MDP}.

We repeat here the standard setting of an infinite horizon MDP from \cite{MR2182753}.
A \emph{controlled Markov chain} $(X^k)_{k \geq 0}$ is a sequence of random variables taking values in the finite state space $\{0, \ldots, M\}$ determined by transition probabilities
\begin{equation}
    \mathbb{P}(X^{k+1} = j \mid X^k = i) = T(i, j, \theta^k(i)) 
    \label{eqn:matrix_transition_probabilities}
\end{equation}
where $\theta = (\theta^k)_{k \geq 0}$, the \emph{control}, is a sequence of functions such that $\theta^k(i) \in \mathcal{P}_i$.
To simplify matters, we assume that the \emph{control set} $\mathcal{P}_i$ is a countable set.\footnote{The assumption that $\mathcal{P}_i$ is countable is required to avoid any measure theoretic issues. More generally, $\mathcal{P}_i$ can be an arbitrary Polish space \cite[Section 1.5]{MR2183196}.}
The controller's objective is to choose a control to maximize the total reward
\begin{equation}
    J(i; \theta) = \mathbb{E} \left[
        \sum_{k \geq 0} R(X^k, \theta^k(X^k)) \prod_{\ell = 0}^{k - 1} D(X^\ell, \theta^\ell(X^\ell))
        \middle | X_0 = i
    \right].
    \label{eqn:matrix_mdp_functional}
\end{equation}
We denote by $V_i = \sup_\theta J(i; \theta)$ the maximum total reward.
The function $R$ determines the reward obtained at each ``step'' $k$, while the function $D$ -- which takes values in $[0,1]$ -- determines the discount factor.
\cite[Lemma 5]{MR2302036} establishes that if the entries of the vector $V = (V_0, \ldots, V_M)^\intercal$ are finite (i.e., $|V_i| < \infty$), then $V$ is a solution of the Bellman problem \cref{eqn:matrix_bellman} with
\[
    [A(P)]_{ij} = \delta_{ij} - T(i, j, P_i) D(i, P_i)
    \textspace \text{ and } \textspace
    [y(P)]_i = R(i, P_i)
\]
where $\delta_{ij}$ is the Kronecker delta and as usual, $P = (P_0, \ldots, P_M) \in \mathcal{P}_0 \times \cdots \times \mathcal{P}_M$.

\begin{remark}
    \cref{eqn:matrix_transition_probabilities} and the laws of probability imply that $T \geq 0$ and $\sum_j T(\cdot, j, \cdot) = 1$.
    Since $D \geq 0$, it follows that for $i \neq j$, $[A(P)]_{ij} = -T(i, j, P_i) D(i, P_i) \leq 0$, and hence $A(P)$ is a Z-matrix.
    Using the fact that $D \leq 1$,
    \[
        \sum_j [A(P)]_{ij}
        = \sum_j \delta_{ij} - D(i, P_i) \sum_j T(i, j, P_i)
        = 1 - D(i, P_i)
        \geq 0,
    \]
    and hence $A(P)$ is \gls*{wdd}
    In particular, row $i$ of $A(P)$ is not \gls*{sdd} if and only if $D(i, P_i) = 1$.
    \label{rem:matrix_mdp_standard_assumption}
\end{remark}

\begin{remark}
    It has long been known that in the undiscounted setting, $V$ cannot be obtained by a na\"{i}ve application of policy iteration \cite[Page 149]{MR2182753}.
    However, if one assumes the nonnegativity of the reward function $R$, $V$ can instead be obtained by solving a related linear program \cite[Page 150]{MR2182753}.
    A similar result holds if the reward function is assumed to be negative (corresponding to costs), though the resulting program is not linear (or even convex) and hence hard to solve in general \cite[Page 150]{MR2182753}.
    Using the ideas developed in this chapter, we are able to consider the much more general case in which no assumptions are made about the sign of the reward function $R$.
    \label{rem:matrix_undiscounted_difficult}
\end{remark}

\begin{remark}
    We are careful not to refer to the index $k$ in \cref{eqn:matrix_mdp_functional} as ``time'' since this analogy only works in the discounted setting.
    In particular, if $D(X^k, \theta^k(X^k)) < 1$, the indices $k$ and $k + 1$ can be interpreted as times $t$ and $t + \Delta t$, since a reward obtained at index $k + 1$ is worth less than the same reward obtained at index $k$.
    Conversely, if $D(X^k, \theta^k(X^k)) = 1$, the indices $k$ and $k + 1$ should be interpreted as referring to the same time $t$, since no discounting is applied.
    This is analogous to impulse control, for which the left hand limit $t-$ and the right hand limit $t$ are used to refer to the instant before and after an impulse occurs (see, e.g., \cref{eqn:introduction_controlled_sde_across}).
    In this sense, undiscounted \glspl*{MDP} can be interpreted as impulse control problems where the time horizon and state space are discrete.
    \label{rem:matrix_time}
\end{remark}

\section{Summary}

Motivated by the direct control scheme, this chapter resolved the issue of singularity in the Bellman problem.
In doing so, some new results for \glspl*{MDP} and \gls*{wcdd} matrices were obtained.
We saw that unlike the penalty scheme, policy iteration applied to the direct control scheme requires additional effort to ensure convergence.

\setcounter{chapter}{3}
\chapter{Convergence of schemes for weakly nonlocal second order equations}
\chaptermark{Convergence of schemes for nonlocal equations}
\label{chap:convergence}

Consider the fully nonlinear second-order \emph{local} PDE
\begin{equation}
    G(x, V(x), D V(x), D^2 V(x)) = 0
    \textspace \text{for } x \in \overline{\Omega}
    \label{eqn:convergence_local_pde}
\end{equation}
where $\Omega$ is an open\footnote{The assumption that $\Omega$ is open is used to simplify presentation, being stronger than what is generally required for the theory of viscosity solutions (see, e.g., \cite[Theorem 3.2]{MR1118699}).} subset of $\mathbb{R}^d$,
$G : \mathbb{R} \times \mathbb{R} \times \mathbb{R}^d \times \mathcal{S}^d \rightarrow \mathbb{R}$
is a locally bounded function, and $\mathcal{S}^d$ is the set of $d \times d$ symmetric matrices. 
In \cref{eqn:convergence_local_pde}, $D V$ and $D^2 V$ are the (formal) gradient vector and Hessian matrix of $V$.
It is established in \cite{MR1115933} that if a numerical scheme for the PDE \cref{eqn:convergence_local_pde} is \emph{monotone}, \emph{stable}, and \emph{consistent}, then it converges to a viscosity solution of \cref{eqn:convergence_local_pde} provided that a \emph{comparison principle}, guaranteeing the uniqueness of solutions, holds. 
This result is referred to as the Barles-Souganidis framework, named after its authors.

The term ``local'' above is used to capture the fact that to determine if a function $V$ satisfies the PDE \cref{eqn:convergence_local_pde} at a particular point $x$ in the domain, it is necessary only to examine the values of $V$ and its partial derivatives evaluated at that point $x$ (at points where $V$ is not smooth, derivatives of particular test functions are considered instead).
However, the HJBQVI \cref{eqn:introduction_hjbqvi} is not a local PDE due to the intervention operator $\mathcal{M}$.

As such, we must broaden the class of PDEs which we consider.
In particular, we consider in this chapter PDEs having the form
\begin{equation}
    F(x, V(x), D V(x), D^2 V(x), [\mathcal{I} V](x)) = 0
    \textspace \text{for } x \in \overline{\Omega}
    \label{eqn:convergence_pde}
\end{equation}
where
$F : \mathbb{R} \times \mathbb{R} \times \mathbb{R}^d \times \mathcal{S}^d \times \mathbb{R} \rightarrow \mathbb{R}$
is a locally bounded function and $\mathcal{I}$ is an operator mapping from a space of functions to itself.
The purpose of $\mathcal{I}$ is to capture the nonlocal features of a PDE (e.g., in the case of the HJBQVI, $\mathcal{I} = \mathcal{M}$).
Since this nonlocality does not extend to the derivatives $D^i V$, we refer to equations of the form \cref{eqn:convergence_pde} as \emph{weakly nonlocal}.

It is well-known that for certain nonlocal operators $\mathcal{I}$ (e.g., integro-differential operators), the Barles-Souganidis framework can be applied to \cref{eqn:convergence_pde} without running into any major issues \cite{MR2182141}.
However, if $\mathcal{I} = \mathcal{M}$ is an intervention operator, this approach fails (see \cref{app:barles} for a detailed discussion of why it fails).
To overcome this issue, we extend the Barles-Souganidis framework to nonlocal PDEs of the form \cref{eqn:convergence_pde} in a very general manner.
This results in a consistency requirement that is stronger than what is typically used in the local case.
We use these results to prove the convergence of the schemes from \cref{chap:schemes} to the unique viscosity solution of the HJBQVI \cref{eqn:introduction_hjbqvi}.

Our contributions in this chapter are:
\begin{itemize}
    \item Extending the Barles-Souganidis framework to nonlocal PDEs in a general manner.
    \item Proving the convergence of the direct control, penalty, and explicit-impulse schemes to the viscosity solution of the HJBQVI.
\end{itemize}

The results of this chapter appear in our articles

\fullcite{azimzadeh2017convergence}

\fullcite{azimzadeh2017zero}

\section{Viscosity solutions}
\label{sec:convergence_viscosity_solution}

Recall that in the classical (i.e., differentiable) setting, establishing a \emph{maximum principle} is the usual technique for obtaining uniqueness of solutions to second order PDEs \cite[Chapter 2]{MR0181836}.
However, for many interesting (e.g., nonlinear) PDEs, it is unreasonable to expect to find differentiable solutions.

Viscosity solutions generalize the classical notion of a solution to a PDE.
In particular, viscosity solutions are ``weak enough'' to be applicable to functions that are only semicontinuous (i.e., neither differentiable nor even a priori continuous) but ``strong enough'' to allow the user to derive maximum principles which, in the viscosity setting, are referred to as \emph{comparison principles} \cite{MR1118699}.

Moreover, viscosity solutions are well-known to be the relevant notion of solution for optimal control problems, making them the relevant solution for problems arising from finance.
While being out of the scope of this thesis, we provide the reader with some references to various modern arguments establishing this fact.
One such argument is \emph{weak} dynamic programming \cite{MR2806570,MR2976505} which, unlike \emph{strong} dynamic programming \cite{MR709164}, does not require sophisticated tools from measure theory (e.g., measurable selection).
Another argument is stochastic Perron's method \cite{MR2929032,MR3124891,MR3162260}, which is arguably even more flexible as it does not rely on any dynamic programming style arguments.

In this section, we introduce the appropriate definition of viscosity solution for the PDEs in this thesis.
In order to do so, we first recall the notion of semicontinuity:

\begin{definition}
    Let $Y$ be a metric space and $u \colon Y \rightarrow \mathbb{R}$ be a locally bounded function.
    We define the upper and lower semicontinuous envelopes $u^*$ and $u_*$ of $u$ by
    \[
        u^*(x) = \limsup_{y \rightarrow x} u(y)
        \textspace \text{and} \textspace
        u_*(x) = \liminf_{y \rightarrow x} u(y).
    \]
    The function $u$ is said to be upper (resp. lower) semicontinuous if $u = u^*$ (resp. $u = u_*$) pointwise (\cref{fig:convergence_semicontinuous}).
\end{definition}


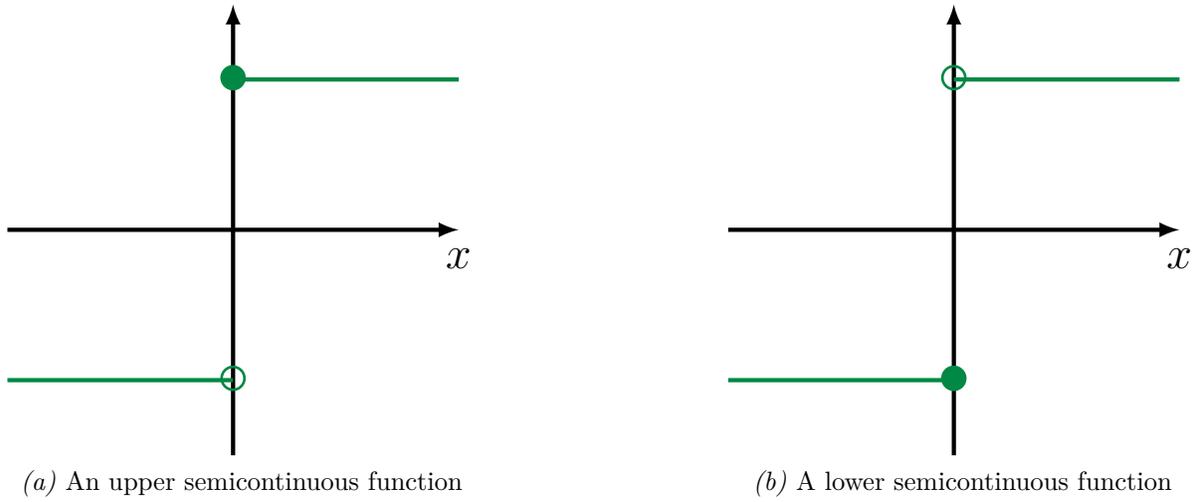
\begin{figure}
    \subfloat[An upper semicontinuous function]{
        \begin{tikzpicture}
            \draw [->, ultra thick] (-3, 0) -- (3, 0);
            \draw [->, ultra thick] (0, -3) -- (0, 3);
            \draw [colB, ultra thick]
                (-3, -2) -- (0, -2)
                (0, 2) -- (3, 2)
            ;
            \node [colB] at (0, -2) {\huge $\circ$};
            \node [colB] at (0, 2) {\huge \textbullet};
            \node [yshift=-0.1cm, below] at (3, 0) {\Large $x$};
        \end{tikzpicture}
    }
    \hfill{}
    \subfloat[A lower semicontinuous function]{
        \begin{tikzpicture}
            \draw [->, ultra thick] (-3, 0) -- (3, 0);
            \draw [->, ultra thick] (0, -3) -- (0, 3);
            \draw [colB, ultra thick]
                (-3, -2) -- (0, -2)
                (0, 2) -- (3, 2)
            ;
            \node [colB] at (0, -2) {\huge \textbullet};
            \node [colB] at (0, 2) {\huge $\circ$};
            \node [yshift=-0.1cm, below] at (3, 0) {\Large $x$};
        \end{tikzpicture}
    }
    \caption{Examples of semicontinuous functions}
    \label{fig:convergence_semicontinuous}
\end{figure}

We are now ready to give the definition of viscosity solution for \cref{eqn:convergence_pde}.

\begin{definition}
    An upper (resp. lower) semicontinuous function $V \colon \overline{\Omega} \rightarrow \mathbb{R}$ is a viscosity subsolution (resp. supersolution) of \cref{eqn:convergence_pde} if for all $\varphi \in C^2(\overline{\Omega})$ and $x \in \overline{\Omega}$ such that $V - \varphi$ has a local maximum (resp. minimum) at $x$ (\cref{fig:convergence_extrema}), we have
    \begin{align*}
        F_*(x, V(x), D \varphi(x), D^2 \varphi(x), [\mathcal{I}V](x)) & \leq 0
        \\
        \text{ (resp. }
        F^*(x, V(x), D \varphi(x), D^2 \varphi(x), [\mathcal{I}V](x)) & \geq 0
        \text{)}.
    \end{align*}
    We say $V$ is a viscosity solution if it is both a subsolution and a supersolution.
    Since viscosity solutions are the only solution concept used in this chapter, we will often omit the prefix ``viscosity'' and simply write ``subsolution'', ``supersolution'', or ``solution''.
    \label{def:convergence_viscosity_solution}
\end{definition}

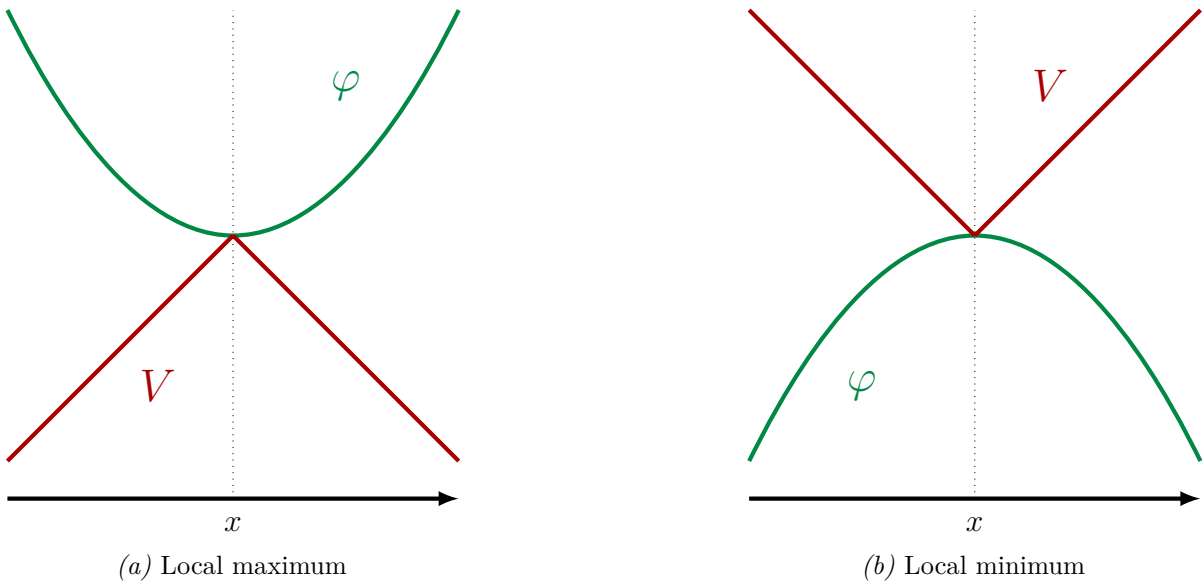
\begin{figure}
    \subfloat[Local maximum]{ 
        \begin{tikzpicture}
            \draw [ultra thick, colB]
                (0,0) parabola (3,3)
                (0,0) parabola (-3,3)
            ;
            \draw [ultra thick, colC]
                (0,0) -- (3,-3)
                (0,0) -- (-3,-3)
            ;
            \node [colB] at (1.5, 2) {\Large $\varphi$};
            \node [colC] at (-1, -2) {\Large $V$};
            \draw [dotted] (0, -3.5) -- (0, 3);
            \node [yshift=-0.1cm, below] at (0, -3.5) {$x$};
            \draw [->, ultra thick] (-3, -3.5) -- (3, -3.5);
        \end{tikzpicture}
    }
    \hfill{}
    \subfloat[Local minimum]{ 
        \begin{tikzpicture}
            \draw [ultra thick, colB]
                (0,0) parabola (3,-3)
                (0,0) parabola (-3,-3)
            ;
            \draw [ultra thick, colC]
                (0,0) -- (3,3)
                (0,0) -- (-3,3)
            ;
            \node [colB] at (-1.5,-2) {\Large $\varphi$};
            \node [colC] at (1, 2) {\Large $V$};
            \draw [dotted] (0, -3.5) -- (0, 3);
            \node [yshift=-0.1cm, below] at (0, -3.5) {$x$};
            \draw [->, ultra thick] (-3, -3.5) -- (3, -3.5);
        \end{tikzpicture}
    }
    \caption{Role of test functions in the definition of viscosity solution}
    \label{fig:convergence_extrema}
\end{figure}


If $F$ is continuous, the definition above agrees with the common definition for nonlocal problems given in \cite{MR2422079}.
We have, however, allowed $F$ to be discontinuous so that we may write both the partial differential equation and its boundary conditions as a single expression as in \cite{MR1115933}.
This is demonstrated in the example below.

\begin{example}
    \label{exa:convergence_dirichlet}
    Consider the Dirichlet problem
    \begin{align*}
        V(x_d) & = y_d
        \\
        - \frac{1}{2} \sigma^2 x^2 V_{xx}
        + \lambda \kappa x V_x
        + \lambda V
        - \lambda \int_{ x_d / x }^{ x_u / x } V(\eta x) \psi(\eta) d \eta
        - \lambda f
        & = 0 \textspace \text{on } (x_d, x_u)
        \\
        V(x_u) & = y_u
    \end{align*}
    where $0 < x_d < x_u$, $\sigma$ and $\lambda$ are positive constants, $\psi$ is a lognormal probability density, $\kappa = \int_0^\infty \eta \psi(\eta) d\eta - 1 \geq 0$, and
    \[
        f(x)
        = y_d \int_0^{ x_d / x } \psi(\eta) d \eta
        + y_u \int_{ x_u / x }^\infty \psi(\eta) d \eta.
    \]
    In a market where the riskless rate of return is zero, this Dirichlet problem corresponds to the price of an infinite horizon barrier option which pays off $y_d$ if the price of the asset, which follows a geometric Brownian motion with lognormally distributed jumps, drops below the barrier $x_d$ and pays off $y_u$ if it goes above the barrier $x_u$ \cite{MR2053597}.

    We can rewrite this problem in the form \eqref{eqn:convergence_pde} by defining $\Omega = (x_d, x_u)$,
    \begin{equation}
        [\mathcal{I} V](x)
        = \int_{x_d / x}^{x_u / x} V(\eta x) \psi(\eta) d \eta,
        \label{eqn:convergence_jump}
    \end{equation}
    and
    \[
        F(x, r, p, A, \wideEll)
        = \begin{cases}
            r - y_d & \text{if } x = x_d
            \\
            H(x, r, p, A, \wideEll)
            & \text{if } x \in \Omega
            \\
            r - y_u & \text{if } x = x_u
        \end{cases}
    \]
    where $
        H(x, r, p, A, \wideEll)
        = - \frac{1}{2} \sigma^2 x^2 A + \lambda \kappa x p + \lambda r - \lambda \wideEll - \lambda f(x)
    $.
    In view of \cref{def:convergence_viscosity_solution}, note that
    \begin{multline*}
        F_*(x, V(x), \varphi^\prime(x), \varphi^{\prime\prime}(x), [\mathcal{I}V](x))
        \\
        = \begin{cases}
            \min \{
                H(x, V(x), \varphi^\prime(x), \varphi^{\prime\prime}(x), [\mathcal{I}V](x))
                ,
                V(x) - y_d
            \}
            & \text{if } x = x_d
            \\
            \phantom{ \min \{ }
                H(x, V(x), \varphi^\prime(x), \varphi^{\prime\prime}(x), [\mathcal{I}V](x))
            \phantom{ \} }
            & \text{if } x \in \Omega
            \\
            \min \{
                H(x, V(x), \varphi^\prime(x), \varphi^{\prime\prime}(x), [\mathcal{I}V](x))
                ,
                V(x) - y_u
            \}
            & \text{if } x = x_u
        \end{cases}
    \end{multline*}
    and similarly,
    \begin{multline*}
        F^*(x, V(x), \varphi^\prime(x), \varphi^{\prime\prime}(x), [\mathcal{I}V](x))
        \\
        = \begin{cases}
            \max \{
                H(x, V(x), \varphi^\prime(x), \varphi^{\prime\prime}(x), [\mathcal{I}V](x))
                ,
                V(x) - y_d
            \}
            & \text{if } x = x_d
            \\
            \phantom{ \max \{ }
                H(x, V(x), \varphi^\prime(x), \varphi^{\prime\prime}(x), [\mathcal{I}V](x))
            \phantom{ \} }
            & \text{if } x \in \Omega
            \\
            \max \{
                H(x, V(x), \varphi^\prime(x), \varphi^{\prime\prime}(x), [\mathcal{I}V](x))
                ,
                V(x) - y_u
            \}
            & \text{if } x = x_u.
        \end{cases}
    \end{multline*}
    In the above, we see a ``mixing'' of the boundary conditions and the PDE operator $H(\cdot)$ at $x = x_d$ and $x = x_u$.
    This phenomenon, which we will encounter again in this chapter, is referred to as a boundary layer \cite[Remark 3.2]{MR2857450}.
\end{example}

\section{Numerical schemes}
\label{sec:convergence_schemes}

In \cite{MR1115933}, the authors find sufficient conditions under which one can approximate the solution $V$ of the local PDE \cref{eqn:convergence_local_pde} using an ``approximate solution'' $V^{\gls*{meshing_parameter}}$ that converges to $V$ as $\gls*{meshing_parameter}$ goes to zero.
Intuitively, $V^{\gls*{meshing_parameter}}$ is obtained by a numerical method (e.g., finite differences) and $\gls*{meshing_parameter}$ is a parameter that controls, roughly speaking, the accuracy of the method (e.g., a smaller value of $h$ corresponds to a finer grid).

We follow a similar approach, approximating a solution $V$ of the nonlocal PDE \cref{eqn:convergence_pde} using an approximate solution $V^{\gls*{meshing_parameter}} \colon \overline{\Omega} \rightarrow \mathbb{R}$.
Precisely, $V^{\gls*{meshing_parameter}}$ is a function
which satisfies
\begin{equation}
    S(
        \gls*{meshing_parameter},
        x,
        V^{\gls*{meshing_parameter}},
        [\mathcal{I}^{\gls*{meshing_parameter}} V^{\gls*{meshing_parameter}}](x)
    ) = 0 \textspace \text{for } x \in \overline{\Omega}
    \label{eqn:convergence_scheme}
\end{equation}
(in the time-dependent case, $x$ is replaced by $t,x$).
Together, the pair $(S, \mathcal{I}^{\gls*{meshing_parameter}})$ are meant be an abstract representation of a numerical scheme with $\mathcal{I}^{\gls*{meshing_parameter}}$ being an approximation to the operator $\mathcal{I}$.
Precisely, letting $\gls*{bdd}$ denote the set of bounded real-valued functions mapping from $\overline{\Omega}$, $S$ is a real-valued function mapping from $(0,\infty) \times \overline{\Omega} \times \gls*{bdd} \times \mathbb{R}$ and the operator $\mathcal{I}^{\gls*{meshing_parameter}}$ is a map from $\gls*{bdd}$ to itself.
The idea behind the abstract representation of a numerical scheme \cref{eqn:convergence_scheme} is made clear in the example below.

\begin{example}
    Consider the Dirichlet problem described in \cref{exa:convergence_dirichlet}.
    Let $\{ x_0, \ldots, x_M \}$ be a partition 
    of the interval $[x_d, x_u]$ and, for convenience, define $x_{-1} = x_0$ and $x_{M+1} = x_M$.
    The change of variables $\zeta = \log(\eta x)$ in the integral in \cref{eqn:convergence_jump} yields
    \[
        [\mathcal{I} V](x)
        = \int_{\log x_d}^{\log x_u}
            (e^\zeta / x) V(e^\zeta) \psi(e^\zeta / x)
        d \zeta
        = \sum_{0 \leq j \leq M}
            \int_{ (\log x_{j-1} + \log x_j)/2 }^{ (\log x_j + \log x_{j+1})/2 }
                (e^\zeta / x) V(e^\zeta) \psi(e^\zeta / x)
            d \zeta
        .
    \]
    The above suggests the numerical scheme $(S, \mathcal{I}^{\gls*{meshing_parameter}})$ given by
    \[
        [\mathcal{I}^{\gls*{meshing_parameter}} V](x_i)
        = \sum_{0 \leq j \leq M}
            V(x_j)
            \int_{ (\log x_{j-1} + \log x_j)/2 }^{ (\log x_j + \log x_{j+1})/2 }
                (e^\zeta / x_i) \psi(e^\zeta / x_i)
            d \zeta
    \]
    and, using the shorthand $V_i = V(x_i)$ and $V = (V_0, \ldots, V_M)^\intercal$,
    \[
        S(\gls*{meshing_parameter}, x_i, V, \wideEll)
        = \begin{cases}
            V_i - y_d & \text{if } i = 0 \\
            H(
                x_i, \,
                V_i , \,
                (\mathcal{D}_- V)_i , \,
                (\mathcal{D}_2 V)_i , \,
                \wideEll
            )
            & \text{if } 0 < i < M \\
            V_i - y_u & \text{if } i = M \\
        \end{cases}
    \]
    where $H$ is defined in \cref{exa:convergence_dirichlet} and $\mathcal{D}_2$ and $\mathcal{D}_-$ are defined in \cref{eqn:schemes_second_derivative,eqn:schemes_first_derivative_pm}.
\end{example}

\begin{remark}
    \label{rem:convergence_grid_points}
    For fixed $\gls*{meshing_parameter} > 0$, the scheme in the above example only defines the numerical solution $V^{\gls*{meshing_parameter}}$ at grid points $x_i$.
    Following \cite[Example 3]{MR1115933}, we extend the numerical solution to points not on the grid by assuming that it is piecewise constant:
    \[
        V^{\gls*{meshing_parameter}}(x) = V^{\gls*{meshing_parameter}}(x_i)
        \textspace \text{if } x \in
        \left[
            \frac{x_{i-1} + x_i}{2},
            \frac{x_i + x_{i+1}}{2}
        \right)
    \]
    where we have used the convention $x_{-1} = x_0$ and $x_{M+1} = x_M$.
    This extension to points not on the grid can be enforced ``indirectly'' by defining $[\mathcal{I}^{\gls*{meshing_parameter}} \, \cdot](x)$ and $S(\gls*{meshing_parameter}, x, \cdot, \cdot)$ for $x \neq x_i$ appropriately. 
\end{remark}

As a technical condition, we will assume that there exists a nonnegative function $f$ mapping from $(0, \infty) \times \overline{\Omega}$ such that
$
	\lim_{\substack{
		\gls*{meshing_parameter} \rightarrow 0 \\
		y \rightarrow x
	}}
	f(y, \gls*{meshing_parameter}) = 0
$
and
\begin{equation}
	\left |
	S(
		\gls*{meshing_parameter},
		x,
		\psi,
		\cdot
	)
	-
	S(
		\gls*{meshing_parameter},
		x,
		\varphi,
		\cdot
	)
	\right |
	\leq f(x, \gls*{meshing_parameter})
	\label{eq:convergence_technical_assumption}
\end{equation}
for all $\gls*{meshing_parameter}$, $x$, $\varphi$, and $\psi$ defined by $\psi(y) = \varphi(y) \pm e^{-1 / \gls*{meshing_parameter}} \boldsymbol{1}_{\{x\}}(y)$.
The reader can verify that this is trivially satisfied for all schemes considered in this thesis (including the one above).

\section{Application to the HJBQVI}
\label{sec:convergence_hjbqvi}

In this section, we revisit the schemes of \cref{chap:schemes}.
Our goal is to show that the solutions of these schemes converge to the viscosity solution of the HJBQVI \cref{eqn:introduction_hjbqvi}.

For the purposes of demonstration, we take the domain $\gls*{domain}$ in \cref{eqn:introduction_hjbqvi} to be the real line $\mathbb{R}$.
To simplify notation, we assume that the grid points $\{x_0, \ldots, x_M\}$ are uniformly spaced.
In particular, we take $x_i = (i - M/2) \Delta x$ so that the approximate solution is computed on the ``numerical domain'' $[0, T] \times [-(M/2) \Delta x, (M/2) \Delta x]$.
To ensure that the numerical grid ``approximates'' the original domain $[0, T] \times \mathbb{R}$ as $\gls*{meshing_parameter} \rightarrow 0$, we assume
\begin{equation}
    \Delta \tau = \gls*{const} \gls*{meshing_parameter},
    \textspace
    \Delta x = \gls*{const} \gls*{meshing_parameter},
    \textspace
    \text{and}
    \textspace
    M \Delta x \rightarrow \infty \text{ as } \gls*{meshing_parameter} \rightarrow 0.
    \label{eqn:convergence_grid}
\end{equation}

    While the assumption $M \Delta x \rightarrow \infty$ simplifies the presentation, it is computationally intractable to implement in practice as it requires ``enlarging'' the numerical domain whenever the spatial grid is refined (e.g., by taking $M = \gls*{const} \lceil \gls*{meshing_parameter}^{- 1 - \alpha} \rceil$ for some $\alpha > 0$ where $\lceil \cdot \rceil$ is the ceiling function).
    Fortunately, for elliptic PDEs, truncating the domain to a bounded region and introducing ``\artificial{}'' boundary conditions is expected to lead to small errors on the interior \cite[Section 5]{MR1470506}.
    In \cref{app:truncated}, we truncate the domain of the HJBQVI to a bounded set, introducing \artificial{} boundary conditions and modifying our schemes accordingly.
    Therein, we also give the necessary modifications required for our convergence proofs to work in the truncated setting.
    Since these arguments are somewhat standard, we confine them to the appendix so that we may, in this chapter, focus on the main difficulties arising from discretizing the HJBQVI.

To write the HJBQVI \cref{eqn:introduction_hjbqvi} in the form \cref{eqn:convergence_pde}, we take $\Omega = [0, T) \times \mathbb{R}$,\footnote{Though $\Omega = [0, T) \times \mathbb{R}$ is not an open set, this poses no issues in our analyses since the coefficient of the time derivative $q = V_t(t, x)$ in the HJBQVI is negative (cf. \cite[Remark 7.1]{MR2976505}).} $\mathcal{I} = \mathcal{M}$, and
\begin{multline}
    F((t, x), r, (q, p), A, \wideEll)
    \\
    = \begin{cases}
        \min \left\{
            - q - \sup_{\textcolor{ctrlcolor}{w} \in W} \left\{
                \frac{1}{2} b(t, x, \textcolor{ctrlcolor}{w})^2 A
                + a(t, x, \textcolor{ctrlcolor}{w}) p
                + f(t, x, \textcolor{ctrlcolor}{w})
            \right\}, r - \wideEll
        \right\}
        & \text{if } t < T \\
        \min \left\{
            r - g(x),
            r - \wideEll
        \right\}
        & \text{if } t = T
    \end{cases}
    \label{eqn:convergence_hjbqvi}
\end{multline}
where we have used the notation $(q, p)$ to distinguish between the time and spatial derivatives (i.e., $q = V_t(t, x)$ and $p = V_x(t, x)$) and, with a slight abuse of notation, $A$ to refer only to the second spatial derivative (i.e., $A = V_{xx}(t, x)$) since no second time derivatives appear.
Throughout this section, when we refer to ``the HJBQVI'', it is understood that we mean the above.

In order to perform our analysis, we must first rewrite the direct control scheme \cref{eqn:schemes_direct_control}, the penalty scheme \cref{eqn:schemes_penalty_rearranged}, and the explicit-impulse scheme \cref{eqn:schemes_explicit_impulse_rearranged} in the form \cref{eqn:convergence_scheme} by defining $S$ and $\mathcal{I}^{\gls*{meshing_parameter}}$ appropriately.
We do so below, using the usual shorthand  $V_i^n = V(\tau^n, x_i)$ and $V^n = (V_0^n, \ldots, V_M^n)^\intercal$.

For the direct control scheme, we define $\mathcal{I}^{\gls*{meshing_parameter}}$ by
\begin{equation}
    [\mathcal{I}^{\gls*{meshing_parameter}} V](\tau^n, x_i)
    = (\mathcal{M}_n V^n)_i
    \label{eqn:convergence_direct_control_nonlocal}
\end{equation}
where $\mathcal{M}_n$ is defined in \cref{eqn:schemes_discretized_intervention} and $S$ by
\begin{equation}
    S(h, (\tau^n, x_i), V, \wideEll)
    = V_i^n - g(x_i)
    \textspace \text{if } n = 0 \text{ and } 0 \leq i \leq M
    \label{eqn:convergence_direct_control_boundary}
\end{equation}
and
\begin{multline}
    S(h, (\tau^n, x_i), V, \wideEll) \\
    = - \sup_{\substack{
        \textcolor{ctrlcolor}{d_i} \in \{0,1\} \\
        \textcolor{ctrlcolor}{w_i} \in W^{\gls*{meshing_parameter}}
    }} \left\{
        \textcolor{ctrlcolor}{\overline{d_i}} \left(
            \frac{V_i^{n-1} - V_i^n}{\Delta \tau}
            + \frac{1}{2} b_i^n(\textcolor{ctrlcolor}{w_i})^2 (\mathcal{D}_2 V^n)_i
            + a_i^n(\textcolor{ctrlcolor}{w_i}) (\mathcal{D} V^n)_i
            + f_i^n(\textcolor{ctrlcolor}{w_i})
        \right)
        + \textcolor{ctrlcolor}{d_i} \left( \wideEll - V_i^n \right)
    \right\} \\
    \text{if } n > 0 \text{ and } 0 \leq i \leq M.
    \label{eqn:convergence_direct_control_interior}
\end{multline}

For the penalty scheme, we define $\mathcal{I}^{\gls*{meshing_parameter}}$ by \cref{eqn:convergence_direct_control_nonlocal} and $S$ by \cref{eqn:convergence_direct_control_boundary} and
\begin{multline}
    S(h, (\tau^n, x_i), V, \wideEll) = -
        \sup_{\substack{
            \textcolor{ctrlcolor}{d_i} \in \{0,1\} \\
            \textcolor{ctrlcolor}{w_i} \in W^{\gls*{meshing_parameter}}
        }} \left\{
            \textcolor{ctrlcolor}{\overline{d_i}} \gamma_i^n(V, \textcolor{ctrlcolor}{w_i})
            + \textcolor{ctrlcolor}{d_i} \left( \wideEll - V_i^n + \epsilon \gamma_i^n(V, \textcolor{ctrlcolor}{w_i}) \right)
        \right\} \\
        \text{if } n > 0 \text{ and } 0 \leq i \leq M
    \label{eqn:convergence_penalty_interior}
\end{multline}
where $\gamma_i^n(V, \textcolor{ctrlcolor}{w_i})$ is defined in \cref{eqn:schemes_penalty_gamma}.
It is understood that the parameter $\epsilon > 0$ in the above vanishes as the grid is made finer.
Precisely, we mean that
\begin{equation}
    \epsilon \rightarrow 0
    \textspace \text{as} \textspace
    h \rightarrow 0.
    \label{eqn:convergence_penalty_vanishes}
\end{equation}

Lastly, for the explicit-impulse scheme, we define $\mathcal{I}^{\gls*{meshing_parameter}}$ by
\begin{equation}
    [\mathcal{I}^{\gls*{meshing_parameter}} V](\tau^n, x_i)
    = (\mathcal{M}_n V^{n-1})_i
    \label{eqn:convergence_explicit_impulse_nonlocal}
\end{equation}
(compare with \cref{eqn:convergence_direct_control_nonlocal}) and $S$ by \cref{eqn:convergence_direct_control_boundary} and
\begin{multline}
    S(h, (\tau^n, x_i), V, \wideEll) \\
    = -\sup_{\substack{
        \textcolor{ctrlcolor}{d_i} \in \{0,1\} \\
        \textcolor{ctrlcolor}{w_i} \in W^{\gls*{meshing_parameter}}
    }}
    \biggl\{
        \textcolor{ctrlcolor}{\overline{d_i}}
        \left(
            \displaystyle{
                \frac{\gls*{interp}(
                    V^{n-1},
                    x_i + a_i^n(\textcolor{ctrlcolor}{w_i}) \Delta \tau
                ) - V_i^n}{\Delta \tau}
            }
            + \frac{1}{2} (b_i^n)^2 (\mathcal{D}_2 V^n)_i
            + f_i^n(\textcolor{ctrlcolor}{w_i})
        \right) \\
        + \textcolor{ctrlcolor}{d_i} \left(
            \wideEll - V_i^n
            + \frac{1}{2} (b_i^n)^2 (\mathcal{D}_2 V^n)_i \Delta \tau
        \right)
    \biggr\} \textspace \text{if } n > 0  \text{ and } 0 \leq i \leq M.
    \label{eqn:convergence_explicit_impulse_interior}
\end{multline}

In order to prove convergence, we make the following assumptions about the quantities appearing in the HJBQVI and intervention operator $\mathcal{M}$:

\begin{enumerate}[label=(H\arabic*)]
    \item
        \label{enum:convergence_continuity}
        \label{enum:convergence_bounded}
        \label{enum:convergence_suboptimality}
        $f$, $g$, $a$, $b$, $\Gamma$, and $K$ are continuous functions with $f$ and $g$ bounded and $\mathcal{M} g \leq g$.
    \item
        $W$ is a nonempty compact metric space and there exists a metric space $Y$ such that $Z(t,x)$ is a nonempty compact subset of $Y$ for each $(t,x)$.
    \item
        \label{enum:convergence_negative_cost}
        $\sup_{t, x, z} K(t, x, z) < 0$.
\end{enumerate}

    The assumption \cref{enum:convergence_negative_cost}, to be used in the stability proofs, 
    can be interpreted as the controller paying a cost for the right to perform an impulse.
    \cref{rem:convergence_relax_cost} discusses how this assumption can be relaxed.

\subsection{Monotonicity}

In this subsection, we prove the monotonicity of our schemes.
A scheme $(S, \mathcal{I}^{\gls*{meshing_parameter}})$ is \emph{monotone} (cf. \cite[Eq. (2.2)]{MR1115933}) if for all 
$h \in (0,\infty)$, $x \in \overline{\Omega}$, $\wideEll \in \mathbb{R}$, and $V, \hat{V} \in \gls*{bdd}$,
\[
    S(h, x, V, \wideEll) \leq S(h, x, \hat{V}, \wideEll)
    \textspace \text{whenever} \textspace
    V \geq \hat{V} \text{ and } V(x) = \hat{V}(x).
\]

    Note that the monotonicity requirement above does not involve the operator $\mathcal{I}^{\gls*{meshing_parameter}}$ and as such, the reader may be tempted to guess that high order discretizations of the nonlocal operator $\mathcal{I}$ are possible. 
    Unfortunately, this is not the case: high order discretizations are generally precluded by the stronger consistency requirement of our framework (see \cref{rem:proofs_higher_order} for further explanation).

Due to our choice of upwind discretization \cref{eqn:schemes_first_derivative} and first order interpolation in \cref{eqn:schemes_lagrangian_derivative}, all the schemes of \cref{chap:schemes} are monotone by construction.
We establish this below in a series of lemmas.

\begin{lemma}
    The direct control scheme is monotone.
    \label{lem:convergence_direct_control_monotone}
\end{lemma}

\begin{proof}
    It is sufficient to check the monotonicity condition at grid points (see \cref{rem:convergence_grid_points}).
    In light of this, let $(\tau^n, x_i)$ be an arbitrary grid point and let $V$ and $\hat{V}$ be two functions satisfying $V \geq \hat{V}$ and $V_i^n = \hat{V}_i^n$ where we have, as usual, employed the shorthand $V_i^n = V(\tau^n, x_i)$ and $\hat{V}_i^n = \hat{V}(\tau^n, x_i)$.
    By \cref{eqn:convergence_direct_control_boundary}, if $n = 0$,
    \[
        S(h, (\tau^n, x_i), V, \wideEll) - S(h, (\tau^n, x_i), \hat{V}, \wideEll)
        = V_i^n - \hat{V}_i^n
        = 0.
    \]
    By \cref{eqn:convergence_direct_control_interior}, if $n > 0$,
    \begin{multline}
        S(h, (\tau^n, x_i), V, \wideEll) - S(h, (\tau^n, x_i), \hat{V}, \wideEll)
        \\
        \leq \sup_{\substack{
            \textcolor{ctrlcolor}{d_i} \in \{0,1\} \\
            \textcolor{ctrlcolor}{w_i} \in W^{\gls*{meshing_parameter}}
        }} \biggl\{
            \textcolor{ctrlcolor}{\overline{d_i}} \biggl(
                \frac{\hat{V}_i^{n-1} - V_i^{n-1}}{\Delta \tau}
                + \frac{1}{2} b_i^n(\textcolor{ctrlcolor}{w_i})^2 (\mathcal{D}_2 U)_i
                + \left| a_i^n(\textcolor{ctrlcolor}{w_i}) \right| (\mathcal{D}_+ U)_i
                \boldsymbol{1}_{ \{ a_i^n(\textcolor{ctrlcolor}{w_i}) > 0 \} }
                \\
                - \left| a_i^n(\textcolor{ctrlcolor}{w_i}) \right| (\mathcal{D}_- U)_i
                \boldsymbol{1}_{ \{ a_i^n(\textcolor{ctrlcolor}{w_i}) < 0 \} }
            \biggr)
        \biggr\}
        \label{eqn:convergence_direct_control_monotone_1}
    \end{multline}
    where, in the above, we have defined the vector $U = \hat{V}^n - V^n = (\hat{V}^n_0 - V^n_0, \ldots, \hat{V}^n_M - V^n_M)^\intercal$ and treated $\mathcal{D}_2$ and $\mathcal{D}_\pm$ (defined in \cref{eqn:schemes_first_derivative_pm,eqn:schemes_second_derivative}) as matrices.
    The indicator functions $\boldsymbol{1}_{ \{ a_i^n(\textcolor{ctrlcolor}{w_i}) > 0 \} }$ and $\boldsymbol{1}_{ \{ a_i^n(\textcolor{ctrlcolor}{w_i}) < 0 \} }$ are used to capture the fact that due to the upwind discretization, the choice of stencil ($\mathcal{D}_+$ or $\mathcal{D}_-$) depends on the sign of the coefficient of the first derivative.
    Using the fact that the vector $U$ has nonpositive entries and $U_i = V_i^n - \hat{V}_i^n = 0$, the inequalities $(\mathcal{D}_2 U)_i \leq 0$, $(\mathcal{D}_+ U)_i \leq 0$, and $(\mathcal{D}_- U)_i \geq 0$ follow immediately from \cref{eqn:schemes_first_derivative_pm,eqn:schemes_second_derivative}.
    Along with the fact that $\hat{V}_i^{n-1} - V_i^{n-1} \leq 0$, this implies that the right hand side of \cref{eqn:convergence_direct_control_monotone_1} is nonpositive, as desired.
\end{proof}

\begin{lemma}
    The penalty scheme is monotone.
\end{lemma}

\begin{proof}
    The proof is identical to that of \cref{lem:convergence_direct_control_monotone} except that instead of \cref{eqn:convergence_direct_control_monotone_1}, we have, by \cref{eqn:convergence_penalty_interior},
    \begin{multline*}
        S(h, (\tau^n, x_i), V, \wideEll) - S(h, (\tau^n, x_i), \hat{V}, \wideEll)
        \leq \sup_{\substack{
            \textcolor{ctrlcolor}{d_i} \in \{0,1\} \\
            \textcolor{ctrlcolor}{w_i} \in W^{\gls*{meshing_parameter}}
        }} \left\{
            \left(
                \textcolor{ctrlcolor}{\overline{d_i}}
                + \epsilon \textcolor{ctrlcolor}{d_i}
            \right) \left(
                \gamma_i^n(\hat{V}, \textcolor{ctrlcolor}{w_i})
                - \gamma_i^n(V, \textcolor{ctrlcolor}{w_i})
            \right)
        \right\} \\
        = \sup_{\substack{
            \textcolor{ctrlcolor}{d_i} \in \{0,1\} \\
            \textcolor{ctrlcolor}{w_i} \in W^{\gls*{meshing_parameter}}
        }} \biggl\{
            \left(
                \textcolor{ctrlcolor}{\overline{d_i}}
                + \epsilon \textcolor{ctrlcolor}{d_i}
            \right)
            \biggl(
                \frac{\hat{V}_i^{n-1} - V_i^{n-1}}{\Delta \tau}
                + \frac{1}{2} b_i^n(\textcolor{ctrlcolor}{w_i})^2 (\mathcal{D}_2 U)_i
                + \left| a_i^n(\textcolor{ctrlcolor}{w_i}) \right| (\mathcal{D}_+ U)_i
                \boldsymbol{1}_{ \{ a_i^n(\textcolor{ctrlcolor}{w_i}) > 0 \} }
                \\
                - \left| a_i^n(\textcolor{ctrlcolor}{w_i}) \right| (\mathcal{D}_- U)_i
                \boldsymbol{1}_{ \{ a_i^n(\textcolor{ctrlcolor}{w_i}) < 0 \} }
            \biggr)
        \biggr\} \leq 0. \qedhere
    \end{multline*}
\end{proof}

\begin{lemma}
    The explicit-impulse scheme is monotone.
\end{lemma}

\begin{proof}
    The proof is identical to that of \cref{lem:convergence_direct_control_monotone} except that instead of \cref{eqn:convergence_direct_control_monotone_1}, we have, by \cref{eqn:convergence_explicit_impulse_interior},
    \begin{multline}
        S(h, (\tau^n, x_i), V, \wideEll) - S(h, (\tau^n, x_i), \hat{V}, \wideEll)
        \leq \sup_{\substack{
            \textcolor{ctrlcolor}{d_i} \in \{0,1\} \\
            \textcolor{ctrlcolor}{w_i} \in W^{\gls*{meshing_parameter}}
        }} \biggl\{
            \left(
                \textcolor{ctrlcolor}{\overline{d_i}} + \Delta \tau \textcolor{ctrlcolor}{d_i}
            \right)
            \frac{1}{2} (b_i^n)^2 (\mathcal{D}_2 U)_i \\
            + \textcolor{ctrlcolor}{\overline{d_i}} \left(
                \frac{
                    \gls*{interp}(
                        \hat{V}^{n-1},
                        x_i + a_i^n(\textcolor{ctrlcolor}{w_i}) \Delta \tau
                    )
                    - \gls*{interp}(
                        V^{n-1},
                        x_i + a_i^n(\textcolor{ctrlcolor}{w_i}) \Delta \tau
                    )
                }{\Delta \tau}
            \right)
        \biggr\}.
        \label{eqn:convergence_explicit_impulse_monotone_1}
    \end{multline}
    By \cref{eqn:schemes_interpolation},
    \begin{multline*}
        \gls*{interp}(
            \hat{V}^{n-1},
            x_i + a_i^n(\textcolor{ctrlcolor}{w_i}) \Delta \tau
        )
        - \gls*{interp}(
            V^{n-1},
            x_i + a_i^n(\textcolor{ctrlcolor}{w_i}) \Delta \tau
        )
        \\
        = \alpha \left(
            \hat{V}_{k+1}^{n-1} - V_{k+1}^{n-1}
        \right)
        + \left(1 - \alpha\right) \left(
            \hat{V}_k^{n-1} - V_k^{n-1}
        \right)
        \leq 0
    \end{multline*}
    where $0 \leq \alpha \leq 1$ and $0 \leq k < M$.
    Therefore, the right hand side of \cref{eqn:convergence_explicit_impulse_monotone_1} is nonpositive, as desired.
\end{proof}

\subsection{Stability}

In this subsection, we prove the stability of our schemes.
A scheme $(S, \mathcal{I}^{\gls*{meshing_parameter}})$ is \emph{stable} (cf. \cite[Eq. (2.3)]{MR1115933}) if
there exists a constant $C$ such that for each solution $V^{\gls*{meshing_parameter}}$ of \cref{eqn:convergence_scheme}, $\Vert V^{\gls*{meshing_parameter}} \Vert_\infty \leq C$.

To establish stability, we prove that the solutions of each scheme are bounded by $C = \Vert f \Vert_\infty T + \Vert g \Vert_\infty$.
Recalling the original optimal control problem defined by \cref{eqn:introduction_functional}, this bound has the intuitive interpretation of being the sum of the maximum possible continuous reward and the maximum possible reward obtained at the final time:
\begin{equation}
    \int_t^T f(u, X_u, \textcolor{ctrlcolor}{w_u}) du + g(X_T) \leq \Vert f \Vert_\infty T + \Vert g \Vert_\infty.
    \label{eqn:convergence_stability_intuition}
\end{equation}
We establish this below in a series of lemmas.

\begin{lemma}
    The direct control scheme is stable.
    \label{lem:convergence_direct_control_stable}
\end{lemma}

\begin{proof}
    Let $V = V^{\gls*{meshing_parameter}}$ be a solution of the direct control scheme.
    First, note that by \cref{eqn:convergence_direct_control_boundary},
    \[
        0
        = S(h, (\tau^0, x_i), V, \cdot)
        = V_i^0 - g(x_i)
    \]
    from which it follows that
    \begin{equation}
        - \Vert g \Vert_\infty \leq V_i^0 \leq \Vert g \Vert_\infty.
        \label{eqn:convergence_direct_control_stability_0}
    \end{equation}

    Letting $n > 0$, \cref{eqn:convergence_direct_control_interior} implies
    \begin{multline}
        0
        = -S(h, (\tau^n, x_i), V, \cdot) \\
        \geq \sup_{
            \textcolor{ctrlcolor}{w_i} \in W^{\gls*{meshing_parameter}}
        } \left\{
            \frac{V_i^{n-1} - V_i^n}{\Delta \tau}
            + \frac{1}{2} b_i^n(\textcolor{ctrlcolor}{w_i})^2 (\mathcal{D}_2 V^n)_i
            + a_i^n(\textcolor{ctrlcolor}{w_i}) (\mathcal{D} V^n)_i
            + f_i^n(\textcolor{ctrlcolor}{w_i})
        \right\}.
        \label{eqn:convergence_direct_control_stability_1}
    \end{multline}
    Now, pick $j$ such that $V_j^n = \min_i V_i^n$.
    It follows that $V_{j \pm 1}^n - V_j^n \geq 0$ so that by \cref{eqn:schemes_second_derivative},
    \[
        (\mathcal{D}_2 V^n)_j
        = \frac{V_{j+1}^n - 2 V_j^n + V_{j-1}^n}{(\Delta x)^2}
        \boldsymbol{1}_{ \{ 0 < j < M \} }
        = \frac{V_{j+1}^n - V_j^n + V_{j-1}^n - V_j^n}{(\Delta x)^2}
        \boldsymbol{1}_{ \{ 0 < j < M \} }
        \geq 0.
    \]
    Similarly, by \cref{eqn:schemes_first_derivative_pm},
    \[
        (\mathcal{D}_+ V^n)_j
        = \frac{V_{j+1}^n - V_j^n}{\Delta x} \boldsymbol{1}_{ \{ j < M \} }
        \geq 0
        \textspace \text{and} \textspace
        (\mathcal{D}_- V^n)_j
        = \frac{V_j^n - V_{j-1}^n}{\Delta x} \boldsymbol{1}_{ \{ j > 0 \} }
        \leq 0.
    \]
    Therefore, by \cref{eqn:schemes_first_derivative},
    \[
        a_j^n(\textcolor{ctrlcolor}{w_j}) (\mathcal{D} V^n)_j
        = |a_j^n(\textcolor{ctrlcolor}{w_j})|
        (\mathcal{D}_+ V^n)_j
        \boldsymbol{1}_{ \{ a_j^n(\textcolor{ctrlcolor}{w_j}) > 0 \} }
        - |a_j^n(\textcolor{ctrlcolor}{w_j})|
        (\mathcal{D}_- V^n)_j
        \boldsymbol{1}_{ \{ a_j^n(\textcolor{ctrlcolor}{w_j}) < 0 \} }
        \geq 0.
    \]
    We summarize the above findings by writing
    \begin{equation}
        (\mathcal{D}_2 V^n)_j \geq 0
        \textspace \text{and} \textspace
        a_j^n(\textcolor{ctrlcolor}{w_j}) (\mathcal{D} V^n)_j \geq 0.
        \label{eqn:convergence_direct_control_stability_2}
    \end{equation}
    Setting $i=j$ in \cref{eqn:convergence_direct_control_stability_1} and applying \cref{eqn:convergence_direct_control_stability_2},
    \[
        0 \geq \frac{V_j^{n-1} - V_j^n}{\Delta \tau} - \Vert f \Vert_\infty
    \]
    and hence
    \begin{equation}
        V_j^n \geq \min_i V_i^{n-1} - \Vert f \Vert_\infty \Delta \tau.
        \label{eqn:convergence_direct_control_stability_3}
    \end{equation}
    By \cref{eqn:convergence_direct_control_stability_0,eqn:convergence_direct_control_stability_3}, it follows by induction on $n$ that
    \[
        \min_i V_i^n
        \geq - \Vert f \Vert_\infty n \Delta \tau - \Vert g \Vert_\infty
        \geq - \Vert f \Vert_\infty T - \Vert g \Vert_\infty
        \textspace \text{for all } n,
    \]
    establishing a lower bound on the solution.

    It remains to establish an upper bound.
    Let $n > 0$ and pick $j$ such that $V_j^n = \max_i V_i^n$.
    By \cref{eqn:convergence_direct_control_nonlocal,eqn:convergence_direct_control_interior},
    \begin{multline*}
        0 = -S(h, (\tau^n, x_j), V, [\mathcal{I}^{\gls*{meshing_parameter}} V](\tau^n, x_j)) \\
        = \sup_{\substack{
            \textcolor{ctrlcolor}{d_j} \in \{0,1\} \\
            \textcolor{ctrlcolor}{w_j} \in W^{\gls*{meshing_parameter}}
        }} \biggl\{
            \textcolor{ctrlcolor}{\overline{d_j}} \left(
                \frac{V_j^{n-1} - V_j^n}{\Delta \tau}
                + \frac{1}{2} b_j^n(\textcolor{ctrlcolor}{w_j})^2 (\mathcal{D}_2 V^n)_j
                + a_j^n(\textcolor{ctrlcolor}{w_j}) (\mathcal{D} V^n)_j
                + f_j^n(\textcolor{ctrlcolor}{w_j})
            \right) \\
            + \textcolor{ctrlcolor}{d_j} \left(
                (\mathcal{M}_n V^n)_j - V_j^n
            \right)
        \biggr\}
    \end{multline*}
    and hence at least one of
    \begin{equation}
        0 = \sup_{
            \textcolor{ctrlcolor}{w_j} \in W^{\gls*{meshing_parameter}}
        } \left\{
            \frac{V_j^{n-1} - V_j^n}{\Delta \tau}
            + \frac{1}{2} b_j^n(\textcolor{ctrlcolor}{w_j})^2 (\mathcal{D}_2 V^n)_j
            + a_j^n(\textcolor{ctrlcolor}{w_j}) (\mathcal{D} V^n)_j
            + f_j^n(\textcolor{ctrlcolor}{w_j})
        \right\}
        \label{eqn:convergence_direct_control_stability_4}
    \end{equation}
    or
    \begin{equation}
        0 = (\mathcal{M}_n V^n)_j - V_j^n
        \label{eqn:convergence_direct_control_stability_5}
    \end{equation}
    has to hold.
    Note that \cref{eqn:convergence_direct_control_stability_5} cannot be true since by \cref{enum:convergence_negative_cost},
    \begin{multline}
        (\mathcal{M}_n V^n)_j
        = \sup_{\textcolor{ctrlcolor}{z_j} \in Z^{\gls*{meshing_parameter}}(\tau^n, x_j)}
        \left\{
            \gls*{interp}(
                V^n,
                \Gamma(\tau^n, x_j, \textcolor{ctrlcolor}{z_j})
            ) + K(\tau^n, x_j, \textcolor{ctrlcolor}{z_j})
        \right\} \\
        < \sup_{\textcolor{ctrlcolor}{z_j} \in Z^{\gls*{meshing_parameter}}(\tau^n, x_j)}
        \gls*{interp}(
            V^n,
            \Gamma(\tau^n, x_j, \textcolor{ctrlcolor}{z_j})
        )
        \leq V_j^n.
        \label{eqn:convergence_direct_control_stability_6}
    \end{multline}
    Our choice of $j$ implies $V_{j \pm 1}^n - V_j^n \leq 0$ which, similarly to \cref{eqn:convergence_direct_control_stability_2}, we can use to establish
    \begin{equation}
        (\mathcal{D}_2 V^n)_j \leq 0
        \textspace \text{and} \textspace
        a_j^n(\textcolor{ctrlcolor}{w_j}) (\mathcal{D} V^n)_j \leq 0.
        \label{eqn:convergence_direct_control_stability_7}
    \end{equation}
    Applying \cref{eqn:convergence_direct_control_stability_7} to \cref{eqn:convergence_direct_control_stability_4},
    \begin{equation}
        0 \leq \frac{V_j^{n-1} - V_j^n}{\Delta \tau} + \Vert f \Vert_\infty
        \label{eqn:convergence_direct_control_stability_8}
    \end{equation}
    and hence
    \begin{equation}
        V_j^n \leq \max_i V_i^{n-1} + \Vert f \Vert_\infty \Delta \tau.
        \label{eqn:convergence_direct_control_stability_9}
    \end{equation}
    By \cref{eqn:convergence_direct_control_stability_0,eqn:convergence_direct_control_stability_9}, it follows by induction on $n$ that
    \[
        \max_i V_i^n
        \leq \Vert f \Vert_\infty n \Delta \tau + \Vert g \Vert_\infty
        \leq \Vert f \Vert_\infty T + \Vert g \Vert_\infty
        \textspace \text{for all } n,
    \]
    establishing an upper bound on the solution.
\end{proof}

\begin{lemma}
    The penalty scheme is stable.
    \label{lem:convergence_penalty_stable}
\end{lemma}

\begin{proof}
    Let $V = V^{\gls*{meshing_parameter}}$ be a solution of the penalty scheme.
    Since this proof is very similar to that of \cref{lem:convergence_direct_control_stable}, we point out only parts in which it differs.

    Letting $n > 0$, \cref{eqn:convergence_direct_control_nonlocal,eqn:convergence_penalty_interior} imply
    \begin{multline*}
        0
        = -S(h, (\tau^n, x_i), V, [\mathcal{I}^{\gls*{meshing_parameter}} V](\tau^n, x_i))
        \\
        = \sup_{\substack{
            \textcolor{ctrlcolor}{d_i} \in \{0,1\} \\
            \textcolor{ctrlcolor}{w_i} \in W^{\gls*{meshing_parameter}}
        }} \left\{
            \textcolor{ctrlcolor}{\overline{d_i}} \gamma_i^n(V, \textcolor{ctrlcolor}{w_i})
            + \textcolor{ctrlcolor}{d_i} \left(
                (\mathcal{M}_n V^n)_i - V_i^n + \epsilon \gamma_i^n(V, \textcolor{ctrlcolor}{w_i})
            \right)
        \right\}.
    \end{multline*}
    Performing some algebra and substituting in the definition of $\gamma_i^n(V, \textcolor{ctrlcolor}{w_i})$ from \cref{eqn:schemes_penalty_gamma}, the above equality becomes
    \begin{multline}
        0
        = \sup_{\substack{
            \textcolor{ctrlcolor}{d_i} \in \{0,1\} \\
            \textcolor{ctrlcolor}{w_i} \in W^{\gls*{meshing_parameter}}
        }} \left\{
            \gamma_i^n(V, \textcolor{ctrlcolor}{w_i})
            + \frac{1}{\epsilon} \textcolor{ctrlcolor}{d_i} \left(
                (\mathcal{M}_n V^n)_i - V_i^n
            \right)
        \right\} \\
        = \sup_{\substack{
            \textcolor{ctrlcolor}{d_i} \in \{0,1\} \\
            \textcolor{ctrlcolor}{w_i} \in W^{\gls*{meshing_parameter}}
        }} \biggl \{
            \frac{V_i^{n-1} - V_i^n}{\Delta \tau}
            + \frac{1}{2} b_i^n(\textcolor{ctrlcolor}{w_i})^2 (\mathcal{D}_2 V^n)_i
            + a_i^n(\textcolor{ctrlcolor}{w_i}) (\mathcal{D} V^n)_i
            + f_i^n(\textcolor{ctrlcolor}{w_i})
            + \frac{1}{\epsilon} \textcolor{ctrlcolor}{d_i} \left(
                (\mathcal{M}_n V^n)_i - V_i^n
            \right)
        \biggr \}.
        \label{eqn:convergence_penalty_stability_1}
    \end{multline}
    By \cref{eqn:convergence_penalty_stability_1},
    \begin{equation}
        0 \geq \sup_{
            \textcolor{ctrlcolor}{w_i} \in W^{\gls*{meshing_parameter}}
        } \biggl \{
            \frac{V_i^{n-1} - V_i^n}{\Delta \tau} \\
            + \frac{1}{2} b_i^n(\textcolor{ctrlcolor}{w_i})^2 (\mathcal{D}_2 V^n)_i
            + a_i^n(\textcolor{ctrlcolor}{w_i}) (\mathcal{D} V^n)_i
            + f_i^n(\textcolor{ctrlcolor}{w_i})
        \biggr \}.
        \label{eqn:convergence_penalty_stability_2}
    \end{equation}
    Noting that the inequalities \cref{eqn:convergence_direct_control_stability_1,eqn:convergence_penalty_stability_2} are identical, we can proceed as in the proof of \cref{lem:convergence_direct_control_stable} to obtain a lower bound on $V$.

    To obtain the upper bound, pick $j$ such that $V_j^n = \max_i V_i^n$.
    By \cref{eqn:convergence_direct_control_stability_6}, $(\mathcal{M}_n V^n)_j - V_j \leq 0$.
    Applying this fact to \cref{eqn:convergence_penalty_stability_1},
    \begin{equation}
        0 \leq \sup_{
            \textcolor{ctrlcolor}{w_j} \in W^{\gls*{meshing_parameter}}
        } \left\{
            \frac{V_j^{n-1} - V_j^n}{\Delta \tau}
            + \frac{1}{2} b_j^n(\textcolor{ctrlcolor}{w_j})^2 (\mathcal{D}_2 V^n)_j
            + a_j^n(\textcolor{ctrlcolor}{w_j}) (\mathcal{D} V^n)_j
            + f_j^n(\textcolor{ctrlcolor}{w_j})
        \right\}.
        \label{eqn:convergence_penalty_stability_3}
    \end{equation}
    As in the proof of \cref{lem:convergence_direct_control_stable} (see, in particular, \cref{eqn:convergence_direct_control_stability_7}), our choice of $j$ implies $(\mathcal{D}_2 V^n)_j \leq 0$ and $a_j^n(\textcolor{ctrlcolor}{w_j}) (\mathcal{D} V^n)_j \leq 0$.
    Applying these facts to \cref{eqn:convergence_penalty_stability_3}, we get
    \begin{equation}
        0 \leq \frac{V_j^{n-1} - V_j^n}{\Delta \tau} + \Vert f \Vert_\infty.
        \label{eqn:convergence_penalty_stability_4}
    \end{equation}
    Noting that the inequalities \cref{eqn:convergence_direct_control_stability_8,eqn:convergence_penalty_stability_4} are identical, we can proceed as in the proof of \cref{lem:convergence_direct_control_stable} to obtain an upper bound on $V$.
\end{proof}

\begin{lemma}
    The explicit-impulse scheme is stable.
    \label{lem:convergence_explicit_impulse_stable}
\end{lemma}

\begin{proof}
    Let $V = V^{\gls*{meshing_parameter}}$ be a solution of the explicit-impulse scheme.
    As usual, we proceed by establishing a lower bound and upper bound separately.

    Letting $n > 0$, \cref{eqn:convergence_explicit_impulse_nonlocal,eqn:convergence_explicit_impulse_interior} imply
    \begin{multline}
        0
        = -S(h, (\tau^n, x_i), V, [\mathcal{I} V](\tau^n, x_i)) \\
        = \sup_{\substack{
            \textcolor{ctrlcolor}{d_i} \in \{0,1\} \\
            \textcolor{ctrlcolor}{w_i} \in W^{\gls*{meshing_parameter}}
        }} \biggl\{
            \textcolor{ctrlcolor}{\overline{d_i}}
            \left(
                \displaystyle{
                    \frac{\gls*{interp}(
                        V^{n-1},
                        x_i + a_i^n(\textcolor{ctrlcolor}{w_i}) \Delta \tau
                    ) - V_i^n}{\Delta \tau}
                }
                + \frac{1}{2} (b_i^n)^2 (\mathcal{D}_2 V^n)_i
                + f_i^n(\textcolor{ctrlcolor}{w_i})
            \right) \\
            + \textcolor{ctrlcolor}{d_i} \left(
                (\mathcal{M}_n V^{n-1})_i - V_i^n
                + \frac{1}{2} (b_i^n)^2 (\mathcal{D}_2 V^n)_i \Delta \tau
            \right)
        \biggr\}.
        \label{eqn:convergence_explicit_impulse_stable_0}
    \end{multline}
    By \cref{eqn:convergence_explicit_impulse_stable_0},
    \begin{multline}
        0
        \geq \sup_{
            \textcolor{ctrlcolor}{w_i} \in W^{\gls*{meshing_parameter}}
        } \left \{
            \displaystyle{
                \frac{\gls*{interp}(
                    V^{n-1},
                    x_i + a_i^n(\textcolor{ctrlcolor}{w_i}) \Delta \tau
                ) - V_i^n}{\Delta \tau}
            }
            + \frac{1}{2} (b_i^n)^2 (\mathcal{D}_2 V^n)_i
            + f_i^n(\textcolor{ctrlcolor}{w_i})
        \right \} \\
        \geq \frac{\min_k V_k^{n-1} - V_i^n}{\Delta \tau}
        + \frac{1}{2} (b_i^n)^2 (\mathcal{D}_2 V^n)_i
        - \Vert f \Vert_\infty.
        \label{eqn:convergence_explicit_impulse_stable_1}
    \end{multline}
    Now, pick $j$ such that $V_j^n = \min V_i^n$.
    As in the proof of \cref{lem:convergence_direct_control_stable} (see, in particular \cref{eqn:convergence_direct_control_stability_2}), this choice of $j$ implies $(\mathcal{D}_2 V^n)_j \geq 0$.
    Setting $i = j$ and substituting this inequality into \cref{eqn:convergence_explicit_impulse_stable_1}, we obtain, after some simplification,
    \begin{equation}
        V_j^n \geq \min_k V_k^{n-1} - \Vert f \Vert_\infty \Delta \tau.
        \label{eqn:convergence_explicit_impulse_stable_2}
    \end{equation}
    Noting that the inequalities \cref{eqn:convergence_direct_control_stability_3,eqn:convergence_explicit_impulse_stable_2} are identical, we can proceed as in the proof of \cref{lem:convergence_direct_control_stable} to obtain a lower bound on $V$.

    It remains to establish an upper bound.
    \cref{enum:convergence_negative_cost} implies that
    \begin{multline}
        (\mathcal{M}_n V^{n-1})_i
        = \sup_{\textcolor{ctrlcolor}{z_i} \in Z^{\gls*{meshing_parameter}}(\tau^n, x_i)}
        \left\{
            \gls*{interp}(
                V^{n-1},
                \Gamma(\tau^n, x_i, \textcolor{ctrlcolor}{z_i})
            ) + K(\tau^n, x_i, \textcolor{ctrlcolor}{z_i})
        \right\} \\
        < \sup_{\textcolor{ctrlcolor}{z_i} \in Z^{\gls*{meshing_parameter}}(\tau^n, x_i)}
        \gls*{interp}(
            V^{n-1},
            \Gamma(\tau^n, x_i, \textcolor{ctrlcolor}{z_i})
        )
        \leq \max_k V_k^{n-1}.
        \label{eqn:convergence_explicit_impulse_stable_3}
    \end{multline}
    Now, pick $j$ such that $V_j^n = \max V_i^n$.
    As in the proof of \cref{lem:convergence_direct_control_stable} (see, in particular, \cref{eqn:convergence_direct_control_stability_7}), this choice of $j$ implies $(\mathcal{D}_2 V^n)_j \leq 0$.
    Applying this fact and \cref{eqn:convergence_explicit_impulse_stable_3} to equation \cref{eqn:convergence_explicit_impulse_stable_0} with $i = j$ yields
    \[
        0
        \leq \sup_{
            \textcolor{ctrlcolor}{d_j} \in \{0,1\}
        } \left \{
            \textcolor{ctrlcolor}{\overline{d_j}}
            \left(
                \displaystyle{
                    \frac{\max_k V_k^{n-1} - V_j^n}{\Delta \tau}
                }
                + \Vert f \Vert_\infty
            \right) \\
            + \textcolor{ctrlcolor}{d_j}
            \left(
                \max_k V_k^{n-1} - V_j^n
            \right)
        \right \}.
    \]
    Equivalently,
    \begin{multline}
        0
        \leq \sup_{
            \textcolor{ctrlcolor}{d_j} \in \{0,1\}
        } \left \{
            \textcolor{ctrlcolor}{\overline{d_j}}
            \left(
                \max_k V_k^{n-1} - V_j^n
                + \Vert f \Vert_\infty \Delta \tau
            \right)
            + \textcolor{ctrlcolor}{d_j}
            \left(
                \max_k V_k^{n-1} - V_j^n
            \right)
        \right \} \\
        = \max_k V_k^{n-1} - V_j^n + \Vert f \Vert_\infty \Delta \tau.
        \label{eqn:convergence_explicit_impulse_stable_4}
    \end{multline}
    Noting that the inequalities \cref{eqn:convergence_direct_control_stability_9,eqn:convergence_explicit_impulse_stable_4} are equivalent, we can proceed as in the proof of \cref{lem:convergence_direct_control_stable} to obtain an upper bound on $V$.
\end{proof}

\begin{remark}
    \label{rem:convergence_relax_cost}
    A close examination of the proof of direct control stability reveals that if we redefine $\mathcal{M}_n$ by
    \[
        (\mathcal{M}_n U)_i =
        \sup_{\textcolor{ctrlcolor}{z_i} \in Z^{\gls*{meshing_parameter}}(\tau^n, x_i)}
        \left\{
            \gls*{interp}(
                U,
                \Gamma(\tau^n, x_i, \textcolor{ctrlcolor}{z_i})
            )
            + K(\tau^n, x_i, \textcolor{ctrlcolor}{z_i})
        \right\}
        - \nu
    \]
    where $\nu > 0$ and $\nu \rightarrow 0$ as $\gls*{meshing_parameter} \rightarrow 0$, we can relax the requirement \cref{enum:convergence_negative_cost} from strict inequality to weak inequality (i.e., $\sup_{t,x,\textcolor{ctrlcolor}{z}} K(t, x, \textcolor{ctrlcolor}{z}) \leq 0$) without losing the stability of the direct control scheme.
    Since $\nu$ vanishes as the grid is made finer, this has no bearing on the consistency proofs of the next section.

    Similarly, a close examination of the proofs of penalty and explicit-impulse stability reveals that, as in the previous paragraph, we can relax the requirement \cref{enum:convergence_negative_cost} from strict inequality to weak inequality.
    In this case, it is not necessary even to redefine $\mathcal{M}_n$.
\end{remark}

\subsection{Nonlocal consistency}
\label{subsec:convergence_nonlocal_consistency}

As mentioned at the beginning of this chapter, our notion of consistency differs from the usual notion defined in \cite[Eq. (2.4)]{MR1115933}.
Before we can introduce our notion of consistency, we review half-relaxed limits, which appear frequently in the theory of viscosity solutions.

\begin{definition}[\cite{MR921827}]
    \label{def:convergence_half_relaxed_limits}
    For a family $(u^{\gls*{meshing_parameter}})_{\gls*{meshing_parameter} > 0}$ of real-valued maps from a metric space such that $(\gls*{meshing_parameter}, x) \mapsto u^{\gls*{meshing_parameter}}(x)$ is locally bounded, we define the upper and lower half-relaxed limits $\overline{u}$ and $\underline{u}$ by
    \[
        \overline{u}(x) = \limsup_{\substack{
            \gls*{meshing_parameter} \rightarrow 0 \\
            y \rightarrow x
        }} u^{\gls*{meshing_parameter}}(y)
        \textspace \text{and} \textspace
        \underline{u}(x) = \liminf_{\substack{
            \gls*{meshing_parameter} \rightarrow 0 \\
            y \rightarrow x
        }} u^{\gls*{meshing_parameter}}(y).
    \]
\end{definition}

Half-relaxed limits are a generalization of semicontinuous envelopes.
To see this, note that if we fix a function $u$ and take $u^h = u$ for all $h$, $\overline{u} = u^*$ and $\underline{u} = u_*$.
See also \cref{fig:convergence_half_relaxed} for an example of a family of functions along with its upper half-relaxed limit.

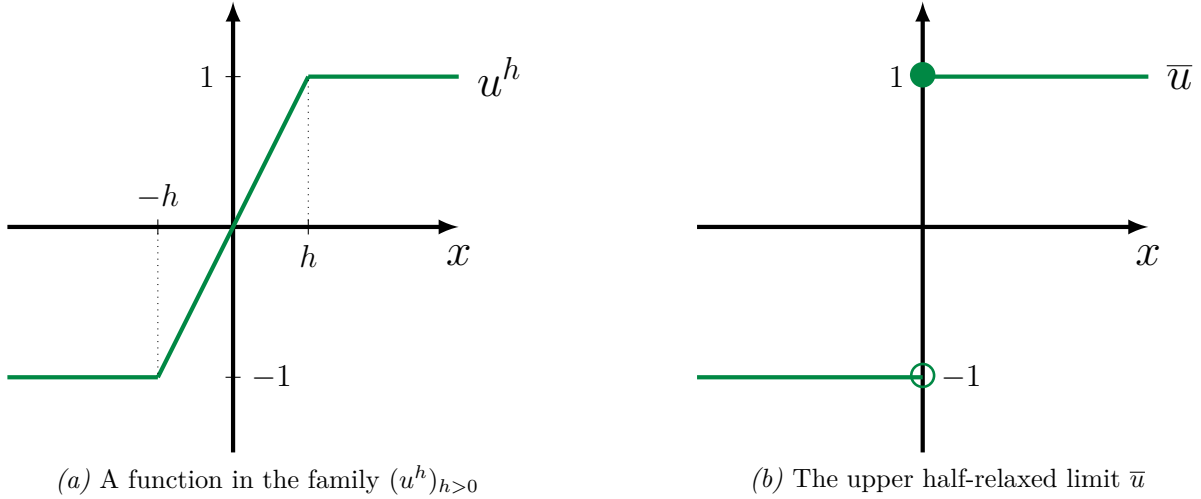
\begin{figure}
    \subfloat[A function in the family $(u^h)_{h > 0}$]{
        \begin{tikzpicture}
            \draw [->, ultra thick] (-3, 0) -- (3, 0);
            \draw [->, ultra thick] (0, -3) -- (0, 3);
            \draw [colB, ultra thick]
                (-1, -2) -- (1, 2)
                (-3, -2) -- (-1, -2)
                (1, 2) -- (3, 2)
            ;
            \draw [dotted]
                (-1, 0) -- (-1, -2)
                (1, 0) -- (1, 2)
            ;
            \draw (-1, -0.1) -- (-1, 0.1);
            \node [yshift=0.1cm, above] at (-1, 0) {$-h$};
            \draw (1, -0.1) -- (1, 0.1);
            \node [yshift=-0.1cm, below] at (1, 0) {$h$};
            \draw (-0.1, -2) -- (0.1, -2);
            \node [xshift=0.1cm, right] at (0, -2) {$-1$};
            \draw (-0.1, 2) -- (0.1, 2);
            \node [xshift=-0.1cm, left] at (0, 2) {$1$};
            \node [xshift=0.1cm, right] at (3, 2) {\Large $u^h$};
            \node [yshift=-0.1cm, below] at (3, 0) {\Large $x$};
        \end{tikzpicture}
    }
    \hfill{}
    \subfloat[The upper half-relaxed limit $\overline{u}$]{
        \begin{tikzpicture}
            \draw [->, ultra thick] (-3, 0) -- (3, 0);
            \draw [->, ultra thick] (0, -3) -- (0, 3);
            \draw [colB, ultra thick]
                (-3, -2) -- (0, -2)
                (0, 2) -- (3, 2)
            ;
            \node [xshift=0.1cm, right] at (3, 2) {\Large $\overline{u}$};
            \node [xshift=0.1cm, right] at (0, -2) {$-1$};
            \node [xshift=-0.1cm, left] at (0, 2) {$1$};
            \node [colB] at (0, -2) {\huge $\circ$};
            \node [colB] at (0, 2) {\huge \textbullet};
            \node [yshift=-0.1cm, below] at (3, 0) {\Large $x$};
        \end{tikzpicture}
    }
    \caption{An example of a half-relaxed limit}
    \label{fig:convergence_half_relaxed}
\end{figure}

We are now ready to define nonlocal consistency.
A scheme $(S, \mathcal{I}^{\gls*{meshing_parameter}})$ is \emph{nonlocally consistent} if for each family $(u^{\gls*{meshing_parameter}})_{\gls*{meshing_parameter} > 0}$ of uniformly bounded real-valued maps from $\overline{\Omega}$, $\varphi \in C^2(\overline{\Omega})$, and $x \in \overline{\Omega}$, we have
\begin{equation}
    \liminf_{\substack{
        \gls*{meshing_parameter} \rightarrow 0 \\
        y \rightarrow x \\
        \xi \rightarrow 0
    }} S(
        \gls*{meshing_parameter},
        y,
        \varphi + \xi,
        [\mathcal{I}^{\gls*{meshing_parameter}} u^{\gls*{meshing_parameter}}](y)
    )
    \geq
    F_*(x, \varphi(x), D \varphi(x), D^2 \varphi(x), [\mathcal{I} \overline{u}](x))
    \label{eqn:convergence_subconsistency}
\end{equation}
and
\begin{equation}
    \limsup_{\substack{
        \gls*{meshing_parameter} \rightarrow 0 \\
        y \rightarrow x \\
        \xi \rightarrow 0
    }} S(
        \gls*{meshing_parameter},
        y,
        \varphi + \xi,
        [\mathcal{I}^{\gls*{meshing_parameter}} u^{\gls*{meshing_parameter}}](y)
    )
    \leq
    F^*(x, \varphi(x), D \varphi(x), D^2 \varphi(x), [\mathcal{I} \overline{u}](x))
    \label{eqn:convergence_superconsistency}
\end{equation}
where $\overline{u}$ and $\underline{u}$ are the half-relaxed limits from \cref{def:convergence_half_relaxed_limits}.

In the absence of the nonlocal operator (i.e., $\mathcal{I} = \mathcal{I}^{\gls*{meshing_parameter}} = 0$), nonlocal consistency is equivalent to the usual notion of consistency \cite[Eq (2.4)]{MR1115933}.
In other words, the only difference between nonlocal consistency and the usual notion of consistency is the explicit handling of the nonlocal operator by half-relaxed limits.

In \cref{chap:schemes}, we replaced the (possibly infinite) control sets $W$ by a finite subset $W^{\gls*{meshing_parameter}}$.
A similar exercise was performed to obtain $Z^{\gls*{meshing_parameter}}(t, x)$.
To prove the nonlocal consistency of our schemes, we will need to ensure that these finite subsets converge, in some sense, to the original sets which they approximate.
The appropriate notion of convergence, in this case, is with respect to the Hausdorff metric.
We recall below the definition of the Hausdorff metric.

\begin{definition}[{\cite[Pg. 281]{MR0464128}}]
    The Hausdorff metric $\gls*{hausdorff}$ between two nonempty compact subsets $X$ and $Y$ of a metric space $(M, d)$ is given by
    \[
        \gls*{hausdorff}(X, Y)
        = \max \left \{
            \adjustlimits \sup_{x \in X} \inf_{y \in Y} d(x, y),
            \adjustlimits \sup_{y \in Y} \inf_{x \in X} d(x, y)
        \right \}.
    \]
    If $X \subset Y$ as in \cref{fig:convergence_hausdorff}, then $\sup_{x \in X} \inf_{y \in Y} d(x, y) = 0$ so that the above simplifies to
    \[
        \gls*{hausdorff}(X, Y)
        = \adjustlimits \sup_{y \in Y} \inf_{x \in X} d(x, y).
    \]
\end{definition}

\begin{figure}
    \centering
    \includegraphics[height=2.5in]{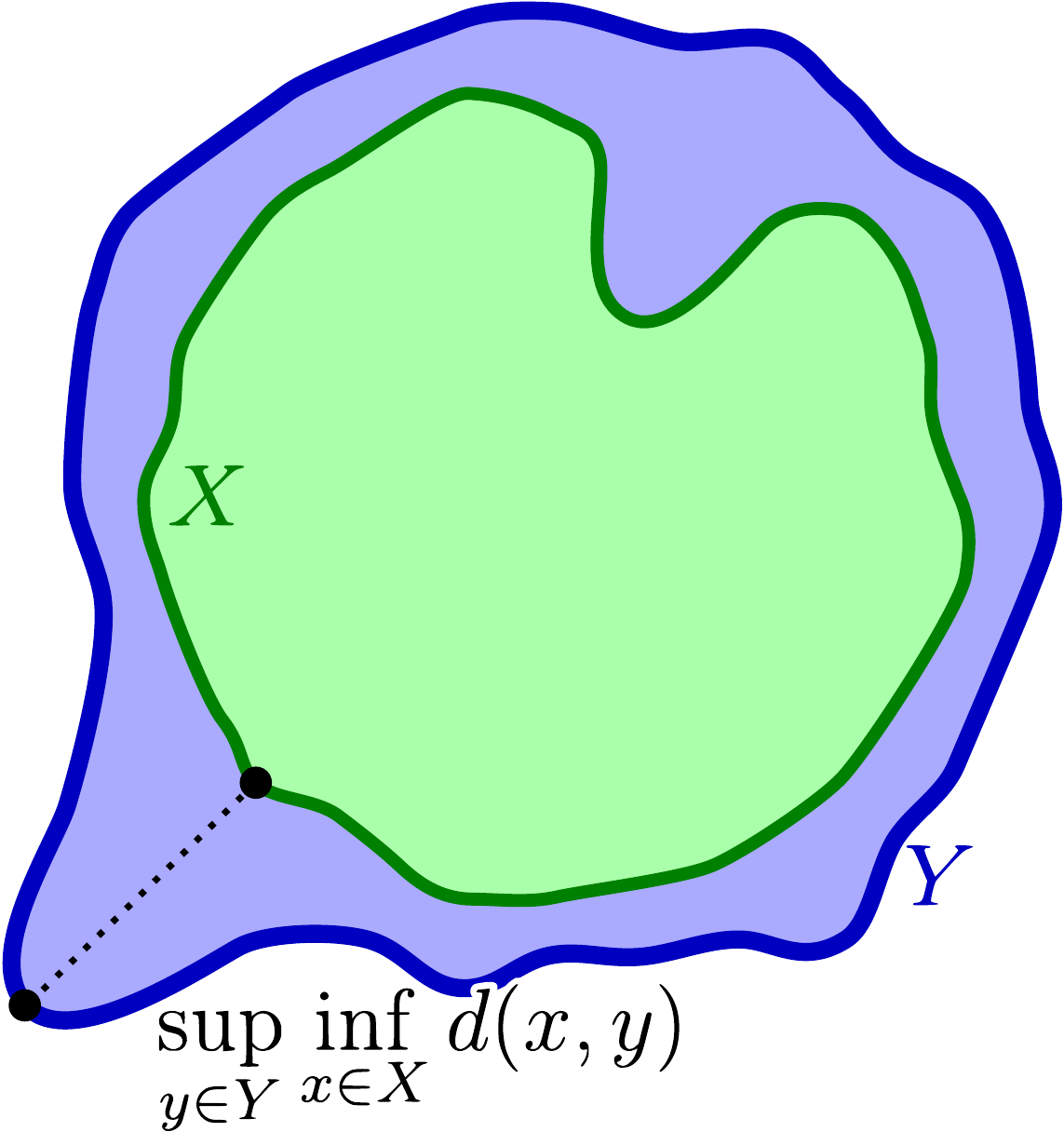}
    \caption[Hausdorff distance between a set and another set containing it]{Hausdorff distance between a set and another set containing it (modified from \url{https://commons.wikimedia.org/wiki/File:Hausdorff_distance_sample.svg})}
    \label{fig:convergence_hausdorff}
\end{figure}

A short example illustrating convergence in the Hausdorff metric is given below.

\begin{example}
    Consider the partition $X_m = \{ 0, \frac{1}{m}, \frac{2}{m}, \ldots, 1 \}$ of the unit interval $Y = [0, 1]$.
    Since $X_m \subset Y$,
    \[
        d_H( X_m, Y )
        = \adjustlimits \sup_{y \in Y} \inf_{x \in X_m} \left| x - y \right|
        = \frac{1}{2m}
    \]
    and hence $d_H( X_m, Y ) \rightarrow 0$ as $m \rightarrow \infty$.
    In other words, the partition $X_m$ converges (in the Hausdorff metric) to the unit interval $Y$.
\end{example}

In the context of discretizations of controlled PDEs, we can think of the unit interval in the above example as the control set (e.g., $W$) and the partition as an approximation of that control set (e.g., $W^{\gls*{meshing_parameter}}$).

We are now able to state the last assumption in this chapter, concerning the convergence of the approximate control sets $W^{\gls*{meshing_parameter}}$ and $Z^{\gls*{meshing_parameter}}(t, x)$.

\begin{enumerate}[label=(H\arabic*),start=4]
    \item
        \label{enum:convergence_stochastic_control_hausdorff}
        \label{enum:convergence_impulse_control_hausdorff}
        As $h \rightarrow 0$, $W^{\gls*{meshing_parameter}}$ converges (in the Hausdorff metric) to $W$ and $Z^{\gls*{meshing_parameter}}$ converges locally uniformly (in the Hausdorff metric) to $Z$.
        Moreover, $(t, x) \mapsto Z^{\gls*{meshing_parameter}}(t, x)$ is continuous (in the Hausdorff metric) for each $\gls*{meshing_parameter}$.
\end{enumerate}

\begin{remark}
    By the uniform limit theorem (which states that a uniformly convergent sequence of continuous functions has a continuous limit \cite[Theorem 21.6]{MR0464128}) and \cref{enum:convergence_impulse_control_hausdorff}, the function $(t, x) \mapsto Z(t, x)$ is also continuous (in the Hausdorff metric).
\end{remark}

Using \cref{enum:convergence_impulse_control_hausdorff}, we will establish the nonlocal consistency of our schemes.
Before we can do so, we require a few lemmas.

\begin{lemma}
    \label{lem:proofs_convergence_direct_control_consistency_1}
    Let $(a_m)_m$, $(b_m)_m$, and $(c_m)_m$ be real sequences with $c_m \geq \min \{ a_m, b_m \}$ (resp. $c_m \leq \min \{ a_m, b_m \}$) for each $m$.
    Then,
    \begin{align}
        \liminf_{m \rightarrow \infty} c_m & \geq \min \{ \liminf_{m \rightarrow \infty} a_m, \liminf_{m \rightarrow \infty} b_m \}
        \label{eqn:proofs_convergence_direct_control_consistency_1_1} \\
        \text{(resp. }
            \limsup_{m \rightarrow \infty} c_m & \leq \min \{ \limsup_{m \rightarrow \infty} a_m, \limsup_{m \rightarrow \infty} b_m \}
        \text{)}.
        \label{eqn:proofs_convergence_direct_control_consistency_1_2}
    \end{align}
\end{lemma}
\begin{proof}
    If $c_m \geq \min \{ a_m, b_m \}$ for each $m$,
    \[
        \inf_{k \geq m} c_k
        \geq \inf_{k \geq m} \min \{ a_k, b_k \}
        = \min \{ \inf_{k \geq m} a_k, \inf_{k \geq m} b_k \}.
    \]
    Taking limits in the above yields
    \[
        \lim_{m \rightarrow \infty} \inf_{k \geq m} c_k
        \geq \lim_{m \rightarrow \infty} \min \{ \inf_{k \geq m} a_k, \inf_{k \geq m} b_k \}
        = \min \{ \lim_{m \rightarrow \infty} \inf_{k \geq m} a_k, \lim_{m \rightarrow \infty} \inf_{k \geq m} b_k \}
    \]
    where the last equality follows from the continuity of the function $(x, y) \mapsto \min\{ x, y \}$.
    Since $\lim_{m \rightarrow \infty} \inf_{k \geq m} x_k = \liminf_{m \rightarrow \infty} x_m$ by definition, this establishes \cref{eqn:proofs_convergence_direct_control_consistency_1_1}.

    Similarly, if $c_m \leq \min \{ a_m, b_m \}$ for each $m$,
    \[
        \sup_{k \geq m} c_k
        \leq \sup_{k \geq m} \min \{ a_k, b_k \}
        \leq \min \{ \sup_{k \geq m} a_k, \sup_{k \geq m} b_k \}.
    \]
    As in the previous paragraph, taking limits establishes \cref{eqn:proofs_convergence_direct_control_consistency_1_2}.
\end{proof}

The next lemma ensures us that we can approximate the term $\sup_{\textcolor{ctrlcolor}{w} \in W}\{\cdot\}$ appearing in \eqref{eqn:convergence_hjbqvi} using the approximate control set $W^{\gls*{meshing_parameter}}$.

\begin{lemma}
    \label{lem:proofs_convergence_direct_control_consistency_2}
    Let $Y$ be a compact metric space and $\rho \colon Y \times W \rightarrow \mathbb{R}$ be a continuous function.
    Let $(\gls*{meshing_parameter}_m)_m$ be a sequence of positive real numbers converging to zero and $(y_m)_m$ be a sequence in $Y$ converging to some $\hat{y} \in Y$.
    Then,
    \[
        \sup_{\textcolor{ctrlcolor}{w} \in W^{\gls*{meshing_parameter}_m}}
        \rho(y_m, \textcolor{ctrlcolor}{w})
        \rightarrow \sup_{\textcolor{ctrlcolor}{w} \in W}
        \rho(\hat{y}, \textcolor{ctrlcolor}{w}).
    \]
\end{lemma}
\begin{proof}
    Let $\rho_m = \sup_{\textcolor{ctrlcolor}{w} \in W^{\gls*{meshing_parameter}_m}} \rho(y_m, \textcolor{ctrlcolor}{w})$.
    Since $\Vert \rho \Vert_{\infty} < \infty$ ($Y \times W$ is compact and $\rho$ is continuous), the sequence $(\rho_m)_m$ is bounded.
    Therefore, it is sufficient to show that every convergent subsequence of $(\rho_m)_m$ converges to $\sup_{\textcolor{ctrlcolor}{w} \in W} \rho(\hat{y}, \textcolor{ctrlcolor}{w})$.

    In light of this, consider an arbitrary convergent subsequence of $(\rho_m)_m$, and relabel it, with a slight abuse of notation, $(\rho_m)_m$.
    Since $W$ is compact, we can find $\textcolor{ctrlcolor}{\hat{w}} \in W$ such that
    \[
        \rho(\hat{y}, \textcolor{ctrlcolor}{\hat{w}})
        = \sup_{\textcolor{ctrlcolor}{w} \in W} \rho(\hat{y}, \textcolor{ctrlcolor}{w}).
    \]
    Moreover, since $W^{\gls*{meshing_parameter}_m} \rightarrow W$ by \cref{enum:convergence_stochastic_control_hausdorff}, we can find a sequence $(\textcolor{ctrlcolor}{w_m})_m$ such that $\textcolor{ctrlcolor}{w_m} \in W^{\gls*{meshing_parameter}_m}$ for each $m$ and $\textcolor{ctrlcolor}{w_m} \rightarrow \textcolor{ctrlcolor}{\hat{w}} \in W$.
    Therefore,
    \begin{equation}
        \lim_{m \rightarrow \infty} \rho_m
        = \adjustlimits \lim_{m \rightarrow \infty} \sup_{\textcolor{ctrlcolor}{w} \in W^{\gls*{meshing_parameter}_m}} \rho(y_m, \textcolor{ctrlcolor}{w})
        \geq \lim_{m\rightarrow\infty} \rho(y_m, \textcolor{ctrlcolor}{w_m})
        = \rho(\hat{y}, \textcolor{ctrlcolor}{\hat{w}})
        = \sup_{\textcolor{ctrlcolor}{w} \in W}\rho(\hat{y}, \textcolor{ctrlcolor}{w}).
        \label{eqn:proofs_convergence_direct_control_consistency_2_1}
    \end{equation}
    Furthermore, since $W^{\gls*{meshing_parameter}_m} \subset W$ by definition,
    \begin{equation}
        \lim_{m \rightarrow \infty} \rho_m
        = \lim_{m \rightarrow \infty} \sup_{\textcolor{ctrlcolor}{w} \in W^{\gls*{meshing_parameter}_m}} \rho(y_m, \textcolor{ctrlcolor}{w})
        \leq \adjustlimits \lim_{m\rightarrow\infty} \sup_{\textcolor{ctrlcolor}{w} \in W} \rho(y_m, \textcolor{ctrlcolor}{w})
        = \sup_{\textcolor{ctrlcolor}{w} \in W} \rho(\lim_{m\rightarrow\infty} y_m, \textcolor{ctrlcolor}{w})
        = \sup_{\textcolor{ctrlcolor}{w} \in W} \rho(\hat{y}, \textcolor{ctrlcolor}{w})
        \label{eqn:proofs_convergence_direct_control_consistency_2_2}
    \end{equation}
    where we have obtained the second last equality in \cref{eqn:proofs_convergence_direct_control_consistency_2_2} by using the continuity of $y \mapsto \sup_{\textcolor{ctrlcolor}{w} \in W} \rho(y, \textcolor{ctrlcolor}{w})$.
    Combining \cref{eqn:proofs_convergence_direct_control_consistency_2_1,eqn:proofs_convergence_direct_control_consistency_2_2}, we obtain
    \[
        \lim_{m \rightarrow \infty} \rho_m
        = \sup_{\textcolor{ctrlcolor}{w} \in W} \rho(\hat{y}, \textcolor{ctrlcolor}{w}),
    \]
    as desired.

    In the previous paragraph, we claimed that $y \mapsto \sup_{\textcolor{ctrlcolor}{w} \in W} \rho(y, \textcolor{ctrlcolor}{w})$ is continuous.
    To see that this is indeed true, let $y_0 \in Y$ and $\epsilon > 0$ be arbitrary.
    Since $\rho$ is defined on a compact set and thereby uniformly continuous, we can pick $\delta > 0$ such that for all $\textcolor{ctrlcolor}{w} \in W$ and $y_1 \in Y$ with $d(y_0, y_1) < \delta$, $|\rho(y_0, \textcolor{ctrlcolor}{w}) - \rho(y_1, \textcolor{ctrlcolor}{w})| < \epsilon$ and hence
    \[
        \left|
            \sup_{\textcolor{ctrlcolor}{w} \in W} \rho(y_0, \textcolor{ctrlcolor}{w})
            - \sup_{\textcolor{ctrlcolor}{w} \in W} \rho(y_1, \textcolor{ctrlcolor}{w})
        \right|
        \leq \sup_{\textcolor{ctrlcolor}{w} \in W} \left|
            \rho(y_0, \textcolor{ctrlcolor}{w}) - \rho(y_1, \textcolor{ctrlcolor}{w})
        \right|
        < \epsilon. \qedhere
    \]
\end{proof}

The next lemma ensures us that we can approximate the intervention operator $\mathcal{M}$ using its discretization $\mathcal{M}_n$ given in \cref{eqn:schemes_discretized_intervention}.

\begin{lemma}
    \label{lem:proofs_convergence_direct_control_consistency_3}
    Let $(u^{\gls*{meshing_parameter}})_{\gls*{meshing_parameter} > 0}$ be a family of uniformly bounded real-valued maps from $[0,T] \times \mathbb{R}$ with half-relaxed limits $\overline{u}$ and $\underline{u}$.
    For brevity, let $u^{n, \gls*{meshing_parameter}} = (u^{\gls*{meshing_parameter}}(\tau^n, x_0), \ldots, u^{\gls*{meshing_parameter}}(\tau^n, x_M))^\intercal$ be a vector whose components are obtained by evaluating the function $u^{\gls*{meshing_parameter}}(\tau^n, \cdot)$ on the spatial grid.
    Then, for any $(t, x) \in [0, T] \times \mathbb{R}$,
    \[
        \mathcal{M} \underline{u}(t, x) 
        \leq \liminf_{\substack{
            \gls*{meshing_parameter} \rightarrow 0 \\
            (\tau^n, x_i) \rightarrow (t, x) 
        }}
        (\mathcal{M}_n u^{n, \gls*{meshing_parameter}})_i
        \leq \limsup_{\substack{
            \gls*{meshing_parameter} \rightarrow 0 \\
            (\tau^n, x_i) \rightarrow (t, x) 
        }}
        (\mathcal{M}_n u^{n, \gls*{meshing_parameter}})_i
        \leq \mathcal{M} \overline{u} (t, x). 
    \]
\end{lemma}

\begin{proof}
    We first prove the leftmost inequality.
    Let $(\gls*{meshing_parameter}_m, s_m, y_m)_m$ be a sequence satisfying
    \[
        h_m \rightarrow 0
        \textspace \text{and} \textspace
        (s_m, y_m) \rightarrow (t, x)
        \textspace \text{as} \textspace
        m \rightarrow \infty
    \]
    chosen such that $s_m = \tau^{n_m} = T - n_m \Delta \tau$ and $y_m = x_{i_m}$ are grid points
    (though not explicit in the notation, both $\Delta \tau = \gls*{const} h_m$ and $\Delta x = \gls*{const} h_m$ depend on $m$ through $h_m$).
    Now, let $\delta > 0$ and choose $\textcolor{ctrlcolor}{z^{\delta}} \in Z(t, x)$ such that
    \[
        \mathcal{M} \underline{u}(t, x)
        = \sup_{\textcolor{ctrlcolor}{z} \in Z(t, x)} \left \{
            \underline{u}(
                t,
                \Gamma(t, x, \textcolor{ctrlcolor}{z})
            )
            + K(t, x, \textcolor{ctrlcolor}{z})
        \right \}
        \leq \underline{u}(
            t,
            \Gamma(t, x, \textcolor{ctrlcolor}{z^{\delta}})
        )
        + K(t, x, \textcolor{ctrlcolor}{z^{\delta}}) + \delta.
    \]
    By \cref{enum:convergence_impulse_control_hausdorff}, we can pick a sequence $(\textcolor{ctrlcolor}{z_m})_m$ such that $\textcolor{ctrlcolor}{z_m} \rightarrow \textcolor{ctrlcolor}{z^{\delta}}$ and $\textcolor{ctrlcolor}{z_m} \in Z^{\gls*{meshing_parameter}_m}(s_m, y_m)$ for each $m$.
    For brevity, we write $u_i^m$ for the quantity $u^{\gls*{meshing_parameter}_m}(s_m, x_i)$.
    Note that since $s_m = \tau^{n_m}$, $u_i^m$ is the $i$-th entry of the vector $u^{n_m, \gls*{meshing_parameter}_m}$ defined in the lemma statement.
    Using this notation,
    \begin{align*}
        (\mathcal{M}_{n_m} u^{{n_m}, \gls*{meshing_parameter}_m})_{i_m}
        & = \sup_{\textcolor{ctrlcolor}{z} \in Z^{\gls*{meshing_parameter}_m}(s_m, y_m)} \left \{
            \gls*{interp}(
                u^{n_m, \gls*{meshing_parameter}_m},
                \Gamma(s_m, y_m, \textcolor{ctrlcolor}{z})
            ) + K(s_m, y_m, \textcolor{ctrlcolor}{z})
        \right \} \\
        & \geq \gls*{interp}(
            u^{n_m, \gls*{meshing_parameter}_m},
            \Gamma(s_m, y_m, \textcolor{ctrlcolor}{z_m})
        ) + K(s_m, y_m, \textcolor{ctrlcolor}{z_m}) \\
        & = \alpha_m u_{k_m + 1}^m
        + \left( 1 - \alpha_m \right) u_{k_m}^m
        + K(s_m, y_m, \textcolor{ctrlcolor}{z_m})
    \end{align*}
    where $0 \leq \alpha_m \leq 1$ and $k_m$ satisfies
    \begin{equation}
        x_{k_m} \rightarrow 
        \Gamma(t, x, \textcolor{ctrlcolor}{z^{\delta}})
        \textspace \text{as} \textspace m \rightarrow \infty
        \label{eqn:proofs_convergence_direct_control_consistency_3_1}
    \end{equation}
    (recall \cref{eqn:schemes_interpolation}).
    Therefore, by \cref{lem:proofs_convergence_direct_control_consistency_1},
    \begin{align}
        \liminf_{m \rightarrow \infty} (\mathcal{M}_{n_m} u^{{n_m}, \gls*{meshing_parameter}_m})_{i_m}
        & \geq \liminf_{m \rightarrow \infty} \{
            \min \{ u_{k_m+1}^m, u_{k_m}^m \} + K(s_m, y_m, \textcolor{ctrlcolor}{z_m})
        \}
        \nonumber \\
        & \geq \min \{
            \liminf_{m \rightarrow \infty} u_{k_m+1}^m,
            \liminf_{m \rightarrow \infty} u_{k_m}^m
        \} + K(t, x, \textcolor{ctrlcolor}{z^{\delta}}).
        \label{eqn:proofs_convergence_direct_control_consistency_3_2}
    \end{align}
    Now, by \cref{eqn:proofs_convergence_direct_control_consistency_3_1}, $(s_m, x_{k_m}) \rightarrow (t, 
    \Gamma(t, x, \textcolor{ctrlcolor}{z^\delta}))$ as $m \rightarrow \infty$.
    Therefore, by the definition of the half-relaxed limit $\underline{u}$,
    \begin{equation}
        \liminf_{m \rightarrow \infty} u_{k_m}^m
        = \liminf_{m \rightarrow \infty} u^{\gls*{meshing_parameter}_m}(s_m, x_{k_m})
        \geq \liminf_{\substack{
            \gls*{meshing_parameter} \rightarrow 0 \\
            (s, y) \rightarrow (t, 
            \Gamma(t, x, \textcolor{ctrlcolor}{z^\delta}))
        }} u^{\gls*{meshing_parameter}}(s, y)
        = \underline{u}(t, 
        \Gamma(t, x, \textcolor{ctrlcolor}{z^\delta})).
        \label{eqn:proofs_convergence_direct_control_consistency_3_3}
    \end{equation}
    Similarly, since $x_{k_m+1} - x_{k_m} = \Delta x = \gls*{const} \gls*{meshing_parameter}_m$,
    \begin{equation}
        \liminf_{m \rightarrow \infty} u_{k_m+1}^m
        \geq \underline{u}(t, 
        \Gamma(t, x, \textcolor{ctrlcolor}{z^\delta})).
        \label{eqn:proofs_convergence_direct_control_consistency_3_4}
    \end{equation}
    Applying \cref{eqn:proofs_convergence_direct_control_consistency_3_3,eqn:proofs_convergence_direct_control_consistency_3_4} to \cref{eqn:proofs_convergence_direct_control_consistency_3_2},
    \begin{equation}
        \liminf_{m \rightarrow \infty} (\mathcal{M}_{n_m} u^{{n_m}, \gls*{meshing_parameter}_m})_{i_m}
        \geq \underline{u}(t, 
        \Gamma(t, x, \textcolor{ctrlcolor}{z^{\delta}}))
        + K(t, x, \textcolor{ctrlcolor}{z^{\delta}})
        \geq \mathcal{M}\underline{u}(t, x) - \delta.
        \label{eqn:proofs_convergence_direct_control_consistency_3_5}
    \end{equation}
    Since $\delta$ is arbitrary, the desired result follows.

    We now handle the rightmost inequality.
    Let $(\gls*{meshing_parameter}_m, s_m, y_m)_m$ be a sequence as in the previous paragraph.
    Since $Z^{\gls*{meshing_parameter}_m}(s_m, y_m)$ is by definition a finite set (and hence compact), for each $m$, there exists a $\textcolor{ctrlcolor}{z_m} \in Z^{\gls*{meshing_parameter}_m}(s_m, y_m)$ such that
    \begin{align*}
        (\mathcal{M}_{n_m} u^{{n_m}, \gls*{meshing_parameter}_m})_{i_m}
        & = \sup_{\textcolor{ctrlcolor}{z} \in Z^{\gls*{meshing_parameter}_m}(s_m, y_m)} \left \{
            \gls*{interp}(
                u^{n_m, \gls*{meshing_parameter}_m},
                \Gamma(s_m, y_m, \textcolor{ctrlcolor}{z})
            ) + K(s_m, y_m, \textcolor{ctrlcolor}{z})
        \right \} \\
        & = \gls*{interp}(
            u^{n_m, \gls*{meshing_parameter}_m},
            \Gamma(s_m, y_m, \textcolor{ctrlcolor}{z_m})
        ) + K(s_m, y_m, \textcolor{ctrlcolor}{z_m}) \\
        & = \alpha_m u_{k_m+1}^m + \left( 1 - \alpha_m \right) u_{k_m}^m
        + K(s_m, y_m, \textcolor{ctrlcolor}{z_m})
    \end{align*}
    where $0 \leq \alpha_m \leq 1$ and $k_m$ satisfies \cref{eqn:proofs_convergence_direct_control_consistency_3_1}.
    By \cref{enum:convergence_impulse_control_hausdorff}, $(\textcolor{ctrlcolor}{z_m})_m$ is contained in a compact set and hence has a convergent subsequence $(\textcolor{ctrlcolor}{z_{m_j}})_j$ with limit $\textcolor{ctrlcolor}{\hat{z}}$.
    By another application of \cref{enum:convergence_impulse_control_hausdorff}, the right hand side of the inequality
    \[
        d(\textcolor{ctrlcolor}{\hat{z}}, Z(t, x))
        \leq d(\textcolor{ctrlcolor}{\hat{z}}, \textcolor{ctrlcolor}{z_{m_j}})
        + d(\textcolor{ctrlcolor}{z_{m_j}}, Z(t, x))
    \]
    approaches zero as $j \rightarrow \infty$ and hence $\textcolor{ctrlcolor}{\hat{z}} \in Z(t, x)$.
    Therefore, similarly to how we established \cref{eqn:proofs_convergence_direct_control_consistency_3_2,eqn:proofs_convergence_direct_control_consistency_3_5} in the previous paragraph,
    \begin{align*}
        \limsup_{m \rightarrow \infty} (\mathcal{M}_{n_m} u^{{n_m}, \gls*{meshing_parameter}_m})_{i_m}
        & \leq \limsup_{m \rightarrow \infty} \{
            \max \{ u_{k_m+1}^m, u_{k_m}^m \}
            + K(s_m, y_m, \textcolor{ctrlcolor}{z_m})
        \} \\
        & \leq \max \{
            \limsup_{m \rightarrow \infty} u_{k_m+1}^m,
            \limsup_{m \rightarrow \infty} u_{k_m}^m
        \} + K(t, x, \textcolor{ctrlcolor}{\hat{z}}) \\
        & \leq \overline{u}(t, 
        \Gamma(t, x, \textcolor{ctrlcolor}{\hat{z}}))
        + K(t, x, \textcolor{ctrlcolor}{\hat{z}}) \\
        & \leq \mathcal{M} \overline{u}(t, x). \qedhere
    \end{align*}
\end{proof}

\begin{remark}
    \label{rem:proofs_higher_order}
    The proof above relies heavily on the fact that $\gls*{interp}$ is a linear interpolant and hence the ``interpolation weights'' $\alpha$ and $(1 - \alpha)$ in \cref{eqn:schemes_interpolation} are nonnegative.
    A quadratic interpolant, for example, takes the form
    \[
        \operatorname{quad-interp}(U, x) = \alpha U_{k-1} + (1 - \alpha - \beta) U_k + \beta U_{k+1}
    \]
    where it is not necessarily the case that the coefficients $\alpha$, $(1 - \alpha - \beta)$, and $\beta$ are nonnegative.
    This suggests that 
    higher order discretizations of $\mathcal{M}$ are generally precluded by the nonlocal consistency requirement.
\end{remark}

We are now ready to prove the nonlocal consistency of our schemes.

\begin{lemma}
    The direct control scheme is nonlocally consistent.
    \label{lem:convergence_direct_control_consistency}
\end{lemma}

\begin{proof}
    Let $\Omega = [0,T) \times \mathbb{R}$ and $\varphi \in C^{1,2}(\overline{\Omega})$.
    Let $(u^{\gls*{meshing_parameter}})_{\gls*{meshing_parameter} > 0}$ be a family of uniformly bounded real-valued maps from $\overline{\Omega}$ with half-relaxed limits $\overline{u}$ and $\underline{u}$.
    Let $(\gls*{meshing_parameter}_m, s_m, y_m, \xi_m)_m$ be an arbitrary sequence satisfying
    \begin{equation}
        \gls*{meshing_parameter}_m \rightarrow 0
        \text{,} \textspace
        (s_m, y_m) \rightarrow (t, x)
        \text{,} \textspace \text{and} \textspace
        \xi_m \rightarrow 0
        \textspace \text{as} \textspace
        m \rightarrow \infty.
        \label{eqn:proofs_convergence_direct_control_consistency_0}
    \end{equation}
    Without loss of generality (see \cref{rem:convergence_grid_points}), we will assume that the sequence is chosen such that $s_m = \tau^{n_m} = T - n_m \Delta \tau$ and $y_m = x_{i_m}$ are grid points (as usual, both $\Delta \tau = \gls*{const} h_m$ and $\Delta x = \gls*{const} h_m$ depend on $m$ through $h_m$).
    For brevity, let $\varphi_i^n$ denote $\varphi(\tau^n, x_i)$ and $\varphi^n = (\varphi_0^n, \ldots, \varphi_M^n)^\intercal$.

    By \cref{eqn:convergence_direct_control_nonlocal,eqn:convergence_direct_control_boundary,eqn:convergence_direct_control_interior},
    \begin{equation}
        S(
            \gls*{meshing_parameter}_m,
            (s_m, y_m), 
            \varphi + \xi_m,
            [\mathcal{I}^{\gls*{meshing_parameter}_m} u^{\gls*{meshing_parameter}_m}](s_m, y_m) 
        )
        = \begin{cases}
            \min \{ S_m^{(1)}, S_m^{(2)} \} & \text{if } n_m > 0 \\
            S_m^{(3)} & \text{if } n_m = 0
        \end{cases}
        \label{eqn:proofs_convergence_direct_control_consistency_1}
    \end{equation}
    where
    \begin{align}
        S_m^{(1)}
        & = -\sup_{\textcolor{ctrlcolor}{w} \in W^{\gls*{meshing_parameter}_m}} \left \{
            \frac{\varphi_{i_m}^{n_m-1} - \varphi_{i_m}^{n_m}}{\Delta \tau}
            + \frac{1}{2} b_{i_m}^{n_m}(\textcolor{ctrlcolor}{w})^2 (\mathcal{D}_2 \varphi^{n_m})_{i_m}
            + a_{i_m}^{n_m}(\textcolor{ctrlcolor}{w}) (\mathcal{D} \varphi^{n_m})_{i_m}
            + f_{i_m}^{n_m}(\textcolor{ctrlcolor}{w})
        \right \}
        \nonumber \\
        S_m^{(2)}
        & = \varphi_{i_m}^{n_m} + \xi_m - (\mathcal{M}_{n_m}u^{n_m,\gls*{meshing_parameter}_m})_{i_m}
        \nonumber \\
        S_m^{(3)}
        & = \varphi_{i_m}^{n_m} + \xi_m - g(y_m).
        \label{eqn:proofs_convergence_direct_control_consistency_2}
    \end{align}
    Note that in the above, we have used the notation $u^{n,\gls*{meshing_parameter}}$ introduced in \cref{lem:proofs_convergence_direct_control_consistency_3}.
    Now, by \cref{lem:proofs_convergence_direct_control_consistency_2},
    \begin{multline}
        \lim_{m \rightarrow \infty} S_m^{(1)} \\
        = - \adjustlimits \lim_{m \rightarrow \infty}
        \sup_{\textcolor{ctrlcolor}{w} \in W^{\gls*{meshing_parameter}_m}} \left\{
            \varphi_t(t, x)
            + \frac{1}{2} b_{i_m}^{n_m}(\textcolor{ctrlcolor}{w})^2 \varphi_{xx}(t, x)
            + a_{i_m}^{n_m}(\textcolor{ctrlcolor}{w}) \varphi_x(t, x)
            + f_{i_m}^{n_m}(\textcolor{ctrlcolor}{w})
            + O(\gls*{meshing_parameter}_m)
        \right\} \\
        = - \varphi_t(t, x) - \sup_{\textcolor{ctrlcolor}{w} \in W} \left\{
            \frac{1}{2} b(t, x, \textcolor{ctrlcolor}{w})^2 \varphi_{xx}(t, x)
            + a(t, x, \textcolor{ctrlcolor}{w}) \varphi_x(t, x)
            + f(t, x, \textcolor{ctrlcolor}{w})
        \right\}.
        \label{eqn:proofs_convergence_direct_control_consistency_3}
    \end{multline}
    Moreover, by \cref{lem:proofs_convergence_direct_control_consistency_3},
    \begin{equation}
        \liminf_{m \rightarrow \infty} S_m^{(2)}
        \geq \varphi(t,x) - \limsup_{m\rightarrow\infty} (\mathcal{M}_{n_m} u^{n_m, \gls*{meshing_parameter}_m})_{i_m}
        \geq \varphi(t,x) - \mathcal{M}\overline{u}(t,x)
        \label{eqn:proofs_convergence_direct_control_consistency_4_sub}
    \end{equation}
    and
    \begin{equation}
        \limsup_{m \rightarrow \infty} S_m^{(2)}
        \leq \varphi(t,x) - \liminf_{m\rightarrow\infty} (\mathcal{M}_{n_m} u^{n_m, \gls*{meshing_parameter}_m})_{i_m}
        \leq \varphi(t,x) - \mathcal{M}\underline{u}(t,x).
        \label{eqn:proofs_convergence_direct_control_consistency_4_sup}
    \end{equation}

    Suppose now that $t < T$.
    Since $s_m \rightarrow T$, we may assume that $s_m < T$ (or, equivalently, $n_m > 0$) for each $m$.
    In this case, taking limit inferiors of both sides of \eqref{eqn:proofs_convergence_direct_control_consistency_1} and applying \cref{lem:proofs_convergence_direct_control_consistency_1} yields
    \begin{multline*}
        \liminf_{m\rightarrow\infty}
        S(
            \gls*{meshing_parameter}_m,
            (s_m, y_m), 
            \varphi + \xi_m,
            [\mathcal{I}^{\gls*{meshing_parameter}_m} u^{\gls*{meshing_parameter}_m}](s_m, y_m) 
        )
        \\
        = \liminf_{m \rightarrow \infty} \min \left\{
            S_m^{(1)},
            S_m^{(2)}
        \right\}
        \geq \min \left\{
            \liminf_{m \rightarrow \infty} S_m^{(1)},
            \liminf_{m \rightarrow \infty} S_m^{(2)}
        \right\}
    \end{multline*}
    Applying \cref{eqn:proofs_convergence_direct_control_consistency_3,eqn:proofs_convergence_direct_control_consistency_4_sub} to the above,
    \begin{multline}
        \liminf_{m\rightarrow\infty} S(
            \gls*{meshing_parameter}_m,
            (s_m, y_m), 
            \varphi + \xi_m,
            [\mathcal{I}^{\gls*{meshing_parameter}_m} u^{\gls*{meshing_parameter}_m}](s_m, y_m) 
        )
        \\
        \geq \min \biggl\{
            - \varphi_t(t, x) - \sup_{\textcolor{ctrlcolor}{w} \in W} \left\{
                \frac{1}{2} b(t, x, \textcolor{ctrlcolor}{w})^2 \varphi_{xx}(t, x)
                + a(t, x, \textcolor{ctrlcolor}{w}) \varphi_x(t, x)
                + f(t, x, \textcolor{ctrlcolor}{w})
            \right\}, \\
            \varphi(t,x) - \mathcal{M}\overline{u}(t,x)
        \biggr\}
        = F_*(
            (t, x),
            \varphi(t, x),
            D \varphi (t, x),
            D^2 \varphi (t, x),
            \mathcal{M} \overline{u}(t, x)
        )
        \label{eqn:proofs_convergence_direct_control_consistency_5}
    \end{multline}
    where $D^2 \varphi = \varphi_{xx}$, $D \varphi = (\varphi_t, \varphi_x)$, and $F$ is given by \cref{eqn:convergence_hjbqvi}.
    In establishing the last equality in the above, we have used the fact that $F = F_*$ since $F$ is continuous away from $t = T$.
    Now, since $(\gls*{meshing_parameter}_m, s_m, y_m, \xi_m)_m$ is an arbitrary sequence satisfying \cref{eqn:proofs_convergence_direct_control_consistency_0}, \cref{eqn:proofs_convergence_direct_control_consistency_5} implies
    \begin{equation}
        \liminf_{\substack{
            h \rightarrow 0 \\
            (s, y) \rightarrow (t, x) \\
            \xi \rightarrow 0
        }}
        S(
            h,
            (s, y),
            \varphi + \xi,
            [\mathcal{I}^{\gls*{meshing_parameter}} u^{\gls*{meshing_parameter}}](s, y)
        )
        \geq F_*(
            (t, x),
            \varphi(t, x),
            D \varphi (t, x),
            D^2 \varphi (t, x),
            \mathcal{M} \overline{u}(t, x)
        ),
        \label{eqn:proofs_convergence_direct_control_consistency_6_sub}
    \end{equation}
    which is exactly the nonlocal consistency inequality \cref{eqn:convergence_subconsistency} in the time-dependent case with $\mathcal{I} = \mathcal{M}$.
    Symmetrically, we can establish the inequality
    \begin{equation}
        \limsup_{\substack{
            h \rightarrow 0 \\
            (s, y) \rightarrow (t, x) \\
            \xi \rightarrow 0
        }}
        S(
            h,
            (s, y),
            \varphi + \xi,
            [\mathcal{I}^{\gls*{meshing_parameter}} u^{\gls*{meshing_parameter}}](s, y)
        )
        \leq F^*(
            (t, x),
            \varphi(t, x),
            D \varphi (t, x),
            D^2 \varphi (t, x),
            \mathcal{M} \overline{u}(t, x)
        ),
        \label{eqn:proofs_convergence_direct_control_consistency_6_sup}
    \end{equation}
    which corresponds to the nonlocal consistency inequality \cref{eqn:convergence_superconsistency}.

    Suppose now that $t = T$.
    Since $s_m \rightarrow t$, it is possible that $s_m = T$ (or, equivalently, $n_m = 0$) for one more indices $m$ in the sequence.
    Therefore, by \cref{eqn:proofs_convergence_direct_control_consistency_1},
    \begin{align}
        S(
            \gls*{meshing_parameter}_m,
            (s_m, y_m), 
            \varphi + \xi_m,
            [\mathcal{I}^{\gls*{meshing_parameter}_m} u^{\gls*{meshing_parameter}_m}](s_m, y_m) 
        )
        & \geq \min \left \{
            S_m^{(1)},
            S_m^{(2)},
            S_m^{(3)}
        \right\}
        \nonumber \\
        & = \min \left (
            \min \left \{
                S_m^{(1)},
                S_m^{(2)}
            \right \},
            S_m^{(3)}
        \right ).
        \label{eqn:proofs_convergence_direct_control_consistency_7}
    \end{align}
    An immediate consequence of the definition of $S_m^{(3)}$ in \cref{eqn:proofs_convergence_direct_control_consistency_2} is that
    \begin{equation}
        \lim_{m \rightarrow \infty} S_m^{(3)}
        = \varphi(t,x) - g(x)
        \geq \min \left \{
            \varphi(t,x) - g(x),
            \varphi(t,x) - \mathcal{M} \overline{u}(t,x)
        \right \}.
        \label{eqn:proofs_convergence_direct_control_consistency_8}
    \end{equation}
    Taking limit inferiors of both sides of \cref{eqn:proofs_convergence_direct_control_consistency_7} and applying \cref{lem:proofs_convergence_direct_control_consistency_1,eqn:proofs_convergence_direct_control_consistency_3,eqn:proofs_convergence_direct_control_consistency_4_sub,eqn:proofs_convergence_direct_control_consistency_8},
    \begin{multline*}
        \liminf_{m\rightarrow\infty}
        S(
            \gls*{meshing_parameter}_m,
            (s_m, y_m), 
            \varphi + \xi_m,
            [\mathcal{I}^{\gls*{meshing_parameter}_m} u^{\gls*{meshing_parameter}_m}](s_m, y_m)
        ) \\
        \geq \min \biggl(
            \min \biggl\{
                -\varphi_t(t, x) - \sup_{\textcolor{ctrlcolor}{w} \in W} \left\{
                    \frac{1}{2} b(t, x, \textcolor{ctrlcolor}{w})^2 \varphi_{xx}(t, x)
                    + a(t, x, \textcolor{ctrlcolor}{w}) \varphi_x(t, x)
                    + f(t, x, \textcolor{ctrlcolor}{w})
                \right\}, \\
                \varphi(t,x) - \mathcal{M}\overline{u}(t,x)
            \biggr\}, \,
            \min \biggl\{
                \varphi(t,x) - g(x),
                \varphi(t,x) - \mathcal{M}\overline{u}(t,x)
            \biggr\}
        \biggr) \\
        = F_*(
            (t, x),
            \varphi(t, x),
            D \varphi (t, x),
            D^2 \varphi (t, x),
            \mathcal{M} \overline{u}(t, x)
        ).
    \end{multline*}
    As in the previous paragraph, the above implies the nonlocal consistency inequality \cref{eqn:proofs_convergence_direct_control_consistency_6_sub}.

    It remains to establish \cref{eqn:proofs_convergence_direct_control_consistency_6_sup} in the case of $t = T$.
    Since $\mathcal{M} g \leq g$ by assumption \cref{enum:convergence_suboptimality}, it follows that $g(y_m) = \max \{ g(y_m), \mathcal{M} g(y_m) \}$ for each $m$.
    Therefore,
    \begin{align*}
        S_m^{(3)}
        & = \varphi_{i_m}^{n_m} + \xi_m - \max \{ g(y_m), \mathcal{M} g(y_m) \} \\
        & = \min \left \{
            \varphi_{i_m}^{n_m} + \xi_m - g(y_m),
            \varphi_{i_m}^{n_m} + \xi_m - \mathcal{M} g(y_m)
        \right \} \\
        & = \min \left \{
            \varphi_{i_m}^{n_m} + \xi_m - g(y_m),
            \varphi_{i_m}^{n_m} + \xi_m - (\mathcal{M}_0 \vec{g} \,)_{i_m} + O( (\Delta x)^2 )
        \right \}
    \end{align*}
    where in the last equality, in which we have used $\vec{g}$ to denote the vector $\vec{g} = (g_0, \ldots, g_M)^\intercal$, we have employed the fact that there is $O( (\Delta x)^2 )$ error in approximating the intervention operator $\mathcal{M}$ by the discretized intervention operator $\mathcal{M}_0$ due to the linear interpolant.
    By \cref{rem:convergence_weaker_consistency}, we can, without loss of generality, assume that $u^{\gls*{meshing_parameter}}$ is a solution of the scheme so that $u_i^{0,\gls*{meshing_parameter}} = V^{\gls*{meshing_parameter}}(\tau^0, x_i) = g(x_i)$ for all $i$, corresponding to the terminal condition.
    It follows that, letting
    \[
        S_m^{(4)}
        = \min \left \{
            \varphi_{i_m}^{n_m} + \xi_m - g(y_m), 
            \varphi_{i_m}^{n_m} + \xi_m - (\mathcal{M}_{n_m} u^{n_m,{\gls*{meshing_parameter}}_m})_{i_m} + O( (\Delta x)^2 )
        \right \},
    \]
    we have $S_m^{(3)} = S_m^{(4)}$ whenever $n_m = 0$.
    Therefore, by \cref{eqn:proofs_convergence_direct_control_consistency_1}
    \begin{equation}
        S(
            \gls*{meshing_parameter}_m,
            (s_m, y_m), 
            \varphi + \xi_m,
            [\mathcal{I}^{\gls*{meshing_parameter}_m} u^{\gls*{meshing_parameter}_m}](s_m, y_m) 
        )
        \leq \max \left (
            \min \left \{
                S_m^{(1)},
                S_m^{(2)}
            \right \},
            S_m^{(4)}
        \right ).
        \label{eqn:proofs_convergence_direct_control_consistency_9}
    \end{equation}
    Moreover, by \cref{lem:proofs_convergence_direct_control_consistency_1,lem:proofs_convergence_direct_control_consistency_3}
    \begin{equation}
        \limsup_{m \rightarrow \infty} S_m^{(4)}
        \leq \min \left \{
            \varphi(t,x) - g(x),
            \varphi(t,x) - \mathcal{M}\underline{u}(t,x)
        \right \}.
        \label{eqn:proofs_convergence_direct_control_consistency_10}
    \end{equation}
    Taking limit superiors of both sides of \cref{eqn:proofs_convergence_direct_control_consistency_9} and applying \cref{lem:proofs_convergence_direct_control_consistency_1,eqn:proofs_convergence_direct_control_consistency_3,eqn:proofs_convergence_direct_control_consistency_4_sup,eqn:proofs_convergence_direct_control_consistency_10},
    \begin{multline*}
        \limsup_{m\rightarrow\infty}
        S(
            \gls*{meshing_parameter}_m,
            (s_m, y_m), 
            \varphi + \xi_m,
            [\mathcal{I}^{\gls*{meshing_parameter}_m} u^{\gls*{meshing_parameter}_m}](s_m, y_m) 
        ) \\
        \leq \max \biggl(
            \min \biggl\{
                - \varphi_t(t, x) - \sup_{\textcolor{ctrlcolor}{w} \in W} \left\{
                    \frac{1}{2} b(t, x, \textcolor{ctrlcolor}{w})^2 \varphi_{xx}(t, x)
                    + a(t, x, \textcolor{ctrlcolor}{w}) \varphi_x(t, x)
                    + f(t, x, \textcolor{ctrlcolor}{w})
                \right\}, \\
                \varphi(t,x) - \mathcal{M}\underline{u}(t,x)
            \biggr\}, \,
            \min \biggl\{
                \varphi(t,x) - g(x),
                \varphi(t,x) - \mathcal{M}\underline{u}(t,x)
            \biggr\}
        \biggr) \\
        = F^*(
            (t, x),
            \varphi(t, x),
            D \varphi (t, x),
            D^2 \varphi (t, x),
            \mathcal{M} \overline{u}(t, x)
        ),
    \end{multline*}
    which establishes \cref{eqn:proofs_convergence_direct_control_consistency_6_sup}, as desired.
\end{proof}

\begin{lemma}
    The penalty scheme is nonlocally consistent.
    \label{lem:convergence_penalty_consistency}
\end{lemma}

\begin{proof}
    Recalling that the penalty scheme can be viewed as an approximation of the direct control scheme (see \cref{eqn:schemes_penalty_rearranged,eqn:schemes_penalty_gamma} and the text following these equations), it is no surprise that the proof of nonlocal consistency for the penalty scheme is nearly identical to the proof of \cref{lem:convergence_direct_control_consistency}.
    Namely, to obtain the proof for the penalty scheme, we need only replace the definition of $S_m^{(2)}$ in \cref{eqn:proofs_convergence_direct_control_consistency_2} by
    \begin{equation}
        \hat{S}_m^{(2)}
        = \varphi_{i_m}^{n_m}
        + \xi_m
        - (\mathcal{M}_{n_m}u^{n_m,\gls*{meshing_parameter}_m})_{i_m}
        + \epsilon S_m^{(1)}
        \label{eqn:convergence_penalty_consistency_1}
    \end{equation}
    where we have used the hat symbol $\hat{\cdot}$ to distinguish the new definition from the old.
    Since $\epsilon S_m^{(1)} \rightarrow 0$ as $m \rightarrow \infty$ by \cref{eqn:convergence_penalty_vanishes}, it follows that
    \[
        \liminf_{m \rightarrow \infty} \hat{S}_m^{(2)} = \liminf_{m \rightarrow \infty} S_m^{(2)}
        \textspace \text{and} \textspace
        \limsup_{m \rightarrow \infty} \hat{S}_m^{(2)} = \limsup_{m \rightarrow \infty} S_m^{(2)}
    \]
    so that no other changes to the proof of \cref{lem:convergence_direct_control_consistency} are necessary.
\end{proof}

Next, we establish the nonlocal consistency of the explicit-impulse scheme.
The arguments in the proof will exploit heavily the fact that $M \Delta x \rightarrow \infty$ (see \cref{eqn:convergence_grid}) in order to ensure that no overstepping error is made in the approximation \cref{eqn:schemes_lagrangian_derivative}.
In \cref{app:truncated}, we consider the explicit-impulse scheme on a truncated domain, in which case we cannot rely on arguments involving $M \Delta x \rightarrow \infty$.
We will see, in \cref{app:truncated}, that to establish nonlocal consistency in the truncated case requires us to modify the spatial grid so that the distance between the first and last two grid points ($x_1 - x_0$ and $x_M - x_{M-1}$) vanishes sublinearly with respect to $\gls*{meshing_parameter}$.
In practice, this is not a grave issue since we are not interested in obtaining high accuracy at the boundaries.

\begin{lemma}
    The explicit-impulse scheme is nonlocally consistent.
    \label{lem:convergence_explicit_impulse_consistency}
\end{lemma}

\begin{proof}
    Let $\Omega = [0, T) \times \mathbb{R}$ and $\varphi \in C^{1,2}(\overline{\Omega})$.
    For brevity, let $\varphi_i^n$ denote $\varphi(\tau^n, x_i)$ and $\varphi^n = (\varphi_0^n, \ldots, \varphi_M^n)^\intercal$.
    Let $(\gls*{meshing_parameter}_m, s_m, y_m, \xi_m)_m$ be a sequence chosen as in the proof of \cref{lem:convergence_direct_control_consistency} so that $s_m = \tau^{n_m}$ and $y_m = x_{i_m}$ are grid points.
    Then,
    \begin{multline}
        \frac{
            \gls*{interp}(
                \varphi^{n_m-1} + \xi_m,
                y_m
                + a_{i_m}^{n_m}(\textcolor{ctrlcolor}{w})
                \Delta \tau
            )
            - \varphi(\tau^{n_m}, y_m) - \xi_m
        }{\Delta \tau} \\
        = \frac{
            \varphi(
                s_m + \Delta \tau,
                y_m
                + a_{i_m}^{n_m}(\textcolor{ctrlcolor}{w})
                \Delta \tau
            )
            - \varphi(s_m, y_m)
            + \xi_m - \xi_m
        }{\Delta \tau} + O\left( \frac{(\Delta x)^2}{\Delta \tau} \right) \\
        = \varphi_t(t, x)
        + a(t, x, \textcolor{ctrlcolor}{w}) \varphi_x(t, x)
        + O\left( \frac{(\Delta x)^2}{\Delta \tau} + \Delta \tau \right)
        = \varphi_t(t, x)
        + a(t, x, \textcolor{ctrlcolor}{w}) \varphi_x(t, x)
        + O\left( \gls*{meshing_parameter}_m \right)
        \label{eqn:proofs_convergence_explicit_impulse_consistency_1}
    \end{multline}
    by a Taylor expansion (compare with \cref{sec:schemes_convergence_rates}).
    As mentioned in the text preceding the lemma statement, we have exploited \cref{eqn:convergence_grid} to ensure that the point $y_m + a_{i_m}^{n_m}(\textcolor{ctrlcolor}{w}) \Delta \tau$ is contained in the interval $[-(M/2) \Delta x, (M/2) \Delta x]$ for $m$ sufficiently large (recall that $M \Delta x \rightarrow \infty$ as $h_m \rightarrow 0$).

    As with the penalty scheme, only a minor modification of the proof of \cref{lem:convergence_direct_control_consistency} is required to obtain the proof of nonlocal consistency of the explicit-impulse scheme.
    In particular, we need only replace the definitions of $S_m^{(1)}$ and $S_m^{(2)}$ in \cref{eqn:proofs_convergence_direct_control_consistency_2} by
    \begin{multline*}
        \hat{S}_m^{(1)}
        = -\sup_{\textcolor{ctrlcolor}{w} \in W^{\gls*{meshing_parameter}_m}} \biggl \{
            \frac{
                \gls*{interp}(
                    \varphi^{n_m-1} + \xi_m,
                    y_m
                    + a_{i_m}^{n_m}(\textcolor{ctrlcolor}{w})
                    \Delta \tau
                )
                - \varphi(\tau^{n_m}, y_m) - \xi_m
            }{\Delta \tau} \\
            + \frac{1}{2} b_{i_m}^{n_m}(\textcolor{ctrlcolor}{w})^2 (\mathcal{D}_2 \varphi^{n_m})_{i_m}
            + f_{i_m}^{n_m}(\textcolor{ctrlcolor}{w})
        \biggr \}
    \end{multline*}
    and
    \[
        \hat{S}_m^{(2)}
        = \varphi_{i_m}^{n_m}
        + \xi_m
        - (\mathcal{M}_{n_m}u^{n_m-1,\gls*{meshing_parameter}_m})_{i_m}.
    \]
    As usual, we have used the hat symbol $\hat{\cdot}$ to distinguish the new definitions from the old.
    Note that by \cref{eqn:proofs_convergence_explicit_impulse_consistency_1},
    \[
        \lim_{m \rightarrow \infty} \hat{S}_m^{(1)} = \lim_{m \rightarrow \infty} S_m^{(1)}.
    \]
    Moreover, by \cref{lem:proofs_convergence_direct_control_consistency_3},
    \[
        \liminf_{m \rightarrow \infty} \hat{S}_m^{(2)}
        \geq \varphi(t,x) - \limsup_{m\rightarrow\infty} (\mathcal{M}_{n_m} u^{n_m-1, \gls*{meshing_parameter}_m})_{i_m}
        \geq \varphi(t,x) - \mathcal{M}\overline{u}(t,x)
    \]
    and
    \[
        \limsup_{m \rightarrow \infty} \hat{S}_m^{(2)}
        \leq \varphi(t,x) - \liminf_{m\rightarrow\infty} (\mathcal{M}_{n_m} u^{n_m-1, \gls*{meshing_parameter}_m})_{i_m}
        \leq \varphi(t,x) - \mathcal{M}\underline{u}(t,x)
    \]
    (compare with \cref{eqn:proofs_convergence_direct_control_consistency_4_sub,eqn:proofs_convergence_direct_control_consistency_4_sup}) so that no other changes to the proof of \cref{lem:convergence_direct_control_consistency} are necessary.
\end{proof}

\subsection{Convergence result}

In this subsection, we give the convergence result for our schemes.
First, we recall what it means for \cref{eqn:convergence_pde} to satisfy a comparison principle.
In particular, we say that PDE \cref{eqn:convergence_pde} satisfies a comparison principle if the following condition is met:
\begin{equation}
    \text{if } U, V \in \gls*{bdd}
    \text{ are a subsolution and supersolution, respectively, of } \eqref{eqn:convergence_pde}
    \text{, then } U \leq V.
    \label{eqn:convergence_comparison}
\end{equation}
We can now state our convergence result.

\begin{theorem}
    Suppose \cref{enum:convergence_continuity}\textendash \cref{enum:convergence_impulse_control_hausdorff} and that the HJBQVI satisfies a comparison principle.
    Then, as $h \rightarrow 0$, solutions of the direct control, penalty, and explicit-impulse schemes converge locally uniformly to the unique bounded solution of the HJBQVI.
    \label{thm:convergence_hjbqvi_result}
\end{theorem}

The above is an immediate corollary of a more general convergence result for the nonlocal PDE \cref{eqn:convergence_pde} given in the next section.

Note that \cref{thm:convergence_hjbqvi_result} depends on the HJBQVI satisfying a comparison principle.
Below, we establish that a comparison principle holds for the HJBQVI if the coefficients $a$ and $b$ do not depend on time and are Lipschitz in space.
This result appears in our article \cite{azimzadeh2017zero}.
The proof, being somewhat technical, is deferred to \cref{app:comparison}.

\begin{theorem}
    Suppose \cref{enum:convergence_continuity}\textendash \cref{enum:convergence_impulse_control_hausdorff} and that the functions $a$ and $b$ are independent of time (i.e., $a(t, x, \textcolor{ctrlcolor}{w}) = a(x, \textcolor{ctrlcolor}{w})$ and similarly for $b$) and satisfy the Lipschitz condition
    \[
            \left| a(x, \textcolor{ctrlcolor}{w}) - a(y, \textcolor{ctrlcolor}{w}) \right|
            + \left| b(x, \textcolor{ctrlcolor}{w}) - b(y, \textcolor{ctrlcolor}{w}) \right|
        \leq \gls*{const} \left | x - y \right|
    \]
    where \gls*{const} does not depend on $\textcolor{ctrlcolor}{w}$.
    Then, the HJBQVI satisfies a comparison principle.
    \label{thm:convergence_comparison}
\end{theorem}

\section{General convergence result}
\label{sec:convergence_extension}

In this section, we show that any monotone, stable, and nonlocally consistent scheme for a nonlocal PDE satisfying a comparison principle converges to the unique solution of that PDE.
As an immediate corollary, we obtain \cref{thm:convergence_hjbqvi_result} of the previous section.
No assumptions are made about the set $\Omega$, other than that it is a subset of $d$-dimensional Euclidean space (in this case, we understand $\varphi \in C^2(\overline{\Omega})$ to mean that $\varphi$ is differentiable in a neighbourhood of $\overline{\Omega}$, relative to $\mathbb{R}^d$).

\begin{theorem}
    \label{thm:convergence_result}
    Suppose the PDE \cref{eqn:convergence_pde} satisfies a comparison principle.
    Let 
    $(S, \mathcal{I}^{\gls*{meshing_parameter}})$ be a monotone, stable, and nonlocally consistent scheme for \cref{eqn:convergence_pde}.
    Then, as $\gls*{meshing_parameter} \rightarrow 0$, $V^{\gls*{meshing_parameter}}$ (a solution of \cref{eqn:convergence_scheme}) converges locally uniformly to the unique bounded solution of \cref{eqn:convergence_pde}.
\end{theorem}

\begin{proof}[Proof of \cref{thm:convergence_result}]
    The proof follows closely that of \cite[Theorem 2.1]{MR1115933}, differing mainly in its use of the nonlocal consistency requirement.

    Let $\overline{V}$ and $\underline{V}$ denote the half-relaxed limits of the family $(V^{\gls*{meshing_parameter}})_{\gls*{meshing_parameter} > 0}$.
    We seek to show that $\overline{V}$ is a subsolution and $\underline{V}$ is a supersolution of \eqref{eqn:convergence_pde}.
    In this case, \eqref{eqn:convergence_comparison} yields $\overline{V} \leq \underline{V}$, while the reverse inequality is a trivial consequence of the definition of $\overline{V}$ and $\underline{V}$.
    Therefore, $V = \overline{V} = \underline{V}$ is a (continuous) solution of \eqref{eqn:convergence_pde}.
    It follows that
    \[
        \lim_{
            \substack{
                \gls*{meshing_parameter} \rightarrow 0 \\
                y \rightarrow x
            }
        } V^{\gls*{meshing_parameter}}(y) = V(x)
        \textspace \text{for } x \in \overline{\Omega},
    \]
    from which we obtain that convergence is locally uniform.

    Returning to our previous claim, we prove that $\overline{V}$ is a subsolution (that $\underline{V}$ is a supersolution is proved similarly).
    To this end, let $x \in \overline{\Omega}$ be a \emph{local} maximum point of $\overline{V} - \varphi$ where $\varphi\in C^2(\overline{\Omega})$.
    By definition, we can find a neighbourhood (relative to $\overline{\Omega}$) $U$ of $x$ whose closure is compact and on which $x$ is a \emph{global} maximum point of $\overline{V} - \varphi$.
    Without loss of generality, we may assume that this maximum is strict, $\overline{V}(x) = \varphi(x)$, and $\varphi \geq \sup_{\gls*{meshing_parameter}} \Vert V^{\gls*{meshing_parameter}} \Vert_{\infty}$ outside $U$ (cf. proof of \cite[Theorem 2.1]{MR1115933}).
    By the definition of $\overline{V}$, we can find a sequence $(\gls*{meshing_parameter}_m,x_m)_m$ such that $\gls*{meshing_parameter}_m \rightarrow 0$, $x_m \rightarrow x$, and $V^{\gls*{meshing_parameter}_m}(x_m) \rightarrow \overline{V}(x)$.
    Now, for each $m$, pick $y_m \in U$ such that
    \[
        V^{\gls*{meshing_parameter}_m}(y_m) - \varphi(y_m) + e^{-1 / \gls*{meshing_parameter}_m}
        \geq \sup_{y \in U} \left\{
            V^{\gls*{meshing_parameter}_m}(y) - \varphi(y)
        \right\}.
    \]
    Due to the compactness of $\overline{U}$, we can pick a subsequence of $(\gls*{meshing_parameter}_m,x_m,y_m)_m$ such that its last argument converges to some point $\hat{y} \in \overline{U}$.
    With a slight abuse of notation, relabel this subsequence $(\gls*{meshing_parameter}_m,x_m,y_m)_m$.
    It follows that
    \begin{align*}
        0 = \overline{V}(x) - \varphi(x) & = \lim_{m\rightarrow\infty} \left\{
            V^{\gls*{meshing_parameter}_m}(x_m) - \varphi(x_m)
        \right\} \\
        & \leq \limsup_{m \rightarrow \infty} \left\{
            V^{\gls*{meshing_parameter}_m}(y_m) - \varphi(y_m) + e^{-1 / \gls*{meshing_parameter}_m}
        \right\} \\
        & \leq \limsup_{
            \substack{
                \gls*{meshing_parameter} \rightarrow 0 \\
                y \rightarrow \hat{y}
            }
        } \left\{
            V^{\gls*{meshing_parameter}}(y) - \varphi(y)
			+ e^{-1 / \gls*{meshing_parameter}_m}
        \right\} \\
        & \leq\overline{V}(\hat{y})-\varphi(\hat{y}).
    \end{align*}
    Because $x$ was assumed to be a strict maximum point, the above inequality implies $\hat{y} = x$.
    Letting $\xi_m = V^{\gls*{meshing_parameter}_m}(y_m) - \varphi(y_m) + e^{-1 / \gls*{meshing_parameter}_m}$, we have $\xi_m \rightarrow 0$ and $V^{\gls*{meshing_parameter}_m} \leq \varphi + \xi_m$ for $m$ sufficiently large (recall that outside of $U$, $\varphi + \xi_m \geq \sup_{\gls*{meshing_parameter}} \Vert V^{\gls*{meshing_parameter}} \Vert_\infty + \xi_m$).
	Let $\psi_m(y) = \varphi(y) - e^{-1 / \gls*{meshing_parameter}_m} \boldsymbol{1}_{\{y_m\}}(y)$.
    Now, the definition of $V^{\gls*{meshing_parameter}}$ and the monotonicity of $S$ yield
    \[
        0 = S(
            \gls*{meshing_parameter}_m,
            y_m,
            V^{\gls*{meshing_parameter}_m},
            [\mathcal{I}^{\gls*{meshing_parameter}_m} V^{\gls*{meshing_parameter}_m}](y_m)
        )
        \geq S(
            \gls*{meshing_parameter}_m,
            y_m,
            \psi_m + \xi_m,
            [\mathcal{I}^{\gls*{meshing_parameter}_m} V^{\gls*{meshing_parameter}_m}](y_m)
        ).
    \]
	Our assumption \cref{eq:convergence_technical_assumption} implies
	\[
		0
		\geq S(
            \gls*{meshing_parameter}_m,
            y_m,
            \varphi + \xi_m,
            [\mathcal{I}^{\gls*{meshing_parameter}_m} V^{\gls*{meshing_parameter}_m}](y_m)
        ) - \epsilon_m
	\]
	where $\epsilon_m \rightarrow 0$ as $m \rightarrow \infty$.
    Taking limit inferiors and employing nonlocal consistency,
    \begin{align*}
        0 & \geq \liminf_{m \rightarrow \infty} S(
            \gls*{meshing_parameter}_m,
            y_m,
            \varphi + \xi_m,
            [\mathcal{I}^{\gls*{meshing_parameter}_m} V^{\gls*{meshing_parameter}_m}](y_m)
        ) \\
        & \geq \liminf_{
            \substack{\gls*{meshing_parameter} \rightarrow 0\\
            y \rightarrow x\\
            \xi \rightarrow 0
        }} S(
            \gls*{meshing_parameter},
            y,
            \varphi + \xi,
            [\mathcal{I}^{\gls*{meshing_parameter}} V^{\gls*{meshing_parameter}}](y)
        ) \\
        & \geq F_*(x, \varphi(x), D \varphi(x), D^{2} \varphi(x), [\mathcal{I} \overline{V}](x)),
    \end{align*}
    which is the desired inequality, since $\overline{V}(x) = \varphi(x)$.
\end{proof}

\begin{remark}
    Recall that the functions $(u^{\gls*{meshing_parameter}})_{\gls*{meshing_parameter} > 0}$ appearing in the nonlocal consistency inequalities \cref{eqn:convergence_subconsistency,eqn:convergence_superconsistency} are an arbitrary family of uniformly bounded real-valued maps.
    However, a close inspection of the above proof reveals that we only use these inequalities with $u^{\gls*{meshing_parameter}} = V^{\gls*{meshing_parameter}}$ where $V^{\gls*{meshing_parameter}}$ is a solution of the scheme \cref{eqn:convergence_scheme}.
    Therefore, in establishing nonlocal consistency we can, without loss of generality, assume that $u^{\gls*{meshing_parameter}}$ is a solution of the scheme.
    \label{rem:convergence_weaker_consistency}
\end{remark}

We close this section by mentioning an extension of \cref{thm:convergence_result} that allows us to approximate solutions that are not necessarily bounded (the proof is identical to that of \cref{thm:convergence_result}, save for minor modifications).
This is of practical importance, since many problems in finance do not have bounded solutions.
Let $B_d(\overline{\Omega})$ be the set of all functions $V \colon \overline{\Omega} \rightarrow \mathbb{R}$ satisfying the polynomial growth condition
\[
    \left| V(x) \right| \leq \gls*{const} (1 + \left| x \right|^d) 
    \text{ for all } x \in \overline{\Omega}.
\]
Now, relax (i) the stability condition to read
\[
    \text{there exists a unique solution } V^{\gls*{meshing_parameter}} \in B_d(\overline{\Omega})
    \text{ of } \cref{eqn:convergence_scheme} \text{ for each } \gls*{meshing_parameter} > 0,
\]
(ii) the consistency condition by requiring \cref{eqn:convergence_subconsistency,eqn:convergence_superconsistency} to hold more generally for families $(u^{\gls*{meshing_parameter}})_{\gls*{meshing_parameter} > 0} \subset B_d(\overline{\Omega})$ not necessarily uniformly bounded, and (iii) the comparison principle \cref{eqn:convergence_comparison} by replacing instances of $B(\overline{\Omega})$ by $B_d(\overline{\Omega})$.
Then, we obtain a relaxation of \cref{thm:convergence_result} which allows for solutions of polynomial growth.

\section{Extensions}

\subsection{The infinite horizon (steady state) case}

Recall that in \cref{chap:schemes}, we gave versions of the direct control and penalty schemes for the infinite horizon HJBQVI \cref{eqn:schemes_infinite_horizon}.
The arguments for these infinite horizon schemes are nearly identical to those in \cref{sec:convergence_hjbqvi}, save that we obtain the stability bound $\Vert f \Vert_\infty \beta^{-1}$ instead of $\Vert f \Vert_\infty T + \Vert g \Vert_\infty$.
Similarly to \cref{eqn:convergence_stability_intuition}, this bound has the intuitive interpretation of being the maximum possible continuous reward in \cref{eqn:schemes_infinite_horizon_functional}:
\[
    \int_t^\infty e^{-\beta u} f(u, X_u, \textcolor{ctrlcolor}{w_u}) du
    \leq \Vert f \Vert_\infty \int_t^\infty e^{-\beta u} du
    \leq \Vert f \Vert_\infty \int_0^\infty e^{-\beta u} du
    = \Vert f \Vert_\infty \beta^{-1}.
\]

\subsection{Higher dimensions}
\label{subsec:convergence_high_dimension_hjbqvi}

The HJBQVI \cref{eqn:introduction_hjbqvi} is a special case of the following higher dimensional HJBQVI \cite[Chapter 8]{MR2109687}:
\begin{align}
    \min \left\{
        - V_t - \sup_{\textcolor{ctrlcolor}{w} \in W} \left\{
                \mathcal{L}^{\textcolor{ctrlcolor}{w}} V(t, x)
                + f(\cdot, \textcolor{ctrlcolor}{w})
        \right\},
        V - \mathcal{M} V
    \right\} & = 0 \text{ on } [0,T) \times \gls*{domain} \nonumber \\
    \min \left \{
        V(T,\cdot) - g,
        V(T,\cdot) - \mathcal{M} V(T,\cdot)
    \right \} & = 0 \text{ on } \gls*{domain}
    \label{eqn:convergence_high_dimension_hjbqvi}
\end{align}
where
\[
    \mathcal{L}^{\textcolor{ctrlcolor}{w}} V(t, x)
    = \frac{1}{2} \gls*{trace}(
        b(\cdot, \textcolor{ctrlcolor}{w})
        b(\cdot, \textcolor{ctrlcolor}{w})^\intercal
        D^2 V
    )
    + \left\langle
        a(\cdot, \textcolor{ctrlcolor}{w}),
        DV
    \right\rangle.
\]
In the above, $DV$ and $D^2 V$ are the gradient vector and Hessian matrix of $V$ (with respect to the spatial coordinate $x = (x_1, \ldots, x_d)$) and $\gls*{domain} \subset \mathbb{R}^d$.
Note that in the above, $a$ and $b$ are no longer real-valued functions but rather $d \times 1$ vector and $d \times d$ matrix valued functions.

If no cross-derivatives appear in \cref{eqn:convergence_high_dimension_hjbqvi} (i.e., if $b$ is a diagonal matrix at each point in its domain), the direct control, penalty, and explicit-impulse schemes can be made to handle the higher-dimensional case by extending the finite difference approximations $\mathcal{D}_2$ and $\mathcal{D}$ to higher dimensions in the obvious way.
An analogous claim can be made for a higher dimensional relaxation of the infinite horizon HJBQVI \cref{eqn:schemes_infinite_horizon}.

However, in the general case involving cross-derivatives, special care must be taken to ensure that the schemes are monotone, either by employing wide-stencils \cite{MR2399429,MR3649430,chen2016monotone,chen2017multigrid} or by techniques involving interpolation \cite{MR3042570}.
We do not handle the cross-derivative case in this thesis, leaving it for future work.

\section{Summary}

In this chapter, we extended the Barles-Souganidis framework to nonlocal PDEs in a general manner.
We used our results to prove the convergence of the direct control, penalty, and explicit-impulse schemes to the viscosity solution of the HJBQVI.

\setcounter{chapter}{4}
\chapter{Numerical results}
\label{chap:results}

In this chapter, we apply the findings of the previous chapters to compute numerical solutions of the following impulse control problems from finance:
\begin{enumerate}[label=(\Alph*)]
    \item
        \label{enum:results_fex}
        Optimal control of the foreign exchange (\gls*{FEX}) rate.
    \item
        \label{enum:results_consumption}
        Optimal consumption with fixed and proportional transaction costs.
    \item
        \label{enum:results_gmwb}
        Guaranteed minimum withdrawal benefits (\glspl*{GMWB}) in variable annuities.
\end{enumerate}
Each of these three problems is a special case of the general impulse control problem involving \cref{eqn:introduction_functional}.
As a result of our numerical tests, we also gain a sense of the relative efficiency of the direct control, penalty, and explicit-impulse schemes.

Our work is the first to give a numerical implementation of problem \cref{enum:results_fex} in the finite horizon setting (i.e., $T < \infty$).
Previously, only ``semi-analytic'' solutions had been considered in the infinite horizon setting (i.e., $T = \infty$) in \cite[Section 5]{MR1705310} and \cite[Section 4]{MR1802595}.
We are also the first to give a numerical implementation of problem \cref{enum:results_consumption} in the finite horizon setting.
Though an infinite horizon implementation of this problem was given in \cite{MR1976512}, the technique used therein is iterated optimal stopping, which is well-known to have a high space complexity and prohibitively slow convergence rate when extended to the finite horizon setting \cite{MR3150265}.\footnote{We mention also the related works \cite{MR3654873,muhle2016primer}, which employ a penalty-like scheme for infinite horizon consumption problems.}

We have seen in \cref{chap:matrix} that for the direct control scheme, policy iteration must be performed on a subset $\mathcal{P}^\prime$ of the original control set $\mathcal{P}$.
Unfortunately, this subset has to be picked on a problem-by-problem basis, making the direct control scheme less robust than the penalty scheme, for which no additional effort is required to ensure convergence of the corresponding policy iteration.
As such, in this chapter, we seek to answer the following natural question: if both schemes converge to the same solution, is there an advantage to using the direct control scheme? (e.g., could it, at least empirically, exhibit faster convergence?)

Our contributions in this chapter are:
\begin{itemize}
    \item Applying the direct control, penalty, and explicit-impulse schemes to three classical impulse control problems from finance (as mentioned above, some of these problems had not been previously considered numerically).
    \item Comparing the performance of the three schemes.
\end{itemize}

The results of this chapter appear in our article 

\fullcite{MR3493959}

\section{Optimal control of the foreign exchange (FEX) rate}
\label{sec:results_fex}

In this section, we study the problem described in \cref{exa:introduction_fex}, in which a government is interested in influencing the \gls*{FEX} rate of its currency.

\subsection{\Artificial{} boundary conditions}

As discussed in \cref{sec:convergence_hjbqvi}, it is computationally intractable to solve the HJBQVI \cref{eqn:results_fex_hjbqvi} of \cref{exa:introduction_fex} since it is posed on the unbounded domain $[0,T] \times \mathbb{R}$.
In light of this, we truncate the domain to a bounded set $[0,T] \times [-R,R]$.
To ensure that that impulses do not leave the truncated domain, we similarly truncate the control set $Z(t, x) = \mathbb{R}$ in \cref{eqn:results_fex_intervention} to $[-R - x, R - x]$.
Artificial boundary conditions are introduced as per \cref{app:truncated}.

\subsection{Implementation details}

%

Unless otherwise specified, we use BiCGSTAB with an ILUT preconditioner \cite[Section 10.4]{MR1990645} to solve all linear systems.
{\sc Policy-Iteration}, used to obtain solutions of the direct control and penalty schemes, is terminated at the $\ell$-th iteration if the error tolerance
\begin{equation}
    \max_i \left\{
        \frac{
            | U_i^\ell - U_i^{\ell - 1} |
        }{
            \max \{ | U_i^\ell |, \text{scale} \}
        }
    \right\} < \text{tolerance}
    \label{eqn:results_convergence_criterion}
\end{equation}
is met.
The parameter $\text{scale}$ is used to ensure that unrealistic levels of accuracy are not imposed if the solution is close to zero.
We use $\text{scale} = 1$ and $\text{tolerance} = 10^{-6}$ for all experiments.
For the penalty scheme, we use the penalty parameter (recall \cref{eqn:convergence_penalty_vanishes}) $\epsilon = 10^{-2} \Delta \tau$.
For the direct control scheme, we use the scaling factor described in \cref{rem:matrix_scaling} with $\delta = 10^{-2}$.
The explicit-impulse scheme is optimized as per \cref{rem:schemes_factorization} whenever possible.
Since these implementation details are the same for all problems considered in this chapter, we do not repeat them.

\subsection{Convergence tests}

\begin{table}
    \centering \footnotesize
    \begin{tabular}{rrr}
        \toprule
        \multicolumn{2}{r}{Parameter} & Value \tabularnewline
        \midrule
        Maximum interest rate differential & $w^{\max}$ & 0.07 \tabularnewline
        Drift speed & $\mu$ & 0.25 \tabularnewline
        Volatility & $\sigma$ & 0.3 \tabularnewline
        Target & $m$ & 0 \tabularnewline
        Effect of interest rate differential & $\gamma$ & 3 \tabularnewline
        Proportional transaction cost & $\kappa$ & 1 \tabularnewline
        Fixed transaction cost & $c$ & 0.1 \tabularnewline
        Discount factor & $\beta$ & 0.02 \tabularnewline
        Horizon & $T$ & 10 \tabularnewline
        Truncated domain boundary & $R$ & 2 \tabularnewline
        \bottomrule
    \end{tabular}
    \caption{Parameters for FEX rate problem}
    \label{tab:results_fex_parameters}
\end{table}

\begin{table}
    \centering \footnotesize
    \begin{tabular}{ccccc}
        \toprule
        $\gls*{meshing_parameter}$ & Timesteps & $x$ points & $\textcolor{ctrlcolor}{w}$ points & $\textcolor{ctrlcolor}{z}$ points \tabularnewline
        \midrule
        1 & 16 & 32 & 8 & 16 \tabularnewline
        1/2 & 32 & 64 & 16 & 32 \tabularnewline
        $\vdots$ & $\vdots$ & $\vdots$ & $\vdots$ & $\vdots$ \tabularnewline
        \bottomrule
    \end{tabular}
    \caption{Numerical grid for FEX rate problem}
    \label{tab:results_fex_grid}
\end{table}

\begin{table}
    \centering \footnotesize
    \subfloat[Direct control scheme]{
        \begin{tabular}{cccccc}
            \toprule
            $\gls*{meshing_parameter}$ & Value $V^{\gls*{meshing_parameter}}(t=0, x=m)$ & Avg. policy its. & Avg. BiCGSTAB its. & Ratio & Norm. time \tabularnewline
            \midrule
            1    & -1.59470667276 & 2.50 & 1.65 &       & 1.24e+01\tabularnewline 
            1/2  & -1.60161214854 & 2.53 & 1.82 &       & 4.30e+01\tabularnewline 
            1/4  & -1.60009885637 & 2.33 & 2.80 & -4.56 & 2.53e+02\tabularnewline 
            1/8  & -1.59882094629 & 2.33 & 2.96 &  1.18 & 1.88e+03\tabularnewline 
            1/16 & -1.59796763572 & 2.36 & 2.97 &  1.50 & 1.73e+04\tabularnewline 
            1/32 & -1.59753341756 & 2.34 & 2.95 &  1.97 & 1.42e+05\tabularnewline 
            1/64 & -1.59730416122 & 2.34 & 2.90 &  1.89 & 1.04e+06\tabularnewline 
            \bottomrule
        \end{tabular}
    }

    \subfloat[Penalty scheme]{
        \begin{tabular}{cccccc}
            \toprule
            $\gls*{meshing_parameter}$ & Value $V^{\gls*{meshing_parameter}}(t=0, x=m)$ & Avg. policy its. & Avg. BiCGSTAB its. & Ratio & Norm. time \tabularnewline
            \midrule
            1    & -1.59542996288 & 2.56 & 1.95 &       & 1.23e+01 \tabularnewline 
            1/2  & -1.60176266672 & 2.53 & 2.13 &       & 4.70e+01 \tabularnewline 
            1/4  & -1.60012316809 & 2.34 & 2.08 & -3.86 & 2.87e+02 \tabularnewline 
            1/8  & -1.59883787204 & 2.33 & 3.08 &  1.28 & 2.09e+03 \tabularnewline 
            1/16 & -1.59796948734 & 2.36 & 3.19 &  1.48 & 1.83e+04 \tabularnewline 
            1/32 & -1.59753376608 & 2.35 & 3.17 &  1.99 & 1.42e+05 \tabularnewline 
            1/64 & -1.59730437362 & 2.34 & 3.18 &  1.90 & 1.04e+06 \tabularnewline 
            \bottomrule
        \end{tabular}
    }

    \subfloat[Explicit-impulse scheme]{
        \begin{tabular}{cccc}
            \toprule
            $\gls*{meshing_parameter}$ & Value $V^{\gls*{meshing_parameter}}(t=0,x=m)$ & Ratio & Norm. time \tabularnewline
            \midrule
            1    & -1.21009825238 &      & 1.00e+00 \tabularnewline
            1/2  & -1.40343492151 &      & 7.14e+00 \tabularnewline
            1/4  & -1.50140778899 & 1.97 & 5.69e+01 \tabularnewline
            1/8  & -1.54909952448 & 2.05 & 4.44e+02 \tabularnewline
            1/16 & -1.57273173354 & 2.02 & 3.33e+03 \tabularnewline
            1/32 & -1.58474899304 & 1.97 & 2.84e+04 \tabularnewline
            1/64 & -1.59084952538 & 1.97 & 2.34e+05 \tabularnewline
            \bottomrule
        \end{tabular}
    }
    \caption{Convergence tests for FEX rate problem}
    \label{tab:results_fex_convergence}
\end{table}

\begin{figure}
    \subfloat[Optimal control at the initial time $t=0$ \label{fig:results_fex_control}]{
        \includegraphics[width=3.25in]{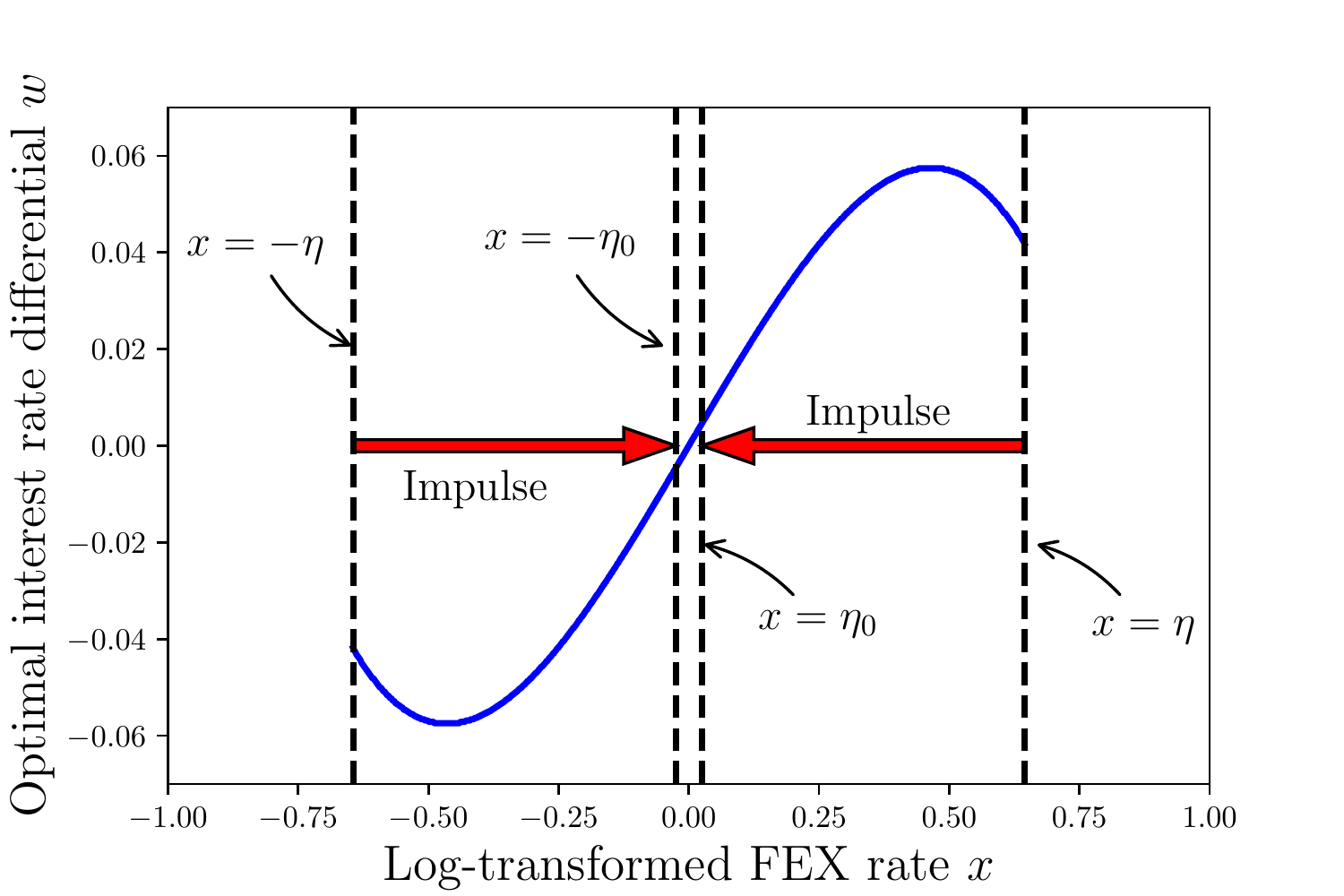}
    }
    \hfill{}
    \subfloat[Cost at the initial time $t=0$ \label{fig:results_fex_value}]{
        \includegraphics[width=3.25in]{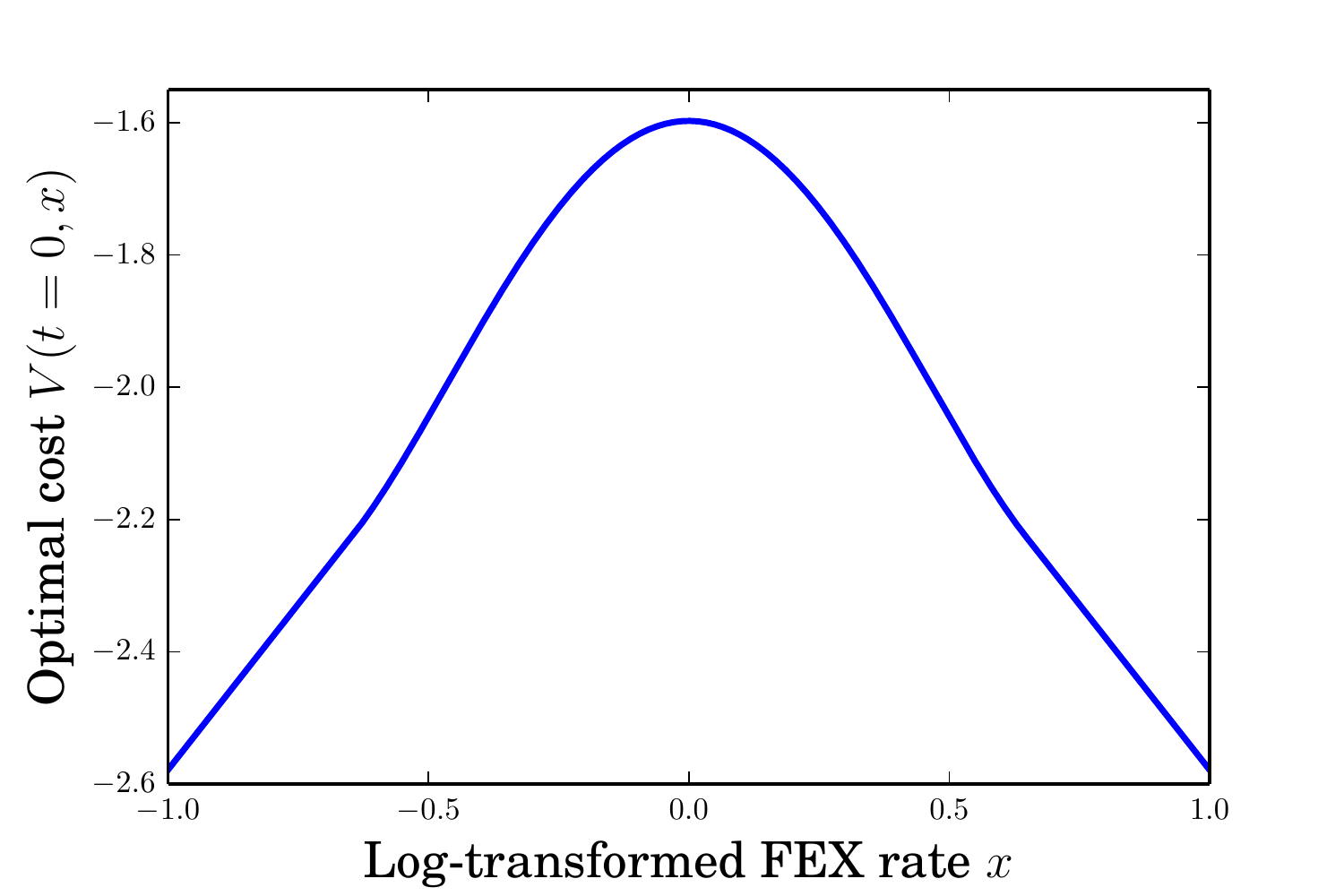}
    }
    \caption{Numerical solution of FEX rate problem}
\end{figure}

For the numerical tests in this section, we use the parameters in \cref{tab:results_fex_parameters}.
\cref{tab:results_fex_grid} reports the size of the numerical grid.
In particular, we report the number of timesteps $\tau^n$, spatial points $x_i$, stochastic control points $\textcolor{ctrlcolor}{w_i}$, and impulse control points $\textcolor{ctrlcolor}{z_i}$.
The value of $R$ is chosen large enough to ensure high accuracy within a region of interest (in this case, at the point $x = m$).

Convergence tests are shown in \cref{tab:results_fex_convergence}.
Times are normalized to the fastest explicit-impulse solve.
The ratio of successive changes in the solution is reported.
The average number of policy iterations and BiCGSTAB iterations \emph{per timestep} are also reported.
For the explicit-impulse scheme applied to the exchange rate problem, we note that the linear system solved at each timestep $n$ involves a tridiagonal matrix $A$ (see \cref{eqn:schemes_tridiagonal} and \cref{rem:schemes_factorization}).
Therefore, in this special case, we use a direct solver.
As such, no iteration counts are reported for the explicit-impulse scheme.

Recall that the explicit-impulse requires only a single linear system solve per timestep.
The direct control and penalty schemes, on the other hand, require the use of policy iteration at each timestep.
Policy iteration can, in the worst case, require $|\mathcal{P}|$ linear system solves \cite{MR3449908}, where $\mathcal{P}$ is the set of controls (recall \cref{sec:matrix_hjbqvi}).
If the discretized control sets $W^{\gls*{meshing_parameter}}$ and $Z^{\gls*{meshing_parameter}}(t, x)$ are each of size $O(1/h)$, this corresponds to a total of $O(|\mathcal{P}|) = O(1/h^2 \cdot 1/(\Delta x)^d) = O(1/h^{d+2})$ iterations, where $d$ is the dimension of the problem (in the case of the \gls*{FEX} rate problem, $d = 1$).
However, this worst case bound is much more pessimistic than what is observed in practice since in \cref{tab:results_fex_convergence}, we see roughly a constant number of policy iterations per timestep.

We close this subsection with a few more observations.
The direct control and penalty schemes are nearly identical in performance and accuracy.
Compared to the explicit-impulse scheme, these two schemes require more execution time but produce solutions that are more accurate even for large values of $\gls*{meshing_parameter}$.
As $\gls*{meshing_parameter}$ approaches zero, we see that all three schemes tend to linear convergence rates (i.e., $\text{Ratio} \approx 2$).

\subsection{Optimal control}

Consider the optimal control in \cref{fig:results_fex_control}.
If the \gls*{FEX} rate is sufficiently close to the target $m = 0$, the government does not intervene in the \gls*{FEX} market.
That is, no impulses occur in the region $(-\eta,\eta)$ for some $\eta$.
In this region, the government influences the \gls*{FEX} rate by choosing the interest rate differential $\textcolor{ctrlcolor}{w}$ (depicted by the smooth curve in \cref{fig:results_fex_control}).
However, if the \gls*{FEX} rate lies outside of this region, the government intervenes by selling large amounts of foreign currency to bring the \gls*{FEX} rate back to some acceptable level $\pm \eta_0$ satisfying $m < \eta_0 < \eta$.
For example, at the point $x = \eta$, the government performs the impulse $\textcolor{ctrlcolor}{z} = \eta_0 - \eta$ (depicted by the rightmost arrow in \cref{fig:results_fex_control}) in order to instantaneously change the \gls*{FEX} rate to $x + \textcolor{ctrlcolor}{z} = \eta + (\eta_0 - \eta) = \eta_0$.

The above suggests that the solution $V$ satisfies
\[
    V(t, x) =
        V(t, \eta_0(t))
        - \kappa \left( \left| x \right| - \eta_0(t) \right)
        - c
    \textspace \text{if } \left| x \right| \geq \eta(t)
\]
where, as reflected in the notation, $\eta$ and $\eta_0$ are allowed to depend on time $t$.
In other words, $V$ is asymptotically linear in space.
This is reflected in \cref{fig:results_fex_value}.

\section{Optimal consumption with fixed and proportional costs}
\label{results:consumption}

The problem of an investor consuming optimally was first studied in the celebrated article \cite{MR0456373}.
To incorporate both fixed and proportional transaction costs, the model was extended to the impulse control setting in \cite{MR1897195}.
In that paper, the authors consider an investor that, at any point in time, has two investment opportunities: a risky investment and a safe bank account.
\begin{itemize}
    \item At all times, the investor picks the rate at which they consume capital from the bank account.
    \item The investor picks specific times at which to transfer money from the bank to the risky investment or vice versa.
\end{itemize}

Let $\textcolor{ctrlcolor}{w_t}$ denote the investor's rate of consumption at time $t$.
Let $\textcolor{ctrlcolor}{\xi_1} \leq \textcolor{ctrlcolor}{\xi_2} \leq \cdots \leq \infty$ be the times at which the investor transfers money between the accounts, with corresponding amounts $\textcolor{ctrlcolor}{z_1}, \textcolor{ctrlcolor}{z_2}, \ldots$
A positive value for $\textcolor{ctrlcolor}{z_\ell}$ indicates that money is moved from the bank account to the risky investment at time $\textcolor{ctrlcolor}{\xi_\ell}$, while a negative value indicates flow of money in the opposite direction.

Subject to the above, the risky investment $S$ and bank account $Q$ evolve according to
    \begin{align*}
        dS_t
        & = \mu S_t dt + \sigma S_t d\gls*{wiener}_t
        & \text{for } t \neq \xi_1, \xi_2, \ldots
        \nonumber \\
        dQ_t
        & = ( r Q_t - \textcolor{ctrlcolor}{w_t} \boldsymbol{1}_{ \{Q_t > 0\} } ) dt
        & \text{for } t \neq \xi_1, \xi_2, \ldots
        \nonumber \\
        S_{\textcolor{ctrlcolor}{\xi_\ell}}
        & = S_{\textcolor{ctrlcolor}{\xi_\ell}-} + \textcolor{ctrlcolor}{z_\ell}
        & \text{for } \ell = \phantom{\xi_1} \mathllap{1}, \phantom{\xi_2} \mathllap{2}, \ldots
        \\
        Q_{\textcolor{ctrlcolor}{\xi_\ell}}
        & =
            Q_{\textcolor{ctrlcolor}{\xi_\ell}-} - \textcolor{ctrlcolor}{z_\ell} - \kappa \left| \textcolor{ctrlcolor}{z_\ell} \right| - c
        & \text{for } \ell = \phantom{\xi_1} \mathllap{1}, \phantom{\xi_2} \mathllap{2}, \ldots
    \end{align*}
where $\mu$, $\sigma$, and $r$ are nonnegative constants representing the drift, volatility, and interest rate.
The term $r Q_t$ in the above corresponds to the riskless return from the bank account, which is reduced by the rate of consumption $\textcolor{ctrlcolor}{w_t}$.
The indicator function $\boldsymbol{1}_{\{Q_t > 0\}}$ is used to ensure that the bank account cannot become negative. 
The constants $c > 0$ and $0 \leq \kappa \leq 1$ parameterize the fixed and proportional costs of moving money to/from the bank account.

To prevent the investor from consuming too quickly, we impose the restriction $|\textcolor{ctrlcolor}{w_t}| \leq w^{\max}$.
We assume that the control $\textcolor{ctrlcolor}{z_\ell}$ must be chosen such that the portfolio remains solvent (i.e., the risky investment and bank account must remain nonnegative at all times).

The investor's objective is to maximize their utility by consuming optimally.
Letting $x = (s, q)$ and $\textcolor{ctrlcolor}{\theta} = (\textcolor{ctrlcolor}{w}; \textcolor{ctrlcolor}{\xi_1}, \textcolor{ctrlcolor}{z_1}; \textcolor{ctrlcolor}{\xi_2}, \textcolor{ctrlcolor}{z_2}; \ldots)$ denote a control, the investor's utility is captured by the objective function
\[
    J(t, x; \textcolor{ctrlcolor}{\theta}) = \mathbb{E} \left[
        \int_t^T
            e^{-\beta u} \frac{
                \textcolor{ctrlcolor}{ w_u^{ \textcolor{black}{\gamma} } }
            }{\gamma}
            \boldsymbol{1}_{\{ Q_u > 0 \}}
        du
        + e^{-\beta T} \frac{
            \max \left\{ Q_T + \left( 1-\kappa \right) S_T - c, 0 \right\}^\gamma
        }{\gamma}
        \middle | (S_{t-}, Q_{t-}) = x
    \right]
\]
where $0 < \gamma \leq 1$ captures the investor's aversion to risk and $\beta$ is a nonnegative discount factor.
The term $\textcolor{ctrlcolor}{ w_u^{ \textcolor{black}{\gamma} } } / \gamma$ corresponds to the utility gained from consumption.
The term $\max \left\{ Q_T + \left( 1-\kappa \right) S_T - c, 0 \right\}^\gamma / \gamma$ corresponds to the investor consuming their net worth at the final time (modeling, for example, retirement).

In \cite[Example 8.3]{MR2109687}, standard dynamic programming arguments are used to transform the optimization problem $V(t, x) = \sup_{\textcolor{ctrlcolor}{\theta}} J(t, x; \textcolor{ctrlcolor}{\theta})$ into the equivalent HJBQVI
    \begin{align}
        \min \left\{
            -V_t - \sup_{\textcolor{ctrlcolor}{w} \in [0, w^{\max}]} \left\{
                    \mathcal{L}^{\textcolor{ctrlcolor}{w}} V
                    + e^{-\beta t} \frac{\textcolor{ctrlcolor}{w}^\gamma}{\gamma}
                    \boldsymbol{1}_{\{ q > 0\}}
            \right\},
            V - \mathcal{M} V
        \right\} & = 0 \text{ on } [0,T) \times [0,\infty)^2
        \nonumber \\
        \min \left\{
            V(T,\cdot) - g,
            V(T, \cdot) - \mathcal{M} V(T, \cdot)
        \right\}
        & = 0 \text{ on } [0,\infty)^2
        \label{eqn:results_consumption_hjbqvi}%
    \end{align}
where
\begin{align}
    \mathcal{L}^{\textcolor{ctrlcolor}{w}} V(t, s, q)
    & = \frac{1}{2} \sigma^2 s^2 V_{ss}(t, s, q) + \mu s V_s(t, s, q)
    + \left( rq - \textcolor{ctrlcolor}{w} \boldsymbol{1}_{\{ q > 0\}} \right) V_q(t, s, q)
    \nonumber \\
    \mathcal{M} V(t, s, q)
    & = \sup_{
        z \in Z(s, q)
    }
    V(
        t,
        s + \textcolor{ctrlcolor}{z},
            q - \textcolor{ctrlcolor}{z} - \kappa \left| \textcolor{ctrlcolor}{z} \right| - c
    )
    \nonumber \\
    Z(s, q)
    & = \left \{
        \textcolor{ctrlcolor}{z} \colon
        s + \textcolor{ctrlcolor}{z} \geq 0
    \right \} \cap \left \{
        \textcolor{ctrlcolor}{z} \colon
        q - \textcolor{ctrlcolor}{z} - \kappa \left| \textcolor{ctrlcolor}{z} \right| - c \geq 0
    \right \}
    \nonumber \\
    g(s, q)
    & = e^{-\beta T} \frac{
        \max \left\{ q + \left( 1-\kappa \right) s - c, 0 \right\}^\gamma
    }{\gamma}.
    \label{eqn:results_consumption_quantities}
\end{align}
Note that the above is a two dimensional HJBQVI with no cross-derivatives and can hence be handled by obvious extensions of the direct control, penalty, and explicit-impulse schemes discussed in the previous chapter (see \cref{subsec:convergence_high_dimension_hjbqvi}).

\subsection{\Artificial{} boundary conditions}
\label{subsec:results_consumption_artificial}

As usual, we do not solve the HJBQVI \cref{eqn:results_consumption_hjbqvi} directly but rather truncate the domain and introduce \artificial{} boundary conditions.
In particular, letting $R_s$ and $R_q$ be large positive numbers, we work on the truncated domain $[0, T] \times [0, R_s] \times [0, R_q]$.

In \cite{MR1976512}, the authors solve an infinite horizon version of the optimal consumption problem.
In that paper, they use the \artificial{} Neumann boundary conditions $V_{ss}(t, R_s, q) = V_s(t, R_s, q) = 0$ and $V_q(t, s, R_q) = 0$.
We follow the approach of \cite{MR1976512} here, defining the operator $\mathcal{L}_R^{\textcolor{ctrlcolor}{w}}$ by
\begin{equation}
    \mathcal{L}_{R}^{\textcolor{ctrlcolor}{w}} V(t, s, q)
    = \left( \frac{1}{2}\sigma^{2}s^{2}V_{ss}(t, s, q)+\mu sV_{s}(t, s, q) \right)
    \boldsymbol{1}_{ \{ s < R_s \}}
    \\
    + \left( rq-\textcolor{ctrlcolor}{w}\boldsymbol{1}_{\{q>0\}} \right) V_{q}(t, s, q)
    \boldsymbol{1}_{ \{q < R_q \} }
    \label{eqn:results_consumption_artificial_operator}
\end{equation}
(compare with $\mathcal{L}^{\textcolor{ctrlcolor}{w}}$ in \cref{eqn:results_consumption_quantities}).
For convenience, we also define the sets
\[
    \Omega = [0,T) \times [0, R_s] \times [0, R_q]
    \textspace \text{and} \textspace
    \partial_T^+ \Omega = \overline{\Omega} \setminus \Omega = \{ T \} \times [0, R_s] \times [0, R_q].
\]
Instead of the HJBQVI \cref{eqn:results_consumption_hjbqvi}, we solve numerically the ``truncated HJBQVI''
\begin{align*}
    \min \left\{
        -V_t - \sup_{\textcolor{ctrlcolor}{w} \in [0, w^{\max}]} \left\{
            \mathcal{L}_R^{\textcolor{ctrlcolor}{w}} V
            + e^{-\beta t} \frac{\textcolor{ctrlcolor}{w}^\gamma}{\gamma}
            \boldsymbol{1}_{\{ q > 0\}}
        \right\},
        V - \mathcal{M}_R V
    \right\} & = 0 \text{ on } \Omega
    \\
    \min \left\{
        V - g,
        V - \mathcal{M}_R V
    \right\}
    & = 0 \text{ on } \partial_T^+ \Omega
\end{align*}
where $\mathcal{L}_R^{\textcolor{ctrlcolor}{w}}$ is given by \cref{eqn:results_consumption_artificial_operator}, $g$ is given by \cref{eqn:results_consumption_quantities}, and
\begin{align*}
    \mathcal{M}_R V(t, s, q)
    & = \sup_{
        \textcolor{ctrlcolor}{z} \in Z_R(s, q)
    }
    V(
        t,
        s + \textcolor{ctrlcolor}{z},
            q - \textcolor{ctrlcolor}{z} - \kappa \left| \textcolor{ctrlcolor}{z} \right| - c
    )
    \\
    Z_R(s, q)
    & =
    \left \{
        \textcolor{ctrlcolor}{z} \colon 
        0 \leq s + \textcolor{ctrlcolor}{z} \leq R_s
    \right \}
    \cap
    \left \{
        \textcolor{ctrlcolor}{z} \colon 
        0 \leq q - \textcolor{ctrlcolor}{z} - \kappa \left| \textcolor{ctrlcolor}{z} \right| - c \leq R_q
    \right \}
\end{align*}
Note that $\mathcal{M}_R$ restricts controls to lie in the set $Z_R(s, q)$ to ensure that impulses do not leave the truncated domain (compare with $\mathcal{M}$ in \cref{eqn:results_consumption_quantities}).

\subsection{Convergence tests}

\begin{table}
    \centering \footnotesize
    \begin{tabular}{rrr}
        \toprule
        \multicolumn{2}{r}{Parameter} & Value \tabularnewline
        \midrule
        Maximum withdrawal rate & $w^{\max}$ & 100 \tabularnewline
        Drift & $\mu$ & 0.11 \tabularnewline
        Volatility & $\sigma$ & 0.3 \tabularnewline
        Interest rate & $r$ & 0.07 \tabularnewline
        Proportional transaction cost & $\kappa$ & 0.1 \tabularnewline
        Fixed transaction cost & $c$ & 0.05 \tabularnewline
        Relative risk aversion & $1-\gamma$ & 0.7 \tabularnewline
        Discount factor & $\beta$ & 0.1 \tabularnewline
        Horizon & $T$ & 40 \tabularnewline
        Initial risky investment value & $s_0$ & 45.20 \tabularnewline
        Initial bank account value & $q_0$ & 45.20 \tabularnewline
        Truncated domain boundaries & $(R_s, R_q)$ & $(200,200)$ \tabularnewline
        \bottomrule
    \end{tabular}
    \caption{Parameters for optimal consumption problem from \cite{MR1976512}}
    \label{tab:results_consumption_parameters}
\end{table}

\begin{table}
    \centering \footnotesize
    \begin{tabular}{cccccc}
        \toprule
        $\gls*{meshing_parameter}$ & Timesteps & $s$ points & $q$ points & $\textcolor{ctrlcolor}{w}$ points & $\textcolor{ctrlcolor}{z}$ points \tabularnewline
        \midrule
        1 & 32 & 20 & 20 & 15 & 15 \tabularnewline
        1/2 & 64 & 40 & 40 & 30 & 30 \tabularnewline
        $\vdots$ & $\vdots$ & $\vdots$ & $\vdots$ & $\vdots$ & $\vdots$ \tabularnewline
        \bottomrule
    \end{tabular}
    \caption{Numerical grid for optimal consumption problem}
    \label{tab:results_consumption_grid}
\end{table}

\begin{table}
    \centering \footnotesize
    \subfloat[Direct control scheme]{
        \begin{tabular}{cccccc}
            \toprule
            $\gls*{meshing_parameter}$ & Value $V^{\gls*{meshing_parameter}}(t=0, s_0, q_0)$ & Avg. policy its. & Avg. BiCGSTAB its. & Ratio & Norm. time \tabularnewline
            \midrule
            1    & 56.0621229141 & 7.63 & 9.77 &      & 1.63e+01 \tabularnewline 
            1/2  & 58.7392240395 & 8.80 & 16.7 &      & 2.93e+02 \tabularnewline 
            1/4  & 59.4201246475 & 10.4 & 23.7 & 3.93 & 5.66e+03 \tabularnewline 
            1/8  & 59.6584129364 & 11.8 & 38.7 & 2.86 & 1.03e+05 \tabularnewline 
            1/16 & 59.7547798553 & 13.3 & 59.2 & 2.47 & 1.85e+06 \tabularnewline 
            1/32 & 59.7972061330 & 14.2 & 92.3 & 2.27 & 3.05e+07 \tabularnewline 
            \bottomrule
        \end{tabular}
    }

    \subfloat[Penalty scheme]{
        \begin{tabular}{cccccc}
            \toprule
            $\gls*{meshing_parameter}$ & Value $V^{\gls*{meshing_parameter}}(t=0, s_0, q_0)$ & Avg. policy its. & Avg. BiCGSTAB its. & Ratio & Norm. time \tabularnewline
            \midrule
            1    & 56.0584963190 & 4.09 & 5.85 &      & 1.03e+01 \tabularnewline 
            1/2  & 58.7390408653 & 3.95 & 6.99 &      & 1.47e+02 \tabularnewline 
            1/4  & 59.4200754123 & 3.40 & 7.99 & 3.94 & 1.99e+03 \tabularnewline 
            1/8  & 59.6583990235 & 3.04 & 11.2 & 2.86 & 2.69e+04 \tabularnewline 
            1/16 & 59.7547779953 & 2.80 & 15.4 & 2.47 & 4.02e+05 \tabularnewline 
            1/32 & 59.7972150004 & 2.58 & 15.4 & 2.27 & 5.86e+06 \tabularnewline 
            \bottomrule
        \end{tabular}
    }

    \subfloat[Explicit-impulse scheme]{
        \begin{tabular}{ccccc}
            \toprule
            $\gls*{meshing_parameter}$ & Value $V^{\gls*{meshing_parameter}}(t=0, s_0, q_0)$ & Avg. BiCGSTAB its. & Ratio & Norm. time \tabularnewline
            \midrule
            1    & 55.6216321734 & 1.00 &      & 1.00e+00 \tabularnewline
            1/2  & 58.7820641022 & 2.00 &      & 1.55e+01 \tabularnewline
            1/4  & 59.4045764001 & 3.00 & 5.08 & 2.60e+02 \tabularnewline
            1/8  & 59.5693702945 & 4.00 & 3.78 & 4.05e+03 \tabularnewline
            1/16 & 59.6511861506 & 6.00 & 2.01 & 6.68e+04 \tabularnewline
            1/32 & 59.7053148416 & 8.00 & 1.51 & 1.11e+06 \tabularnewline
            1/64 & 59.7483254658 & 10.8 & 1.26 & 1.85e+07 \tabularnewline
            \bottomrule
        \end{tabular}
    }
    \caption{Convergence tests for optimal consumption problem}
    \label{tab:results_consumption_convergence}
\end{table}

For the numerical tests in this section, we use the parameters in \cref{tab:results_consumption_parameters}.
\cref{tab:results_consumption_grid} reports the size of the numerical grid.
Convergence tests are shown in \cref{tab:results_consumption_convergence}.

As in the FEX rate problem, the direct control and penalty schemes are nearly identical in accuracy and exhibit linear convergence.
Unlike the FEX rate problem, the average number of BiCGSTAB iterations per timestep increases much faster for the direct control scheme than the penalty scheme, resulting in larger execution times.
This is due to the matrices associated with the penalty scheme being generally better conditioned than those associated with the direct control scheme, a claim which we have verified numerically.
A possible explanation for this phenomenon is to recall that by \cref{eqn:matrix_penalty_matrix_2} and an application of the Gershgorin circle theorem, the matrix $A(P)$ associated with the penalty scheme has eigenvalues bounded away from zero by a distance of $1 / \Delta \tau$ (\cref{fig:results_gershgorin}).
In the case of the direct control scheme, no such guarantees can be made (see \cref{eqn:matrix_direct_control_matrix_2}).

\begin{figure}
    \centering
    \begin{tikzpicture}[scale=2]
        \draw [thick, ->] (-1,0) -- (3.5,0) node [below] {$\mathbb{R}$};
        \draw [thick, ->] (0,-1) -- (0,1);
        \draw (2,0) circle (1cm);
        \node at (2,0) {\textbullet};
        \node at (2,0.2) {$[A(P)]_{ii}$};
        \node at (0,0) {\textbullet};
        \node at (-0.125,-0.15) {$0$};
        \draw [thick,decorate,decoration={brace,amplitude=10pt}](0,0) -- (1,0) node [midway,yshift=0.7cm] {$1/\Delta \tau$};
    \end{tikzpicture}
    \caption{The $i$-th Gershgorin disc (in the complex plane) for the penalty scheme}
    \label{fig:results_gershgorin}
\end{figure}
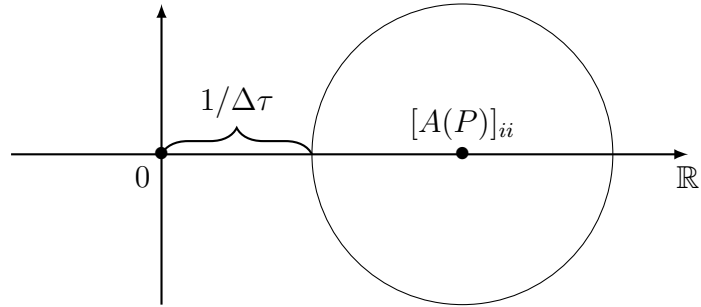

For the explicit-impulse scheme, we note that the average number of BiCGSTAB iterations \emph{per timestep} are nearly integer, suggesting that the number of BiCGSTAB iterations is roughly a constant independent of the timestep $n$.
A possible explanation for this phenomenon is to recall that at each timestep $n$, the explicit-impulse scheme involves solving, for the vector $V^n$, a linear system of the form
\[
    A V^n = y^n
\]
(see \cref{eqn:schemes_linear_system,rem:schemes_factorization}).
In the above, we have made explicit the dependence of the right hand side on the timestep $n$ by writing $y^n$.
Since the matrix $A$ does not depend on the timestep $n$, we expect BiCGSTAB to have similar performance from timestep to timestep.

\subsection{Optimal control}

\begin{figure}
    \subfloat[Optimal control at the initial time $t=0$ \label{fig:results_consumption_control}]{
        \includegraphics[width=3.25in]{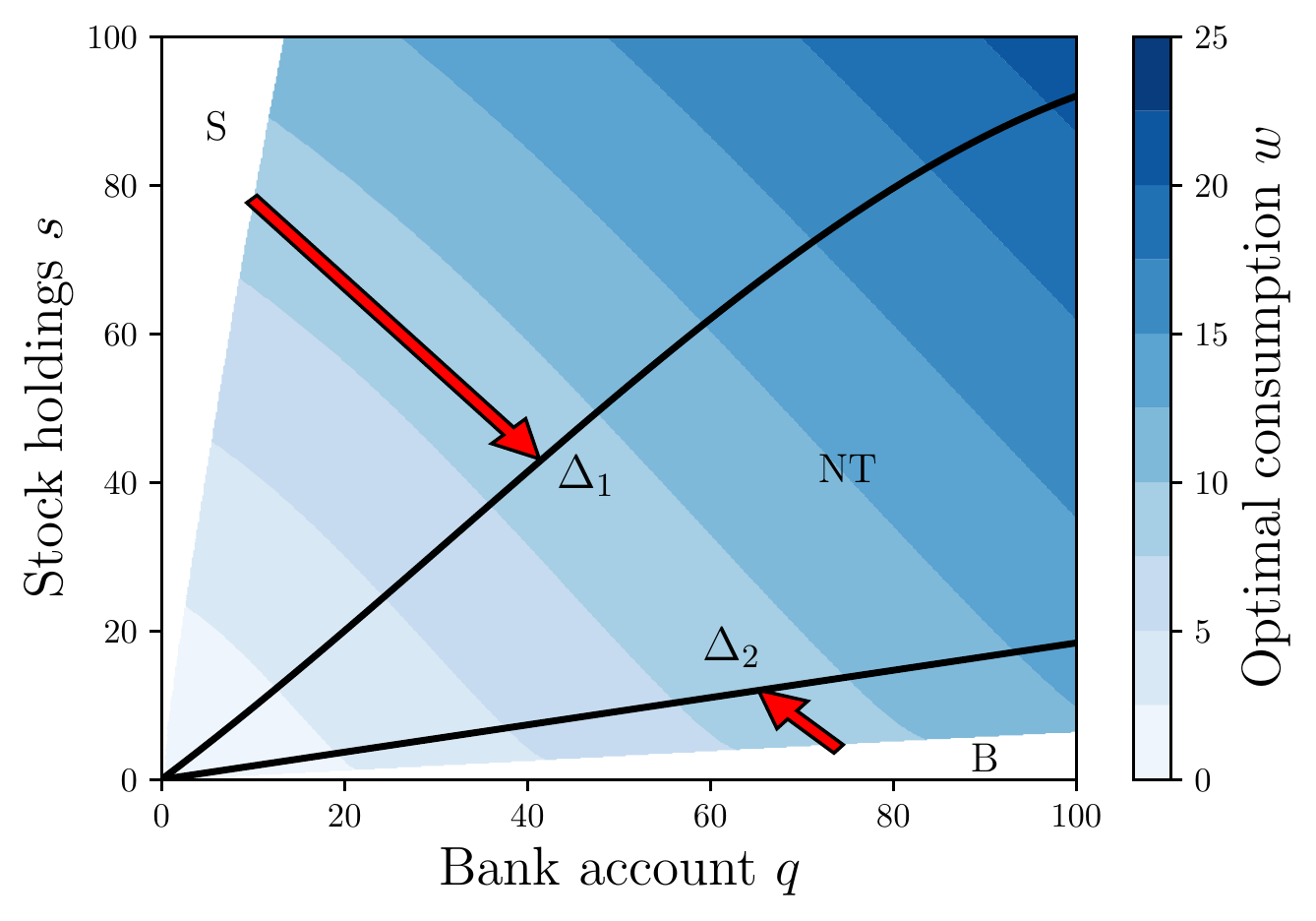}
    }
    \hfill{}
    \subfloat[Value at the initial time $t=0$]{
        \includegraphics[width=3.25in]{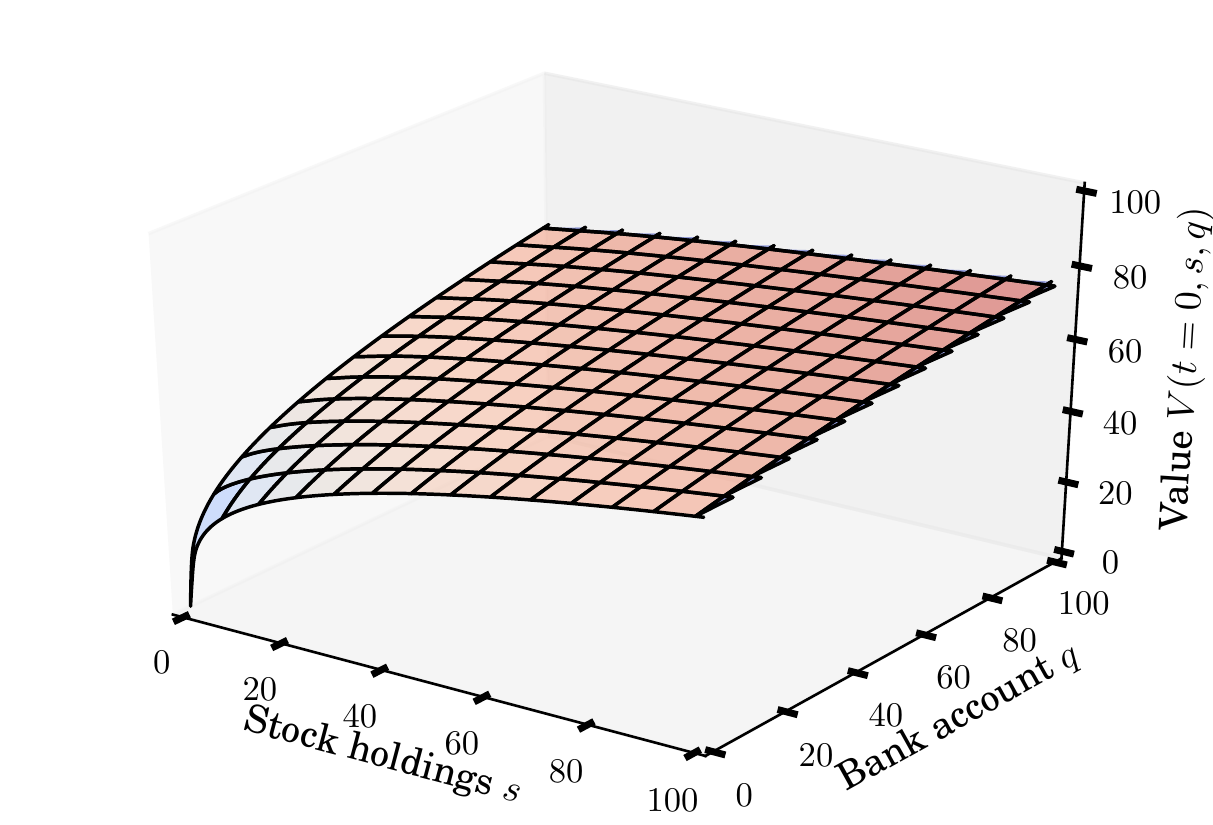}
    }
    \caption{Numerical solution of optimal consumption problem}
    \label{fig:results_consumption_value}
\end{figure}

Consider the optimal control in \cref{fig:results_consumption_control}.
As in \cite{MR1976512}, three regions are observed in an optimal control: the \emph{buy} (B), \emph{sell} (S), and \emph{no transaction} (NT) regions.
In the B and S regions, the controller performs an impulse to jump to the closest of the two lines marked $\Delta_1$ and $\Delta_2$.
In NT, the controller consumes capital.

\section{Guaranteed minimum withdrawal benefits (GMWBs)} 
\label{sec:results_gmwb}

A \variableAnnuitiesDescription

A particularly popular variable annuity, the \gls*{GMWB}, was first studied as a singular control problem in \cite{MR2454673}.
Shortly thereafter, it was recast as an impulse control problem in \cite{MR2407322} and solved numerically therein.
Roughly speaking, impulse control is more general than singular control:
when the fixed cost of an impulse approaches zero, a solution of an impulse control problem tends to that of a corresponding singular control problem; see, e.g., \cite{MR701523,MR709283,MR1709328}.

A \gls*{GMWB} is composed of a risky investment $S$ and a guarantee account $Q$.
It is bootstrapped by an up-front payment $s_0$ to an insurer, placed in the risky investment (i.e., $S_0 = s_0$).
A \gls*{GMWB} promises to pay back at least the lump sum $s_0$, assuming that the holder does not withdraw above a certain contract-specified rate.
This is captured by setting $Q_0 = s_0$ and reducing both the risky investment and guarantee account on a dollar-for-dollar basis upon withdrawals.
The holder can continue to withdraw as long as the guarantee account remains positive.
In particular:
\begin{itemize}
    \item At all times, the holder picks the rate at which they withdraw capital continuously.
    \item The holder picks specific times at which to withdraw a finite amount subject to a penalty imposed by the insurer.
\end{itemize}

Let $\textcolor{ctrlcolor}{w_t}$ denote the withdrawal rate at time $t$.
Let $\textcolor{ctrlcolor}{\xi_1} \leq \textcolor{ctrlcolor}{\xi_2} \leq \cdots \leq \infty$ be the times at which the holder withdraws amounts $\textcolor{ctrlcolor}{z_1}, \textcolor{ctrlcolor}{z_2}, \ldots$
Subject to the above, the risky investment $S$ and guarantee account $Q$ evolve according to
    \begin{align}
        dS_t
        & = (
            (r - \eta ) S_t
            - \textcolor{ctrlcolor}{w_t} \boldsymbol{1}_{ \{S_t, Q_t > 0\} }
        ) dt + \sigma S_t d\gls*{wiener}_t
        & \text{for } t \neq \xi_1, \xi_2, \ldots
        \nonumber \\
        dQ_t
        & = - \textcolor{ctrlcolor}{w_t} \boldsymbol{1}_{\{Q_t > 0\}} dt
        & \text{for } t \neq \xi_1, \xi_2, \ldots
        \nonumber \\
        S_{\textcolor{ctrlcolor}{\xi_\ell}}
        & = \max \{
            S_{\textcolor{ctrlcolor}{\xi_\ell}-} - \textcolor{ctrlcolor}{z_\ell},
            0
        \}
        & \text{for } \ell = \phantom{\xi_1} \mathllap{1}, \phantom{\xi_2} \mathllap{2}, \ldots
        \nonumber \\
        Q_{\textcolor{ctrlcolor}{\xi_\ell}}
        & = Q_{\textcolor{ctrlcolor}{\xi_\ell}-} - \textcolor{ctrlcolor}{z_\ell}
        & \text{for } \ell = \phantom{\xi_1} \mathllap{1}, \phantom{\xi_2} \mathllap{2}, \ldots
        \label{eqn:results_gmwb_sde}%
    \end{align}
where $r$ and $\sigma$ are nonnegative constants representing the interest rate and volatility and $\eta \leq r$ corresponds to a proportional management fee paid to the insurer.
As is usually the case in risk-neutral options pricing, the ``real'' drift $\mu$ of the stock does not make an appearance in \cref{eqn:results_gmwb_sde}.

To prevent the holder from withdrawing too quickly, we impose the restriction $|\textcolor{ctrlcolor}{w_t}| \leq w^{\max}$.
We also assume the holder cannot deposit into or withdraw more than the value of the guarantee account, captured by the restriction $0 \leq \textcolor{ctrlcolor}{z_\ell} \leq Q_{\textcolor{ctrlcolor}{\xi_\ell}-}$.

Letting $x = (s, q)$ and $\textcolor{ctrlcolor}{\theta} = (\textcolor{ctrlcolor}{w}; \textcolor{ctrlcolor}{\xi_1}, \textcolor{ctrlcolor}{z_1}; \textcolor{ctrlcolor}{\xi_2}, \textcolor{ctrlcolor}{z_2}; \ldots)$ denote a control, the value of the \gls*{GMWB} is captured by the objective function
\begin{multline}
    J(t, x; \textcolor{ctrlcolor}{\theta}) = \mathbb{E} \biggl[
        \int_t^T
            e^{-r u} \textcolor{ctrlcolor}{w_u} \boldsymbol{1}_{\{ Q_u > 0 \}}
        du
        + \sum_{t \leq \textcolor{ctrlcolor}{\xi_\ell} \leq T}
            e^{-r \textcolor{ctrlcolor}{\xi_\ell}} \left(
                \left(1 - \kappa\right) \textcolor{ctrlcolor}{z_\ell} - c
            \right) \\
        + e^{-r T} \max \left\{
            S_T,
            \left(1 - \kappa\right) Q_T - c
        \right\}
        \biggm| (S_{t-}, Q_{t-}) = x
    \biggr]
    \label{eqn:results_gmwb_functional}
\end{multline}
where $c > 0$ and $0 \leq \kappa \leq 1$ are fixed and proportional penalties paid for withdrawing.
The term $\max \{ S_T, \left(1 - \kappa\right) Q_T - c \}$ corresponds to the holder receiving the greater of the risky investment or a full withdrawal of the guarantee account subject to any withdrawal penalties.

In \cite[Appendix G]{chen2008numerical}, standard dynamic programming arguments are used to transform the optimization problem $V(t, x) = \sup_{\textcolor{ctrlcolor}{\theta}} J(t, x; \textcolor{ctrlcolor}{\theta})$ into the equivalent HJBQVI\footnote{\cite{chen2008numerical} considers instead the function $V_0(t, x) = e^{r t} \sup_{\textcolor{ctrlcolor}{\theta}} J(t, x; \textcolor{ctrlcolor}{\theta})$. As such, the HJBQVI \cref{eqn:results_gmwb_hjbqvi} and the HJBQVI in \cite{chen2008numerical} are equivalent up to the change of variables $V_0 = e^{r t} V$.}
    \begin{align}
        \min \left\{
            -V_t - \sup_{\textcolor{ctrlcolor}{w} \in [0, w^{\max}]} \left\{
                    \mathcal{L}^{\textcolor{ctrlcolor}{w}} V
                    + e^{-r t} \textcolor{ctrlcolor}{w} \boldsymbol{1}_{\{q > 0\}}
            \right\},
            V - \mathcal{M}V
        \right\} & = 0 \text{ on } [0,T) \times [0,\infty)^2
        \nonumber \\
        \min \left\{
            V(T,\cdot) - g,
            V(T, \cdot) - \mathcal{M} V(T, \cdot)
        \right\}
        & = 0 \text{ on } [0,\infty)^2
        \label{eqn:results_gmwb_hjbqvi}%
    \end{align}
where
\begin{align}
    \mathcal{L}^{\textcolor{ctrlcolor}{w}} V(t, s, q)
    & = \frac{1}{2} \sigma^2 s^2 V_{ss}(t, s, q)
    + \left(
        \left(r - \eta\right) s
        - \textcolor{ctrlcolor}{w} \boldsymbol{1}_{\{s, q > 0\}}
    \right) V_s(t, s, q)
    - \textcolor{ctrlcolor}{w} \boldsymbol{1}_{\{q > 0\}} V_q(t, s, q)
    \nonumber \\
    \mathcal{M} V(t, s, q)
    & = \sup_{ \textcolor{ctrlcolor}{z} \in [0, q] } \left\{
        V(t, \max\left\{ s - \textcolor{ctrlcolor}{z}, 0 \right\}, q - \textcolor{ctrlcolor}{z})
        + e^{-r t} \left( \left( 1 - \kappa \right) \textcolor{ctrlcolor}{z} - c \right)
    \right\}
    \nonumber \\
    g(s, q)
    & = e^{-r T} \max \left\{
        s,
        \left(1 - \kappa\right) q - c
    \right\}.
    \label{eqn:results_gmwb_quantities}
\end{align}

\subsection{\Artificial{} boundary conditions}
\label{subsec:results_gmwb_artificial}

As usual, we do not solve the HJBQVI \cref{eqn:results_gmwb_hjbqvi} directly but rather truncate the domain and introduce \artificial{} boundary conditions.
In particular, letting $R_s$ and $R_q$ be large positive numbers, we work on the truncated domain $[0, T] \times [0, R_s] \times [0, R_q]$.

As remarked in \cite{MR2454673,MR2407322}, an \artificial{} boundary condition is not needed at $q = R_q$. This is due to the fact that the coefficient of $V_q$ in \cref{eqn:results_gmwb_quantities} is nonpositive, and hence the characteristics of the PDE are outgoing in the $q$ direction at $q = R_q$.

To derive an appropriate \artificial{} boundary condition at $s = R_s$, we assume as per \cite{MR2454673,MR2407322} that $V$ is asymptotically linear in $s$.
That is, we assume the existence of a smooth function $A$ such that
\[
    V(t, s, q) \sim A(t, q) s
    \textspace \text{as} \textspace s \rightarrow \infty
\]
which implies
\[
    V_t(t, s, q) \sim A_t(t, q) s
    \text{,} \textspace
    V_s(t, s, q) \sim A(t, q)
    \text{,} \textspace \text{and} \textspace
    V_{ss}(t, s, q) \sim 0
    \textspace \text{as} \textspace s \rightarrow \infty.
\]
Substituting the above into $\mathcal{L}^{\textcolor{ctrlcolor}{w}} V(t, s, q)$ defined in \cref{eqn:results_gmwb_quantities}, we find that
\[
    \mathcal{L}^{\textcolor{ctrlcolor}{w}} V(t, s, q)
    = \left(
        \left(r - \eta\right) s
        - \textcolor{ctrlcolor}{w} \boldsymbol{1}_{\{s, q > 0\}}
    \right) V(t, s, q) / s
    - \textcolor{ctrlcolor}{w} \boldsymbol{1}_{\{q > 0\}} V_q(t, s, q)
    \textspace \text{as} \textspace s \rightarrow \infty.
\]
This leads us to define the operator $\mathcal{L}_R^{\textcolor{ctrlcolor}{w}}$ by
\begin{multline}
    \mathcal{L}_R^{\textcolor{ctrlcolor}{w}} V(t, s, q)
    = - \textcolor{ctrlcolor}{w} \boldsymbol{1}_{\{q > 0\}} V_q(t, s, q)
    \\
    + \begin{cases}
        \frac{1}{2} \sigma^2 s^2 V_{ss}(t, s, q)
        + \left(
            \left(r - \eta\right) s
            - \textcolor{ctrlcolor}{w} \boldsymbol{1}_{\{s, q > 0\}}
        \right) V_s(t, s, q)
        & \text{if } s < R_s
        \\
        \left(
            \left(r - \eta\right) s
            - \textcolor{ctrlcolor}{w} \boldsymbol{1}_{\{s, q > 0\}}
        \right) V(t, s, q) / s
        & \text{if } s = R_s.
    \end{cases}
    \label{eqn:results_gmwb_artificial_operator}
\end{multline}
For convenience, we also define the sets
\[
    \Omega = [0,T) \times [0, R_s] \times [0, R_q]
    \textspace \text{and} \textspace
    \partial_T^+ \Omega = \overline{\Omega} \setminus \Omega = \{ T \} \times [0, R_s] \times [0, R_q].
\]
Instead of the HJBQVI \cref{eqn:results_gmwb_hjbqvi}, we solve numerically the ``truncated HJBQVI''
\begin{align*}
    \min \left\{
        -V_t - \sup_{\textcolor{ctrlcolor}{w} \in [0, w^{\max}]} \left\{
            \mathcal{L}_R^{\textcolor{ctrlcolor}{w}} V
            + e^{-r t} \textcolor{ctrlcolor}{w} \boldsymbol{1}_{\{q > 0\}}
        \right\},
        V - \mathcal{M}V
    \right\} & = 0 \text{ on } \Omega
    \nonumber \\
    \min \left\{
        V - g,
        V - \mathcal{M} V
    \right\}
    & = 0 \text{ on } \partial_T^+ \Omega
\end{align*}
where $\mathcal{L}_R^{\textcolor{ctrlcolor}{w}}$ is given by \cref{eqn:results_gmwb_artificial_operator} and $g$ and $\mathcal{M}$ are given by \cref{eqn:results_gmwb_quantities}.

\begin{remark}
    The particular form of the intervention operator $\mathcal{M}$ in the \gls*{GMWB} problem does not satisfy assumption \cref{enum:convergence_negative_cost} of \cref{chap:convergence}, so that the stability arguments of that chapter do not apply.
    However, the stability of the various schemes in this thesis applied to the \gls*{GMWB} problem can be handled by minor modifications of the arguments in \cite[Lemma 5.1]{MR2407322} for the explicit-impulse scheme and \cite[Lemma 5.1]{MR2875254} for the direct control and penalty schemes.
\end{remark}

\subsection{Convergence tests}

\begin{table}
    \centering \footnotesize
    \begin{tabular}{rrr}
        \toprule
        \multicolumn{2}{r}{Parameter} & Value \tabularnewline
        \midrule
        Maximum withdrawal rate & $w^{\max}$ & 10 \tabularnewline
        Risk-free interest rate & $r$ & 0.05 \tabularnewline
        Management fee & $\eta$ & 0 \tabularnewline
        Volatility & $\sigma$ & 0.3 \tabularnewline
        Excess withdrawal penalty & $\kappa$ & 0.1 \tabularnewline
        Fixed transaction cost & $c$ & $10^{-6}$ \tabularnewline
        Horizon & $T$ & 10 \tabularnewline
        Truncated domain boundaries & $(R_s, R_q)$ & $(1000,100)$ \tabularnewline
        \bottomrule
    \end{tabular}
    \caption{Parameters for GMWB problem from \cite{MR2407322}}
    \label{tab:results_gmwb_parameters}
\end{table}

\begin{table}
    \centering \footnotesize
    \begin{tabular}{cccccc}
        \toprule
        $\gls*{meshing_parameter}$ & Timesteps & $s$ points & $q$ points & $\textcolor{ctrlcolor}{w}$ points & $\textcolor{ctrlcolor}{z}$ points \tabularnewline
        \midrule
        1 & 32 & 64 & 50 & 2 & 2\tabularnewline
        1/2 & 64 & 128 & 100 & 2 & 4\tabularnewline
        $\vdots$ & $\vdots$ & $\vdots$ & $\vdots$ & $\vdots$ & $\vdots$\tabularnewline
        \bottomrule
    \end{tabular}
    \caption{Numerical grid for optimal consumption problem}
    \label{tab:results_gmwb_grid}
\end{table}

\begin{table}
    \centering \footnotesize
        \subfloat[Direct control scheme]{
        \begin{tabular}{cccccc}
            \toprule
            $\gls*{meshing_parameter}$ & $V^{\gls*{meshing_parameter}}(t=0, s=100, q=100)$ & Avg. policy its. & Avg. BiCGSTAB its. & Ratio & Norm. time\tabularnewline
            \midrule
            1    & 107.683417498 & 3.47 & 1.47 &      & 1.71e+01 \tabularnewline
            1/2  & 107.706787394 & 4.25 & 1.64 &      & 2.03e+02 \tabularnewline
            1/4  & 107.718780318 & 4.34 & 1.85 & 1.95 & 2.60e+03 \tabularnewline
            1/8  & 107.725782831 & 4.43 & 2.22 & 1.71 & 3.46e+04 \tabularnewline
            1/16 & 107.729641357 & 4.31 & 2.71 & 1.81 & 4.75e+05 \tabularnewline
            1/32 & 107.731755456 & 4.15 & 3.40 & 1.83 & 7.55e+06 \tabularnewline
            \bottomrule
        \end{tabular}
    }

    \subfloat[Penalty scheme]{
        \begin{tabular}{cccccc}
            \toprule
            $\gls*{meshing_parameter}$ & $V^{\gls*{meshing_parameter}}(t=0, s=100, q=100)$ & Avg. policy its. & Avg. BiCGSTAB its. & Ratio & Norm. time\tabularnewline
            \midrule
            1    & 107.682425551 & 3.47 & 1.58 &      & 1.80e+01 \tabularnewline
            1/2  & 107.706388904 & 4.08 & 1.65 &      & 2.06e+02 \tabularnewline
            1/4  & 107.718700668 & 3.95 & 1.76 & 1.95 & 2.45e+03 \tabularnewline
            1/8  & 107.725763667 & 3.98 & 1.97 & 1.74 & 3.22e+04 \tabularnewline
            1/16 & 107.729637421 & 3.71 & 2.39 & 1.82 & 4.34e+05 \tabularnewline
            1/32 & 107.731754559 & 3.32 & 3.01 & 1.83 & 6.62e+06 \tabularnewline
            \bottomrule
        \end{tabular}
    }

    \subfloat[Explicit-impulse scheme]{
        \begin{tabular}{ccccc}
            \toprule
            $\gls*{meshing_parameter}$ & $V^{\gls*{meshing_parameter}}(t=0, s=100, q=100)$ & Avg. BiCGSTAB its. & Ratio & Norm. time\tabularnewline
            \midrule
            1    & 107.423506170 & 1.00 &      & 1.00e+00 \tabularnewline
            1/2  & 107.684431768 & 1.00 &      & 1.02e+01 \tabularnewline
            1/4  & 107.708405901 & 1.00 & 10.9 & 1.39e+02 \tabularnewline
            1/8  & 107.722569027 & 1.00 & 1.70 & 2.03e+03 \tabularnewline
            1/16 & 107.730146084 & 1.00 & 1.87 & 3.31e+04 \tabularnewline
            1/32 & 107.732241020 & 1.98 & 3.62 & 6.10e+05 \tabularnewline
            1/64 & 107.733372937 & 2.90 & 1.85 & 1.13e+07 \tabularnewline
            \bottomrule
        \end{tabular}
    }
    \caption{Convergence tests for GMWB problem}
    \label{tab:results_gmwb_convergence}
\end{table}

\begin{figure}
    \subfloat[Optimal control at the initial time $t=0$ with parameter $\eta=0.03126$ chosen as in \cite{MR2407322}\label{fig:results_gmwb_control}]{
        \includegraphics[width=3.25in]{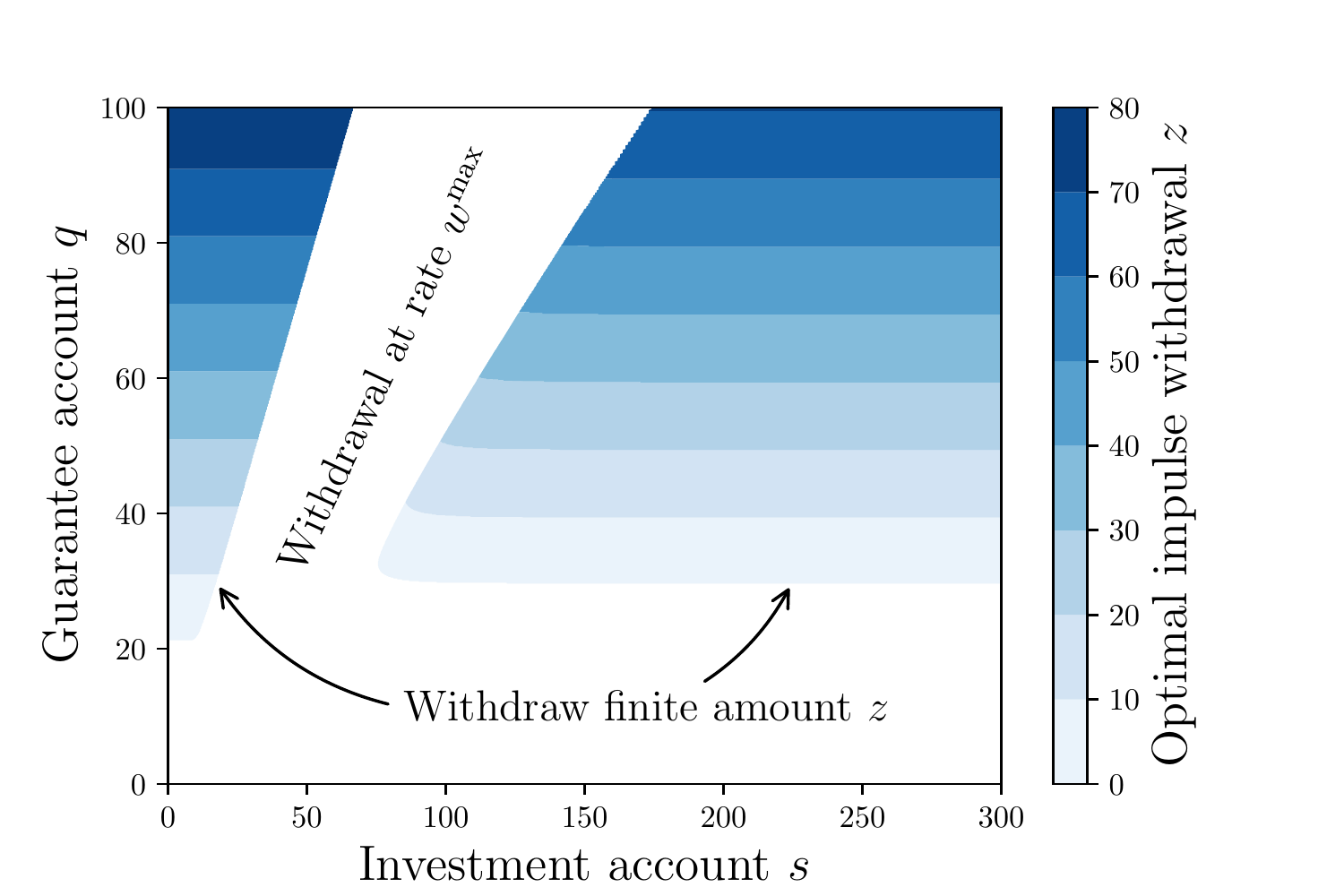}
    }
    \hfill{}
    \subfloat[Value at the initial time $t=0$]{
        \includegraphics[width=3.25in]{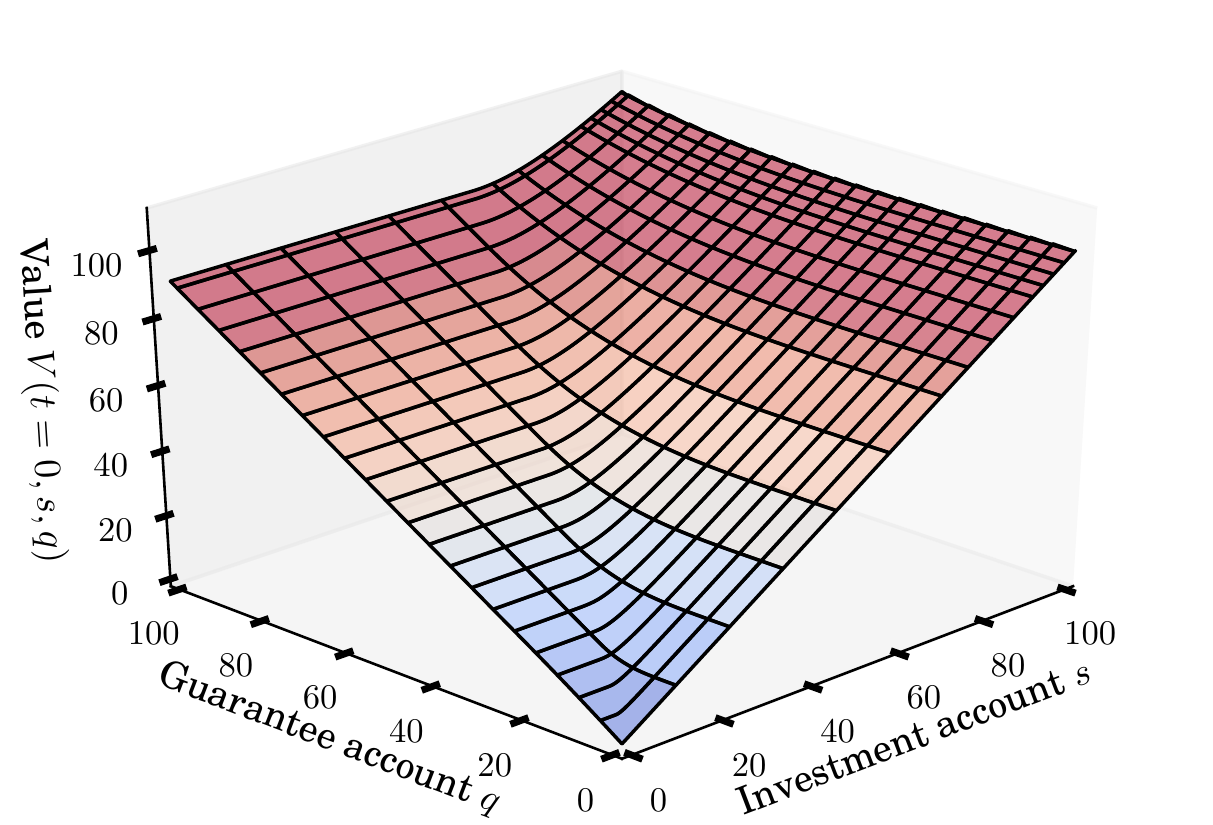}
    }
    \caption{Numerical solution of GMWB problem}
\end{figure}

For the numerical tests in this section, we use the parameters in \cref{tab:results_gmwb_parameters}.
We point out that it is sufficient to take $W^{\gls*{meshing_parameter}} = \{ 0, w^{\max} \}$ as the discretization of the control set $[0, w^{\max}]$ since the argument appearing in the supremum of the HJBQVI
is linear in $\textcolor{ctrlcolor}{w}$. 
\cref{tab:results_gmwb_grid} reports the size of the numerical grid.
Convergence tests are shown in \cref{tab:results_gmwb_convergence}.
Our results agree closely with \cite{MR2407322}, which computes the value as 107.7313.

As in the FEX rate problem, the direct control and penalty schemes are nearly identical in performance and accuracy.
The explicit-impulse scheme, which seems to produce comparably accurate solutions for the GMWB problem, executes much faster than its counterparts.

\subsection{Optimal control}

\cref{fig:results_gmwb_control} shows an optimal control for a \gls*{GMWB}.
Three regions are witnessed.
In two of the three regions, the contract holder performs an impulse to withdraw a finite amount instantaneously.
In the remaining region, the controller withdraws at the constant rate $w^{\max}$.
A more detailed explanation of these regions is given in \cite{MR2407322}.

\section{Infinite horizon}

So far, we have only considered finite horizon problems.
As a proof of concept, we solve in this section the optimal consumption problem over an infinite horizon by the penalty scheme \cref{eqn:schemes_penalty_infinite_horizon}.
This problem was previously considered using iterated optimal stopping in \cite{MR1976512}.
The setting is identical to \cref{results:consumption}, with the exception that $T = \infty$.
In this case, the corresponding HJBQVI is \cite[Eq. (2.31)]{MR1976512}
\[
    \min \left\{
        \beta V - \sup_{\textcolor{ctrlcolor}{w} \in [0, w^{\max}]} \left\{
                \mathcal{L}^{\textcolor{ctrlcolor}{w}} V
                + \frac{\textcolor{ctrlcolor}{w}^\gamma}{\gamma} \boldsymbol{1}_{\{ q > 0\}}
        \right\},
        V - \mathcal{M}V
    \right\} = 0 \text{ on } [0,T) \times [0,\infty)^2
\]
where $\mathcal{L}^{\textcolor{ctrlcolor}{w}}$ and $\mathcal{M}$ are defined in \cref{eqn:results_consumption_quantities}.

Similarly to \cref{subsec:results_consumption_artificial}, we can introduce \artificial{} boundary conditions into the above HJBQVI to make it computationally tractable (we omit the details).
Convergence results are shown in \cref{tab:results_infinite_horizon_consumption_convergence} with times normalized to the fastest solve.
We note that the computed values agree roughly with those shown in \cite[Fig. 5.2]{MR1976512}.
The results are obtained by invoking policy iteration exactly once (there are no timesteps in the infinite horizon case) to solve \cref{eqn:schemes_penalty_infinite_horizon}.
The total number of policy iterations and BiCGSTAB iterations per timestep are also reported.

\begin{table}
    \centering \footnotesize
    \begin{tabular}{cccccc}
        \toprule
        $\gls*{meshing_parameter}$ & Value $V^{\gls*{meshing_parameter}}(t=0, s_0, q_0)$ & Tot. policy its. & Tot. BiCGSTAB its. & Ratio & Norm. time \tabularnewline
        \midrule
        1    & 56.2664380074 & 8 &  21 &      & 1.00e+00 \tabularnewline
        1/2  & 59.1841989658 & 9 &  36 &      & 4.41e+00 \tabularnewline
        1/4  & 59.8299121028 & 9 &  58 & 4.52 & 3.19e+01 \tabularnewline
        1/8  & 60.0469132389 & 9 &  82 & 2.98 & 2.43e+02 \tabularnewline
        1/16 & 60.1330774909 & 9 & 121 & 2.52 & 1.93e+03 \tabularnewline
        1/32 & 60.1706514583 & 8 & 185 & 2.29 & 1.40e+04 \tabularnewline
        1/64 & 60.1881729190 & 8 & 328 & 2.14 & 1.30e+05 \tabularnewline
        \bottomrule
    \end{tabular}
    \caption{Convergence tests for infinite horizon optimal consumption problem}
    \label{tab:results_infinite_horizon_consumption_convergence}
\end{table}

\section{Summary}

In this chapter, we applied the direct control, penalty, and explicit-impulse schemes to three classical impulse control problems from finance.
As mentioned in the introduction, some of these problems had not been previously considered numerically.

The explicit-impulse scheme, which requires no policy iteration whatsoever, is much faster than its counterparts.
However, as pointed out in \cref{chap:schemes}, the explicit-impulse scheme can only be used if the horizon is finite and the second derivative coefficient is independent of the control.
Moreover, our results indicate that the direct control and penalty schemes produce nearly identical results and often require roughly the same amount of computation.
In the specific case of the optimal consumption problem, the penalty scheme even outperforms the direct control scheme, requiring far fewer policy iterations per timestep.
Due to the additional effort required to ensure convergence of the direct control scheme as discussed in the introduction of this chapter, we advise against using this scheme altogether.

\iftoggle{thesis}{%
    \input{chapters/bang}
}{}
\iftoggle{thesis}{%
\setcounter{chapter}{6}
}{
\setcounter{chapter}{5}
}
\chapter{Summary and future work}
\label{chap:summary}

This thesis explored the numerical aspects of impulse control.
\iftoggle{thesis}{%
The bulk of the thesis focused on HJBQVIs arising from impulse control, for which various schemes (some new) were considered.
}{
In particular, we considered various schemes for HJBQVIs arising from impulse control problems.
}
Some of these schemes took the form of nonlinear matrix equations (a.k.a. Bellman problems) suitable for policy iteration.
To ensure that policy iteration could be applied to these schemes required existing and new results from the theory of \gls*{wcdd} matrices.
To prove the convergence of these schemes to the viscosity solution, classical results of Barles and Souganidis were extended to the nonlocal setting and applied to the schemes, which were shown to be monotone, stable, and nonlocally consistent.
The schemes were then applied to numerically compute solutions of various classical problems from finance, from which it was determined that the direct control scheme (proposed in \cite{MR1976512}) is dominated by the penalty scheme.

\iftoggle{thesis}{%
The remainder of the thesis focused on PDEs arising from impulse control with fixed impulse times.
In this case, sufficient conditions were found under which it is possible to perform a control reduction, thereby leading to faster numerical methods.
These results were applied to price \gls*{GLWB} contracts.
}{}

In terms of future work, there are several avenues of research that remain to be explored:
\begin{itemize}
    \item
        In \cref{subsec:convergence_high_dimension_hjbqvi}, we mentioned that extending the direct control, penalty, and explicit-impulse schemes to higher dimensions is a nontrivial matter if cross-derivatives are present.
        One possible avenue for future research is to resolve this by incorporating wide-stencils \cite{MR2399429,MR3649430,chen2016monotone} or the interpolation techniques in \cite{MR3042570}.
    \item
        In \cref{app:truncated}, we introduced a truncated approximation to the HJBQVI.
        However, we gave no rigorous bounds on the error between the solutions of the truncated HJBQVI and the original HJBQVI \cref{eqn:introduction_hjbqvi}.
        Obtaining such a result would rigorously justify the use of the truncated approximation, and may also suggest how large to pick our truncated domain in practice.
    \iftoggle{thesis}{%
        \item
            In \cref{chap:bang}, we were able to optimize the cost of approximating the intervention operator in a problem involving fixed impulse times.
            It would be interesting to extend this analysis to the setting of the HJBQVI.
            Such an extension would undoubtedly require us to establish convexity and monotonicity properties of solutions to HJBQVIs.
    }{}
    \item
        In \cite{MR2642892}, the authors obtain an alternate representation of the HJBQVI by backwards stochastic differential equations (BSDEs).
        Very recently, researchers have used deep neural networks to implement fast BSDE solvers \cite{han2017deep}.
        Though such methods are not provably convergent, they are extremely fast and can handle high-dimensional (e.g., 100 dimensions) problems.
        One possible avenue for future research is to combine the results of \cite{MR2642892} and \cite{han2017deep} to obtain a fast high-dimensional HJBQVI solver.
    \item
        In the series of works \cite{MR2310697,MR2295621,MR2271747,MR2269875}, Lasry and Lions created the study of mean field games, concerning strategic decision making in very large populations of small interacting agents.
        At the time of writing, no analytic or numeric work exists incorporating impulses into the mean field game setting.
        This is a promising avenue for future research.
\end{itemize}


\appendix

\chapter*{APPENDICES}
\addcontentsline{toc}{chapter}{APPENDICES}

\chapter{Supplementary material for Chapter 3}
\label{app:matrix_direct_control}

In this appendix, we take the setting of \cref{exa:matrix_direct_control} in order to prove that the assumptions of \cref{thm:matrix_restriction} are satisfied.
We first establish the \gls*{wcdd} condition:

\begin{lemma}
    $A(P)$ is \gls*{wcdd} for each $P \in \mathcal{P}^\prime$.
\end{lemma}

\begin{proof}
    Let $P$ be an arbitrary control in $\mathcal{P}^\prime$.
    It is sufficient to show that for each row $i_1$ of $A(P)$ that is not \gls*{sdd}, we can find a walk $i_1 \rightarrow \cdots \rightarrow i_k$ to an \gls*{sdd} row $i_k$.
    As such, let $i_1$ be a row that is not \gls*{sdd}
    Since the $i$-th row of $A(P)$ is not \gls*{sdd} if and only if $\textcolor{ctrlcolor}{d_i} = 1$ (recall \cref{lem:matrix_sdd_iff}), it follows that $\textcolor{ctrlcolor}{d_{i_1}} = 1$.
    
    By symmetry, it is sufficient to consider the case of $i_1 < M/2$.
    In this case, $\textcolor{ctrlcolor}{z_{i_1}} = x_{i_2} - x_{i_1}$ for some row $i_1 < i_2 \leq M / 2$ by \cref{eqn:matrix_direct_control_filter}.
    If $i_2$ is an \gls*{sdd} row, then $i_1 \rightarrow i_2$ is a walk to an \gls*{sdd} row.
    Otherwise, we can repeat this procedure to produce $i_3, i_4, \ldots$ until arriving at an \gls*{sdd} row $i_k$.
    In this case, $i_1 \rightarrow \cdots \rightarrow i_k$ is a walk to an \gls*{sdd} row.
    We are guaranteed that this procedure terminates since row $M/2$ is \gls*{sdd} due to $\textcolor{ctrlcolor}{d_{M/2}} = 0$ (see \cref{eqn:matrix_direct_control_filter}).
\end{proof}

We now establish \cref{eqn:matrix_invariance}.
The proof, albeit elementary, is somewhat arduous.
As such, we simply sketch the ideas below.

\begin{proof}[Proof sketch for \cref{eqn:matrix_invariance}]
    Fix $n$ and suppose that $V^{n-1} = (V_0^{n-1}, \ldots, V_M^{n-1})^\intercal$ satisfies $V_{M/2-i}^n = V_{M/2+i}^n$ for all $i$.
    Now, let $U = (U_0, \ldots, U_M)^\intercal$ be the unique vector satisfying $H(U; \mathcal{P}^\prime) = 0$.
    By some straightforward yet arduous algebra, one can establish that $U_{M/2-i} = U_{M/2+i}$ for all $i$ and
    \[
        \cdots \leq U_{M/2-2} \leq U_{M/2-1} \leq U_{M/2} \geq U_{M/2+1} \geq U_{M/2+2} \geq \cdots
    \]
    Using these two facts, it is straightforward to establish \cref{eqn:matrix_invariance}.
    Now, by \cref{thm:matrix_restriction}, it follows that $U$ is the unique vector satisfying $H(U; \mathcal{P}) = 0$, and hence $V^n = U$.
    
    Now, since $V^0$ is the zero vector (recall that the terminal condition in \cref{exa:introduction_fex} is $g(x) = 0$), it trivially satisfies $V_{M/2-i}^0 = V_{M/2+i}^0$ for all $i$.
    Therefore, \cref{eqn:matrix_invariance} is established (for all $n$) by induction.
\end{proof}
\chapter{Barles-Souganidis framework and the HJBQVI}
\label{app:barles}

In this appendix, we discuss why it is not possible to na\"{i}vely apply the Barles-Souganidis framework \cite{MR1115933} to prove the convergence of the direct control, penalty, and explicit-impulse schemes for the HJBQVI.
We recommend reading this appendix after \cref{sec:convergence_extension}.

Adapting the notation to match the notation of this thesis, recall that in \cite{MR1115933}, the authors consider numerical schemes of the form
\begin{equation}
    S(
        \gls*{meshing_parameter},
        x,
        V^{\gls*{meshing_parameter}}
    ) = 0 \textspace \text{for } x \in \overline{\Omega}.
    \label{eqn:barles_scheme}
\end{equation}
Compared to our notion of numerical scheme \cref{eqn:convergence_scheme}, the above does not have an explicit treatment of the nonlocal operator $\mathcal{I}$.

Consider now the following alternate definition of viscosity solution for \cref{eqn:convergence_pde}.
\begin{definition}
    An upper (resp. lower) semicontinuous function $V \colon \overline{\Omega} \rightarrow \mathbb{R}$ is a viscosity subsolution (resp. supersolution) of \cref{eqn:convergence_pde} if for all $\varphi \in C^2(\overline{\Omega})$ and $x \in \overline{\Omega}$ such that $V(x) - \varphi(x) = 0$ is a local maximum (resp. minimum) of $V - \varphi$, we have
    \begin{align}
        F_*(x, \varphi(x), D \varphi(x), D^2 \varphi(x), [\mathcal{I}\varphi](x))
        & \leq 0 \nonumber \\
        \text{ (resp. }
        F^*(x, \varphi(x), D \varphi(x), D^2 \varphi(x), [\mathcal{I}\varphi](x))
        & \geq 0
        \text{)}.
        \label{eqn:barles_viscosity_solution}
    \end{align}
    We say $V$ is a viscosity solution if it is both a subsolution and a supersolution.
    \label{def:barles_viscosity_solution}
\end{definition}
Compared to \cref{def:convergence_viscosity_solution}, which is the relevant definition for impulse control problems \cite{MR2568293,MR2109687}, the above alternate definition uses $\mathcal{I} \varphi$ instead of $\mathcal{I} V$ in \cref{eqn:barles_viscosity_solution}.
As such, \cref{def:convergence_viscosity_solution,def:barles_viscosity_solution} are \emph{not equivalent in general}.

In \cite{MR2182141} it is demonstrated that one can use schemes of the form \cref{eqn:barles_scheme} to approximate the solution of \cref{eqn:convergence_pde} when working under the alternate notion of solution specified by \cref{def:barles_viscosity_solution}.
In this case, the relevant notion of consistency is as follows \cite[Eq. (4.1)]{MR2182141}: a scheme $S$ is consistent if for each $\varphi \in C^2(\overline{\Omega})$ and $x \in \overline{\Omega}$
\[
    \liminf_{\substack{
        \gls*{meshing_parameter} \rightarrow 0 \\
        y \rightarrow x \\
        \xi \rightarrow 0
    }} S(
        \gls*{meshing_parameter},
        y,
        \varphi + \xi
    )
    \geq
    F_*(x, \varphi(x), D \varphi(x), D^2 \varphi(x), [\mathcal{I}\varphi](x))
\]
and
\[
    \limsup_{\substack{
        \gls*{meshing_parameter} \rightarrow 0 \\
        y \rightarrow x \\
        \xi \rightarrow 0
    }} S(
        \gls*{meshing_parameter},
        y,
        \varphi + \xi
    )
    \leq
    F^*(x, \varphi(x), D \varphi(x), D^2 \varphi(x), [\mathcal{I}\varphi](x)).
\]
Compared to our notion of nonlocal consistency given by \cref{eqn:convergence_subconsistency,eqn:convergence_superconsistency}, the above does not require the use of half-relaxed limits.
At least intuitively, half-relaxed limits are not needed in this case since the alternate notion of viscosity solution specified by \cref{def:barles_viscosity_solution} uses $\mathcal{I} \varphi$ instead of $\mathcal{I} V$ in \cref{eqn:barles_viscosity_solution}, as discussed in the previous paragraph.

For the HJBQVI, we have the following result, which appears in \cite[Theorem 3.1]{MR2486085}, \cite[Section 4]{MR2735526}, and \cite[Proposition 1.2]{MR784578}.

\begin{proposition}
    \label{prop:barles_uniform_continuity}
    Let $F$ be given by \cref{eqn:convergence_hjbqvi} and $\mathcal{I} = \mathcal{M}$, corresponding to the HJBQVI.
    Let $V \colon \overline{\Omega} \rightarrow \mathbb{R}$ be {\bf uniformly continuous}.
    Then, $V$ is a subsolution (resp. supersolution) in the sense of \cref{def:convergence_viscosity_solution} if and only if it is a subsolution (resp. supersolution) in the sense of \cref{def:barles_viscosity_solution}.
\end{proposition}

The need for uniform continuity is summarized succinctly in \cite{MR2486085}:

\begin{quote}
    \emph{``However, note that the operator $\mathcal{M}$ is nonlocal; i.e., $\mathcal{M} \varphi (t, x)$ is not determined by values of $\varphi$ in a neighborhood of $(t, x)$, and $\mathcal{M} \varphi (t, x)$ might be very small if $\varphi$ is small away from $(t, x)$.
    Therefore, one has no control over $\mathcal{M} \varphi (t, x)$ by simply requiring that $u - \varphi$ have a local maximum (resp. minimum) at $(t, x)$.''}
\end{quote}

Since the Barles-Souganidis framework operates by constructing a subsolution and supersolution pair that are only {\bf semicontinuous} (see the details of the proof in \cite[Pg. 276]{MR1115933}), the continuity requirement in \cref{prop:barles_uniform_continuity} is too strong for us to apply the Barles-Souganidis framework to obtain a solution of the HJBQVI in the sense of \cref{def:convergence_viscosity_solution}.
It is this theoretical issue that ultimately led us to the notion of nonlocal consistency as a means to obtain the provable convergence of our schemes.

    It is worthwhile to mention that one way to sidestep the above issue is by using a technique called iterated optimal stopping, in which a solution of the HJBQVI is obtained by considering a \emph{sequence of local PDEs}.
    This approach is described and analyzed in \cite{MR2314777} for the case of infinite horizon (i.e., $T = \infty$) HJBQVIs such as \cref{eqn:schemes_infinite_horizon}.
    However, it is well-known that when extended to the finite horizon (i.e., $T < \infty$) setting, iterated optimal stopping has a high space complexity and prohibitively slow convergence rate \cite{MR3150265}, making it unsuitable for the problems considered in this thesis.

\chapter{Truncated domain approximation}
\label{app:truncated}

As discussed in \cref{chap:convergence} (see, in particular, the text below \cref{eqn:convergence_grid}), computing the numerical solution of the HJBQVI as posed on an unbounded domain is computationally intractable.
In light of this, we restrict the HJBQVI to a truncated domain, introducing ``\artificial{}'' boundary conditions posed on the boundary of that domain.

In Subsection \ref{app:truncated_hjbqvi}, we introduce the truncated HJBQVI.
In Subsection \ref{app:truncated_consistency}, we modify the 
convergence arguments of \cref{chap:convergence} to ensure convergence to the solution of the truncated HJBQVI.

\section{Truncated HJBQVI}
\label{app:truncated_hjbqvi}

The truncated HJBQVI is
\begin{subequations}
    \begin{align}
        \min \left\{
            -V_t - \sup_{\textcolor{ctrlcolor}{w} \in W} \left\{
                \frac{1}{2} b(\cdot, \textcolor{ctrlcolor}{w})^2 V_{xx}
                + a(\cdot, \textcolor{ctrlcolor}{w}) V_x
                + f(\cdot, \textcolor{ctrlcolor}{w})
            \right\},
            V - \mathcal{M}V
        \right\} & = 0 \text{ on } \Omega
        \label{eqn:truncated_hjbqvi_interior} \\
        \min \left\{
            -V_t - \sup_{\textcolor{ctrlcolor}{w} \in W} \left\{
                f(\cdot, \textcolor{ctrlcolor}{w})
            \right\},
            V - \mathcal{M}V
        \right\} & = 0 \text{ on } \partial_R^+ \Omega
        \label{eqn:truncated_hjbqvi_spatial_boundary} \\
        \min \left \{
            V - g,
            V - \mathcal{M} V
        \right \}
        & = 0 \text{ on } \partial_T^+ \Omega
        \label{eqn:truncated_hjbqvi_boundary}
    \end{align}
    \label{eqn:truncated_hjbqvi}%
\end{subequations}
where 
\[
    \Omega = [0,T) \times (-R, R)
    \text{,} \textspace
    \partial_R^+ \Omega = [0, T) \times \{ -R, R \}
    \text{,} \textspace \text{and} \textspace
    \partial_T^+ \Omega = \{ T \} \times [-R, R].
\]

\begin{remark}
    \cref{eqn:truncated_hjbqvi_spatial_boundary} corresponds to imposing artificial Neumann boundary conditions $V_{xx}(t, \pm R) = V_x(t, \pm R) = 0$ for all $t \in [0, T)$.
    The techniques in this appendix can be used to handle other boundary conditions, such as replacing \cref{eqn:truncated_hjbqvi_spatial_boundary} by
    \[
        \min \left\{
            V - g,
            V - \mathcal{M}V
        \right\} = 0
        \text{ on } \partial_R^+ \Omega.
    \]
    With the above choice of boundary condition, the truncated HJBQVI is related to maximizing, over all controls $\textcolor{ctrlcolor}{\theta}$, the objective function (compare with \cref{eqn:introduction_functional})
    \[
        J_R(t, x; \textcolor{ctrlcolor}{\theta}) = \mathbb{E} \left[
            \int_t^{\pi} f(u, X_u, \textcolor{ctrlcolor}{w_u}) du
            + \sum_{t \leq \textcolor{ctrlcolor}{\xi_\ell} \leq T} K(\textcolor{ctrlcolor}{\xi_\ell}, X_{\textcolor{ctrlcolor}{\xi_\ell}-}, \textcolor{ctrlcolor}{z_\ell})
            + g(X_{\pi})
            \middle | X_{t-} = x
        \right]
    \]
    where $\pi$ is the minimum of the terminal time $T$ and the first time the process $X$ ``escapes'' the interval $(-R, R)$ (\cref{fig:truncated_first_escape}), defined as
    \[
        \pi = \min \left \{ T, \pi_0 \right \}
        \textspace \text{where} \textspace
        \pi_0 = \inf \left \{
            s \geq t \colon |X_s| \geq R
        \right \}.
    \]
\end{remark}

\begin{figure}
    \centering
    \subfloat{
        \begin{tikzpicture}[scale=5]
            \pgfmathsetseed{48}
            \draw [thick, ->] (0,0) -- (1.25,0);
            \draw [thick, ->] (0,-0.35) -- (0,1);
            \drawArithmeticBrownianMotionPath{colA,thick}{A}{1/\gratio}{1/\gratio}{190}
            \node at (ArithmeticBrownianMotionEnd_A) {\textbullet};
            \node [above=of ArithmeticBrownianMotionEnd_A, yshift=-1cm] {$X_{\pi}$};
            \draw [dotted]
                ($(0,0)!(ArithmeticBrownianMotionEnd_A)!(0,2)$) -- 
                ($(1.25,0)!(ArithmeticBrownianMotionEnd_A)!(1.25,2)$)
            ;
            \node [xshift=-0.5cm] at ($(0,0)!(ArithmeticBrownianMotionEnd_A)!(0,2)$) {$R$};
            \node [yshift=-0.5cm] at ($(0,0)!(ArithmeticBrownianMotionEnd_A)!(2,0)$) {$\pi$};
            \node [yshift=-0.5cm] at (1.15,0) {$T$};
            \draw
                ($(0,-0.025)!(ArithmeticBrownianMotionEnd_A)!(2,-0.025)$) -- ($(0,0.025)!(ArithmeticBrownianMotionEnd_A)!(2,0.025)$)
                ($(-0.025,0)!(ArithmeticBrownianMotionEnd_A)!(-0.025,2)$) -- ($(0.025,0)!(ArithmeticBrownianMotionEnd_A)!(0.025,2)$)
                (1.15,-0.025) -- (1.15,0.025)
            ;
            \pgfmathsetseed{8}
            \draw [thick, ->] (1.75,0) -- (3,0);
            \draw [thick, ->] (1.75,-0.35) -- (1.75,1);
            \drawArithmeticBrownianMotionPath{colA,thick,xshift=1.75cm}{B}{1/\gratio}{1/\gratio}{190}
            \node at (ArithmeticBrownianMotionEnd_B) {$\circ$};
            \node at ($(0,0.95)!(ArithmeticBrownianMotionEnd_B)!(3,0.95)$) {\textbullet};
            \draw [dotted]
                ($(1.75,0)!(ArithmeticBrownianMotionEnd_A)!(1.75,2)$) -- 
                ($(3,0)!(ArithmeticBrownianMotionEnd_A)!(3,2)$)
            ;
            \draw [dashed] (ArithmeticBrownianMotionEnd_B) -- ($(0,0.95)!(ArithmeticBrownianMotionEnd_B)!(3,0.95)$);
            \node [xshift=0.6cm] at (ArithmeticBrownianMotionEnd_B) {$X_{\pi-}$};
            \node [xshift=0.5cm] at ($(0,0.95)!(ArithmeticBrownianMotionEnd_B)!(3,0.95)$) {$X_{\pi}$};
            \node [yshift=-0.5cm] at ($(0,0)!(ArithmeticBrownianMotionEnd_B)!(3,0)$) {$\pi$};
            \node [yshift=-0.5cm] at (2.9,0) {$T$};
            \draw
                ($(1.75,-0.025)!(ArithmeticBrownianMotionEnd_B)!(3,-0.025)$) -- ($(1.75,0.025)!(ArithmeticBrownianMotionEnd_B)!(3,0.025)$)
                ($(1.725,0)!(ArithmeticBrownianMotionEnd_A)!(1.725,1)$) --
                ($(1.775,0)!(ArithmeticBrownianMotionEnd_A)!(1.775,1)$)
                (2.9,-0.025) -- (2.9,0.025)
            ;
            \node [xshift=-0.5cm] at ($(1.75,0)!(ArithmeticBrownianMotionEnd_A)!(1.75,2)$) {$R$};
        \end{tikzpicture}
    }
    \caption{Possible realizations of the first escape time}
    \label{fig:truncated_first_escape}
\end{figure}
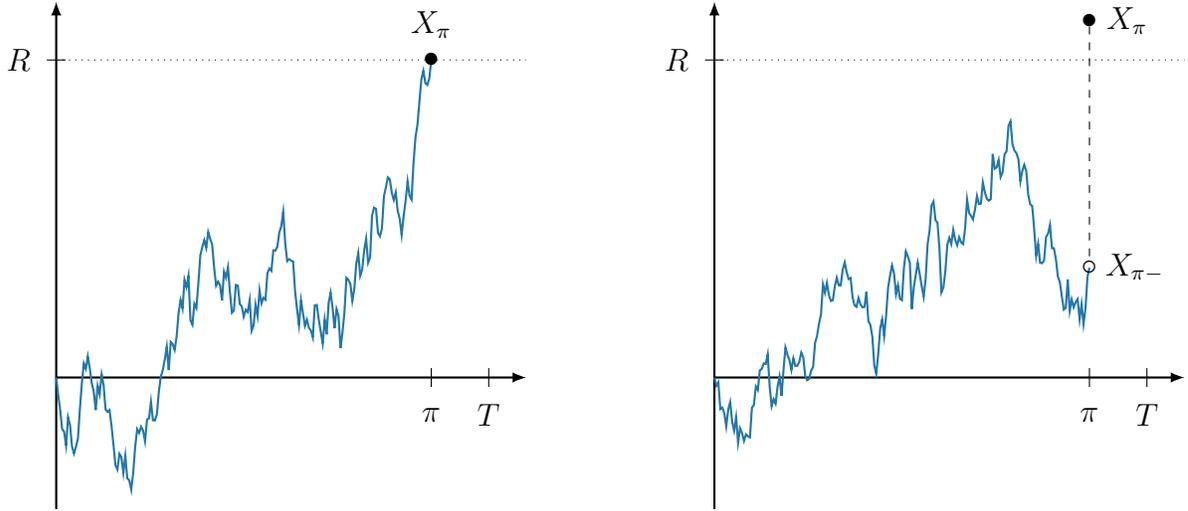

Throughout this appendix, we assume \cref{enum:convergence_continuity}\textendash \cref{enum:convergence_impulse_control_hausdorff} of \cref{chap:convergence} hold. 
To ensure that impulses do not leave the truncated domain, we also assume the following:

\begin{enumerate}[label=(H\arabic*),start=5]
    \item $(t, \Gamma(t, x, \textcolor{ctrlcolor}{z})) \in \overline{\Omega}$ for all $(t, x) \in \overline{\Omega}$ and $\textcolor{ctrlcolor}{z} \in Z(t, x)$.
\end{enumerate}

Lastly, we mention that in order to write \cref{eqn:truncated_hjbqvi} in the form \cref{eqn:convergence_pde}, we take $\mathcal{I} = \mathcal{M}$ and
\begin{multline}
    F((t, x), r, (q, p), A, \wideEll)
    \\
    = \begin{cases}
        \min \left\{
            - q - \sup_{\textcolor{ctrlcolor}{w} \in W} \left\{
                \frac{1}{2} b(t, x, \textcolor{ctrlcolor}{w})^2 A
                + a(t, x, \textcolor{ctrlcolor}{w}) p
                + f(t, x, \textcolor{ctrlcolor}{w})
            \right\}, r - \wideEll
        \right\}
        & \text{if } (t, x) \in \Omega \\
        \min \left\{
            - q - \sup_{\textcolor{ctrlcolor}{w} \in W} \left\{
                f(t, x, \textcolor{ctrlcolor}{w})
            \right\},
            r - \wideEll
        \right\}
        & \text{if } (t, x) \in \partial_R^+ \Omega \\
        \min \left\{
            r - g(x),
            r - \wideEll
        \right\}
        & \text{if } (t, x) \in \partial_T^+ \Omega
    \end{cases}
    \label{eqn:truncated_hjbqvi_operator}
\end{multline}
(compare with \cref{eqn:convergence_hjbqvi}).

\begin{remark}
    If, in addition to the requirements of \cref{thm:convergence_comparison}, $a(\cdot, \pm R, \cdot) = b(\cdot, \pm R, \cdot) = 0$, a trivial modification of the arguments in \cref{app:comparison} yield a comparison principle for the truncated HJBQVI.
\end{remark}

\section{Modifying the convergence arguments}
\label{app:truncated_consistency}

In this subsection, we discuss how to modify the convergence arguments in order to ensure convergence to the solution of the truncated HJBQVI \cref{eqn:truncated_hjbqvi}.

Instead of the uniform spacing condition $x_{i+1} - x_i = \Delta x$ and the condition on $M$ in \cref{eqn:convergence_grid}, we work under the assumption that $\{ x_0, \ldots, x_M \}$ is a partition of the interval $[-R, R]$ which satisfies (using Bachmann-Landau notation),
\begin{equation}
    x_{i+1} - x_i = \Theta(\gls*{meshing_parameter}).
    \label{eqn:truncated_spacing}
\end{equation}
Therefore, while the number of points on the spatial grid increases as $h \rightarrow 0$, the location of the boundary points $x_0 = -R$ and $x_M = R$ remains fixed.

The monotonicity and stability of the direct control, penalty, and explicit-impulse schemes are not affected by the modified grid described above.
In light of this, it is sufficient only to show that these schemes are nonlocally consistent (with respect to \cref{eqn:truncated_hjbqvi_operator}).

\begin{lemma}
    \label{lem:truncated_direct_control_consistency}
    The direct control scheme is nonlocally consistent (with respect to \cref{eqn:truncated_hjbqvi_operator}).
\end{lemma}

\begin{proof}

    Let $\Omega = [0,T) \times (-R, R)$ and $\varphi \in C^{1,2}(\overline{\Omega})$.
    Let $\varphi_i^n$ and $\varphi^n$, $(u^{\gls*{meshing_parameter}})_{\gls*{meshing_parameter} > 0}$, and $(\gls*{meshing_parameter}_m, s_m, y_m, \xi_m)_m$ be defined as in the proof of \cref{lem:convergence_direct_control_consistency}.
    By the definition of the scheme,
    \begin{equation}
        S(
            \gls*{meshing_parameter}_m,
            (s_m, y_m),
            \varphi + \xi_m,
            [\mathcal{I}^{\gls*{meshing_parameter}_m} u^{\gls*{meshing_parameter}_m}](s_m, y_m)
        )
        = \begin{cases}
            \min \{ S_m^{(1)}, S_m^{(2)} \} & \text{if } n_m > 0 \text{ and } 0 < i < M \\
            \min \{ S_m^{(5)}, S_m^{(2)} \} & \text{if } n_m > 0 \text{ and } i = 0, M \\
            S_m^{(3)} & \text{if } n_m = 0
        \end{cases}
        \label{eqn:truncated_direct_control_consistency_1}
    \end{equation}
    (compare with \cref{eqn:proofs_convergence_direct_control_consistency_1}) where $S_m^{(1)}$, $S_m^{(2)}$, and $S_m^{(3)}$ are defined in \cref{eqn:proofs_convergence_direct_control_consistency_2} and
    \[
        S_m^{(5)}
        = -\sup_{\textcolor{ctrlcolor}{w} \in W^{\gls*{meshing_parameter}_m}} \left \{
            \frac{\varphi_{i_m}^{n_m-1} - \varphi_{i_m}^{n_m}}{\Delta \tau}
            + f_{i_m}^{n_m}(\textcolor{ctrlcolor}{w})
        \right \}.
    \]
    Similarly to \cref{lem:proofs_convergence_direct_control_consistency_1}, we have
    \begin{equation}
        \lim_{m \rightarrow \infty} S_m^{(5)}
        = -\varphi_t(t, x) - \sup_{\textcolor{ctrlcolor}{w} \in W} \left \{
            f(t, x, \textcolor{ctrlcolor}{w})
        \right \}.
        \label{eqn:truncated_direct_control_consistency_2}
    \end{equation}
    
    Suppose that $(s_m, y_m) \rightarrow (t, x)$ where $t < T$ and $x = \pm R$, corresponding to the \artificial{} boundary condition (the other cases can be handled as in the proof of \cref{lem:convergence_direct_control_consistency}).
    Due to symmetry, we will focus only on the case of $x = + R$.
    
    Since $s_m \rightarrow t$, we may assume that $s_m < T$ (or, equivalently, $n_m > 0$) for each $m$.
    In this case, by \cref{eqn:truncated_direct_control_consistency_1},
    \[
        S(
            \gls*{meshing_parameter}_m,
            (s_m, y_m),
            \varphi + \xi_m,
            [\mathcal{I}^{\gls*{meshing_parameter}_m} u^{\gls*{meshing_parameter}_m}](s_m, y_m)
        )
        \geq \min \left(
            \min \{ S_m^{(1)}, S_m^{(2)} \},
            \min \{ S_m^{(5)}, S_m^{(2)} \}
        \right).
    \]
    Taking limit inferiors of both sides of the above inequality and applying 
    \cref{lem:proofs_convergence_direct_control_consistency_1,eqn:proofs_convergence_direct_control_consistency_3,eqn:proofs_convergence_direct_control_consistency_8,eqn:truncated_direct_control_consistency_2},
    \begin{multline*}
        \liminf_{m\rightarrow\infty}
        S(
            \gls*{meshing_parameter}_m,
            (s_m, y_m),
            \varphi + \xi_m,
            [\mathcal{I}^{\gls*{meshing_parameter}_m} u^{\gls*{meshing_parameter}_m}](s_m, y_m)
        ) \\
        \geq \min \biggl(
            \min \biggl\{
                -\varphi_t(t, x) - \sup_{\textcolor{ctrlcolor}{w} \in W} \left\{
                    \frac{1}{2} b(t, x, \textcolor{ctrlcolor}{w})^2 \varphi_{xx}(t, x)
                    + a(t, x, \textcolor{ctrlcolor}{w}) \varphi_x(t, x)
                    + f(t, x, \textcolor{ctrlcolor}{w})
                \right\}, \\
                \varphi(t,x) - \mathcal{M}\overline{u}(t,x)
            \biggr\}, \,
            \min \biggl\{
                -\varphi_t(t, x) - \sup_{\textcolor{ctrlcolor}{w} \in W} \left\{
                    f(t, x, \textcolor{ctrlcolor}{w})
                \right\},
                \varphi(t,x) - \mathcal{M}\overline{u}(t,x)
            \biggr\}
        \biggr) \\
        = F_*(
            (t, x),
            \varphi(t, x),
            D \varphi (t, x),
            D^2 \varphi (t, x),
            \mathcal{M} \overline{u}(t, x),
        )
    \end{multline*}
    where $F$ is defined in \cref{eqn:truncated_hjbqvi_operator}.
    As in the proof of \cref{lem:convergence_direct_control_consistency}, this implies \cref{eqn:proofs_convergence_direct_control_consistency_6_sub}.
    \cref{eqn:proofs_convergence_direct_control_consistency_6_sup} is established similarly.
\end{proof}

Recall that in the case of an unbounded domain, nonlocal consistency of the (untruncated) penalty scheme in \cref{lem:convergence_penalty_consistency} was obtained by a simple modification of the proof of nonlocal consistency of the (untruncated) direct control scheme in \cref{lem:convergence_direct_control_consistency}.
The situation is analogous here.
Namely, to obtain the nonlocal consistency of the truncated penalty scheme, we need only replace the definition of $S_m^{(2)}$ by \cref{eqn:convergence_penalty_consistency_1}.
We summarize below.

\begin{lemma}
    The penalty scheme is nonlocally consistent (with respect to \cref{eqn:truncated_hjbqvi_operator}).
\end{lemma}

Unlike the truncated direct control and truncated penalty schemes, we cannot immediately prove the nonlocal consistency of the truncated explicit-impulse scheme.
In particular, due to the fact that the point $x_i + a_i^n(\textcolor{ctrlcolor}{w}) \Delta \tau$ appearing in \cref{eqn:schemes_lagrangian_derivative} may be outside of the truncated domain $[-R, R]$, the approximation \cref{eqn:schemes_lagrangian_derivative} may introduce overstepping error.
In order to get around this issue, we will have to modify the spatial grid so that the distance between the first and last two grid points vanishes sublinearly:

\begin{lemma}
    Suppose that, instead of \cref{eqn:truncated_spacing}, the points $\{ x_0, \ldots, x_M \}$ are a partition of $[-R, R]$ which satisfies (using Bachmann-Landau notation)
    \begin{align*}
        x_{i+1} - x_i
        & = \Theta(\gls*{meshing_parameter})
        & \text{for } 0 < i < M-1 \phantom{.} \\
        x_{i+1} - x_i
        & = o(\sqrt{\gls*{meshing_parameter}}) \cap \omega(\gls*{meshing_parameter})
        & \text{for } \phantom{0 < i < M-1} \mathllap{i = 0, M-1} .
    \end{align*}
    Then, the explicit-impulse scheme is nonlocally consistent (with respect to \cref{eqn:truncated_hjbqvi_operator}).
\end{lemma}

Before giving the proof, recall that a function is in $\omega(\gls*{meshing_parameter})$ if it vanishes (as $h \rightarrow 0$)  strictly \emph{slower} than $h$ (i.e., sublinearly) while a function is in $o(\sqrt{\gls*{meshing_parameter}})$ if it vanishes strictly \emph{faster} than the function $\sqrt{h}$.
A function is in $o(\sqrt{\gls*{meshing_parameter}}) \cap \omega(\gls*{meshing_parameter})$ if it satisfies both these requirements.

\begin{proof}
    Let $\Omega = [0,T) \times (-R, R)$ and $\varphi \in C^{1,2}(\overline{\Omega})$.
    Let $\varphi_i^n$ and $\varphi^n$ and $(\gls*{meshing_parameter}_m, s_m, y_m, \xi_m)_m$ be defined as in the proof of \cref{lem:convergence_direct_control_consistency}.
    
    Now, suppose $-R < y_m < R$ (or, equivalently, $0 < i_m < M$).
    Let
    \[
        \Vert a \Vert_{\Omega \times W}
        = \sup_{(t,x,\textcolor{ctrlcolor}{w}) \in \Omega \times W}
        a(t, x, \textcolor{ctrlcolor}{w})
        < \infty.
    \]
    Then,
    \begin{align*}
        y_m + a_{i_m}^{n_m}(\textcolor{ctrlcolor}{w}) \Delta \tau
        & \leq x_{M-1} + \Vert a \Vert_{\Omega \times W} \Delta \tau \\
        & = x_M - \left( x_M - x_{M-1} \right)
        + \gls*{const} \Vert a \Vert_{\Omega \times W} \gls*{meshing_parameter}_m \\
        & = R - \left( x_M - x_{M-1} \right)
        + \gls*{const} \gls*{meshing_parameter}_m
    \end{align*}
    so that for $m$ sufficiently large, $y_m + a_{i_m}^{n_m}(\textcolor{ctrlcolor}{w}) \Delta \tau \leq R$ since $x_M - x_{M-1} = \omega(\gls*{meshing_parameter}_m)$.
    By a symmetric argument, we also have that $y_m + a_{i_m}^{n_m}(\textcolor{ctrlcolor}{w}) \Delta \tau \geq -R$ (i.e., there is no overstepping error for $m$ sufficiently large).
    Therefore,
    \begin{multline*}
        \frac{
            \gls*{interp}(
                \varphi^{n_m-1} + \xi_m,
                y_m
                + a_{i_m}^{n_m}(\textcolor{ctrlcolor}{w})
                \Delta \tau
            )
            - \varphi(\tau^{n_m}, y_m) - \xi_m
        }{\Delta \tau} \\
        = \frac{
            \varphi(
                s_m + \Delta \tau,
                y_m
                + a_{i_m}^{n_m}(\textcolor{ctrlcolor}{w})
                \Delta \tau
            )
            - \varphi(s_m, y_m)
            + \xi_m - \xi_m
        }{\Delta \tau} + O\left( \frac{(x_{k_m+1} - x_{k_m})^2}{\Delta \tau} \right) \\
        = \varphi_t(t, x)
        + a(t, x, \textcolor{ctrlcolor}{w}) \varphi_x(t, x)
        + O\left( \frac{(x_{k_m+1} - x_{k_m})^2}{\Delta \tau} + \Delta \tau \right) \\
        = \varphi_t(t, x)
        + a(t, x, \textcolor{ctrlcolor}{w}) \varphi_x(t, x)
        + O\left( \frac{(x_{k_m+1} - x_{k_m})^2}{\Delta \tau} + \Delta \tau \right).
    \end{multline*}
    Now, by the $o(\sqrt{\gls*{meshing_parameter}})$ requirement,
    \[
        O\left( \frac{(x_{k_m+1} - x_{k_m})^2}{\Delta \tau} + \Delta \tau \right)
        = o\left( \frac{(\sqrt{\gls*{meshing_parameter}_m})^2}{\gls*{meshing_parameter}_m} \right)
        + O( \gls*{meshing_parameter}_m )
        \rightarrow 0
        \textspace \text{as} \textspace
        m \rightarrow \infty.
    \]
    The remainder of the proof follows the usual arguments as in \cref{lem:convergence_explicit_impulse_consistency}.
\end{proof}

\begin{remark}
    If $a(\cdot, R, \cdot) \leq 0$, the condition $x_M - x_{M-1} = o(\sqrt{h}) \cap \omega(h)$ is not needed, since no overstepping can occur.
    A symmetric claim is true for the left boundary in the case of $a(\cdot, -R, \cdot) \geq 0$.
\end{remark}

\chapter{Comparison principle proof}
\label{app:comparison}

In this appendix, we prove the comparison principle in \cref{thm:convergence_comparison}.
Throughout this appendix, we assume \cref{enum:convergence_continuity}\textendash \cref{enum:convergence_impulse_control_hausdorff} of \cref{chap:convergence} along with the conditions on the coefficients $a$ and $b$ in the statement of \cref{thm:convergence_comparison}.
Furthermore, we use the notation $\Omega = [0, T) \times \mathbb{R}$, $D^2 \varphi = \varphi_{xx}$, and $D \varphi = (\varphi_t, \varphi_x)$.

Our first order of business is to perform a change of variables to introduce a discount factor $\beta > 0$ into the HJBQVI.
In particular, we define $\mathcal{M}_\beta$ by
\[
    \mathcal{M}_\beta V(t, x)
    = \sup_{\textcolor{ctrlcolor}{z} \in Z(t, x)} \left\{
        V(
            t,
            \Gamma(t, x, \textcolor{ctrlcolor}{z})
        )
        + K_\beta(t, x, \textcolor{ctrlcolor}{z})
    \right\}
\]
and consider the discounted HJBQVI
\begin{align*}
    \min \left\{
        - V_t + \beta V - \sup_{\textcolor{ctrlcolor}{w} \in W} \left\{
            \frac{1}{2} b(\cdot, \textcolor{ctrlcolor}{w})^2 V_{xx}
            {+} a(\cdot, \textcolor{ctrlcolor}{w}) V_x
            {+} f_\beta(\cdot, \textcolor{ctrlcolor}{w})
        \right\},
        V - \mathcal{M}_\beta V
    \right\} & = 0 \text{ on } [0,T) \times \mathbb{R}
    \\
    \min \left \{
        V(T, \cdot) - g_\beta,
        V(T, \cdot) - \mathcal{M}_\beta V(T, \cdot)
    \right \}
    & = 0 \text{ on } \mathbb{R}
\end{align*}
where $K_\beta(t, x) = e^{\beta t} K(t, x)$, $f_\beta(t, x) = e^{\beta t} f(t, x)$, and $g_\beta(x) = e^{\beta T} g(x)$.
To write the discounted HJBQVI in the form \cref{eqn:convergence_pde}, we take $\mathcal{I} = \mathcal{M}_\beta$ and
\begin{multline}
    F((t, x), r, (q, p), A, \wideEll) \\
    = \begin{cases}
        \min \left\{
            - q + \beta r - \sup_{\textcolor{ctrlcolor}{w} \in W} \left\{
                \frac{1}{2} b(x, \textcolor{ctrlcolor}{w})^2 A
                + a(x, \textcolor{ctrlcolor}{w}) p
                + f_\beta(t, x, \textcolor{ctrlcolor}{w})
            \right\}, r - \wideEll
        \right\} & \text{if } t > 0 \\
        \min \left\{
            r - g_\beta(x),
            r - \wideEll
        \right\}
        & \text{if } t = 0.
    \end{cases}
    \label{eqn:comparison_discounted_hjbqvi}
\end{multline}
Note that setting $\beta = 0$ returns to the undiscounted HJBQVI specified by \cref{eqn:convergence_hjbqvi}.
Now, note that if $V$ is a viscosity subsolution of the \emph{undiscounted} HJBQVI, then the function $(t, x) \mapsto e^{\beta t} V(t, x)$ is a viscosity subsolution (resp. supersolution) of the \emph{discounted} HJBQVI.
Therefore, it is sufficient to prove the comparison principle for the discounted HJBQVI.

We now repeat two results regarding the intervention operator, taken from \cite{MR2568293}.

\begin{lemma}[{\cite[Lemma 4.3]{MR2568293}}]
    \label{lem:comparison_intervention_1}
    Let $U, V \in \gls*{bdd}$.
    If $U \geq V$, then $\mathcal{M}_\beta U \geq \mathcal{M}_\beta V$ (i.e., $\mathcal{M}_\beta$ is monotone).
    Moreover, $\mathcal{M}_\beta U_*$ (resp. $\mathcal{M}_\beta U^*$) is lower (resp. upper) semicontinuous and $\mathcal{M}_\beta U_* \leq (\mathcal{M}_\beta U)_*$ (resp. $(\mathcal{M}_\beta U)^* \leq \mathcal{M}_\beta U^*$).
\end{lemma}

\begin{lemma}[{\cite[Lemma 5.5]{MR2568293}}]
    \label{lem:comparison_intervention_2}
    Let $U, V \in \gls*{bdd}$.
    Then,
    \[
        \mathcal{M}_\beta (\lambda U + \left(1 - \lambda\right) V)
        \leq \lambda \mathcal{M}_\beta U + \left(1 - \lambda\right) \mathcal{M}_\beta V
        \textspace \text{for }
        0 \leq \lambda \leq 1
    \]
    (i.e., $\mathcal{M}_\beta$ is convex).
\end{lemma}

The following, also to be used in the proof of comparison, is a well-known result from real analysis:

\begin{proposition}[{\cite[Problem 2.4.17]{MR1751334}}]
    \label{prop:comparison_limsup_product}
    Let $(a_{n})_{n}$ and $(b_{n})_{n}$ be sequences of nonnegative numbers. If $a_{n}$ converges to a positive number $a$, $\limsup_{n\rightarrow\infty}a_{n}b_{n}=a\limsup_{n\rightarrow\infty}b_{n}$.
\end{proposition}

We now give a result that describes the regularity of the ``non-impulse'' part of the discounted HJBQVI.
Below, we use $\preceq$ to denote the positive semidefinite order.

\begin{lemma}
    \label{lem:comparison_non_impulse_continuity}
    Let $H$ be given by
    \[
        H(t, x, p, A)
        = 
        - \sup_{\textcolor{ctrlcolor}{w} \in W} \left \{
            \frac{1}{2} b(x, \textcolor{ctrlcolor}{w})^2 A
            + a(x, \textcolor{ctrlcolor}{w}) p
            + f_\beta(t, x, \textcolor{ctrlcolor}{w})
        \right \}
    \]
    (compare with \cref{eqn:comparison_discounted_hjbqvi}).
    Then, there exists a positive constant $C$ such that for each compact set $D \subset \mathbb{R}$, there exists a modulus of continuity $\varpi$ such that for all $(t, x, X), (s, y, Y) \in [0,T] \times D \times \mathbb{R}$ satisfying
    \[
        \begin{pmatrix}
            X \\
            & -Y
        \end{pmatrix}
        \preceq
        3 \alpha
        \begin{pmatrix}
            1 & -1 \\
            -1 & 1
        \end{pmatrix}
    \]
    and all positive constants $\alpha$ and $\epsilon$,
    \begin{multline*}
        H(
            s, y,
            \alpha \left( x - y \right) - \epsilon y,
            Y - \epsilon
        )
        - H(
            t, x,
            \alpha \left( x - y \right) + \epsilon x,
            X + \epsilon
        ) \\
        \leq
        C\, ( \alpha \, |x-y|^2 + \epsilon\, (1 + |x|^2 + |y|^2))
        + \varpi( \left| (t,x) - (s,y) \right| ).
    \end{multline*}
\end{lemma}

\begin{proof}
    First, note that
    \begin{multline}
        H(
            s, y,
            \alpha \left( x - y \right) - \epsilon y,
            Y - \epsilon
        )
        - H(
            t, x,
            \alpha \left( x - y \right) + \epsilon x,
            X + \epsilon
        )
        \\
        \leq
        \sup_{\textcolor{ctrlcolor}{w} \in W} \left \{
            \begin{gathered}
                \frac{1}{2} \left(
                    b(x, \textcolor{ctrlcolor}{w})^2 X
                    - b(y, \textcolor{ctrlcolor}{w})^2 Y
                \right)
                + \frac{\epsilon}{2} \left(
                    b(x, \textcolor{ctrlcolor}{w})^2
                    + b(y, \textcolor{ctrlcolor}{w})^2
                \right) \\
                + \alpha \left(
                    a(x, \textcolor{ctrlcolor}{w})
                    - a(y, \textcolor{ctrlcolor}{w})
                \right) \left( x - y \right)
                + \epsilon \left(
                    a(x, \textcolor{ctrlcolor}{w}) x
                    + a(y, \textcolor{ctrlcolor}{w}) y
                \right)
            \end{gathered}
        \right \} 
        + \varpi( \left| (t,x) - (s,y) \right| )
        \label{eqn:comparison_hamiltonian_difference}
    \end{multline}
    where we have used the fact that since $[0,T] \times D \times W$ is compact and $f$ is continuous, we can find a modulus of continuity $\varpi$ satisfying
    \[
        f_\beta(t, x, \textcolor{ctrlcolor}{w})
        - f_\beta(s, y, \textcolor{ctrlcolor}{w})
        \leq \varpi( \left| (t,x) - (s,y) \right| )
        \textspace
        \text{for } (t, x), (s, y) \in [0,T] \times D \text{ and } \textcolor{ctrlcolor}{w} \in W.
    \]
    Omitting the dependence on $\textcolor{ctrlcolor}{w}$ for brevity, we can write the argument in the supremum in \cref{eqn:comparison_hamiltonian_difference} as
    \begin{equation}
        \frac{1}{2} \left( b(x)^2 X - b(y)^2 Y \right)
        + \frac{\epsilon}{2} \left( b(x)^2 + b(y)^2 \right)
        + \alpha \left( a(x) - a(y) \right) \left( x - y \right)
        + \epsilon \left( a(x) x + a(y) y \right).
        \label{eqn:comparison_hamiltonian_supremum_argument}
    \end{equation}
    Now, recall that any Lipschitz function $\ell(\cdot)$ satisfies
    \[
        \left| \ell(x) - \ell(y) \right| \leq \gls*{const} \left| x - y \right|
        \textspace \text{and} \textspace
        \left| \ell(x) \right| \leq \gls*{const} \left(1 + \left| x \right| \right).
    \]
    Therefore, since $a$ is Lipschitz, we have
    \[
        \left( a(x) - a(y) \right) \left( x - y \right)
        \leq \left| a(x) - a(y) \right| \left| x - y \right|
        \leq \gls*{const} \left| x - y \right|^2.
    \]
    and
    \begin{multline*}
        a(x) x + a(y) y
        \leq \left| a(x) \right| \left| x \right| + \left| a(y) \right| \left| y \right|
        \leq \gls*{const} \left(
            \left( 1 + \left| x \right| \right) \left | x \right|
            + \left( 1 + \left| y \right| \right) \left | y \right|
        \right) \\
        \leq \gls*{const} \, ( 1 + \left| x \right|^2 + \left| y \right|^2 ).
    \end{multline*}
    Similarly, since $b$ is Lipschitz, we have
    \[
        b(x)^2 + b(y)^2
        \leq \gls*{const} \, ( (1 + \left| x \right| )^2 + (1 + \left| y \right| )^2 )
        \leq \gls*{const} \, ( 1 + \left| x \right|^2 + \left| y \right|^2 ).
    \]
    and
    \begin{align*}
        b(x)^2 X - b(y)^2 Y
        & = \gls*{trace} \left(
            \begin{pmatrix}
                b(x)^2 & b(x)b(y) \\
                b(y)b(x) & b(y)^2
            \end{pmatrix}
            \begin{pmatrix}
                X \\
                & -Y
            \end{pmatrix}
        \right) \\
        & \leq 3 \alpha \gls*{trace} \left(
            \begin{pmatrix}
                b(x)^2 & b(x)b(y) \\
                b(y)b(x) & b(y)^2
            \end{pmatrix}
            \begin{pmatrix}
                1 & -1 \\
                -1 & 1
            \end{pmatrix}
        \right) \\
        & = 3 \alpha \left( b(x) - b(y) \right)^2 \\
        & \leq 3 \alpha \gls*{const} \, ( 1 + \left| x \right|^2 + \left| y \right|^2 ).
    \end{align*}
    Now, the desired result follows by applying the above inequalities to \cref{eqn:comparison_hamiltonian_supremum_argument}.
\end{proof}

The next result establish that the discounted HJBQVI has no boundary layer (recall \cref{exa:convergence_dirichlet}) at time $t = T$.
The proof follows closely \cite[Remark 3.2]{MR2857450}, which establishes the absence of a boundary layer for a Hamilton-Jacobi-Bellman (HJB) equation.

\begin{lemma}
    \label{lem:comparison_no_boundary_layer}
    Let $V$ be a subsolution (resp. supersolution) of the discounted HJBQVI.
    Then, for all $x \in \mathbb{R}$,
    \begin{align*}
        \min \left \{
            V(T, x) - g_\beta(x),
            V(T, x) - \mathcal{M}_\beta V(T, x)
        \right \} & \leq 0 \\
        \text{(resp. } \min \left \{
            V(T, x) - g_\beta(x),
            V(T, x) - \mathcal{M}_\beta V(T, x)
        \right \} & \geq 0 \text{)}.
    \end{align*}
\end{lemma}

\begin{proof}
    Let $V$ be a subsolution.
    Let $\varphi\in C^{1,2}(\overline{\Omega})$ and $(T,x)$ be a maximum point of $V - \varphi$.
    Define $\psi(t, \cdot) = \varphi(t, \cdot) + c(T - t)$ where $c$ is a positive constant, to be chosen later.
    Since $c(T - t)$ is nonnegative, the point $(T, x)$ is also a maximum point of $V - \psi$.
    Since $\psi_t = \varphi_t - c$, it follows that
    \begin{multline*}
        F_*(
            (T, x),
            V(T,x),
            D \psi(T,x),
            D^2 \psi(T,x),
            [\mathcal{I} V](T,x)
        )
        = \min \biggl \{
            - \varphi_t(T, x)
            + \beta V(T, x) \\
            -
            H(T, x, \psi_x(T, x), \psi_{xx}(T, x))
            + c,
            V(T, x) - g_\beta(x),
            V(T, x) - \mathcal{M}_\beta V(T, x)
        \biggr \}
        \leq 0
    \end{multline*}
    where $F$ is given by \cref{eqn:comparison_discounted_hjbqvi}, $\mathcal{I} = \mathcal{M}_\beta$, and $H$ is defined in \cref{lem:comparison_non_impulse_continuity}.
    By picking $c$ large enough,
    \[
        \min \left \{
            V(T, x) - g_\beta(x),
            V(T, x) - \mathcal{M}_\beta V(T, x)
        \right \} \leq 0,
    \]
    as desired.
    The supersolution case is handled symmetrically.
\end{proof}

The next lemma allows us to construct a family of ``strict'' supersolutions $\{ V_\lambda \}_{\lambda \in (0, 1)}$ of the discounted HJBQVI by taking combinations of an ordinary supersolution and a specific constant.
A similar technique is used for a related problem in \cite[Lemma 3.2]{MR1232083}.

\begin{lemma}
    \label{lem:comparison_strict_supersolutions}
    Let $V$ be a supersolution of the discounted HJBQVI.
    Define
    \[
        c
        = \max\{
            (\Vert f_\beta \Vert_{\infty} + 1) / \beta,
            \Vert g_\beta \Vert_{\infty} + 1
        \}
    \] and $\xi = \min \{1, K_0 \}$ where $K_0 = -\sup_{t, x, \textcolor{ctrlcolor}{z}} K(t, x, \textcolor{ctrlcolor}{z}) > 0$.
    Let $V_\lambda = (1 - \lambda) V + \lambda c$ where $0 < \lambda < 1$.
    Then, for all $\varphi \in C^{1,2}(\overline{\Omega})$ and $(t, x) \in \Omega$ such that $V_\lambda - \varphi$ has a local maximum (resp. minimum) at $(t, x)$, we have
    \begin{equation}
        F((t, x), V_\lambda(t, x), D \varphi(t,x), D^2 \varphi(t,x), [\mathcal{I} V_\lambda](t, x))
        \geq \lambda \xi
        \label{eqn:comparison_strict_supersolutions_result_1}
    \end{equation}
    where $F$ is given by \cref{eqn:comparison_discounted_hjbqvi} and $\mathcal{I} = \mathcal{M}_\beta$.
    Moreover, for all $x \in \mathbb{R}$,
    \begin{equation}
        \min \left \{
            V_\lambda(T, x) - g_\beta(x),
            V_\lambda(T, x) - \mathcal{M}_\beta V_\lambda(T, x)
        \right \}
        \geq \lambda \xi.
        \label{eqn:comparison_strict_supersolutions_result_2}
    \end{equation}
\end{lemma}

\begin{proof}
    Let $(t, x, \psi) \in \Omega \times C^{1,2}(\overline{\Omega})$ be such that $V_\lambda(t,x) - \psi(t,x) = 0$ is a local minimum of $V_\lambda - \psi$.
    Letting $\lambda^\prime = 1 - \lambda$ for brevity and $\varphi = (\psi - \lambda c) / \lambda^\prime$, it follows that $(t,x)$ is also a local minimum point of $V - \varphi$ since
    \[
        \lambda^{\prime}\left(V-\varphi\right)
        =\lambda^{\prime}\left(V-\left(\psi-\lambda c\right)/\lambda^{\prime}\right)
        =\lambda^{\prime}V+\lambda c-\psi
        =V_{\lambda}-\psi.
    \]
    Since $V$ is a supersolution, we can find $\textcolor{ctrlcolor}{w} \in W$ such that
    \begin{align}
        0
        & \geq
        \lambda^\prime \left(
            \varphi_t(t, x)
            - \beta V(t, x)
            + \frac{1}{2} b(x, \textcolor{ctrlcolor}{w})^2 \varphi_{xx}(t, x)
            + a(x, \textcolor{ctrlcolor}{w}) \varphi_x(t, x)
            + f_\beta(t, x, \textcolor{ctrlcolor}{w})
        \right)
        \nonumber \\
        & =
            \psi_t(t, x)
            - \beta \left( V_\lambda(t, x) - \lambda c \right)
            + \frac{1}{2} b(x, \textcolor{ctrlcolor}{w})^2 \psi_{xx}(t, x)
            + a(x, \textcolor{ctrlcolor}{w}) \psi_x(t, x)
            + \lambda^\prime f_\beta(t, x, \textcolor{ctrlcolor}{w})
        \nonumber \\
        & =
            \psi_t(t, x)
            - \beta V_\lambda(t, x)
            + \frac{1}{2} b(x, \textcolor{ctrlcolor}{w})^2 \psi_{xx}(t, x)
            + a(x, \textcolor{ctrlcolor}{w}) \psi_x(t, x)
            + f_\beta(t, x, \textcolor{ctrlcolor}{w})
        \nonumber \\
        & \textspace \textspace
            + \lambda \left(
                \beta c
                - f_\beta(t, x, \textcolor{ctrlcolor}{w})
            \right)
        \nonumber \\
        & \geq
            \psi_t(t, x)
            - \beta V_\lambda(t, x)
            + \frac{1}{2} b(x, \textcolor{ctrlcolor}{w})^2 \psi_{xx}(t, x)
            + a(x, \textcolor{ctrlcolor}{w}) \psi_x(t, x)
            + f_\beta(t, x, \textcolor{ctrlcolor}{w})
            + \lambda \xi
        \label{eqn:comparison_strict_supersolutions_1}
    \end{align}
    where the last inequality follows from
    \[
        \beta c
        - f_\beta(t, x, \textcolor{ctrlcolor}{w})
        \geq 
        \beta \left( \Vert f_\beta \Vert_\infty + 1 \right) / \beta
        - \Vert f_\beta \Vert_\infty
        = 1
        \geq \xi.
    \]
    
    Now, note that the convexity of $\mathcal{M}_\beta$ (\cref{lem:comparison_intervention_2}) yields
    \[
        V_\lambda - \mathcal{M}_\beta V_\lambda
        = V_\lambda - \mathcal{M}_\beta ( \lambda^\prime V + \lambda c )
        \geq V_\lambda - \lambda^\prime \mathcal{M}_\beta V - \lambda \mathcal{M}_\beta c
    \]
    where we have, with a slight abuse of notation, used $c$ in the above to denote a constant function (i.e., $c(t, x) = c$).
    Note also that, for each point $(t, x)$,
    \[
        c - \mathcal{M}_\beta c(t, x)
        = c - \sup_{ \textcolor{ctrlcolor}{z} \in Z(t, x) } \left \{
            c + K(t, x, \textcolor{ctrlcolor}{z})
        \right \}
        \geq K_0
        \geq \xi.
    \]
    Now, using once again the fact that $V$ is a supersolution along with \cref{lem:comparison_no_boundary_layer}, we have that $V \geq \mathcal{M}_\beta V$.
    Therefore,
    \begin{equation}
        V_\lambda - \mathcal{M}_\beta V_\lambda
        \geq V_\lambda - \lambda^\prime V - \lambda \mathcal{M}_\beta c
        \geq \lambda \left( c - \mathcal{M}_\beta c \right)
        \geq \lambda \xi.
        \label{eqn:comparison_strict_supersolutions_2}
    \end{equation}
    
    Now, let $x \in \mathbb{R}$ be arbitrary.
    Using once again the fact that $V$ is a supersolution along with \cref{lem:comparison_no_boundary_layer}, we have that $V(T, x) \geq g_\beta(x)$.
    Therefore,
    \begin{equation}
        V_\lambda(T, x) - g_\beta(x)
        = \lambda^\prime \left( V(x) - g_\beta(x) \right) + \lambda \left( c - g_\beta(x) \right)
        \geq \lambda \left( c - g_\beta(x) \right)
        \geq \lambda \xi
        \label{eqn:comparison_strict_supersolutions_3}
    \end{equation}
    where the last inequality follows from
    \[
        c - g_\beta(x)
        \geq \Vert g_\beta \Vert + 1 - \Vert g_\beta \Vert_\infty
        = 1
        \geq \xi.
    \]
    
    The proof is concluded by noting that inequalities \cref{eqn:comparison_strict_supersolutions_1,eqn:comparison_strict_supersolutions_2} imply \cref{eqn:comparison_strict_supersolutions_result_1}
    while inequalities \cref{eqn:comparison_strict_supersolutions_2,eqn:comparison_strict_supersolutions_3} imply \cref{eqn:comparison_strict_supersolutions_result_2}.
\end{proof}

We are finally ready to prove the comparison principle.
The main tool that we use is the Crandall-Ishii lemma, for which we refer to the ``User's Guide to Viscosity Solutions'' \cite[Theorem 3.2]{MR1118699}.
In the proof, we use the notation $\mathscr{P}_\Omega^{2, \pm}$ and $\overline{\mathscr{P}}_\Omega^{2, \pm}$ to denote the parabolic semijets defined in \cite[Section 8]{MR1118699}.

\begin{proof}[Proof of \cref{thm:convergence_comparison}]
    Let $U$ be a bounded subsolution and $V$ be a bounded supersolution of the discounted HJBQVI.
    Let $c$ be given as in \cref{lem:comparison_strict_supersolutions} and define $V_m = (1 - 1 / m) V + c / m$ for all integers $m > 1$.
    Note that
    \[
        \sup_{\Omega} \left \{ U - V_m \right \}
        = \sup_{\Omega} \left \{ U - V + \left( V - c \right) / m \right \}
        \geq \sup_{\Omega} \left \{ U - V \right \} - \left( \Vert V \Vert_\infty + c \right) / m.
    \]
    Therefore, to prove the comparison principle, it is sufficient to show that $U - V_m \leq 0$ along a subsequence of $(V_m)_m$
    We establish it for all $m$.
    
    To that end, fix $m$ and suppose $\delta = \sup_{\Omega} \{ U - V_m \} > 0$.
    Letting $\nu > 0$, we can find $(t^\nu, x^\nu) \in \Omega$ such that $U(t^\nu, x^\nu) - V(t^\nu, x^\nu) \geq \delta - \nu$.
    Let
    \[
        \varphi(t, x, s, y)
        =\frac{\alpha}{2}\left(\left|t-s\right|^{2}+\left|x-y\right|^{2}\right)
        +\frac{\epsilon}{2}\left(\left|x\right|^{2}+\left|y\right|^{2}\right)
    \]
    be a smooth function parameterized by constants $\alpha > 0$ and $0 < \epsilon \leq 1$.
    Further let $\Phi(t,x,s,y) = U(t,x) - V_m(s,y) - \varphi(t,x,s,y)$ and note that
    \begin{align*}
        \sup_{(t,x,s,y) \in \Omega^2} \Phi(t, x, s, y)
        & \geq \sup_{(t, x) \in \Omega} \Phi(t, x, t, x) \\
        & = \sup_{(t, x) \in \Omega} \left \{
            U(t,x) - V_m(t,x) - \epsilon \left| x \right|^2
        \right \} \\
        & \geq U(t^\nu, x^\nu) - V_m(t^\nu, x^\nu) - \epsilon \left| x^\nu \right|^2 \\
        & \geq \delta - \nu - \epsilon \left| x^\nu \right|^2.
    \end{align*}
    We henceforth assume $\nu$ and $\epsilon$ are small enough (e.g., pick $\nu \leq \delta / 4$ and $\epsilon \leq \delta / (4 |x^\nu|^2)$) to ensure that $\delta - \nu - \epsilon | x^\nu |^2$ is positive.
    
    Since $U$ and $V_m$ are bounded (and thus trivially of subquadratic growth), it follows that $\Phi$ admits a maximum at $(t_\alpha, x_\alpha, s_\alpha, y_\alpha) \in \overline{\Omega}^2$ such that
    \begin{equation}
        \Vert U \Vert_\infty + \Vert V_m \Vert_\infty
        \geq U(t_\alpha, x_\alpha) - V_m(s_\alpha, y_\alpha)
        \geq \delta - \nu - \epsilon |x^\nu|^2 + \varphi(t_\alpha, x_\alpha, s_\alpha, y_\alpha)
        \label{eqn:comparison_principle_proof_inequality_0}
    \end{equation}
    Since $-\epsilon | x^\nu |^2 \geq - | x^\nu |^2$, the above inequality implies that
    \[
        \alpha\left(|t_{\alpha}-s_{\alpha}|^{2}+|x_{\alpha}-y_{\alpha}|^{2}\right)
        +\epsilon\left(|x_{\alpha}|^{2}+|y_{\alpha}|^{2}\right)
    \]
    is bounded independently of $\alpha$ and $\epsilon$ (but not of $\nu$ since $| x^\nu |$ may be arbitrarily large).
    
    Now, for fixed $\epsilon$, consider some sequence of increasing $\alpha$, say $(\alpha_{n})_{n}$, such that $\alpha_{n}\rightarrow\infty$.
    To each $\alpha_{n}$ is associated a maximum point $(t_{n},x_{n},s_{n},y_{n}) = (t_{\alpha_{n}},x_{\alpha_{n}},s_{\alpha_{n}},y_{\alpha_{n}})$.
    By the discussion above, $\{(t_{n},x_{n},s_{n},y_{n})\}_{n}$ is contained in a compact set. 
    Therefore, $(\alpha_{n},t_{n},x_{n},s_{n},y_{n})_{n}$ admits a subsequence whose four last components converge to some point $(\hat{t},\hat{x},\hat{s},\hat{y})$.
    With a slight abuse of notation, we relabel this subsequence $(\alpha_{n},t_{n},x_{n},s_{n},y_{n})_{n}$,
    forgetting the original sequence.
    It follows that $\hat{x}=\hat{y}$ since otherwise $|\hat{x}-\hat{y}|>0$ and \cref{prop:comparison_limsup_product} implies
    \[
        \limsup_{n \rightarrow \infty} \left\{
            \alpha_n |x_n - y_n|^2
        \right \}
        = \limsup_{n \rightarrow \infty} \alpha_n \left| \hat{x} - \hat{y} \right|^2
        = \infty,
    \]
    contradicting the boundedness discussed in the previous paragraph.
    The same exact argument yields $\hat{t}=\hat{s}$.
    Moreover, letting $\varphi_{n} = \varphi(t_{n},x_{n},s_{n},y_{n};\alpha_{n})$,
    \begin{align}
        0 \leq \limsup_{n \rightarrow \infty} \varphi_n
        & \leq \limsup_{n \rightarrow \infty} \left\{
            U(t_n, x_n) - V_m(s_n, y_n)
        \right\} - \delta + \nu + \epsilon|x^\nu|^2
        \nonumber \\
        & \leq U(\hat{t}, \hat{x}) - V_m(\hat{t}, \hat{x})
        - \delta + \nu + \epsilon | x^\nu|^2
        \label{eqn:comparison_principle_proof_inequality_1}
    \end{align}
    and hence
    \begin{equation}
        0
        < \delta - \nu - \epsilon | x^\nu |^2
        \leq U(\hat{t}, \hat{x}) - V_m(\hat{t}, \hat{x}).
        \label{eqn:comparison_principle_proof_inequality_2}
    \end{equation}
    
    By \cref{lem:comparison_strict_supersolutions}, $V_m(s_n, y_n) - \mathcal{M}_\beta V_m(s_n, y_n) \geq \xi / m$.
    Suppose, in order to arrive at a contradiction, $(\alpha_n, t_n, x_n, s_n, y_n)_n$ admits a subsequence along which $U(t_n, x_n) - \mathcal{M}_\beta U(t_n, x_n) \leq 0$.
    As usual, we abuse slightly the notation and temporarily refer to this subsequence as $(\alpha_n, t_n, x_n, s_n, y_n)_n$.
    Combining the inequalities $V_m(s_n, y_n) - \mathcal{M}_\beta V_m(s_n, y_n) \geq \xi / m$ and $U(t_n, x_n) - \mathcal{M}_\beta U(t_n, x_n) \leq 0$, we get
    \begin{align*}
        -\xi/m
        & \geq U(t_n, x_n) - V_m(s_n, y_n) - \left(
            \mathcal{M}_\beta U(t_n, x_n) - \mathcal{M}_\beta V_m(s_n, y_n)
        \right) \\
        & \geq \delta - \nu - \epsilon |x^\nu|^2
        + \mathcal{M}_\beta V_m(s_n, y_n) - \mathcal{M}_\beta U(t_n, x_n).
    \end{align*}
    Taking limit inferiors with respect to $n \rightarrow \infty$ of both sides of the above inequality and using the semicontinuity established in \cref{lem:comparison_intervention_1} yields
    \[
        - \xi / m
        \geq \delta - \nu - \epsilon|x^\nu|^2
        + \mathcal{M}_\beta V_m(\hat{t}, \hat{x}) - \mathcal{M}_\beta U(\hat{t}, \hat{x}).
    \]
    It follows, by the upper semicontinuity of $U$, that the supremum in $\mathcal{M}_\beta U(\hat{t}, \hat{x})$ is achieved at some $\textcolor{ctrlcolor}{\hat{z}} \in Z(\hat{t}, \hat{x})$.
    Therefore,
    \[
        - \xi / m
        \geq \delta - \nu - \epsilon| x^\nu |^2
        + V_m(
            \hat{t}, 
            \Gamma(\hat{t}, \hat{x}, \textcolor{ctrlcolor}{\hat{z}})
        )
        - U(
            \hat{t}, 
            \Gamma(\hat{t}, \hat{x}, \textcolor{ctrlcolor}{\hat{z}})
        )
        \geq - \nu - \epsilon| x^\nu |^2.
    \]
    Taking $\nu$ and $\epsilon$ small enough yields a contradiction.
    By virtue of the above, we may assume that our original sequence $(\alpha_n, t_n, x_n, s_n, y_n)_n$ whose four last components converge to $(\hat{t}, \hat{x}, \hat{s}, \hat{y})$ satisfies $U(t_n, x_n) - \mathcal{M}_\beta U(t_n, x_n) > 0$ for all $n$.
    
    Now, suppose $\hat{t} = T$.
    By \cref{lem:comparison_strict_supersolutions}, $V_m(T, \hat{x}) - \mathcal{M}_\beta V_m(T, \hat{x}) \geq \xi / m$ and $V_m(T, \hat{x}) - g_\beta(\hat{x}) \geq 0$.
    If $U(T,\hat{x}) - \mathcal{M}_\beta U(T,\hat{x}) \leq 0$, we arrive at a contradiction by an argument similar to the argument in the previous paragraph.
    It follows that $U(T, \hat{x}) - g_\beta(\hat{x}) \leq 0$ and hence $U(T,\hat{x}) - V_m(T,\hat{x}) \leq 0$, contradicting \cref{eqn:comparison_principle_proof_inequality_2}.
    We conclude that $\hat{t} < T$ so that we may safely assume $(t_n, x_n, s_n, y_n) \in \Omega$ for all $n$.
    
    Define the shorthand derivative notation
    \[
        \partial_t \varphi_n = \frac{\partial \varphi}{\partial t}(t_n, x_n, s_n, y_n; \alpha_n)
    \]
    and $\partial_x \varphi_n$, $\partial_s \varphi_n$, and $\partial_y \varphi_n$ similarly.
    We are now in a position to apply the Crandall-Ishii lemma \cite[Theorem 3.2]{MR1118699}, which implies the existence of $X_n, Y_n \in \mathbb{R}$ satisfying
    \[
        ( \partial_t \varphi_n, \partial_x \varphi_n, X_n + \epsilon )
        \in \overline{\mathscr{P}}_\Omega^{2, +} U(t_n, x_n),
        \textspace
        ( \partial_s \varphi_n, \partial_y \varphi_n, Y_n - \epsilon )
        \in \overline{\mathscr{P}}_\Omega^{2, -} V_m(s_n, y_n),
    \]
    and
    \[
        -3 \alpha_n
        \begin{pmatrix}
            1 \\
            & 1
        \end{pmatrix}
        \preceq
        \begin{pmatrix}
            X_n \\
            & -Y_n
        \end{pmatrix}
        \preceq
        3 \alpha_n
        \begin{pmatrix}
            1 & -1 \\
            -1 & 1
        \end{pmatrix}.
    \]
    Due to our choice of $\varphi$, we get
    \[
        \partial_t \varphi_n
        = \alpha_n \left(t_n -s_n \right)
        = - \partial_s \varphi_n
    \]
    along with
    \[
        \partial_x \varphi_n
        = \alpha_n (x_n - y_n) + \epsilon x_n
        \textspace \text{and} \textspace
        \partial_y \varphi_n
        = - \alpha_n(x_n - y_n) +\epsilon y_n.
    \]
    Therefore, since $U(t_n, x_n) - \mathcal{M}_\beta U(t_n, x_n) > 0$,
    \begin{align}
        - \partial_t \varphi_n
        + \beta U(t_n, x_n)
        + H(
            t_n, x_n,
            \alpha_n(x_n - y_n) + \epsilon x_n,
            X_n + \epsilon
        )
        & \leq 0
        \nonumber \\
        - \partial_t \varphi_n
        + \beta V_m(s_n, y_n)
        + H(
            s_n, y_n,
            \alpha_n(x_n - y_n) - \epsilon y_n,
            Y_n - \epsilon
        )
        & \geq 0
        \label{eqn:comparison_principle_proof_inequality_3}
    \end{align}
    where $H$ is defined in \cref{lem:comparison_non_impulse_continuity}.
    We can combine the inequalities in \cref{eqn:comparison_principle_proof_inequality_3} and apply \cref{lem:comparison_non_impulse_continuity} to get
    \begin{align}
        0
        & \leq
        H(
            s_n, y_n,
            \alpha_n(x_n - y_n) - \epsilon y_n,
            Y_n - \epsilon
        )
        - H(
            t_n, x_n,
            \alpha_n(x_n - y_n) + \epsilon x_n,
            X_n + \epsilon
        )
        \nonumber \\
        & \textspace \textspace
        + \beta \left( V_m(s_n, y_n) - U(t_n, x_n) \right)
        \nonumber \\
        & \leq
        C \, (
            \alpha_n| x_n - y_n|^2
            + \epsilon(1 + | x_n |^2 + | y_n |^2)
        )
        + \varpi( \left| (t_n, x_n) - (s_n, y_n) \right| )
        \nonumber \\
        & \textspace \textspace
        + \beta \left( V_m(s_n, y_n) - U(t_n, x_n) \right)
        \nonumber \\
        & \leq
        2 C \left( \varphi_n + \epsilon \right)
        + \varpi( \left| (t_n, x_n) - (s_n, y_n) \right| )
        + \beta \left( V_m(s_n, y_n) - U(t_n, x_n) \right)
        \label{eqn:comparison_principle_proof_inequality_4}
    \end{align}
    where $\varpi$ is a modulus of continuity.
    Moreover, by \cref{eqn:comparison_principle_proof_inequality_0},
    \begin{equation}
        V_m(s_n, y_n) - U(t_n, x_n)
        \leq - \delta + \nu + \epsilon |x^\nu|^2,
        \label{eqn:comparison_principle_proof_inequality_5}
    \end{equation}
    and by \cref{eqn:comparison_principle_proof_inequality_1},
    \begin{equation}
        \limsup_{n \rightarrow \infty} \varphi_n
        \leq \nu + \epsilon |x^\nu|^2.
        \label{eqn:comparison_principle_proof_inequality_6}
    \end{equation}
    Applying \cref{eqn:comparison_principle_proof_inequality_5} to \cref{eqn:comparison_principle_proof_inequality_4}, taking the limit superior as $n\rightarrow\infty$ of both sides, and finally applying \cref{eqn:comparison_principle_proof_inequality_6} to the resulting expression yields
    \[
        \delta
        \leq \gls*{const} \left(
            \nu + \epsilon + \epsilon |x^\nu|^2
        \right)
    \]
    ($\gls*{const}$ above depends on $\beta$ and $C$).
    Since $\delta$ is positive, picking $\nu$ small enough and taking $\epsilon \rightarrow 0$ in the above inequality yields the desired contradiction.
\end{proof}
\iftoggle{thesis}{%
    \input{appendices/greens}
    \input{appendices/mortality}
}{}



\cleardoublepage 
\phantomsection  

\renewcommand*{\bibname}{References}

\addcontentsline{toc}{chapter}{\textbf{References}}

\begin{refcontext}[sorting=nty]
\printbibliography
\end{refcontext}

\end{document}